\documentclass{report}

\usepackage{amsmath,amscd}
\usepackage{amscd}
\usepackage{amsfonts}
\usepackage{amssymb}

\numberwithin{equation}{section}

\newtheorem{theorem}{Theorem}[section]
\newtheorem{lemma}[theorem]{Lemma}
\newtheorem{proposition}[theorem]{Proposition}

\newtheorem{example}[theorem]{Example}
\newtheorem{definition}[theorem]{Definition}
\newtheorem{corollary}[theorem]{Corollary}
\newtheorem{remark}[theorem]{Remark}

\newenvironment{proof}{\bf Proof. \rm}{$\Box$}

\newcommand{\be}{\begin{equation}}
\newcommand{\ee}{\end{equation}}
\newcommand{\bes}{\begin{equation*}}
\newcommand{\ees}{\end{equation*}}

\newcommand{\cA}{\mathcal{A}}
\newcommand{\cB}{\mathcal{B}}

\newcommand{\cE}{\mathcal{E}}
\newcommand{\cI}{\mathcal{I}}
\newcommand{\cH}{\mathcal{H}}
\newcommand{\cK}{\mathcal{K}}
\newcommand{\cM}{\mathcal{M}}
\newcommand{\cN}{\mathcal{N}}
\newcommand{\cF}{\mathcal{F}}
\newcommand{\cL}{\mathcal{L}}
\newcommand{\cT}{\mathcal{T}}

\newcommand{\cO}{\mathcal{O}}
\newcommand{\cP}{\mathcal{P}}
\newcommand{\cQ}{\mathcal{Q}}
\newcommand{\cR}{\mathcal{R}}
\newcommand{\cS}{\mathcal{S}}
\newcommand{\cX}{\mathcal{X}}
\newcommand{\tT}{\widetilde{T}}
\newcommand{\tV}{\widetilde{V}}
\newcommand{\mfp}{\mathfrak{p}}
\newcommand{\mfq}{\mathfrak{q}}

\newcommand{\mfB}{\mathfrak{B}}
\newcommand{\Rp}{\mathbb{R}_+}
\newcommand{\Rpt}{\mathbb{R}_+^2}
\newcommand{\Rpk}{\mathbb{R}_+^k}
\newcommand{\lel}{\left\langle}
\newcommand{\rir}{\right\rangle}
\newcommand{\diad}{\mathbb{D}^2}
\newcommand{\diadp}{\mathbb{D}_+^2}
\newcommand{\dpp}{\mathbb{D}_{++}}
\newcommand{\diadpp}{\mathbb{D}_{++}^2}
\newcommand{\mb}[1]{\mathbb{#1}}

\begin{document}


\title{\bf Product systems, subproduct systems and dilation theory of completely positive semigroups
\\
}

\author{Orr Moshe Shalit}
\date{}
\maketitle


\pagestyle{empty}
\thispagestyle{empty}
\
\pagebreak[4]

\pagestyle{empty}
\thispagestyle{empty}
\vspace*{5cm}
\begin{center}
\Large
\textbf{
Product systems, subproduct systems\\
and dilation theory of completely\\
positive semigroups\\
}
\end{center}
\vspace{1cm}
\begin{center}
Research Thesis
\end{center}
\vspace{1cm}
\begin{center}
In partial fulfillment of the requirements of the\\
Degree of Doctor of Philosophy\\
\end{center}
\vspace{2cm}
\begin{center}
Orr Moshe Shalit
\end{center}
\vspace{1cm}
\begin{center}
Submitted to the Senate of the\\
Technion - Israel Institute of Technology\\
\end{center}
\vspace{1cm}
\begin{center}
Tamuz 5769 \hspace{\fill} Haifa \hspace{\fill} July 2009\\
\end{center}
\pagebreak[4]

%
\pagestyle{empty}
\thispagestyle{empty}
\
\pagebreak[4]
%
%

\pagestyle{empty}
\thispagestyle{empty}
\vspace*{2cm}

\begin{center}

\Large To Nohar

\end{center}

\pagebreak[4]

\pagestyle{empty}
\thispagestyle{empty}
\
\pagebreak[4]

%
\pagestyle{empty}
\thispagestyle{empty}
%
%


\begin{center}
\Large
\textbf{
Acknowledgments
}
\end{center}

\noindent
The research was carried out under the supervision of
Professor Baruch Solel in the Department of Mathematics.

\noindent I thank Baruch for the excellent professional guidance, for being patient and meticulous, and for his concern for my development and well being.

%

\vspace{1cm}

\noindent I thank Daniel Markiewicz and Eliahu Levy for collaborating with me (separately), and for allowing me to publish our joint results in this thesis. I could not have reached these results without their help.

%

\vspace{1cm}
\noindent I thank the Department of Mathematics - staff and faculty - that made me feel at home from the very first day that I arrived at the Technion eight years ago. A special thanks goes to Hana Kaplan, for her dedication.

\vspace{1cm}
\noindent
I want to thank my friends at the Technion for making the time spent at the Technion a very pleasant one, and especially I want to thank Daniel Reem, Itay Ben-Dan and Sedi Bartz, for their special support.

\vspace{1cm}
\noindent
The generous financial support of the Technion---Israel Institute of
Technology is gratefully acknowledged. In addition, financial support
during parts of my graduate studies was provided by
  \begin{itemize}
    \item the Pollak Fellowship,
    \item the Jacobs-Qualcom Fellowship / Gutwirth Prize,
    \item the Promotion of Excellence in Mathematics Prize given by the Department of Mathematics,
  \end{itemize}
for which I am deeply grateful.

\vspace{1cm}
\noindent I am grateful to my loving and beloved family:

\noindent My wife's parents, Braha of blessed memory and Ilan, supported my family and me during all the years of my studies, in every way.

\noindent My parents, Malka and Meir, have always believed and me and helped me, and are with me wherever I go.

\noindent My beloved children - Anna, Tama, Gev, Em and Shem - accept me the way I am and give me the strength to go in my way.

\noindent Nohar, my dear wife, has given me so much, and received so little. This work is dedicated with great love to her.


\pagebreak[4]
%
%
%
\pagestyle{empty}
\thispagestyle{empty}
\
\pagebreak[4]
%

%
%
%
%
%
%
%
%
%
%

\addtocontents{toc}{\protect\thispagestyle{empty}}
\addtocontents{lot}{\protect\thispagestyle{empty}}
\addtocontents{lof}{\protect\thispagestyle{empty}}

\tableofcontents
\pagestyle{empty}
\thispagestyle{empty}

%
%

\pagestyle{plain}
\chapter*{Abstract}
\pagenumbering{arabic}
\setcounter{page}{1}
\addcontentsline{toc}{chapter}{\protect\numberline{}{Abstract}}
\vspace{-10pt}

Let $\cM$ be a von Neumann algebra acting on a Hilbert space $H$, and let $\{\Theta_s \}_{s \in \cS}$ be a semigroup of contractive, normal and completely positive maps (henceforth: CP maps) on $\cM$ over a unital and abelian semigroup $\cS$, that is:
\bes
\Theta_{s + t}  = \Theta_s \circ \Theta_t ,
\ees
for all $s, t \in \cS$, and $\Theta_0 = {\bf id}$.

A quadruple $(K, u, \cR, \{\alpha_s \}_{s \in \cS})$ consisting of a Hilbert space $K$, an isometric embedding $u: H \rightarrow K$, a von Neumann algebra $\cR$ (of operators acting on $K$) and a semigroup of normal $*$-endomorphisms $\{\alpha_s \}_{s \in \cS}$ is called a \emph{$*$-endomorphic dilation of $\{\Theta_s \}_{s \in \cS}$} (or an \emph{E-dilation}, for short) if $\cM=u^* \cR u$ and if for all $T \in \cR, s \in \cS$ the following equation holds:
\bes
\Theta_s (u^* T u) = u^* \alpha_s (T) u .
\ees

The purpose of this thesis is to develop dilation theory of semigroups of CP maps.
The central results of this thesis are
\begin{enumerate}
\item When $\cS = \Rpt$ and $\Theta_{s,t} = \phi_s \circ \psi_t$, where $\{\phi_t\}_{t \geq 0}$ and $\{\psi_t\}_{t \geq 0}$ are two \emph{strongly commuting} semigroups of \emph{unit preserving} CP maps, an E-dilation exists.
\item When $\cS = \Rpt$, $\cM = B(H)$ and $\Theta_{s,t} = \phi_s \circ \psi_t$, where $\{\phi_t\}_{t \geq 0}$ and $\{\psi_t\}_{t \geq 0}$ are two \emph{strongly commuting} semigroups of (not necessarily unit preserving) CP maps, and E-dilation exists.
\item For general $\cS$ we formulate some necessary and sufficient conditions for the existence of a (minimal) dilation.
\end{enumerate}
Moreover, we show that in general, when $\cS = \mb{N}^k$, an E-dilation does not necessarily exist
(thereby settling an open problem raised by Bhat in 1998).

The main tools used to prove the above results are \emph{product systems}, \emph{subproduct systems},
and their \emph{representations}. With every semigroup of CP maps we associate a subproduct system and a
representation from which the semigroup can be reconstructed. This association reduces the problem of constructing
an E-dilation to a semigroup of CP maps to the problem of constructing an isometric dilation to a subproduct system
representation. Regarding the latter problem, we obtain the following results:
\begin{enumerate}
\item When $\cS = \diadp$, then every completely contractive product system representation has an isometric dilation.
\item When $\cS \subseteq \mb{R}_+^k$ is sufficiently ``nice" (e.g., $\mb{R}_+^k$ or $\mb{N}^k$), then every fully coisometric product system representation has a fully coisometric and isometric dilation.
\item When $\cS \subseteq \mb{R}_+^k$ is \emph{commensurable} (e.g. $\mb{Q}_+^k$ or $\mb{N}^k$), then under suitable assumptions a product system representation has an isometric dilation.
\end{enumerate}
The latter two results are proved by a novel construction that allows to reduce these problems to analogous, well-studied problems in classical dilation theory.

The application of subproduct systems to the solution of the above problems have led us to study subproduct
systems in their own. Special attention is given to subproduct
systems over the semigroup $\mb{N}$, which are used as a framework for studying tuples of operators satisfying
homogeneous polynomial relations, and the operator algebras they generate. As applications we obtain a
noncommutative (projective) Nullstellansatz, a model for tuples of operators subject to homogeneous
polynomial relations, a complete description of all representations of Matsumoto's subshift C$^*$-algebra
when the subshift is of finite type, and a classification of certain operator algebras -- including an interesting
non-selfadjoint generalization of the noncommutative tori.


%
\chapter*{List of Symbols}
\addcontentsline{toc}{chapter}{\protect\numberline{}{List of Symbols}}
%
%
\noindent\textbf{Sets}\\
\hspace*{2cm}
\begin{tabular}{p{1.3in}p{3.8in}}
$ \mb{N} $
        \dotfill &
        the natural numbers $0,1,2,\ldots$\\
$ \mb{Z} $
        \dotfill &
        the integers\\
$ \mb{Q} $
        \dotfill & the rationals\\
$ \mb{R} $
        \dotfill & the reals\\
$ \mb{R}_+ $
        \dotfill & the nonnegative reals $[0,\infty)$\\
$ \Rpt,\,\, \Rpk $
        \dotfill & $\mb{R}_+ \times \mb{R}_+ $, $\,\,\mb{R}_+ \times \mb{R}_+ \times \cdots \times \mb{R}_+ \,$ ($k$ times)\\
$ \cS $
        \dotfill & usually an arbitrary subsemigroup $0 \in \cS \subseteq \Rpk$\\                     $ \mb{C} $
        \dotfill & the complex numbers\\
$ M_n(\mb{C}) $
        \dotfill & the complex $n \times n$ matrices\\
$ \mb{D}_+ $
        \dotfill & the nonnegative dyadic rationals $\{\frac{m}{2^n}:m,n \in \mb{N}\}$\\
$ \diad $
        \dotfill & the pairs of dyadic rationals $\left\{\left(\frac{k}{2^{n}},\frac{m}{2^{n}}\right) : k,m, \in \mathbb{Z}\right\}$ \\

\textbf{BUT ALSO}    \\
$ \mb{D} $
        \dotfill & the open unit disc $\{z \in \mb{C} : |z| < 1\}$\\

\end{tabular}

%
%
\noindent\textbf{Hilbert spaces}\\
\hspace*{2cm}
\begin{tabular}{p{1.3in}p{3.8in}}
$ B(H) $
        \dotfill & the bounded operators on a Hilbert space $H$\\
$ P_H $
        \dotfill &
        the orthogonal projection onto $H$\\
$ K \ominus H $
        \dotfill &
        the orthogonal complement of $H$ in $K$\\
$ \bigvee A $
        \dotfill & the closed linear span of the set $A$ in a Hilbert space\\
$ [A] $
        \dotfill & also the closed linear span of the set $A$ in a Hilbert space\\
\end{tabular}

%
%
\noindent\textbf{Subproduct systems and algebras}\\
\hspace*{2cm}
\begin{tabular}{p{1.3in}p{3.8in}}
$ X, Y $ \dotfill &
         will usually denote subproduct systems\\
$ \cF_X $
        \dotfill &
        the Fock space associated with a subproduct system $X$\\
$ \cK(\cF_X) $
        \dotfill &
        the compact operator on the Fock space\\
$ S^X $
        \dotfill & the shift representation of $X$ on $\cF_X$ \\
$ \cA_X $
        \dotfill & the operator algebra generated by the shift
        $\overline{\textrm{alg}}\{S^X(\xi) : \xi \in X\}$\\
$ \cE_X $
        \dotfill & the operator space $\overline{\textrm{span}}\cA_X \cA_X^*$\\
$ \cT_X $
        \dotfill & the C$^*$-algebra $C^*(\cA_X)$ generated by the shift\\
$ \cO_X $
        \dotfill & the C$^*$-algebra $\cT_X / \cK(\cF_X)$\\
\end{tabular}

\pagebreak[4]

%
%

\chapter*{Introduction}
\addcontentsline{toc}{chapter}{\protect\numberline{}{Introduction}}

\section*{Motivation: dilation theory of CP-semigroups}

Let $H$ be a separable Hilbert space, and let $\cM \subseteq B(H)$ be a von Neumann algebra. A \emph{CP map} on $\cM$ is a contractive, normal and completely positive map. A  \emph{CP-semigroup} on $\cM$ is a family $\Theta = \{\Theta_t\}_{t\geq0}$ of CP maps on $\cM$ satisfying the semigroup property
$$\Theta_{s+t}(a) = \Theta_s (\Theta_t(a)) \,\, ,\,\, s,t\geq 0, a\in \cM ,$$
$$\Theta_{0}(a) = a \,\, , \,\,  a\in B(H) ,$$
and the continuity condition
\be\label{eq:intro_cont}
\lim_{t\rightarrow t_0} \langle \Theta_t(a)h,g\rangle = \langle \Theta_{t_0}(a)h,g\rangle \,\, , \,\, a\in \cM, h,g \in H .
\ee
A CP-semigroup is called an \emph{E-semigroup} if each of
its elements is a $*$-endomorphism.

Let $\Theta$ be a CP-semigroup acting on $\cM$, and let $\alpha$ be an
E-semigroup acting on $\cR$, where $\cR$ is a von Neumann subalgebra of $B(K)$ and $K\supseteq H$. Denote the orthogonal projection of $K$ onto $H$ by $P_H$. We say that $\alpha$ is an \emph{E-dilation} of
$\Theta$ if for all $t \geq 0$ and $B \in \cR$
\begin{equation}\label{eq:dilation}
\Theta_t(P_H B P_H) = P_H \alpha_t (B) P_H .
\end{equation}
In the mid 1990's Bhat (building on earlier works with Parthasarathy) proved the following result, known today as ``Bhat's Theorem" (see \cite{Bhat96} for a discussion of the case where $\cM = B(H)$ and the CP maps are unit preserving, and also the works by SeLegue \cite{SeLegue}, Bhat-Skeide \cite{BS00}, Muhly-Solel \cite{MS02}, and Arveson \cite{Arv03} for different proofs and for the general case):
\begin{theorem}\label{thm:Bhat}
\emph{{\bf (Bhat).}} Every CP-semigroup has a unique minimal
E-dilation.
\end{theorem}

Naturally, there is also a version of the above theorem with $\mb{R}_+$ replaced with $\mb{N}$. In \cite{Bhat98}, Bhat showed that given a pair of commuting CP maps $\Theta$ and $\Phi$ on $B(H)$, there is a Hilbert space $K\supseteq H$ and a pair of commuting normal $*$-endomorphisms $\alpha$ and $\beta$ acting on $B(K)$ such that
\bes
\Theta^m \circ \Phi^n (P_H B P_H) = P_H \alpha^m \circ \beta^n (B) P_H \,\, , \,\, B \in B(K)
\ees
for all $m,n \in \mathbb{N}$. Later on Solel, using a different method (product systems and their representations), proved this result for commuting CP maps on arbitrary von Neumann algebras \cite{S06}.

A natural question is then this: \emph{given {\bf two} commuting CP-semigroups, can one simultaneously dilate them
to a pair of commuting E-semigroups?} This question was the starting point of the research documented here. We still do
not know the answer to it, but that is all for the good. The difficulty of this question has led us into unknown
mathematical territories, where we have made several exciting discoveries.

This thesis contains the mathematical results that we (the author, under the supervision of Professor Baruch Solel) obtained while trying to solve the problem of dilating two
commuting CP-semigroups. A detailed account of the results in this thesis is given in the section
\emph{Detailed overview of the thesis} below. The results in this thesis are taken from several papers that the author
has written: \cite{LevyShalit} (with Eliahu Levy), \cite{MarkiewiczShalit} (with Daniel Markiewicz), \cite{ShalitReprep,ShalitCP0Dil,ShalitWhatType,ShalitCPDil}, and \cite{ShalitSolel} (with Baruch Solel).

\section*{The role of product systems and subproduct systems}

Product systems of Hilbert spaces over $\mb{R}_+$ were introduced by Arveson some 20 years ago in his study of E$_0$-semigroups \cite{Arv89}. In a few imprecise words, a product system of Hilbert spaces over $\mb{R}_+$ is a bundle $\{X(t)\}_{t\in\mb{R}_+}$ of Hilbert spaces such that
\bes
X(s+t) = X(s) \otimes X(t) \,\, , \,\, s,t \in \mb{R}_+.
\ees
We emphasize immediately that Arveson's definition of product systems required also that the bundle carry a certain Borel measurable structure, but we do not deal with these matters here. To every E$_0$-semigroup Arveson associated a product system, and it turns out that the product system associated to an E$_0$-semigroup is a complete cocycle conjugacy invariant of the E$_0$-semigroup.  See \cite{Arv03} for a detailed account.

Later, product systems of Hilbert $C^*$-correspondences over $\mb{R}_+$ appeared (see the survey \cite{SkeideSurvey} by Skeide). In \cite{BS00}, Bhat and Skeide associate with every semigroup of completely positive maps on a C$^*$-algebra $A$ a product system of Hilbert $A$-correspondences. This product system was then used in showing that every semigroup of completely positive maps can be dilated to a semigroup of $*$-endomorphisms. Muhly and Solel introduced a different construction \cite{MS02}: to every CP$_0$-semigroup on a von Neumann algebra $\cM$ they associated a product system of Hilbert W$^*$-correspondences over $\cM'$, the commutant of $\cM$. Again, this product system is then used in constructing an E$_0$-dilation for the original CP-semigroup.

Product systems of C$^*$-correspondences over semigroups other than $\mb{R}_+$ were first studied by Fowler \cite{Fowler2002}, and they have been studied since then by many authors. In \cite{S06}, product systems over $\mb{N}^2$ (and their representations) were studied, and the results were used to prove that every pair of commuting CP maps has a $*$-endomorphic dilation. Product systems over $\Rpt$ were also central to the proof of Theorem \ref{thm:EDil} below, where every pair of strongly commuting CP-semigroups is associated with a product system over $\Rpt$. However, the construction of the product system is one of the hardest parts in that proof. Furthermore, that construction fails when one drops the assumption of strong commutativity.

On the other hand, there is another object that may be naturally associated with a semigroup of CP maps over \emph{any} semigroup: this object is the \emph{subproduct system}, which, when the CP maps act on $B(H)$, is the bundle of Arveson's ``metric operator spaces" (introduced in \cite{Arv97b}). Roughly, a subproduct system of correspondences over a semigroup $\cS$ is a bundle $\{X(s)\}_{s\in\cS}$ of correspondences such that
\bes
X(s+t) \subseteq X(s) \otimes X(t) \,\, , \,\, s,t \in \cS.
\ees
See Definition \ref{def:subproduct_system} below. Of course, a difficult problem cannot be made easy just by introducing a new notion, and the problem of dilating commuting CP-semigroups remains unsolved. However, subproduct systems did already provide us with an efficient general framework for tackling various problems in operator algebras, and in particular it has led us to progress dilation theory of completely positive semigroups. We have our reasons to believe that they will play an important part in noncommutative analysis, in particular in the further development of dilation theory of completely positive semigroups.

\section*{Brief overview of the thesis}

The thesis contains three major achievements, and is divided into three parts. The first achievement, to which the Part \ref{part:I} is devoted, is a partial solution to the problem of dilating two commuting CP-semigroups. Here is one of the main results we obtain\footnote{we also have a result for general von Neumann algebras $\cM$ instead of $B(H)$, under the additional assumption that the semigroups are unit preserving.}:
\begin{theorem}\label{thm:EDil}{\bf (Theorem \ref{thm:main})}
Let $\{\phi_t\}_{t\geq0}$ and $\{\theta_t\}_{t\geq0}$ be two strongly
commuting CP-semigroups on $B(H)$, where $H$ is a separable Hilbert space. Then there is a separable Hilbert space $K$ containing $H$ and an orthogonal projection $P_H:K \rightarrow H$, and two commuting
E-semigroups $\alpha$ and $\beta$ on $B(K)$ such that
$$\phi_s \circ \theta_t (P_H B P_H) = P_H \alpha_s \circ \beta_t (B) P_H$$
for all $s,t \geq 0$ and all $B \in B(K)$.
\end{theorem}

In other words: every two-parameter CP-semigroup that satisfies an additional condition of \emph{strong commutativity} has a two-parameter E-dilation. 

In order to prove Theorem \ref{thm:EDil} we proved several new results in the theory of isometric dilations of product
system representations. To this end, we came up with a very neat trick that reduces (in some cases) the problem of
dilating a product system representation to an isometric one, to the problem of dilating a semigroup of contractions on
a Hilbert space to a semigroup
of isometries on some larger Hilbert space. The latter problem is classical and well studied.

The second major achievement of this thesis is progress in the development of a general dilation theory for semigroups of
CP maps (over an arbitrary semigroup $\cS$). This progress is the subject of Part \ref{part:II}. Building on
earlier works of Arveson and of Muhly and Solel, we associate with each semigroup $\Theta = \{\Theta_s\}_{s\in\cS}$ of
CP maps a
\emph{subproduct system} $X$ and a \emph{subproduct system representation} $T$ from which $\Theta$ can be reconstructed
(Proposition \ref{prop:semigroup} and Theorem \ref{thm:reprep}). We show that the pairing $\Theta \leftrightarrow (X,T)$
is, essentially, bijective. From here we deduce several new results regarding $*$-endomorphic dilation of $k$-tuples
of commuting CP maps.

The third major achievement of this thesis is the progress made in the study of subproduct systems over $\mb{N}$. We show
how these subproduct systems fit naturally in the study of noncommutative algebraic geometry, and we find
that certain universal operator algebras (generated by a row contraction subject to homogeneous
noncommutative polynomial identities) can be realized explicitly in terms of (a representation of) a subproduct
system. We study two classes of examples of such universal algebras, and we show how these algebras can be classified -
up to isometric isomorphism - by the subproduct systems that give rise to them. In the last chapter of the thesis,
we show that Matsumoto's C$^*$-algebras associated to \emph{subshifts} also fit into the framework provided by
subproduct systems, and we study their representation theory using the tools developed here.

\noindent {\bf Note to the reader:}
The three parts of this thesis can be read independently and in any order, with the exceptions that one should read Chapter \ref{chap:subproduct} before Part \ref{part:III}, and also, some results in Part \ref{part:II} depend on Chapter \ref{chap:representing_representations}.

\section*{Detailed overview of the thesis}

Following this introduction there is a chapter containing some preliminary material.

Chapters \ref{chap:cont_and_ext} through \ref{chap:nonunital} constitute the first part of the thesis.
In Chapter \ref{chap:cont_and_ext} we study some fine continuity and extendability properties of semigroups. Section \ref{sec:point_strong} contains a proof that every CP-semigroup is continuous in the \emph{point-strong operator topology}. By that we mean that a CP-semigroup $\Theta$, which is \emph{a priori} only assumed to satisfy the continuity condition (\ref{eq:intro_cont}), actually satisfies the seemingly stronger continuity condition
\bes
\lim_{t \rightarrow t_0}\|\Theta_t(a)h - \Theta_{t_0}(a)h\| = 0,
\ees
for all $a \in \cM, h \in H$.
This result is taken from the paper \cite{MarkiewiczShalit}, written jointly by Daniel Markiewicz and the author of this thesis.

In Section \ref{sec:extension} it is shown that a semigroup of CP maps $\{\Theta_t\}_{t \in \cS}$ (acting on $B(H)$), that is parameterized by
a dense subsemigroup $\cS \subseteq \mb{R}_+$, and that satisfies the continuity condition (\ref{eq:intro_cont}) (with $t,t_0 \in \cS$), can be extended to a CP-semigroup (parameterized by $\mb{R}_+$). This result is taken from \cite{ShalitCPDil}, and the proof uses ideas from \cite{SeLegue}.
Section \ref{sec:continuous_extension} contains a similar result, but for weakly continuous semigroups of contractions on a separable and reflexive Banach space instead of semigroups of CP maps, a result that appeared in a joint paper with Eliahu Levy \cite{LevyShalit}. This section is not used anywhere else in the thesis.

Chapter \ref{chap:representing_representations} contains a construction by which we ``represent" a product system representation as a semigroup of contractions on a Hilbert space. That is, for a given product system $X = \{X(s)\}_{s\in\cS}$ and a representation $T$ of $X$ on a Hilbert space $H$, we construct a Hilbert space $\cH$ and semigroup of contractions $\hat{T} = \{\hat{T}_s\}_{s\in\cS}$ on $\cH$ that encodes all the information about $T$ (Section \ref{sec:rep}). Using this construction we derive several new existence theorems for isometric dilations of product system representations. Most importantly for dilation of commuting CP-semigroups, we prove that every fully-coisometric representation (of a subproduct system over $\mb{R}_+^k$) has an isometric and fully-coisometric dilation (Theorem \ref{thm:isoDilFC}). Several other sufficient conditions for the existence of a \emph{regular} isometric dilation are given, among them a ``double commutation" condition and a norm condition. The results in this chapter are taken from the papers \cite{ShalitCP0Dil} and \cite{ShalitReprep} (the existence of a regular isometric dilation under a norm condition has not appeared before).

In Chapter \ref{chap:strong_commutativity} we define the notion strong commutativity (introduced in \cite{S06}), and we study it. Most of the material in this chapter is borrowed from \cite{ShalitCP0Dil}, Section \ref{sec:GNS} is from \cite{ShalitCPDil}. Section \ref{subsec:QCS} is new.

Chapter \ref{chap:unital} contains our first main result regarding dilation of commuting CP-semigroups, Theorem \ref{thm:scudil}:
\begin{theorem}
Let $\{\phi_t\}_{t\geq0}$ and $\{\theta_t\}_{t\geq0}$ be two strongly
commuting unit preserving CP-semigroups on a von Neumann algebra $\cM \subseteq B(H)$, where $H$ is a separable Hilbert space. Then there is a separable Hilbert space $K$ containing $H$ and an orthogonal projection $P_H:K \rightarrow H$, a von Neumann algebra $\cR \subseteq B(K)$ such that $\cM = P_H \cR P_H \subseteq \cR$, and two commuting
unit preserving E-semigroups $\alpha$ and $\beta$ on $\cR$ such that
$$\phi_s \circ \theta_t (P_H B P_H) = P_H \alpha_s \circ \beta_t (B) P_H$$
for all $s,t \geq 0$ and all $B \in \cR$.
\end{theorem}
This result was published in \cite{ShalitCP0Dil}. The proof of this Theorem depends on Theorem \ref{thm:isoDilFC} together with the constructions of Section \ref{sec:rep_SC}, where we associate with each pair of strongly commuting CP-semigroups a product system (over $\Rpt$) and a product system representation which ``represents" in some way the CP-semigroups.

We close Chapter \ref{chap:unital} with Section \ref{sec:type}, in which we study the cocycle conjugacy classes of the dilating semigroups $\alpha$ and $\beta$ (in the notation of the above theorem). We show that (in the interesting cases) $\alpha$ is cocycle conjugate to $\phi$'s minimal dilation, and that $\beta$ is cocycle conjugate to $\theta$'s minimal dilation.
Combining this result with Daniel Markiewicz' earlier work on \emph{quantized convolution semigroups}, we deduce Corollary \ref{cor:quantized}, which says that for every pair of quantized convolution semigroups, there exists a pair of commuting type I E$_0$-semigroups that dilate them simultaneously. This is a very concrete yet nontrivial class of examples to which our theory applies.

The final Chapter of Part \ref{part:I} is Chapter \ref{chap:nonunital}, where Theorem \ref{thm:EDil} is proved
(Theorem \ref{thm:main}). That is, we prove the existence of an E-dilation to a pair of strongly commuting
CP-semigroups on $B(H)$, where
the difference from the previous chapter is that we drop the assumption of unitality and add an assumption that the
semigroups act on a type I factor. The proof is significantly different from the proof of Theorem \ref{thm:scudil}.
The isometric dilation theorems obtained in Chapter \ref{chap:representing_representations} are not applicable in this
case, and we need to develop an appropriate isometric dilation theorem, Theorem \ref{thm:IsoDilDiad}. The results in
this chapter are taken from \cite{ShalitCPDil}.

Subproduct systems, their representations, and their units, are defined in Chapter \ref{chap:subproduct}. The following two Chapters, \ref{chap:subncp} and \ref{chap:subunitsncp},
can be viewed as a significant reorganization and sharpening of some known results, including several new observations.

Chapter \ref{chap:subncp} establishes the correspondence between $cp$-semigroups and subproduct systems. It is shown that given a subproduct system $X$ of $\cN$- correspondences and a subproduct system representation $R$ of $X$ on $H$, we may construct a $cp$-semigroup $\Theta$ acting on $R_0(\cN)'$. We denote this assignment as $\Theta = \Sigma(X,R)$. Conversely, it is shown that given a  $cp$-semigroup $\Theta$ acting on $\cM$, there is a subproduct system $E$ (called the \emph{Arveson-Stinespring subproduct system} of $\Theta$) of $\cM'$-correspondences and an \emph{injective} representation $T$ of $E$ on $H$ such that $\Theta = \Sigma(E,T)$. Denoting this assignment as $(E,T) = \Xi(\Theta)$, we have that $\Sigma \circ \Xi$ is the identity. In Theorem \ref{thm:essentially_inverse} we show that $\Xi \circ \Sigma$ is also, after restricting to pairs $(X,R)$ with $R$ an injective representation (and up to some ``isomorphism"), the identity. This allows us to deduce (Corollary \ref{cor:onlyproductisorep}) that a subproduct system that is not a product system has no isometric representations.
We introduce the \emph{Fock spaces} associated to a subproduct system and the
canonical \emph{shift representations}. These constructs allow us to show that every subproduct system
is the Arveson-Stinespring subproduct system of some $cp$-semigroup.

In Chapter \ref{chap:subunitsncp} we briefly sketch the picture that is dual to that of Chapter \ref{chap:subncp}.
It is shown that given a subproduct system and a unit of that subproduct system one may construct a $cp$-semigroup, and that every $cp$-semigroup arises this way (via the GNS construction).

In Chapter \ref{chap:aut_dil}, we construct for every subproduct system $X$ and every fully coisometric subproduct system
representation $T$ of $X$ on a Hilbert space, a semigroup $\hat{T}$ of contractions on a Hilbert space that captures ``all the information" about $X$ and $T$. This construction is a modification of the construction introduced in Chapter \ref{chap:representing_representations} for the case where $X$ is a \emph{product} system. It turns out that when $X$ is merely a \emph{sub}product system, it is hard to apply $\hat{T}$ to obtain new results about the representation $T$. However, when $X$ is a true \emph{product} system $\hat{T}$ is very handy, and we use it to prove that every $e_0$-semigroup has a $*$-automorphic dilation (in a certain sense).

Chapter \ref{chap:dil} contains some general remarks regarding dilations and pieces of subproduct system representations,
and the connection between the dilation theories of $cp$-semigroups and of representations of subproduct systems is made.
We define the notion of a \emph{subproduct subsystem} and then we define \emph{dilations} and \emph{pieces} of subproduct
system representations. These notions generalize the notions of \emph{commuting piece} or \emph{$q$-commuting piece} of
\cite{BBD03} and \cite{Dey07}, and also generalizes the definition of \emph{dilation} of a product system
representation of \cite{MS02}. Proposition \ref{prop:dil_rep_dil_CP}, Theorem \ref{thm:edil_repdil} and Corollary \ref{cor:edil_repdil} show that the 1-1 correspondences $\Sigma$ and $\Xi$ between $cp$-semigroups and subproduct systems with representations take isometric dilations of representations to $e$-dilations and vice-versa.
This is used to obtain an example of three commuting, normal and contractive CP maps on $B(H)$ for which there exists no
$e$-dilation acting on a $B(K)$, and no \emph{minimal} dilation acting on any von Neumann algebra (Theorem \ref{thm:parrot}).
This resolves an open problem raised by Bhat in 1998.

In Chapter \ref{chap:dil} we also present a reduction of both the problem of constructing an $e_0$-dilation to a $cp_0$-semigroup, and the problem of constructing an $e$-dilation to a $k$-tuple of commuting CP maps with \emph{small enough norm},
to the problem of embedding a subproduct system in a larger \emph{product system}. We show that not every subproduct system can be embedded in a product system (Proposition \ref{prop:sprdctcntrexample}), and we use this to construct an example of three commuting CP maps $\theta_1, \theta_2, \theta_3$ such that for \emph{any} $\lambda >0$ the three-tuple $\lambda \theta_1, \lambda \theta_2, \lambda \theta_3$ has no $e$-dilation (Theorem \ref{thm:strongparrot}). This unexpected phenomenon has no counterpart in the classical theory of isometric dilations, and provides the first example of a theorem in classical dilation theory that cannot be generalized to the theory of $e$-dilations of $cp$-semigroups.

The developments described in the second part of the thesis indicate that subproduct systems are
worthwhile objects of study, but to make progress we must look at plenty concrete examples.
In the third and final part of the paper we begin studying subproduct systems of Hilbert spaces over the semigroup $\mb{N}$. In Chapter \ref{chap:subproductN} we show that every subproduct system (of W$^*$-correspondences) over $\mb{N}$ is isomorphic to a \emph{standard} subproduct system, that is, it is a subproduct subsystem of the full product system $\{E^{\otimes n}\}_{n\in\mb{N}}$ for some W$^*$-correspondence $E$. Using the results of the previous section, this gives a new proof to the discrete analogue of Bhat's Theorem: \emph{every $cp_0$-semigroup over $\mb{N}$ has an $e_0$-dilation}. Given a subproduct system we define the \emph{standard $X$-shift}, and we show that if $X$ is a subproduct subsystem of $Y$, then the standard $X$-shift is the maximal $X$-piece of the standard $Y$-shift, generalizing and unifying results from \cite{BBD03,Dey07,Popescu06}.

In Chapter \ref{chap:projective} we explain why subproduct systems are convenient for studying noncommutative
projective algebraic geometry. We show that every homogeneous ideal $I$ in the algebra
$\mb{C}\langle x_1, \ldots, x_d\rangle$ of noncommutative polynomials corresponds to a unique subproduct system
$X_I$, and vice-versa. The representations of $X_I$ on a Hilbert space $H$ are precisely determined by the $d$-tuples in the zero set of $I$,
\bes
Z(I) = \{\underline{T} = (T_1, \ldots, T_d) \in B(H)^d : \forall p \in I. p(\underline{T}) = 0\}.
\ees
A noncommutative version of the Nullstellansatz is obtained, stating that
\bes
\{p \in \mb{C}\langle x_1, \ldots, x_d\rangle : \forall \underline{T} \in Z(I). p(\underline{T}) =0\} = I.
\ees

Chapter \ref{chap:universal} starts with a review of a powerful tool, Gelu Popescu's ``Poisson Transform" \cite{Popescu99}.
Using this tool we derive some basic results (obtained previously by Popescu in \cite{Popescu06}) which allow us to
identify the operator algebra $\cA_X$ generated by the $X$-shift as the universal unital operator algebra generated by a row
contraction subject to homogeneous polynomial identities. We then prove that every completely bounded representation of
a subproduct system $X$ is a piece of a scaled inflation of the $X$-shift, and derive a related ``von Neumann inequality".

In Chapter \ref{chap:operator_algebra} we discuss the relationship between a subproduct system $X$ and $\cA_X$, the (non-selfadjoint)
operator algebra generated by the $X$-shift. The main result in this section is Theorem \ref{thm:algebra_iso}, which
says that $X \cong Y$ if and only if $\cA_X$ is completely isometrically isomorphic to $\cA_Y$ by an isomorphism that
preserves the vacuum state. This result is used in Chapter \ref{chap:qcommuting}, where we study the universal
norm closed unital operator algebra generated by a row contraction $(T_1, \ldots, T_d)$ satisfying the relations
\bes
T_i T_j = q_{ij}T_j T_i \,\, , \,\, 1 \leq i < j \leq d,
\ees
where $q = (q_{ij})_{i,j=1}^d \in M_n(\mb{C})$ is a matrix such that $q_{ji} = q_{ij}^{-1}$. These non-selfadjoint analogues of the noncommutative tori, are shown to be classified by their subproduct systems when $q_{ij}\neq 1$ for all $i,j$. In particular, when $d=2$, we obtain the universal algebra for the relation
\bes
T_1 T_2 = q T_2 T_1,
\ees
which we call $\cA_q$. It is shown that $\cA_q$ is isomorphically isomorphic to $\cA_r$ if and only if $q = r$ or $q = r^{-1}$.

In Chapter \ref{chap:A} we describe all standard maximal subproduct systems $X$ with $\dim X(1) = 2$ and
$\dim X(2) = 3$, and classify their algebras up to isometric isomorphisms. It is shown that the algebras are isometrically
isomorphic if and only if the subproduct systems from which they come are isomorphic, and we give an effective
description of when this happens.

In the closing section of this paper, Chapter \ref{chap:subshift}, we find that subproduct systems are
also closely related to \emph{subshifts} and to the \emph{subshift C$^*$-algebras} introduced by K. Matsumoto \cite{Ma}.
We show how every subshift gives rise to a subproduct system, and characterize the subproduct systems that come from
subshifts. We use this connection together with the results of Chapter \ref{chap:universal} to describe all
representations of subshift C$^*$-algebras that come from a subshift of \emph{finite type}
(Theorem \ref{thm:rep_subshift}).

All of the results in Parts \ref{part:II} and \ref{part:III} are taken from \cite{ShalitSolel}.

\chapter*{Preliminaries}
\addcontentsline{toc}{chapter}{\protect\numberline{}{Preliminaries}}

\section*{Who is $\cS$?}
Throughout this thesis, $\cS$ will denote a subsemigroup of $\Rpk$ that contains $0$. In some places we will need to
put some additional restrictions on $\cS$, but in most places one can take arbitrary unital semigroups, not necessarily
abelian. Many of the parts that do not work for arbitrary semigroups, should work for, at least, \emph{Ore semigroups}.
However, we are most interested in the case where $\cS$ is either $\mb{N}^k$ or $\Rpk$.

\section*{$C^*$/$W^*$-correspondences, product systems and their representations}\label{subsec:correspondences}

We begin by defining Hilbert C$^*$-modules. Our references for this material are \cite{cL94} and \cite{Pas}.

\begin{definition}
Let $\cA$ be a $C^*$-algebra. A vector space $E$ (over $\mathbb{C}$) is called an \emph{inner product $A$-module} if it is
a right $\cA$-module with a map $\langle \cdot , \cdot \rangle : E \times E \rightarrow \cA$ such that for all
$x,y,z \in E, \alpha$, $\beta \in \mathbb{C}$, and $a \in \cA$
\begin{align}
\langle x, \alpha y + \beta z \rangle
&= \alpha \langle x, y \rangle + \beta \langle x, z \rangle \,\,, \\
\langle x, ya \rangle & = \langle x, y \rangle a \,\,, \\
\langle x, y \rangle &= \langle y, x \rangle ^* \,\,, \\
\langle x, x \rangle \geq 0 ; \,\, &\mathrm{and} \,\, \langle x, x \rangle = 0 \Rightarrow x = 0 \,\,,
\end{align}
and
\bes
\alpha (xa) = (\alpha x)a = x (\alpha a) .
\ees
\end{definition}

For an inner product $\cA$-module $E$ we can define a norm
\begin{equation}\label{eq:normE}
\|x\|_E := \sqrt{ \| \langle x, x \rangle \|_{\cA}} .
\end{equation}
This norm satisfies an analog of the Cauchy-Schwartz inequality
$$\|\langle x, y \rangle \|_{\cA} \leq \|x \|_E \|y\|_E .$$
It is sometimes convenient to define an $\cA$-valued ``norm" $|\cdot |$ on $E$, by
$$|x| := \langle x, x \rangle^{\frac{1}{2}}.$$
If $E$ is complete with respect to the norm (\ref{eq:normE}) then it is called a \emph{Hilbert $\cA$-module}, or a
\emph{Hilbert $C^*$-module} (over the $C^*$-algebra $\cA$).

The simplest example of a Hilbert $C^*$-module is a Hilbert space, the $C^*$-algebra being $\mathbb{C}$. Any $C^*$-algebra
is a Hilbert $C^*$-module over itself, with the inner product $\langle a, b \rangle := a^* b$.

\begin{definition}
Let $E$ and $F$ be Hilbert $C^*$-modules over a $C^*$-algebra $\cA$. A mapping $T : E \rightarrow F$ is called
\emph{adjointable} if there exists a map $S : F \rightarrow E$ such that for all $x \in E, y \in F$
$$\langle T(x), y \rangle = \langle x, S(y) \rangle .$$
If $T$ and $S$ are as above then we denote $T^*=S$.
\end{definition}

It is a fact that, just like in the Hilbert space setting, every adjointable map is $\cA$-linear and bounded.
On the other hand, and contrary to what one is used to from Hilbert spaces, not every bounded, $\cA$-linear mapping is
adjointable. It is the space of adjointable operators between Hilbert modules, rather than the space of bounded ones,
that serves as the useful generalization of $B(H)$. We denote by $\cL (E,F)$ the space of adjointable operators
between $E$ and $F$, and if $E=F$ this space is denoted by $\cL(E)$. $\cL(E)$ is a $C^*$-algebra (with respect to
the usual operations and the operator norm).

\begin{definition}
Let $E$ be a Hilbert $C^*$-module over $\cA$. $E$ is called \emph{self-dual} if every (norm) continuous $\cA$-linear map
$T : E \rightarrow \cA$ is given by the formula
$$T(x) = \langle u_T , x \rangle \,\, , \,\, x \in E \,\,,$$
for some $u_T \in E$.
\end{definition}

\begin{definition}
Let $E$ be a Hilbert $C^*$-module over the von Neumann algebra $\cM$. $E$ is called a \emph{Hilbert $W^*$-module over
$\cM$} in case $E$ is self-dual.
\end{definition}

\begin{definition}
Let $\cA$ and $\cB$ be $C^*$-algebras. A \emph{Hilbert $C^*$-correspondence from $\cA$ to $\cB$} is a Hilbert $C^*$-module $E$ over
$\cB$, endowed with the structure of a left module over $\cA$ via a $*$-homomorphism $\varphi_E : \cA \rightarrow \cL(E)$.
A \emph{$C^*$-correspondence over $\cA$} is simply a $C^*$-correspondence from $\cA$ to $\cA$. If $\cA$ and $\cB$ are
von Neumann
algebras and $E$ is a $W^*$-module over $\cB$, then $E$ is called a $W^*$-correspondence from $\cA$ to $\cB$ (or
a \emph{$W^*$-correspondence over $\cA$, in case $\cA=\cB$}) if $\varphi_E : \cA \rightarrow \cL(E)$ is normal.
\end{definition}

To explain the statement ``$\varphi_E$ is normal" we note that if $E$ is a Hilbert $W^*$-module over the von Neumann
algebra $\cM$, then $\cL(E)$ is known to be a von Neumann algebra, i.e., $\cL(E)$ is a $C^*$-algebra which is also a
dual space and which can be
represented faithfully on a Hilbert space such that the weak-$*$ topology on
$\cL(E)$ (generated by its pre-dual) coincides with the $\sigma$-weak topology induced by the representation.

We will usually omit $\varphi_E$ when writing the left action: for $a \in \cA$ and $x \in E$ we shall write
$ax$ for $\varphi_E(a)x$.

The following notion of representation of a $C^*$-correspondence was studied extensively in \cite{MS98}, and
turned out to be a very useful tool.
\begin{definition}
Let $E$ be a $C^*$-correspondence over $\cA$, and let $H$ be a
Hilbert space. A pair $(\sigma, T)$ is called a \emph{covariant representation} of $E$ on $H$ if
\begin{enumerate}
    \item $T: E \rightarrow B(H)$ is a completely bounded linear map;
    \item $\sigma : \cA \rightarrow B(H)$ is a nondegenerate $*$-homomorphism; and
    \item $T(xa) = T(x) \sigma(a)$ and $T(a\cdot x) = \sigma(a) T(x)$ for all $x \in E$ and  all $a\in\cA$.
\end{enumerate}
It called a \emph{completely contractive covariant representation} (or, for
brevity, a \emph{c.c. representation}) if $T$ is completely contractive.
If $\cA$ is a $W^*$-algebra and $E$ is $W^*$-correspondence then
we also require that $\sigma$ be normal.
\end{definition}
We shall usually omit the adjective ``covariant" (see the following line for an example).

Given a $C^*$-correspondence $E$ and a representation
$(\sigma,T)$ of $E$ on $H$, one can form the Hilbert space $E
\otimes_\sigma H$, which is defined as the Hausdorff completion of
the algebraic tensor product with respect to the inner product
$$\langle x \otimes h, y \otimes g \rangle = \langle h, \sigma (\langle x,y\rangle) g \rangle .$$
One then defines $\widetilde{T} : E \otimes_\sigma H \rightarrow H$ by
$$\widetilde{T} (x \otimes h) = T(x)h .$$

As in the theory of contractions on a Hilbert space, there are
certain particularly nice representations which deserve to be
singled out.
\begin{definition}
A representation $(\sigma, T)$ is called \emph{isometric} if
for all $x, y \in E$,
\begin{equation*}
T(x)^*T(y) = \sigma(\langle x, y \rangle) .
\end{equation*}
(This is the case if and only if $\widetilde{T}$ is an isometry.) It
is called \emph{fully coisometric} if $\widetilde{T}$ is a coisometry.
\end{definition}
Note: $T$ is completely contractive if and only if $\|\widetilde{T}\|\leq 1$.

Given two Hilbert $C^*$-correspondences $E$ and $F$ over $\cA$,
the \emph{balanced} (or \emph{inner}) tensor product $E
\otimes_{\cA} F$ is a Hilbert $C^*$-correspondence over $\cA$
defined to be the Hausdorff completion of the algebraic tensor
product with respect to the inner product
$$\langle x \otimes y, w \otimes z \rangle = \langle y , \langle x,w\rangle \cdot z \rangle \,\, , \,\,  x,w\in E, y,z\in F .$$
The left and right actions are defined as $a \cdot (x \otimes y) =
(a\cdot x) \otimes y$ and $(x \otimes y)a = x \otimes (ya)$,
respectively, for all $a\in A, x\in E, y\in F$. We shall usually
omit the subscript $\cA$, writing just $E \otimes F$. When working
in the context of $W^*$-correspondences, that is, if $E$ and $F$
are $W$*-correspondences and $\cA$ is a $W^*$-algebra, then $E
\otimes_{\cA} F$ is understood to be the \emph{self-dual
extension} of the above construction (see \cite{Pas}).

Suppose $\cS$ is an abelian cancellative semigroup with identity
$0$ and $p : X \rightarrow \cS$ is a family of
$W^*$-correspondences over $\cA$. Write $X(s)$ for the
correspondence $p^{-1}(s)$ for $s \in \cS$. We say that $X$ is a
(discrete) \emph{product system} over $\cS$ if $X$ is a semigroup,
$p$ is a semigroup homomorphism and, for each $s,t \in \cS
\setminus \{0\}$, the map $X(s) \times X(t) \ni (x,y) \mapsto xy
\in X(s+t)$ extends to an isomorphism $U_{s,t}$ of correspondences
from $X(s) \otimes_{\cA} X(t)$ onto $X(s+t)$. The associativity of
the multiplication means that, for every $s,t,r \in \cS$,
\begin{equation}
U_{s+t,r} \left(U_{s,t} \otimes I_{X(r)} \right) = U_{s,t+r} \left(I_{X(s)} \otimes U_{t,r} \right).
\end{equation}
We also require that $X(0) = \cA$ and that the multiplications
$X(0) \times X(s) \rightarrow X(s)$ and $X(s) \times X(0)
\rightarrow X(s)$ are given by the left and right actions of $\cA$
and $X(s)$.

Usually in this thesis, we will think of a product system as a bundle $X = \{X(s)\}_{s\in\cS}$ of C$^*$-correspondences together with a family $\{U_{s,t}\}_{s,t \in \cS}$ of unitary correspondence maps $U_{s,t} : X(s) \otimes X(t) \rightarrow X(s+t)$.

\begin{definition}
Let $H$ be a Hilbert space, $\cA$ a $W^*$-algebra and $X$ a
product system of Hilbert $\cA$-correspondences over the semigroup
$\cS$. Assume that $T : X \rightarrow B(H)$, and write $T_s$ for
the restriction of $T$ to $X(s)$, $s \in \cS$, and $\sigma$ for
$T_0$. $T$ (or $(\sigma, T)$) is said to be a \emph{covariant representation} of $X$ if
\begin{enumerate}
    \item For each $s \in \cS$, $(\sigma, T_s)$ is a covariant representation of $X(s)$; and
    \item $T(xy) = T(x)T(y)$ for all $x, y \in X$.
\end{enumerate}
$T$ is said to be an isometric/fully coisometric/completely-contractive representation if it is an isometric/fully coisometric/completely-contractive representation on every fiber $X(s)$.
\end{definition}

\section*{Completely positive maps and Stinespring's Theorem}

Let $A$ and $B$ be $C^*$-algebras, and let $\Theta$ be a linear map from
$A$ to $B$. $\Theta$ is said to be \emph{positive} if for all $a \in A$
\bes
a \geq 0 \Rightarrow \Theta(a) \geq 0.
\ees
We denote by $M_n (A)$ the $C^*$-algebra of $n \times
n$ matrices with entries in $A$. For any positive integer $n$, we
define $\Theta_n : M_n (A) \rightarrow M_n (B)$ to be the map that
acts elementwise as $\Theta$, i.e.,
$$\Theta_n (a_{ij}) = (\Theta (a_{ij})) .$$
\begin{definition}
A linear map $\Theta: A \rightarrow B$ is said to be \emph{completely positive}, if for all $n \geq 0$, $\Theta_n$ is a positive map $M_n (A) \rightarrow M_n (B)$.
\end{definition}
To keep terminology short, we shall use the term \emph{CP map} to
mean a completely positive, \emph{contractive} (because
dilation theory is sensible only for contractive maps) and \emph{normal} map.

\noindent {\bf Examples.} The transpose map $M_2 (\mathbb{C})
\rightarrow M_2 (\mathbb{C})$ is positive but not completely
positive. Every $*$-endomorphism is completely positive, and so is
any composition of completely positive maps. If $\Theta: A
\rightarrow B$ is positive, and if either $A$ or $B$ is
commutative, then $\Theta$ is completely positive. The
prototypical example of a completely positive, normal map is the
map \bes T \mapsto \sum S_n T S_n^* , \ees defined for $T$ in some
$B(H)$, where $S_n$ are bounded operators such that $\sum S_n
S_n^*$ converges strongly. In fact, any normal completely positive map on
$B(H)$ is of this form, due to the following theorem and some
$C^*$-representation theory.

\begin{theorem}\label{thm:Stinespring}
\emph{(Stinespring's Dilation Theorem, \cite{Stinespring}.)} Let $A$ be a $C^*$-algebra, let $H$ be a Hilbert space, and let $\Theta : A \rightarrow B(H)$ be a linear map. $\Theta$ is completely positive if and only if there exists a Hilbert space $K$, a bounded linear map $V \in B(H,K)$ and a $*$-representation $\rho$ of $A$ on $K$ such that
\bes
\Theta (a) = V^* \rho (a) V ,
\ees
for all $a \in A$.
\end{theorem}
The triple $(K,V,\rho)$ is called \emph{the Stinespring dilation} of $\Theta$. If
\bes
K = \bigvee\{\rho(a) V h : a \in A, h \in H\},
\ees
then $(K,V,\rho)$ is called a \emph{minimal Stinespring dilation}. The minimal Stinespring dilation is unique up to unitary equivalence. If $\Theta$ is unital and $(K,V,\rho)$ is minimal, then $V$ is an isometry. If $A$ is a von Neumann algebra and $\Theta$ is normal, then $\rho$ is also normal.

The goal of this thesis can be described as a version of Stinespring's Theorem that works for a semigroup $\Theta = \{\Theta_s \}_{s \in \cS}$ of CP maps rather than for a single one.

\section*{$cp$-semigroups and $e$-dilations}

Let $\cS$ be a unital subsemigroup of $\Rp^k$, and let $\cM$ be a
von Neumann algebra acting on a Hilbert space $H$. A \emph{CP map}
is a completely positive, contractive and normal map on $\cM$. A
\emph{CP-semigroup over $\cS$} is a family $\{\Theta_s\}_{s \in
\cS}$ of CP maps on $\cM$ such that:
\begin{enumerate}
    \item For all $s,t \in \cS$
    $$\Theta_s \circ \Theta_t = \Theta_{s + t} \,;$$
    \item $\Theta_0 = {\bf id}_{\cM}$;
    \item For all $h,g\in H$ and all $a\in \cM$, the function
    $$\cS\ni s \mapsto \langle \Theta_s(a) h,g \rangle $$
    is continuous.
\end{enumerate}
A CP-semigroup is called an \emph{E-semigroup} if it consists of $*$-endomorphisms. A CP (E) - semigroup is
called a \emph{CP$_0$ (E$_0$)-semigroup} if all its elements are unital. A semigroup will be called a $cp$($cp_0$/$e$/$e_0$)-semigroup if it satisfies all the conditions that CP(CP$_0$/E/E$_0$)-semigroup does, except (perhaps) the continuity condition 3 above. Let's say it again, a CP-semigroup, with uppercase letters, is a $cp$-semigroup (with lowercase letters) that satisfies the continuity condition 3 above. Don't worry: in the body of the thesis we never discuss CP-semigroups and $cp$-semigroups at the same time, so there will be no confusion.

\begin{definition}\label{def:dilation}
Let $\cM$ be a von Neumann algebra of operators acting on a
Hilbert space $H$, and let $\Theta = \{\Theta_s\}_{s \in \cS}$ be
a $cp$-semigroup over the semigroup $\cS$. An \emph{$e$-dilation of
$\Theta$} is a quadruple $(K,u,\cR,\alpha)$, where $K$ is a
Hilbert space, $u: H \rightarrow K$ is an isometry, $\cR$ is a von
Neumann algebra satisfying $u^* \cR u = \cM$, and $\alpha$ is an
$e$-semigroup over $\cS$ such that \be\label{eq:CPdef_dil1} \Theta_s
(u^* a u) = u^* \alpha_s (a) u \,\, , \,\, a \in \cR \ee for all
$s \in \cS$.

If $(K,u,\cR,\alpha)$ is a dilation of $\Theta$, then $(\cM,
\Theta)$ is called a \emph{compression} of $(K,u,\cR,\alpha)$.
\end{definition}

In the above definition, if $\alpha$ is an E-semigroup (or E$_0$-semigroup), then we call it an \emph{E-dilation} (\emph{E$_0$-dilation}).

Let us review some basic facts regarding $e$-dilations. Most of the
content of the following paragraphs is spelled out in \cite[Chapter 8]{Arv03},
for the case where $\cS = \mathbb{R}_+$.

Note that by putting $a = u x u^*$ in (\ref{eq:CPdef_dil1}), for
any $x \in \cM$, one has \be\label{eq:CPdef_dil2} \Theta_s (x) =
u^* \alpha_s (u x u^*) u \,\, , \,\, x \in \cM. \ee

If one identifies $\cM$ with $u \cM u^*$, $H$ with $u H$, and $p$
with $u u^*$, one may give the following equivalent definition,
which we shall use interchangeably with definition
\ref{def:dilation}: \emph{a triple $(p,\cR,\alpha)$ is called an
\emph{$e$-dilation} of $\Theta$ if $\cR$ is a von Neumann algebra
containing $\cM$, $\alpha$ is an $e$-semigroup on $\cR$ and $p$ is a
projection in $\cR$ such that $\cM = p \cR p$ and
$$\Theta_s (p a p) = p \alpha_s (a) p $$
holds for all $s \in \cS, a \in \cR$.}

With this change of notation, we have
$$p\alpha_s (a) p = \Theta_s (pap) = \Theta_s (p^2 a p^2) = p\alpha_s (pap) p, $$
so, taking $a = 1-p$,
$$0 = p \alpha_s (p (1-p) p) p = p \alpha_s (1-p) p .$$
This means that for all $s \in \cS$, $\alpha_s (1-p) \leq 1 - p$.
A projection with this property is called \emph{coinvariant} (note
that if $\alpha$ is an $e$$_0$-semigroup then $p$ is a coinvariant
projection if and only if it is increasing, i.e., $\alpha_s(p)
\geq p$ for all $s\in \cS$). Equivalently,
\be\label{eq:CPdef_dil3} u u^* \alpha_s(1) = u u^* \alpha_s (u
u^*) \,\, , \,\, s \in \cS . \ee One can also show that
(\ref{eq:CPdef_dil2}) and (\ref{eq:CPdef_dil3}) together imply
(\ref{eq:CPdef_dil1}), and this leads to another equivalent
definition of $e$-dilation of a $cp$-semigroup.

Let $\Theta = \{\Theta_s \}_{s \in \cS}$ be a $cp$-semigroup on a
von Neumann algebra $\cM$, and let $(K,u,\cR,\alpha)$ be an
$e$-dilation of $\Theta$. Assume that $q \in \cR$ is a projection
satisfying $u u^* \leq q$. Assume furthermore that $q$ is
coinvariant. Then one can show that the maps
$$\beta_s : a \mapsto  q \alpha_s (a) q $$
are the elements of a $cp$-semigroup on $q \cR q$.

If the maps $\{\beta_s\}$ happen to be multiplicative on $q\cR q$,
then we  say that $q$ is \emph{multiplicative}. In this case,
$(qK,u, q \cR q, \beta)$ is an $e$-dilation of $\Theta$, which is in
some sense ``smaller" than $(K,u,\cR,\alpha)$.

On the other hand, consider the von Neumann algebra
$$\widetilde{\cR} =
W^*\left(\bigcup_{s \in \cS} \alpha_s (u \cM u^*)\right) .$$ This
algebra is clearly invariant under $\alpha$, and it contains $u
\cM u^*$. Thus, restricting $\alpha$ to $\widetilde{\cR}$, we obtain a
``smaller" dilation. This discussion leads to the following natural
definition.
\begin{definition}\label{def:min_dil}
Let $(K,u,\cR,\alpha)$ be an $e$-dilation of the $cp$-semigroup
$\Theta$.  $(K,u,\cR,\alpha)$ is said to be a \emph{minimal} dilation
if there is no multiplicative, coinvariant projection $1 \neq q
\in \cR$ such that $u u^* \leq q$, and if
\be\label{eq:W^*-generator} \cR = W^*\left(\bigcup_{s \in \cS}
\alpha_s (u \cM u^*)\right) . \ee
\end{definition}

In \cite{Arv03} Arveson defines a minimal dilation slightly
differently:
\begin{definition}\label{def:min_dil_Arv}
Let $(K,u,\cR,\alpha)$ be an $e$-dilation of the $cp$-semigroup
$\Theta$.  $(K,u,\cR,\alpha)$ is said to be a \emph{minimal} dilation
if the central support of $u u^*$ in $\cR$ is $1$, and if
(\ref{eq:W^*-generator}) holds.
\end{definition}
The two definitions have been shown to be equivalent in the case
where $\Theta$ is a $cp_0$-semigroup over $\Rp$ (\cite[Section 8.9]{Arv03}). We now treat the general case.

\begin{proposition}\label{prop:equiv_def_min}
Definition \ref{def:min_dil} holds if \ref{def:min_dil_Arv} does.
\end{proposition}
\begin{proof}
Assume that Definition
\ref{def:min_dil} is violated. If (\ref{eq:W^*-generator}) is
violated, then Definition \ref{def:min_dil_Arv} is, too. So assume
that (\ref{eq:W^*-generator}) holds, and that there is a
multiplicative, coinvariant projection $1 \neq q \in \cR$ such
that $u u^* \leq q$. Denote $\cA = \{\alpha_s (a) : a \in u \cM
u^* , s \in \cS \}$. By an easy generalization of \cite[Proposition
8.9.4]{Arv03}, $q$ commutes with $\alpha_s(q \cR q)$ for
all $s \in \cS$, so $q$ commutes with $\cA$, thus $q$ commutes
with $W^*(\bigcup_{s \in \cS} \alpha_s (u \cM u^*))$. In
other words, $q$ is central in $\cR$.
\end{proof}

Below, we will work with Definition \ref{def:min_dil_Arv}.
Whether or not the two definitions are equivalent remains an interesting open question.
To prove that they are, it would be enough to show that the central support of $p = uu^*$  in
$W^*\left(\bigcup_{s \in \cS} \alpha_s (u \cM u^*)\right)$ is a coinvariant projection, because the central support is clearly a multiplicative projection. This has been done by Arveson in \cite[Proposition 8.3.4]{Arv03}, for the case of a $cp_0$-semigroup over $\cS = \Rp$. Arveson's proof makes use of the order structure of $\Rp$ and cannot be extended to the case $\Rpt$ or other semigroups with which we are concerned in this thesis.


\part{E-dilation of two-parameter CP-semigroups}\label{part:I}


\chapter{Continuity and extendability of CP-semigroups}\label{chap:cont_and_ext}

\section{Continuity of CP-semigroups in the point-strong operator topology}\label{sec:point_strong}

\subsection{Introduction}
Let $\cM$ be a von Neumann algebra acting on a Hilbert space $H$. Let $\phi = \{\phi_t\}_{t\geq 0}$ be a CP-semigroup on $\cM$. Recall that $\phi$ is assumed to satisfy the continuity condition
\begin{equation}\label{eq:WOT}
\lim_{t\to t_0} \langle \phi_t(A)\xi,\eta \rangle = \langle
\phi_{t_0}(A)\xi, \eta \rangle, \quad A\in \cM, \xi, \eta \in H.
\end{equation}
We shall refer to this type of continuity as \emph{continuity in the point-weak
operator topology}. It is equivalent to \emph{continuity in the point-$\sigma$-weak topology}, that is, for all $A \in \cM$ and $\omega \in \cM_*$, $\lim_{t\to t_0}
\omega(\phi_t(A)) = \omega(\phi_{t_0}(A))$,
where $\cM_*$ denotes the predual of $\cM$.

In this section we prove that CP-semigroups satisfy a seemingly stronger continuity condition, namely
\begin{equation}\label{eq:STOP}
\lim_{t\rightarrow t_0} \|\phi_t(A)\xi -\phi_{t_0}(A)\xi \| = 0 ,
\end{equation}
for all $A\in \cM, \xi \in H$. A semigroup satisfying (\ref{eq:STOP}) will be said
to be \emph{continuous in the point-strong operator topology}.
The proposition that CP-semigroups are continuous in the point-strong operator
topology has appeared in the literature earlier, but the proofs that
are available seem to
be incomplete. In the proofs of which we are aware, only
continuity \emph{from the right} in the point-strong operator topology is
established. By this we mean that (\ref{eq:STOP}) holds for limits taken with
$t\searrow t_0$.

We consider the continuity of CP-semigroups in the point-strong
operator topology to be an important property, because it plays a
crucial role in the existence of dilations of CP-semigroups to
E-semigroups. We are aware of five different proofs for
the fact that every CP-semigroup has a dilation to an E-semigroup:
Bhat~\cite{Bhat96}, SeLegue~\cite{SeLegue},
Bhat--Skeide~\cite{BS00}, Muhly--Solel~\cite{MS02} and
Arveson~\cite{Arv03} (some of the authors require some
additional conditions, notably that the CP-semigroup be unital or
that the Hilbert space be separable). In order to show that the
minimal dilation of a CP-semigroup to an E-semigroup is continuous
in the point-weak operator topology, all authors make use of
continuity of the CP-semigroup in the point-strong operator topology, either
implicitly or explicitly. Thus, in effect there was a gap in the proof of what is known as
``Bhat's Theorem" (although Bhat did avoid making this mistake in his early paper \cite{Bhat96} - he did not claim there that a CP-semigroup has an E-dilation - rather, he stated that a CP-semigroup that is continuous in the point-strong operator topology has an E-dilation).

\noindent {\bf Author's note:} After two years of working on my PhD. thesis I came to the conclusion that all of the known proofs to Bhat's Theorem rely on the fact that CP-semigroups are continuous in the point-strong operator topology, and that this continuity of CP-semigroups is not justified in the literature. I notified Daniel Markiewicz of this problem, and he came up with the proof given below, modulo a couple of minor changes for which I am responsible. The following proof will appear in \cite{MarkiewiczShalit}. I am grateful to Daniel for giving me permission to include the proof here.

\subsection{Proof}

Let $\cM$ be a von Neumann algebra acting on a Hilbert space. Let $\phi = \{\phi_t\}_{t\geq 0}$ be a
CP-semigroup acting on $\cM$. We denote by $\cM_*$ the set of
$\sigma$-weakly continuous linear functionals on $\cM$. We shall denote by
$\sigma(\cM_*,\cM)$ the pointwise convergence topology of $\cM_*$ as a subset of the
dual space of $\cM$.

Let $\delta$ be the generator of $\phi$, and let $D(\delta)$ denote its
domain:
\begin{equation*}
D(\delta) = \{ A \in \cM : \exists \delta(A) \in \cM\,  \forall \omega \in
\cM_* \, \lim_{t\to 0+} t^{-1} \omega(\phi_t(A) - A)=\omega(\delta(A))  \}.
\end{equation*}

\begin{lemma}\label{lemma:right-cts} For every $A\in \cM$ and $\xi \in H$, the map
$t \mapsto \phi_t(A)\xi$ is continuous from the right (in norm).
\end{lemma}

The proof of this result can be found in the literature, for example as
Lemma A.1 of \cite{Bhat2001} or Proposition 4.1 item 1 in
\cite{MS02}. For completeness, let us present the argument from
\cite{Bhat2001}.
Let $A \in \cM$, $\xi \in H$ and $t \geq 0$. For all $h>0$, we have, using the
Schwartz inequality for completely positive maps,
\begin{multline*}
\|\phi_{t+h}(A)\xi - \phi_{t}(A)\xi\| = \\
= \langle \phi_{t+h}(A)^* \phi_{t+h}(A)
\xi, \xi \rangle - 2 \Re \langle \phi_{t+h}(A) \xi, \phi_{t}(A) \xi \rangle +
\|\phi_{t}(A)\xi \|^2 \hphantom{XXX} \\
\leq  \langle \phi_h(\phi_{t}(A)^* \phi_{t}(A)) \xi, \xi \rangle - 2 \Re \langle
\phi_{t+h}(A) \xi, \phi_{t}(A) \xi \rangle + \|\phi_{t}(A)\xi \|^2
\xrightarrow[h \to 0]{} 0.
\end{multline*}

We remark, however, that two-sided continuity
does not follow directly from continuity from the right. This is in contrast
with the situation of the classical theory of $C_0$-semigroups on Banach spaces
(see for example \cite{HillePhillips1974}). If $T = \{T_t\}_{t\geq 0}$ is a
semigroup of contractions on a Banach space $X$ such that the maps
\begin{equation*}
t \mapsto T_t (x)
\end{equation*}
are continuous from the right in norm for all $x \in X$, then it is easy to show
that these maps are also continuous from the left in
norm\footnote{
for given $x \in X$, $0\leq t \leq a$, $\| T_{a-t}(x) - T_a(x) \| = \|
T_{a-t} (x - T_t(x)) \| \leq \| x - T_t(x)  \|.$
}.
In fact, when $X$ is separable, for
instance, it can be proved by measurability and integrability techniques that if
the  maps $t \mapsto f(T_t (x))$ are measurable for all $x \in X$ and $f \in
X^*$, then the maps $t \mapsto T_t(x)$ are continuous in norm for $t > 0$. In
the case of CP-semigroups on von Neumann algebras, however, these techniques
seem to require considerable modification. We provide here an alternative
approach to the problem.

Recall that a function $g: [0, 1] \to
H$ is \emph{weakly measurable} if for all $\eta \in H$, the complex-valued
function
$g_{\eta}(t)=\langle\eta, g(t) \rangle$ is measurable. We will say that the function $g$
is \emph{strongly measurable} if there exists a family of countably-valued
functions
(i.e. assuming a set of values which is at most countable)
converging Lebesgue almost everywhere to $g$. (For more details, see Definition
3.5.4, p.~72, and the surrounding discussion in \cite{HillePhillips1974}).

\begin{lemma}\label{lemma:bochner-int}
For all $\xi\in H, A\in M$, the function $f:[0,1] \to H$ given by $f(t)
= \phi_t(A)\xi$ is strongly measurable and Bochner integrable on the interval
$[0,1]$.
\end{lemma}
\begin{proof}  The function $f$ is weakly continuous, since
$\phi$ is continuous in the point-weak operator topology. In particular, it is
weakly measurable.  Furthermore, by Lemma \ref{lemma:right-cts}, the function
$f$ is continuous from the right in norm, hence it is separably valued (i.e.,
its range is contained in a separable subspace of $H$). By Theorem~3.5.3 of
\cite{HillePhillips1974}, the function $f$ is strongly measurable because it
is weakly measurable and separably valued.

Thanks to Theorem 3.7.4, p.~80 of \cite{HillePhillips1974}, in order to show
that $f$ is Bochner integrable it is enough to show that $f$ is strongly
measurable and that
\begin{equation*}
\int_0^1 \|f(t)\|dt < \infty.
\end{equation*}
The latter condition is easy to verify, as $t \mapsto \|f(t)\|$ is a
right-continuous, bounded function on $[0,1]$.
\end{proof}

We thank Michael Skeide for the idea to use the continuity of $f$ from the
right in order to avoid making the assumption that $H$ is separable.

\begin{lemma}\label{lemma:domain-density}
Let $A \in B(H)$ be positive. Then there exists a sequence $A_n \in
D(\delta)$ of positive operators such that $A_n \to A$ in the $\sigma$-strong*
topology.
\end{lemma}
\begin{proof}
Recall that the sequence
\begin{equation*}
A_n = n \int_0^{1/n} \phi_t(A) dt
\end{equation*}
(integral taken in the $\sigma$-weak sense) converges in the
$\sigma$-weak topology to $A$. Furthermore $A_n \in D(\delta)$ and it is a
positive operator for each $n>0$ since $\phi_t$ is a CP map for all $t$. It is
easy to check that $\| A_n \| \leq \|A \|$ for all $n$ since $\phi_t$ is
contractive.

By Lemma~\ref{lemma:bochner-int}, for each $\xi \in H$ the map $t \mapsto \phi_t(A) \xi$ is
Bochner integrable on $[0,1]$, hence in
fact we have
\begin{equation*}
A_n \xi = n \int_0^{1/n} \phi_t(A)\xi dt
\end{equation*}
where the integral is taken in the Bochner sense. The identity holds because
for all $\eta \in H$, $n\in \mb{N}$ we have:
\begin{equation*}
\langle A_n \xi , \eta \rangle = n \int_0^{1/n} \langle \phi_t(A)\xi, \eta \rangle dt = \langle
 n \int_0^{1/n} \phi_t(A)\xi dt, \eta \rangle.
\end{equation*}

We now show that $A_n \to A$ strongly. Let $\xi \in H$ be fixed.
\begin{align*}
\|A\xi -  A_n \xi\| &= \|n \int_0^{1/n}A\xi dt  -
n \int_0^{1/n}\phi_t(A)\xi  dt \| \\
&\leq n \int_0^{1/n} \|A\xi - \phi_t(A)\xi \|dt  .
\end{align*}
The latter goes to zero by continuity from the
right (Lemma~\ref{lemma:right-cts}).
Since $A_n$, $A$ are positive operators, by considering adjoints we obtain that
$A_n \to A$ in the strong* topology. Finally, since the sequence is
bounded, we have convergence in the $\sigma$-strong* topology.
\end{proof}

\begin{lemma}\label{lemma:uniformity}
Let $A_n$ be a bounded sequence of operators in $\cM$ converging to $A$ in the
$\sigma$-strong* topology and let $t_0\geq0$. Then for every sequence $t_k \to
t_0$, $\xi \in H$ and $\epsilon>0$, there exists $N\in \mb{N}$ such that for $n\geq
N$,
\begin{equation*}
\| \phi_{t_k}(A_n -A) \xi  \| < \epsilon, \quad \text{for all }k .
\end{equation*}
\end{lemma}
\begin{proof}
Let $B_n=(A_n-A)^*(A_n-A)$, $\omega_k(X)= \langle \phi_{t_k}(X) \xi, \xi \rangle$ and
$\omega(X)=\langle  \phi_{t_0}(X) \xi, \xi \rangle$. Then we have that
\begin{equation*}
\| \phi_{t_k}(A_n - A)\xi  \|^2 = \langle \phi_{t_k}(A_n -A)^*
\phi_{t_k}(A_n -A) \xi, \xi \rangle \leq \omega_k(B_n)
\end{equation*}
since $\phi_t$ is a CP map for all $t$. Since $\phi$ is a point-$\sigma$-weakly
continuous semigroup, we have that $\{\omega_k\}_{k \in \mb{N}}$ is a
sequence of $\sigma$-weakly continuous linear functionals such that $\omega_k(X)
\to \omega(X)$ for all $X \in \cM$. Furthermore, $B_n$ is a bounded sequence
converging in the $\sigma$-strong* topology to 0. The latter holds because $A_n$
is a bounded sequence converging to $A$ in the $\sigma$-strong* topology and
multiplication is jointly continuous with respect to this topology in bounded
sets (of course * is also continuous). Finally, we obtain the desired
conclusion by applying Lemma III.5.5, p.151 of
\cite{Takesaki2002}, which states the following. Let $\cM$ be a von Neumann
algebra and let $\rho_k$ be a sequence in $\cM_*$ converging to $\rho_0 \in
\cM_*$ in the $\sigma(\cM_*, \cM)$ topology. If a bounded sequence $\{a_n\}_{n \in \mb{N}}$ converges
$\sigma$-strongly* to $0$, then $\lim_{n\to \infty} \rho_k(a_n) = 0$ uniformly
in  $k$.
\end{proof}

\begin{theorem}\label{thm:pointstrongcont}
Let $\phi$ be a CP-semigroup acting on a von Neumann algebra $\cM \subseteq B(H)$. Then for all
$\xi \in H$, $A \in \cM$ and $t_0 \geq 0$,
\begin{equation*}
\lim_{t \rightarrow t_0}\|\phi_t(A)\xi - \phi_{t_0}(A)\xi\| = 0
\end{equation*}
\end{theorem}
\begin{proof}
Let $\epsilon >0$ be given, and let $\{t_k\}_{k \in \mb{N}}$ be a sequence converging to
$t_0$. By Lemma~\ref{lemma:domain-density}, there is a bounded sequence $\{A_n\}_{n\in\mb{N}}$
of operators in $D(\delta)$ such that $A_n
\to A$ in the $\sigma$-strong* topology. By Lemma~\ref{lemma:uniformity}, there
exists
$N \in \mb{N}$ such that for $n \geq N$,
\begin{equation*}
\| \phi_{t_k}(A_n -A) \xi  \| < \frac{\epsilon}{3}, \quad \text{for all }
k\geq 0 .
\end{equation*}
By an application of the Principle of Uniform Boundedness, if $X \in D(\delta)$
there exists $C_X > 0$ such that
\begin{equation*}
\sup_{s >0 } \frac{1}{s} \| \phi_s(X) - X \| \leq C_X <\infty  .
\end{equation*}
Now notice that $A_n \in D(\delta)$ for all $n$, and in particular $\exists
C>0$ such that
\begin{equation*}
\sup_{s >0 } \frac{1}{s} \| \phi_s(A_N) - A_N \| \leq C  .
\end{equation*}
Because $\phi_s$ is a contraction for all $s$, we obtain that for all $k$,
\begin{align*}
 \|  \phi_{t_k}(A_N)\xi - \phi_{t_0}(A_N)\xi \| & \leq \|  \phi_{t_k}(A_N) -
\phi_{t_0}(A_N) \|\; \| \xi \| \\
& \leq \|  \phi_{|t_k-t_0|}(A_N) -A_N \| \; \| \xi
\| \\
& \leq C \| \xi \| \; | t_k - t_0 | .
\end{align*}
In particular, we must have that $\| \phi_{t_k}(A_N)\xi
- \phi_{t_0}(A_N)\xi \| \to 0$ as $k \to \infty$. Thus there is $K\in \mb{N}$
such that for $k\geq K$,
\begin{equation*}
\| \phi_{t_k}(A_N)\xi - \phi_{t_0}(A_N)\xi \| < \frac{\epsilon}{3} .
\end{equation*}
We conclude that for $k \geq K$,
\begin{multline*}
\| \phi_{t_k}(A)\xi - \phi_{t_0}(A)\xi \| \leq
\| \phi_{t_k}(A -A_N) \xi  \| + \\
+ \| \phi_{t_k}(A_N)\xi - \phi_{t_0}(A_N)\xi \|  +
\| \phi_{t_0}(A_N -A) \xi  \| < \epsilon.
\end{multline*}
\end{proof}

\section{Extension of densely parameterized positive semigroups}\label{sec:extension}
Let $\cS$ be a dense subsemigroup of $\Rp$, and let $\phi = \{\phi_s\}_{s\in \cS}$ be a bounded semigroup over $\cS$ acting on $\cM$, such that $\phi_s$ is a normal positive
linear map for all $s\in\cS$. It is the purpose of this section to give sufficient conditions under which the semigroup $\phi$ can be extended to a continuous semigroup (of positive normal maps)
$\hat{\phi} = \{\hat{\phi}_t\}_{t\geq 0}$
such that $\hat{\phi}_s = \phi_s$ for all $s \in \cS$.
The results of this section will be used in Chapter \ref{chap:nonunital} in the construction of an E-dilation to a two-parameter CP-semigroup. The results of this section are taken from \cite{ShalitCPDil}.

We follow the ideas of SeLegue, who in \cite[pages 37-38]{SeLegue} proved that a semigroup of unital, normal
$*$-endomorphisms over the positive dyads, which is known to be weakly continuous only on a strong operator
dense subalgebra of $B(H)$, can be extended continuously to an E$_0$-semigroup (over $\Rp$). The crucial
step in SeLegue's argument was to use a result of Arveson \cite[Proposition 1.6]{Arv97a} regarding convergence
of nets of normal states on $B(H)$. As we are interested in non-unital semigroups as well, we will have to generalize
a bit Arveson's result. The proof, however, remains very much the same.

\noindent{\bf Author's note:} I am grateful to Daniel Markiewicz who, during the Be'er-Sheva/Haifa/Tel-Aviv Operator Algebras/Operator Theory Seminar, told me about SeLegue's thesis.

\begin{lemma}\label{lem:arveson}(Arveson \cite[Proposition 1.6]{Arv97a})
Let $\cM$ be a direct sum of type $I$ factors, let $\{\rho_i\}_{i}$ be a bounded net of positive linear functionals on $\cM$, and let $\omega$ be a positive normal linear functional on $\cM$ such that
$$\lim_{i} \rho_i(x) = w(x)$$
for all compact $x \in \cM$, and also
$$\lim_{i} \rho_i(1) = w(1) .$$
Then
$$\lim_{i} \|\rho_i - \omega\| = 0.$$
\end{lemma}
\begin{proof}
Without loss of generality, we shall assume that all of the functionals involved have norm $\leq 1$.
We shall show that if $i$ is large enough then $ \|\rho_i - \omega\|$ is arbitrarily small. Let $\epsilon > 0$. Since $\omega$ is normal, there exists a finite rank projection $p \in \cM$ such that
\be\label{eq:1a}
\omega(1-p) \leq \epsilon.
\ee
Since $p\cM p$ is a von Neumann algebra on the finite dimensional space $pH$, and since $pxp$ is compact for all $x \in \cM$, we have that
\be\label{eq:2a}
\lim_i \sup_{x \in \cM_1}| \rho_i(pxp) - \omega(pxp)| = 0,
\ee
where $\cM_1$ denotes the unit ball of $\cM$.
Now,
\begin{align*}
\|\rho_i - \omega\| &= \sup_{x\in\cM_1} |\rho_i(x) - \omega(x)| \\
&\leq \sup_{x \in \cM_1}|\rho_i(pxp) - \omega(pxp)| + \sup_{x\in \cM_1}|\rho_i(x-pxp)|
+ \sup_{x\in \cM_1}|\omega(x-pxp)|.
\end{align*}
By (\ref{eq:2a}), the first term in the last expression is smaller than $\epsilon$ when $i$ is large. We now estimate the second and third terms.
Write $x - pxp = (1-p)x + px(1-p)$. Then
$$\sup_{x\in \cM_1}|\mu(x-pxp)| \leq \sup_{x\in \cM_1}|\mu((1-p)x)|+ \sup_{x\in \cM_1}|\mu(px(1-p))|,$$
with $\mu = \rho_i$ or $\mu = \omega$. But by the Schwartz inequality,
$$|\mu((1-p)x)| \leq \mu(1-p)^{1/2}\|x\|$$
and
$$|\mu(px(1-p))| \leq \mu(1-p)^{1/2}\|px\| \leq \mu(1-p)^{1/2}\|x\| .$$
Thus, using (\ref{eq:1a}), we obtain the following estimate for the third term:
$$\sup_{x\in \cM_1}|\omega(x-pxp)| \leq 2 \epsilon^{1/2} .$$
Now, $\rho_i(1) \rightarrow \omega(1)$ and $\rho_i(p) \rightarrow \omega(p)$, thus for all $i$ large enough,
$$\rho_i(1-p) \leq \omega(1-p) + \epsilon \leq 2 \epsilon,$$
so
$$\sup_{x\in \cM_1}|\rho_i(x-pxp)| \leq 4 \epsilon^{1/2} .$$
We conclude that for all $i$ large enough, $\|\rho_i - \omega\| \leq 6 \epsilon^{1/2} + \epsilon$. This completes the proof.
\end{proof}

We now give a somewhat generalized version of SeLegue's Theorem discussed above.
\begin{theorem}\label{thm:SeLegue}
(SeLegue, \cite[pp.~ 37-38]{SeLegue})
Let $\cM \subseteq B(H)$ be direct sum of type $I$ factors. Let $\cS$ be a dense subsemigroup of $\Rp$, and let $\phi = \{\phi_s\}_{s\in \cS}$ be a bounded semigroup over $\cS$ acting on $\cM$, such that $\phi_s$ is a normal positive
linear map for all $s\in\cS$. Assume that for all compact $x \in \cM$ and all $\rho \in \cM_*$,
$$\lim_{\cS \ni s \rightarrow 0} \rho (\phi_s(x)) = \rho(x) \quad \text{\rm and} \quad  \lim_{\cS \ni s \rightarrow 0} \rho (\phi_s(1)) = \rho(1).$$
Then $\phi$ can be extended to a semigroup of normal positive linear maps $\hat{\phi} = \{\hat{\phi}_t\}_{t\geq 0}$
such that $\hat{\phi}_s = \phi_s$ for all $s \in \cS$, satisfying the continuity condition
\be\label{eq:weakstarcont}
\lim_{t \rightarrow t_0} \rho (\hat{\phi}_t (a)) = \rho (\hat{\phi}_{t_0} (a)) \,\,\, \text{\rm for all} \,\,\,  a \in \cM, \rho \in \cM_*.
\ee
Moreover, if $\phi$ consists of
contractions/completely positive maps/unital maps/$*$-endomorphisms then so does $\hat{\phi}$.
\end{theorem}
\begin{proof}
We assume, without loss of generality, that $\|\phi_s\| \leq 1$ for all $s \in \cS$. As $\phi_s$ is normal for all $s$, there is a contraction semigroup $T = \{T_s\}_{s\in\cS}$ acting on the predual
$\cM_*$ of $\cM$ such that $T_s^* = \phi_s$ for all $s\in \cS$.
The assumed continuity of $\phi$ at $0$ implies that for all compact $a\in \cM$ and all $\rho \in \cM_*$,
$$T_s \rho (a) \rightarrow \rho(a) \quad \text{\rm and} \quad T_s \rho (1) \rightarrow \rho(1)$$
as $\cS \ni s \rightarrow 0$. Let $\rho$ be a normal state in $\cM_*$. For all $s \in \cS$,
the functional $T_s \rho = \rho \circ \phi_s$ is a positive and normal, because $\phi$ is positive and normal.
Applying Lemma \ref{lem:arveson} to the net $\{T_s \rho\}_{s\in\cS \setminus \{0\}}$, we
obtain that
$$\lim_{\cS \ni s \rightarrow 0}\|T_s \rho - \rho\| = 0 .$$
Any $\rho \in \cM_*$ is a linear combination of normal states, thus
$\lim_{\cS \ni s \rightarrow 0}\|T_s \rho - \rho\| = 0$ for all $\rho \in \cM_*$,
and it follows that for all $s_0 \in \cS$, $\rho \in \cM_*$,
$$\lim_{\cS \ni s \rightarrow s_0}\|T_s \rho - T_{s_0}\rho\| = 0 .$$
In fact, by standard operator-semigroup methods, for every $\rho$ the map $\cS \ni s \mapsto T_s \rho \in \cM_*$ is uniformly continuous on bounded intervals, thus it may be extended
to a unique uniformly continuous map $\Rp \longrightarrow \cM_*$. For all $t \in \Rp$ this gives rise to a well defined
contraction $\hat{T}_t$, such that for $s \in \cS$, $\hat{T}_s = T_s$. It is easy to see that $\{\hat{T}\}_{t \geq 0}$ is a semigroup.

Now define $\hat{\phi}_t = \hat{T}^*_t$. Then $\hat{\phi} = \{\hat{\phi}_t\}_{t\geq 0}$ is a semigroup of normal linear maps
extending $\phi$ and satisfying the continuity condition (\ref{eq:weakstarcont}). With (\ref{eq:weakstarcont}) in mind, the positivity of $\hat{\phi}$ is obvious, and
the final claim of the theorem is also quite clear,
except, perhaps, the part about $*$-endomorphisms. Assume that $\phi$ is a semigroup of $*$-endomorphisms. For $t \in \Rp$,
$a,b \in \cM$, we have
$$\hat{\phi}_t(ab) = \lim_{\cS \ni s \rightarrow t} \phi_s(ab) = \lim_{\cS \ni s \rightarrow t} \phi_s(a)\phi_s(b) ,$$
where convergence is in the weak operator topology. But $\hat{\phi}$ is a CP-semigroup. Thus, by Theorem \ref{thm:pointstrongcont}, for all $x \in \cM$,
$\phi_s(x)$ converges to $\phi_t(x)$ in the \emph{strong} operator topology as $s \rightarrow t$, so
$$\hat{\phi}_t(ab) =  \lim_{\cS \ni s \rightarrow t} \phi_s(a)\phi_s(b) = \hat{\phi}_t(a)\hat{\phi}_t(b) ,$$
because on bounded sets of $\cM$ multiplication is jointly continuous with respect to the strong operator topology.
\end{proof}

\section{Continuous extension of a densely parameterized semigroup on a Banach space}\label{sec:continuous_extension}

\subsection{Author's preface}
In the course of the construction of an E-dilation to a CP-semigroup, the problem of continuously extending a densely parameterized CP-semigroup arose (see Chapter \ref{chap:nonunital}). This problem was solved in the previous section. However, the problem also makes sense for general densely parameterized semigroups of operators on Banach spaces, and my first attempt to solve the problem was in this general setting. I asked Eliahu Levy about it, and he come up with an idea of a proof. After a few months of working on this idea, we came out with the partial solution that appears below, and that should see light in \cite{LevyShalit}. In Theorem \ref{thm:weak}, we show that every right weakly-continuous semigroup of (linear) operators on a reflexive and separable Banach space, parameterized by a dense subsemigroup of $\mb{R}_+$, can be extended to weakly continuous semigroup parameterized by $\mb{R}_+$. Eliahu Levy has managed to prove the analog of this theorem for arbitrary Banach spaces using a different approach \cite{Levy}.

It is honest to say that the best parts of the proof of Theorem \ref{thm:weak} (the clever use of the ``Ulam-Kuratowski Theorem" and the semigroup argument at the end) were Eliahu's ideas. I am grateful to Eliahu for giving me permission to include the proof here.

Finally, I should also say that the main part of this section, Theorem \ref{thm:weak}, is inappropriate for applications in the theory of CP-semigroups, as there are no infinite dimensional von Neumann algebras with a separable and reflexive predual.

\subsection{Introduction}
Let $X$ be a Banach space, and let $\cS$ be a dense sub-semigroup of
$\Rp$. A \emph{semigroup of operators over $\cS$} is a family $T =
\{T_s\}_{s\in\cS}$ of operators on $X$ such that
$$T_{s+t} = T_s \circ T_t \,\, , \,\, s,t \in \cS .$$
If $0 \in \cS$, we also require that $T_0 = I$.
We shall refer below to such a semigroup
as a \emph{densely parameterized} semigroup.
A \emph{contractive semigroup
on $X$} (over $\cS$) is simply a semigroup $T = \{T_s\}_{s\in\cS}$
such that $T_s$ is a contraction for all $s\in\cS$, that is, $\|T_s\|$ is a linear operator
such that $\|T_s\| \leq 1$.

The word
\emph{operator} shall mean henceforth \emph{linear operator}
unless otherwise stated.
A semigroup $T$ (over $\cS$)
is said to be \emph{weakly continuous} if for all $x\in X, y \in
X^*$, the function $\cS\ni s \mapsto y(T_s(x))$ is a continuous
function. Left and right weak continuity are defined similarly.

The theory of weakly continuous semigroups over $\Rp$ is highly developed \cite{HillePhillips1974, EngelNagel}. Some of
the techniques used for semigroups over $\Rp$ cannot be used when
one considers a semigroup of operators over an arbitrary semigroup
$\cS$. For example, the existence of a generator for the semigroup
can be proved using Bochner integration. But if one has a
semigroup of operators, say, over the rational numbers, then one cannot
integrate. The main result of this section is that if $\cS$ is a
dense sub-semigroup of $\Rp$ and $X$ is a separable, reflexive
Banach space, then every right weakly continuous contractive
semigroup on $X$ over $\cS$ can be extended to a weakly continuous
semigroup over $\Rp$.

A similar but weaker result is also obtained for
semigroups of nonlinear operators. A nonlinear map $A$ is said to be \emph{nonexpansive}
if $A$ is Lipschitz continuous with a Lipschitz constant, denoted by $\|A\|$, that is not larger than $1$.
We shall show that, under the same assumptions on $X$ and $\cS$, every right weakly continuous
semigroup of nonexpansive maps that are continuous with respect to the weak topology on $X$ can be extended to a right weakly continuous semigroup over $\Rp$.

The result that every densely parameterized semigroup of (linear) contractions
that is weakly continuous from the right may be extended to a continuous semigroup
parameterized by $\Rp$ may seem rather expected. Indeed, if the semigroup is assumed to
be \emph{strongly} continuous from the right, that is, if for all $x\in X$ the function $\cS\ni s \mapsto T_s(x)$ is continuous from the right (where $X$ is given the norm topology),
then constructing a continuous extension is straightforward. One is tempted to think that a densely parameterized
semigroup that is continuous with respect to any reasonable topology can always be extended to a continuous
semigroup (with respect to the same topology) over $\Rp$. The following example may serve to illustrate
that things do not always work as expected.
\begin{example}
\emph{
Let $X$ be the closed subspace of $L^\infty(\mathbb{R})$ spanned by the functions $x \mapsto e^{iqx}$ with $q \in \mathbb{Q}$. We endow $X$ with the topology inherited from the weak-$*$ topology on $L^\infty(\mathbb{R})$.
We call this topology the $L^1$ weak topology on $X$. Let $T = \{T_s\}_{s \in \mathbb{Q}}$ be a group of isometric multiplication operators on $X$ given by
\bes
(T_s f)(x) = e^{isx}f(x).
\ees
For every $f \in X$, the function $s \mapsto T_s f$ is continuous with respect to the $L^1$ weak topology,
but $T$ cannot be extended to an $L^1$ weakly continuous semigroup over $\mathbb{R}$. Indeed,
if $T$ was extendable then for $r \notin \mathbb{Q}$ and for all $g \in L^1$ we would have
\bes
\lim_{s \rightarrow r} \int_{\mathbb{R}} g(x)e^{isx}f(x)dx = \int_{\mathbb{R}} g(x)(T_r f)(x)dx ,
\ees
from which it follows (using Lebesgue's Dominated Convergence Theorem) that $T_r$ must be given by multiplication by
$e^{irx}$. However, $X$ is not closed under multiplication by $e^{irx}$.}
\end{example}

\subsection{The main result}
For the remainder of this section, let $X$ be a separable and reflexive
Banach space, with a dual $X^*$, and let $\cS$ be a dense
sub-semigroup of $\Rp = [0,\infty)$.

Recall that the pair $(X,X^*)$ satisfies:
\begin{equation}\label{eq:norm}
\|x\| = \max_{y \in X^*, \|y\|=1} |y(x)|,
\end{equation}
and that $X$ is \emph{weakly sequentially complete}, that is,
it has the property: if $\{x_n\} \subset X$ is such that
for all $y\in X^*$, $\{y(x_n)\}$ converges, then there is $x\in X$ such that $y(x_n) \rightarrow y(x)$ for all
$y \in X^*$.

\begin{theorem}\label{thm:weak}
Let $X$ and $\cS$ be as above. Let $T = \{T_s\}_{s\in\cS}$ be a
contractive semigroup on $X$, such that
\begin{equation}\label{eq:continuity}
\lim_{\cS\ni s\rightarrow 0^+} y(T_s(x)) = y(x) \,\, , \,\, x \in X, y \in X^* .
\end{equation}
Then $T$ can be extended to a weakly continuous contractive
semigroup $\{T_t\}_{t\geq 0}$.
\end{theorem}
\begin{remark}
\emph{If $T$ is a \emph{nonlinear} semigroup of nonexpansive
maps satisfying, in addition to the above conditions, the
assumption that for all $s\in \cS$, $T_s$ is continuous in the
weak topology of $X$, then the following proof will guarantee that
we can extend $T$ to a \emph{right} weakly continuous semigroup over $\Rp$ of
nonexpansive maps. Throughout the proof, we shall indicate where the differences between
linear and nonlinear semigroups  occur.}
\end{remark}
\begin{proof} We shall split the proof into a number of logical steps.

\vspace{0.3cm}
\noindent{\bf 1. Simplifying assumptions.}

We assume that $X$ is a real Banach space,
as the complex case follows easily by considering
the real and imaginary parts of the functionals appearing in the proof.
%
We also assume that $T$ is right continuous at any $s\in \cS$, as
this clearly follows from (\ref{eq:continuity}).

\vspace{0.3cm}
\noindent{\bf 2. Preliminary definitions.}

For any (real valued) continuous function $\varphi$ on $\cS$ we define a function $\varphi^-$ on $\Rp$ by
\bes
\varphi^- (t) = \inf\{h(t) : h\in RUSC(\Rp), \forall s\in\cS . \varphi(s) \leq h(s)\}
\ees
for all $t \in \Rp$, where $RUSC(\Rp)$ denotes the space of right upper-semicontinuous%
\footnote{A \emph{right upper semicontinuous} function is just an
upper semicontinuous function with respect to the half-open topology
generated on $\Rp$ by the half open intervals of the type: $[a,b)$.

Note, that the open sets for the latter topology are characterized as those
whose connected components (with respect to the usual topology) are all intervals
open above, necessarily at most countable in number. Thus any set open for the
half-open topology turns into a usual open set by deleting an at most countable
set of points, hence the half-open interior (resp.\ closure) of any set differs
from the usual one by an at most countable set. One concludes that the properties
of a set being dense, resp.\ Baire, meager, residual, coincide for the half-open
and the usual topologies. In particular, $\Rp$ with the half-open topology is
a Baire space.}
(RUSC) functions on $\Rp$.
Similarly, we define $\varphi_-$ as the supremum of all right lower-semicontinuous functions (RLSC) smaller than
$\varphi$. It is clear that $\varphi_- \leq \varphi \leq \varphi^-$, $\varphi^-$ is RUSC, and $\varphi_-$ is RLSC.

For every fixed $x \in X, y \in X^*$ we can define a right continuous function on $\cS$ by
\be\label{eq:deff}
f(s;x,y) = y(T_s (x)) .
\ee
Our aim is to prove
\be\label{eq:what_to_show}
\left(f(t;x,y)\right)^- = \left(f(t;x,y)\right)_- \,\, , \,\, t\in \Rp, x\in X,y\in X^*.
\ee
Before we do that, we concentrate in the next two steps to show how the theorem follows from this fact.

\vspace{0.3cm}
\noindent{\bf 3. Showing how (\ref{eq:what_to_show}) gives rise to a weakly right-continuous contractive semigroup.}

Define
\be\label{eq:E}
E = \{t\in\Rp : \forall x\in X,y\in X^* . \left(f(t;x,y)\right)^- = \left(f(t;x,y)\right)_- \} .
\ee

Observe that $\cS
\subseteq E$. This follows from the fact that for all fixed $s\in
\cS$, $x\in X$ and $y\in X^*$, the functions
$$(y(T_s(x))+\epsilon)\cdot \chi_{[s,s+\delta)} + \infty\cdot\chi_{[s,s+\delta)^c}$$
and
$$(y(T_s(x))-\epsilon)\cdot \chi_{[s,s+\delta)} - \infty\cdot\chi_{[s,s+\delta)^c}$$
are right continuous, and for some $\delta>0$ they dominate and are dominated by the function  $\cS \ni t \mapsto f(t;x,y)$, respectively. We then have $f(s;x,y)^- - f(s;x,y)_- < 2\epsilon$, for all $\epsilon$, so $s \in E$.

For any $t\in \Rp$, if $\cS\ni s_n \searrow t$, then for all $x\in X, y\in X^*$,
\begin{align*}
& \left(f(t;x,y)\right)_- \leq \liminf \left(f(s_n;x,y)\right)_- = \liminf y(T_{s_n}(x)) \leq\\
& \leq \limsup y(T_{s_n}(x)) = \limsup \left(f(s_n;x,y)\right)^- \leq \left(f(t;x,y)\right)^- ,
\end{align*}
because $\left(f(\cdot;x,y)\right)_-$ is RLSC and $\left(f(\cdot;x,y\right)^-$ is RUSC.
If $t\in E$, then this means that $y(T_{s_n}(x)) \longrightarrow \left(f(t;x,y)\right)^-$ regardless of the choice of
$\{s_n\}$ and for all $x\in X$ and $y\in X^*$. Thus for $t\in E \setminus \cS$ we may define $T_t x$ to be the
weak limit $\lim_n T_{s_n} x$, where $\{s_n\}$ is any sequence in $\cS$ converging to $t$ from the right (this is where we
use the fact that $X$ is weakly sequentially complete). Note that
for $t\in E \cap \cS$ this weak limit would turn out to be the same $T_t$ that we started with. We will use this
below before we shall actually prove that $E = \Rp$.

Now if $E = \Rp$, then we get a family $\{T_t\}_{t\geq 0}$ of
linear operators on $X$. Equation (\ref{eq:norm}) implies that the
operators in this family are contractions. $\{T_t\}_{t\geq 0}$ is
weakly continuous from the right, since $y(T_t(x)) =
\left(f(t;x,y) \right)^- = \left(f(t;x,y) \right)_-$, a
right-continuous function in $t$. Also, in either case $0\in \cS$
or $0\notin\cS$, $T_0 = I$ by assumption.

To show that $\{T_t\}_{t\geq 0}$ is a semigroup, we first show
that \be\label{eq:halfSG} T_{s+t} = T_s \circ T_t \,\, , \,\, s
\in \cS, t \in \Rp. \ee Let $\cS\ni t_n \searrow t$, and fix $x\in
X, y\in X^*$. On one hand \be\label{eq:T converges} y(T_s\circ
T_{t_n}(x)) = y(T_{s+t_n}(x)) \longrightarrow y(T_{s+t}(x)) . \ee
On the other hand, \bes y(T_s\circ T_{t_n}(x)) =
y(T_s(T_{t_n}(x))) \longrightarrow y(T_s(T_t(x))) = y(T_s \circ
T_t (x)) , \ees because $T_{t_n}(x)$ converges weakly to
$T_{t}(x)$, and $T_s$ is continuous in the weak topology (as any
bounded operator. This is the main reason why in the nonlinear
case we assume that $T_s$ is weakly continuous, for all $s\in
\cS$). Together with (\ref{eq:T converges}) and (\ref{eq:norm}),
this means that (\ref{eq:halfSG}) holds.

Now let $s,t\in \Rp$, and let $\cS\ni s_n \searrow s$. On one
hand, from equation (\ref{eq:halfSG}) and the weak right
continuity of $\{T_t\}_{t\geq 0}$, it follows that for all $x\in
X$ \bes y(T_{s_n} \circ T_{t} (x)) = y(T_{s_n + t}(x))
\longrightarrow y(T_{s+t} (x)). \ees On the other hand, for all
$x$ and $y$, \bes y(T_{s_n}(T_t(x))) \rightarrow y(T_{s}(T_t(x)))
, \ees where we used again the weak right-continuity of
$\{T_t\}_{t\geq 0}$. Thus \bes T_{s+t} = T_s \circ T_t \,\, , \,\,
s, t \in \Rp. \ees

\vspace{0.3cm} \noindent{\bf 4. $\{T_t\}_{t\geq 0}$ (once defined)
is two sided weakly continuous.}

From the previous step, it follows that the semigroup $T$ extends
to a right weakly continuous contractive semigroup which we shall
also call $T$. It follows from classical results that $T$ is
weakly (and, in fact, strongly) continuous from the left as well
(see the corollary on page 306, \cite{HillePhillips1974}. This step does not hold
for the nonlinear case).

\vspace{0.3cm} \noindent{\bf 5. Two Lemmas.}

In this step we prove two technical lemmas, in order to make the
main parts of the proof smoother.

\begin{lemma}\label{lem:subspace}
For every $t\in\Rp$, the set
$$A_t = \{(x,y) \in X \times X^* :  \left(f(t;x,y)\right)^- = \left(f(t;x,y)\right)_- \}$$
is closed in $X \times X^*$.
\end{lemma}
\begin{proof}
Let $(x,y) \in \overline{A_t}$. We shall show that
$f(t;x,y)^- - f(t;x,y)_- \leq \epsilon$ for every $\epsilon > 0$. Indeed, given $\epsilon \in (0,1)$, let $(w,z) \in A_t$
such that  $\|w-x\|, \|z-y\| < \frac{\epsilon}{6 (M + N)}$, where $M:= \max\{\|x\|,\|y\|\} + 1$, and
$N$ is a bound for $\|T_s(x)\|$ for all $s \in \cS\cap[0,t+1]$. The existence of such a bound $N$ follows from
(\ref{eq:continuity}), together with the semigroup property and the Principle of Uniform Boundedness
(of course, if $T$ is a semigroup of linear operators, $N$ can be taken to be $\|x\|$).

Because $(w,z)\in A_t$, there is a $\delta\in(0,1)$ such that for all $s\in [t,t+\delta)\cap \cS$,
$$f(t;w,z)_- - \epsilon/6 < f(s;w,z) < f(t;w,z)^- + \epsilon/6 .$$
But then for all $s \in [t,t+\delta)\cap \cS$
\begin{align*}
f(s;x,y) &= y(T_s(w)) + y(T_s(x)) - y(T_s(w))\\
&\leq y(T_s(w)) + \|y\|\|T_s\|\|x-w\| \\
&\leq z(T_s(w)) + (y-z)(T_s(w)) + \epsilon/6 \\
&< f(t;w,z)^- + \epsilon/2 .
\end{align*}
Similarly, for all such $s$, $f(s;x,y)
> f(t;w,z)_- - \epsilon/2$. It follows that $f(t;x,y)^- - f(t;x,y)_- \leq
\epsilon$, for all $\epsilon$, in other words, $(x,y)\in A_t$.
\end{proof}

\begin{lemma}\label{lem:varphipsi}
Let $\varphi,\psi: \cS \rightarrow \mathbb{R}$ be right continuous, and let $c \in \Rp$ be such that
$$\varphi^-(s+c) \leq \psi(s) \,\, , \,\, s\in \cS .$$
Then
\be\label{eq:varphi}
\varphi^-(t+c) \leq \psi^-(t) \,\, , \,\, t\in \Rp .
\ee
A similar statement, with inequalities reversed and using $\varphi_-,\psi_-$ instead of $\varphi^-,\psi^-$, is also true.
\end{lemma}
\begin{proof}
Let $h$ be a RUSC function dominating $\psi$ on $\cS$. Then the
function $h_c$ given by $h_c(t) = h(t-c)$ for $t\geq c$, and
$h_c(t) = \infty$ for $t<c$, is RUSC and dominates $\varphi^-$ on
$c+\cS$. Let $c\leq s \in \cS$, and let $c + s_n \in c+\cS$ such
that $c + s_n \searrow s$. Since
$\varphi$ is right continuous at
$s$, we have
$$\varphi(s)  = \lim_{n} \varphi^-(s_n + c) \leq \lim\sup_{n} \psi(s_n) \leq \lim\sup_{n} h(s_n) \leq h_c(s) .$$
Thus, $h_c$ is RUSC and dominates $\varphi^-$ on $\cS$, so
$\varphi^-(t+c) \leq h(t)$ for all $t\in\Rp$, from which
(\ref{eq:varphi}) follows.

The similar statement, involving $\varphi_-,\psi_-$ instead of $\varphi^-,\psi^-$, is obtained immediately by multiplying by $-1$.
\end{proof}

\vspace{0.3cm}
\noindent{\bf 6. Proof of (\ref{eq:what_to_show}).}

Now we turn to prove that $\left(f(t;x,y)\right)^- =
\left(f(t;x,y)\right)_-$, for all $t\in \Rp, x\in X,y\in X^*$.
That is, we turn to prove that $E = \Rp$.

Consider the space $\cX = \Rp \times X \times X^*$ with
half-open$\times$norm$\times$norm topology. Recall that with the
half-open topology $\Rp$ is a Baire space.
Denote the subspace $\cS \times X \times X^*$ by $\cX_0$. A
straightforward computation shows that $f$ is jointly continuous
on $\cX_0$. It then follows that $(t,x,y)\mapsto f(t;x,y)^-$ is
upper and $(t,x,y)\mapsto f(t;x,y)_-$ is lower semicontinuous on
$\cX$, which means that the sets
$$A_n := \{(t,x,y) \in \cX : f(t;x,y)^- - f(t;x,y)_- < 1/n \}$$
are all open and contain the dense set $\cX_0$. We conclude that
the set
$$A := \bigcap_{n=1}^\infty A_n = \{(t,x,y) \in \cX : f(t;x,y)^- - f(t;x,y)_- = 0 \}$$
is a dense $G_\delta$ in $\cX$.

By the results in
\cite[Section II.22.V]{Kur}, (sometimes referred to as the
Kuratowski-Ulam Theorem. To apply this theorem we need the
separability assumption), it follows
that there is a dense $G_\delta$ (in the half-open topology) set $E' \subseteq \Rp$ of points $t$ for
which the set
$$A_t = \{(x,y) \in X \times X^* :  \left(f(t;x,y)\right)^- = \left(f(t;x,y)\right)_- \}$$
is residual, and, in particular, dense in $X \times X^*$. But by
Lemma \ref{lem:subspace}, $A_t$ is closed, so for all $t\in E'$,
$A_t = X \times X^*$. In other words, we obtain that $E$ contains a dense $G_\delta$ in $\Rp$ in the
half-open topology, and it follows that $E$ is residual in $\Rp$
in the standard topology (because every open set $U$ in the half
open topology contains an open set $V$ in the standard one, such
that $V$ is dense in $U$).

By the discussion following the definition of $E$, we can define
$T_t x$ for all $t\in E$ and all $x\in X$, consistently with the
definition of $T_t x$ for $t\in \cS$. For $s,t \in \cS$, we have
\bes f(t+s;x,y) = f(t;T_s(x),y). \ees It follows that for
$t\in\Rp$, $s\in\cS$, \be\label{eq:s+t1} \left(f(t+s;x,y)\right)^-
= \left(f(t;T_s (x),y)\right)^-, \ee and similarly for $f_-$. So
whenever $t\in E$ and $s\in\cS$, then $t+s$ is also in $E$. Now in
(\ref{eq:s+t1}) we put $\cS \ni s_n \searrow s \in E$, to get, for all $t\in \cS$,
\begin{align*}
\left(f(t+s;x,y)\right)^-
&= \lim_n \left(f(t+s_n;x,y)\right)^- \\
&= \lim_n f(t;T_{s_n} (x),y) \\
(*)&= y(T_t(T_s(x))) \\
&= f(t;T_s(x),y)
\end{align*}
(equality $(*)$ follows from the fact that
$T_t$ is weakly continuous).
It follows using Lemma \ref{lem:varphipsi} that
$$\left(f(t+s;x,y)\right)^- \leq \left(f(t;T_s(x),y)\right)^- \,\, , \,\, s\in E, t \in \Rp .$$
Similarly,
$$\left(f(t+s;x,y)\right)_- \geq \left(f(t;T_s(x),y)\right)_- \,\, , \,\, s\in E, t \in \Rp .$$
In particular, if $s,t \in E$, then
$$ \left(f(t+s;x,z)\right)^- \leq \left(f(t;T_s(x),y)\right)^- = \left(f(t;T_s(x),y)\right)_- \leq \left(f(t+s;x,z)\right)_- .$$
Thus, $E$ is a semigroup.

But then $E$ must be $\Rp$. Indeed, for $0<r\in\Rp$, $r - E$ contains a dense $G_\delta$ in $[0,r]$, so it must
intersect $E$. Thus $r$ is a sum of two elements in $E$, and hence is in $E$. It follows that $E = \Rp$, and the proof is complete.
\end{proof}

\begin{remark}
\emph{Note that for Hilbert spaces the above result is trivial,
because weak continuity implies strong continuity at $0$:
$$\|T_t h - h \|^2 = \|T_t h\|^2 - 2 \Re \langle T_t h, h \rangle + \|h\|^2 \leq 2\|h\|^2 - 2 \Re\langle T_t h, h \rangle \rightarrow 0 $$
as $t\rightarrow 0$ (see, for example, \cite[Section I.6]{SzNF70}), and strong continuity at $0$ implies uniform strong
continuity
(this remark -- that is, the \emph{triviality} of the
result -- is not true, in our opinion at least, for nonlinear
semigroups).}
\end{remark}

One might ask where in the proof we used the reflexivity of $X$. Checking the proof, one can see that
we need both $X$ and $X^*$ to be separable (in order to use the Kuratowski-Ulam Theorem), and that we need $X$ to be
weakly sequentially complete. These two conditions turn out to be equivalent to having $X$ separable and reflexive.

Another condition one might question is the contractiveness of the semigroup. This condition is not essential, as the following result shows.
\begin{corollary}
Let $X$ and $\cS$ be as above, and let $T = \{T_s\}_{s\in \cS}$ be a semigroup of operators on $X$ such that
(\ref{eq:continuity}) holds. Then $T$ can be extended to a weakly continuous semigroup of operators over $\Rp$ if and only if there exist $M,a\geq0$ such that for all $t \in \cS$,
\be\label{eq:ineq}
\|T_t\| \leq Me^{at} .
\ee
\end{corollary}
\begin{remark}
\emph{Any semigroup bounded on all bounded subsets of $\cS$ will satisfy (\ref{eq:ineq}) for appropriate $M$ and $a$.
Assuming that each $T_s$ is weakly continuous, the above result also holds for nonlinear semigroups, with the extended semigroup being only \emph{right}-weakly continuous.}
\end{remark}
\begin{proof}
It is a well known result that any weakly continuous semigroup over $\Rp$ satisfies (\ref{eq:ineq}) for appropriate $M$ and $a$, and for all $t\in\Rp$. Thus, if $T$ can be extended to a semigroup over $\Rp$, it must satisfy (\ref{eq:ineq}).

Conversely, if $T$ satisfies (\ref{eq:ineq}), then one can define a new semigroup $U$ by
$$U_s = e^{-as}T_s \,\, , \,\, s\in\cS .$$
Now one defines a new norm on $X$ by
$$\|x\|_{\rm new} = \sup_{s\in \cS} \|U_s x\| ,$$
and with this norm $U$ is a contractive semigroup (this is a standard construction). (\ref{eq:continuity}) and (\ref{eq:ineq}) together imply that $\|\cdot \|_{\rm new}$ is equivalent to $\|\cdot\|$. One checks that the normed space $(X, \|\cdot\|_{\rm new})$ is a separable, reflexive Banach space. Thus, with this new norm, $U$ satisfies the assumptions of Theorem \ref{thm:weak}, so it can be extended. Then one puts
$$T_t = e^{at}U_t \,\, , \,\, t \in \Rp$$
to obtain the desired extension of $T$.
\end{proof}

\chapter{Representing representations as contractive semigroups on a Hilbert space and applications to isometric dilations}\label{chap:representing_representations}


In this chapter we introduce one of our key constructions, which allows to prove the existence of isometric dilations to representations of product systems via reduction to classical dilation theory of semigroups of contractions on a Hilbert space. This construction comes from
\cite{ShalitReprep}, and was used (with several technical differences) also in \cite{ShalitCP0Dil}.

\section{Introduction}
In many ways, representations of product systems are analogous to semigroups of contractions on Hilbert spaces.
For example, when $A = \mathbb{C}$ and $E$ is the trivial product system $\mathbb{C}\times [0,\infty)$, 
then $\{T_t(1)\}_{t\geq 0}$ is a contractive semigroup whenever $T$ is a completely contractive representation of $E$. 
Many proofs of results concerning
representations
are based on the ideas of the proofs of the analogous results concerning contractions on a Hilbert space,
with the appropriate, sometimes highly non-trivial, modifications made. For example, the proof given in
\cite{MS02} that every c.c. representation has an isometric dilation uses some methods from the classical proof that every contraction on a Hilbert space has an isometric dilation.

A point of view that has proved fruitful is that one may try to
exploit the \emph{results} rather than the  \emph{methods} of the
theory of contractive semigroups on a Hilbert space when attacking
problems concerning representations of product systems. In other
words, we wish to find a systematic way to \emph{reduce} (problems
concerning) a representation of a product system to (analogous
problems concerning) a \emph{semigroup of contractions on a
Hilbert space}. This chapter contains the
first steps in this direction. In Section \ref{sec:rep}, given a
product system $X$ over a semigroup $\cS$ and c.c representation
$(\sigma,T)$ of $X$ on a Hilbert space $H$, we construct a Hilbert
space $\cH$ and a contractive semigroup $\hat{T} =
\{\hat{T}_s\}_{s\in\cS}$ on $\cH$, such that $\hat{T}$ contains
all the information regarding the representation. In Section
\ref{sec:dil} we show that if $\hat{T}$ has a regular isometric
dilation, then so does $T$.

In Section \ref{sec:dbly}, we prove that doubly commuting
representations of product systems of Hilbert correspondences over
certain subsemigroups of $\Rp^k$ have doubly commuting, regular
isometric dilations. This was already proved in \cite{S08} for the
case $\cS = \mathbb{N}^k$. Our proof is based on the construction
made in Section \ref{sec:rep}.

Section \ref{sec:iso_dil} contains the first result that will directly be applied to dilation theory of CP-semigroups: the existence of an isometric dilation to a fully coisometric product system representation. This result will be used both in Chapter \ref{chap:unital}, in the construction of an E$_0$-dilation to a (strongly commuting) two-parameter CP$_0$-semigroup, and in Part \ref{part:II}, in the construction of an E$_0$-dilation to a $k$-tuple of commuting unital CP maps (under additional assumptions).

This is a good point to remark that our approach has some
limitations. For example, the construction introduced in section
\ref{sec:rep} does not seem to be canonical in any nice way, and
we cannot obtain all of the
results in \cite{S08}. We will illustrate these limitations in
section \ref{sec:further}, after proving another sufficient
condition for the existence of a regular, isometric dilation. One
might wonder, indeed, how far can one get by trying to reduce
representations of product systems to semigroups of operators on a
Hilbert space, as the former are certainly ``much more
complicated". Indeed, in Section \ref{sec:iso_dil_d} we will construct an isometric dilation of a c.c. representation of a product system over $\diadp$, and we have not been able to do that using the methods of this chapter.

\subsection{Notation for this chapter}

Throughout this chapter, $\Omega$ will denote some fixed set, and $\mb{R}_+^\Omega$ will denote the semigroup product of $\mb{R}_+$ with itself $|\Omega|$ times, that is, the space of functions $\Omega \rightarrow \mb{R}_+$. $\cS$ will be any subsemigroup of $\mb{R}_+^\Omega$, and in different sections we will impose different additional conditions on $\cS$ (in many places $\cS$ can be taken to be any abelian cancellative semigroup with identity $0$ and an appropriate partial ordering, or, more generally, an Ore semigroup).

We denote by $\cS - \cS$ the subgroup of $\mathbb{R}^\Omega$ generated by $\cS$ (with addition and subtraction defined in the obvious way). For $s \in \cS - \cS$ we shall denote by $s_+$ the element in $\cS$ that sends $j\in \Omega$ to $\max\{0,s(j)\}$, and $s_- = s_+ - s$.

$\cS$ becomes a partially ordered set if one introduces the relation
$$s \leq t \Longleftrightarrow \forall j\in\Omega . s(j) \leq t(j) .$$
$s<t$ means that $s \leq t$ and $s \neq t$; $s \ngeq t$ means that $s \geq t$ is false.

\section{Representing representations as contractive semigroups on a Hilbert space}\label{sec:rep}

Let $A$ be a $C^*$-algebra, and let $X$ be a discrete product
system of $C^*$-correspondences over $\cS$. Let $(\sigma,T)$ be a
completely contractive covariant representation of $X$ on the
Hilbert space $H$. Our assumptions do not imply that $X(0) \otimes
H \cong H$. This unfortunate fact will not cause any real trouble,
but it will make our exposition a little clumsy.

Define $\cH_0$ to be the space of all finitely supported functions
$f$ on $\cS$ such that for all $0 \neq s \in \cS$, $f(s) \in X(s)
\otimes_{\sigma} H$ and such that $f(0)\in H$. We equip $\cH_0$
with the inner product
$$\langle \delta_s \cdot \xi, \delta_t \cdot \eta \rangle = \delta_{s,t} \langle \xi, \eta \rangle  ,$$
for all $s,t \in \cS - \{0\}, \xi \in X(s) \otimes H, \eta \in
X(t) \otimes H$ (where the $\delta$'s on the left-hand side are
Dirac deltas, the $\delta$ on the right-hand side is Kronecker's
delta). If $s$ or $t$ is $0$, then the inner product is
defined similarly. Let $\cH$ be the completion of $\cH_0$ with
respect to this inner product. Note that
$$\cH \cong H \oplus \Big(\oplus_{0 \neq s \in \cS} X(s)\otimes H \Big) .$$
We
define a family $\hat{T} = \{\hat{T}_s\}_{s \in \cS}$ of operators
on $\cH_0$ as follows. First, we define
$\hat{T}_0$ to be the identity. Now assume that $s>0$.
It is more convenient to define the adjoint of $\hat{T}_s$, and we do that by the formula
$$\hat{T}^*_s \big(\delta_t \cdot x_t \otimes h \big) = \delta_{t+s}\cdot x_t \otimes \widetilde{T}^*_s h ,$$
for $x_t \in X(t), h \in H$, with $t \neq 0$ (of course, $\delta_{t+s}\cdot x_t \otimes \widetilde{T}^*_s h$ means $\delta_{t+s}\cdot (U_{t,s} \otimes I_H) x_t \otimes \widetilde{T}^*_s h$, that is, we identify $X(t) \otimes X(s)$ with $X(t+s)$). We also define $\hat{T}^*_s \delta_0 h = \delta_s \widetilde{T}^*_s h$.
Since $\widetilde{T}_s^*$ is a contraction, $\hat{T}_s^*$
extends uniquely to a contraction in $B(\cH)$.

The family $\{\hat{T}_s\}_{s\in\cS}$ can be described as follows. If $t\in \cS$ and $t \ngeq s$, then  $\hat{T}_s (\delta_t \cdot \xi ) = 0$ for all
$\xi \in X(t) \otimes_{\sigma} H$ (or all $\xi \in H$, if $t=0$). If $\xi \in X(s) \otimes_\sigma H$, then
$\hat{T}_s (\delta_s \cdot \xi ) = \delta_0 \cdot \widetilde{T}_s \xi$. Finally, if $t
> s > 0$, then
\be\label{eq:def:hat} \hat{T}_s \left(\delta_t \cdot (x_{t-s}
\otimes x_s \otimes h) \right) = \delta_{t-s} \cdot
\left(x_{t-s}\otimes \widetilde{T}_s (x_s \otimes h) \right) .
\ee

%
%
We now show that $\hat{T}$ is a semigroup. Let $s,t \in \cS$. If
either $s = 0$ or $t = 0$ then it is clear that the semigroup
property $\hat{T}_s \hat{T}_t = \hat{T}_{s+t}$ holds. Assume that
$s,t >0$. Then
\begin{align*}
\hat{T}_s^* \hat{T}_t^* \delta_u x_u \otimes h
&= \hat{T}_s^* \delta_{u+t} x_u \otimes \widetilde{T}_t^* h \\
&= \delta_{u+t+s} x_u \otimes (I_{X(t)} \otimes \widetilde{T}_s^*) \widetilde{T}_t^* h \\
&= \delta_{u+t+s} x_u \otimes \widetilde{T}_{s+t}^* h \\
&= \hat{T}_{s+t}^* \delta_u x_u \otimes h .
\end{align*}

Note that if $T$ is a fully coisometric representation, then $\hat{T}$ is a semigroup of coisometries.

We summarize the construction in the following proposition.
\begin{proposition}\label{prop:technology}
Let $A$, $X$, and $\cS$ and $(\sigma,T)$ be as above, and let
$$\cH = H \oplus \Big(\oplus_{0 \neq s \in \cS} X(s)\otimes_\sigma H \Big).$$
There exists a contractive semigroup $\hat{T} = \{\hat{T}_s\}_{s\in\cS}$ on $\cH$ such for all
$0\neq s\in\cS$, $x \in X(s)$ and $h\in H$,
$$\hat{T}_s \left(\delta_s \cdot x \otimes h \right) = T_s(x)h .$$
If $T$ is a fully coisometric representation, then $\hat{T}$ is a semigroup of coisometries.
\end{proposition}

\section{Regular isometric dilations of product system representations}\label{sec:dil}

\subsection{Notation for Sections \ref{sec:dil}, \ref{sec:dbly} and \ref{sec:further}}\label{subsec:notation}
A \emph{commensurable semigroup} is a semigroup $\Sigma$ such that for every $N$
elements $s_1, \ldots, s_N \in \Sigma$, there exist $s_0 \in \Sigma$ and $a_1, \ldots, a_N \in \mathbb{N}$ such that
$s_i = a_i s_0$ for all $i = 1, \ldots N$. For example, $\mathbb{N}$ is a commensurable semigroup. If $r\in \Rp$, then $r\cdot \mathbb{Q}_+$ is commensurable, and any commensurable subsemigroup of $\Rp$ is contained in such a semigroup.

Throughout this section and the next two, $\cS$ will denote a semigroup
$$\cS = \sum_{i\in \Omega}\cS_i ,$$
where $\cS_i$ is a commensurable and unital (i.e., contains $0$) subsemigroup of $\Rp$.
To be more precise,
$\cS$ is the subsemigroup of $\Rp^\Omega$ consisting of finitely supported functions $s$ such that $s(j) \in \cS_j$ for all $j \in \Omega$. Still another way to describe $\cS$ is the following:
\bes
\cS = \left\{ \sum_{j\in \Omega} {\bf e_j}(s_j) : s_j \in \cS_j, {\rm \,\,all \,\, but\,\, finitely\,\, many\,} s_j {\rm's \,\,are\, }0 \right\},
\ees
where ${\bf e_i}$ is the inclusion of $\cS_i$ into $\prod_{j\in \Omega}\cS_j$. Here is a good example to keep in mind: if $|\Omega| = k \in \mathbb{N}$, and if $\cS_i = \mathbb{N}$ for all $i\in\Omega$, then $\cS = \mathbb{N}^k$.

If $u = \{u_1, \ldots, u_N\} \subseteq \Omega$, we let $|u|$ denote the number of elements in $u$ (this notation will only be used for finite sets). We shall denote by ${\bf e}[u]$ the element of $\mathbb{R}^\Omega$ having $1$ in the $i$th place for every $i\in u$, and having $0$'s elsewhere, and we denote $s[u]: = {\bf e}[u]\cdot s$, where multiplication is pointwise.

\subsection{Regular isometric dilations of product system representations}\label{sec:regular}
Let $H$ be a Hilbert space, and let $T = \{T_s\}_{s\in\cS}$ be a
semigroup of contractions over $\cS$. A semigroup $V =
\{V_s\}_{s\in\cS}$ on a Hilbert space $K \supseteq H$ is said to
be a \emph{regular dilation of $T$} if for all $s\in\cS - \cS$
$$P_H V_{s_-}^*V_{s_+} \big|_H = T_{s_-}^*T_{s_+} .$$
Here and henceforth $P_H$ will denote the orthogonal projection from $K$ onto $H$.
$V$ is said to be an \emph{isometric} dilation if it consists of
isometries. An isometric dilation $V$ is said to be a
\emph{minimal} isometric dilation if
$$K = \bigvee_{s\in\cS}V_s H .$$

The notion of regular isometric dilations can be naturally
extended to representations of product systems.
\begin{definition}
Let $X$ be a product system over $\cS$, and let $(\sigma,T)$ be a
c.c. representation of $X$ on a Hilbert space $H$. An isometric
representation $(\rho,V)$ on a Hilbert space $K \supseteq H$ is said
to be a \emph{regular isometric dilation} if for all $a\in A =
X(0)$, $H$ reduces $\rho(a)$ and
$$\rho(a)\big|_H = \sigma(a) ,$$
and for all $s\in\cS - \cS$
$$P_{X(s_-) \otimes H} \widetilde{V}_{s_-}^* \widetilde{V}_{s_+}\big|_{X(s_+) \otimes H} = \widetilde{T}_{s_-}^* \widetilde{T}_{s_+}.$$
Here, $P_{X(s_-) \otimes H}$ denotes the orthogonal projection of
${X(s_-) \otimes_{\rho} K}$ onto ${X(s_-) \otimes_{\rho} H}$.
$(\rho,V)$ is said to be a \emph{minimal} dilation if
$$K = \bigvee \{V(x)h : x \in X, h \in H\} .$$
\end{definition}
In \cite{S08}, Solel studied regular isometric dilations of product
system representations over $\mathbb{N}^k$, and proved some
necessary and sufficient conditions for the existence of a regular
isometric dilation. One of our aims in this chapter is to show how
the construction of Proposition \ref{prop:technology} can be used
to generalize \emph{some} of the results in \cite{S08}. The
following proposition is the main tool.

\begin{proposition}\label{prop:mainprop}
Let $A$ be a $C^*$-algebra, let $X = \{X(s)\}_{s \in \cS}$ be a
product system of $A$-correspondences over $\cS$, and let
$(\sigma,T)$ be a c.c. representation of $X$ on a Hilbert space $H$.
Let $\hat{T}$ and $\cH$ be as in Proposition
\ref{prop:technology}. Assume that $\hat{T}$ has a regular
isometric dilation. Then there exists a Hilbert space $K \supseteq
H$ and an isometric representation $V$ of $X$ on $K$, such that
\begin{enumerate}
    \item\label{it:V_0} $P_H$ commutes with $V_0 (A)$, and $V_0(a) P_H = \sigma(a) P_H$, for all $a \in
    A$;
    \item\label{it:regDil} $P_{X(s_-) \otimes H} \widetilde{V}_{s_-}^* \widetilde{V}_{s_+}\big|_{X(s_+) \otimes H} = \widetilde{T}_{s_-}^* \widetilde{T}_{s_+}$ for all $s \in \cS -
    \cS$;
    \item $K = \bigvee \{V(x)h : x \in X, h\in H\} $
    ;
    \item\label{it:V*1} $P_H V_s(x)\big|_{K \ominus H} = 0$ for all $s \in \cS$, $x \in X(s)$.
\end{enumerate}
That is, if $\hat{T}$ has a regular isometric dilation, then so
does $T$. If $\sigma$ is nondegenerate and $X$ is essential (that is, $A X(s)$ is dense in $X(s)$ for all $s\in \cS$)
then $V_0$ is also nondegenerate.
\end{proposition}
\begin{proof}
Construct $\cH$ and $\hat{T}$ as in Proposition \ref{prop:technology}.

Let $\hat{V} = \{\hat{V}_s \}_{s \in \cS}$ be a minimal, regular,
isometric dilation of $\hat{T}$ on some Hilbert space $\cK$.
Minimality means that
\bes \cK = \bigvee \{\hat{V}_t(\delta_s
\cdot(x \otimes h)) : s,t \in \cS, x \in X(s), h \in H \} .
\ees

Introduce the Hilbert space $K$,
\bes K = \bigvee
\{\hat{V}_s(\delta_s \cdot(x \otimes h)) : s \in \cS, x \in X(s),
h \in H \} .
\ees

We consider $H$ as embedded in $K$ (or in $\cH$ or in $\cK$) by
the identification
\bes h \leftrightarrow \delta_0 \cdot h .
\ees

Next, we define a left action of $A$ on $\cH$ by
\bes a \cdot
(\delta_s \cdot x \otimes h) = \delta_s \cdot ax \otimes h ,
\ees
for all $a \in A, s \in \cS \setminus \{0\}, x \in X(s)$ and $h \in H$, and
\be\label{eq:V_0onH}
 a \cdot (\delta_0 \cdot h) = \delta_0 \cdot
\sigma(a) h \,\, , \,\, a \in A, h \in H .
\ee

By Lemma 4.2 in \cite{cL94}, this extends to a bounded linear
operator on $\cH$. Indeed, this follows from the following
inequality:
\begin{align*}
\left\|\sum_{i=1}^n a x_i \otimes h_i \right\|^2 & = \sum_{i,j=1}^n \langle
h_i,
\sigma (\langle a x_i, a x_j \rangle) h_j \rangle \\
& = \left \langle \big(\sigma (\langle a x_i, a x_j \rangle ) \big) (h_1, \ldots, h_n)^T,  (h_1, \ldots, h_n)^T \right \rangle_{H^{(n)}} \\
(*)& \leq \|a\|^2 \left \langle \big(\sigma (\langle x_i, x_j \rangle ) \big) (h_1, \ldots, h_n)^T,  (h_1, \ldots, h_n)^T \right \rangle_{H^{(n)}} \\
& = \|a\|^2 \left\|\sum_{i=1}^n x_i \otimes h_i \right\|^2 .
\end{align*}
The inequality (*) follows from the complete positivity of $\sigma$
and from $(\langle a x_i, a x_j \rangle ) \leq \|a\|^2 (\langle
x_i, x_j \rangle )$, which is the content of the cited lemma.

In fact, this is a $*$-representation (and it is faithful if $\sigma$
is). Explanation: it is clear that this is a homomorphism of
algebras. To see that it is a $*$-representation it is enough to
take $s \in \cS, x,y \in X(s)$ and $h,k \in H$ and to compute
\begin{align*}
\langle ax \otimes h, y \otimes k \rangle = \langle  h, \sigma
(\langle ax,  y \rangle)  k \rangle = \langle  h, \sigma (\langle x,
a^* y \rangle)  k \rangle = \langle x \otimes h, a^*y \otimes k
\rangle ,
\end{align*}
(recall that the left action of $A$ on X(s) is adjointable). Note that
this left action commutes with $\hat{T}$:
\bes a \hat{T}_s
(\delta_t x_{t-s} \otimes x_s \otimes h) = \delta_{t-s} a x_{t-s}
\otimes T_s (x_s) h = \hat{T}_s (\delta_t a x_{t-s} \otimes x_s
\otimes h) , \ees or \bes a \hat{T}_s (\delta_s x_s \otimes h) =
\delta_{0} \sigma (a)  T_s (x_s) h = \delta_{0}   T_s (a x_s) h=
\hat{T}_s (\delta_s a x_s \otimes h) .
\ees

We shall now define a representation $V$ of $X$ on $K$. We wish to
define $V_0$ by the rules
\be \label{eq:V_0 definition} V_0(a)
\hat{V}_s (\delta_s \cdot x_s \otimes h) = \hat{V}_s (\delta_s
\cdot a x_s \otimes h) , \ee and \bes V_0(a)  (\delta_0 \cdot h) =
\delta_0 \cdot \sigma(a) h .
\ees
To see that this extends
to a bounded, linear operator on $K$, let $\sum_{t} \hat{V}_t
(\delta_t \cdot x_t \otimes h_t) \in K$ (a finite sum), and
compute
\begin{align*}
\| \sum_{t} \hat{V}_t (\delta_t \cdot a x_t \otimes h_t) \|^2 &=
\sum_{s,t} \langle \hat{V}_s (\delta_s \cdot a x_s \otimes
h_s) , \hat{V}_t (\delta_t \cdot a x_t \otimes h_t) \rangle \\
&= \sum_{s,t} \langle \hat{V}_{(s-t)_-}^* \hat{V}_{(s-t)_+}
(\delta_s \cdot a x_s \otimes h_s) , \delta_t \cdot
a x_t \otimes  h_t \rangle \\
(*)&= \sum_{s,t} \langle \hat{T}_{(s-t)_-}^* \hat{T}_{(s-t)_+}
(\delta_s \cdot a x_s \otimes h_s) , \delta_t \cdot
a x_t \otimes  h_t \rangle \\
&= \sum_{s,t} \langle \hat{T}_{(s-t)_-}^* \hat{T}_{(s-t)_+}
(\delta_s \cdot a^* a x_s \otimes h_s) , \delta_t \cdot
 x_t \otimes  h_t \rangle \\
&= \sum_{s,t} \langle \hat{V}_s (\delta_s \cdot a^* a x_s \otimes
h_s) , \hat{V}_t (\delta_t \cdot  x_t \otimes h_t) \rangle .
\end{align*}
(The computation would have worked for finite sums including
summands from $H$, also). Step (*) is justified because $\hat{V}$
is a regular dilation of $\hat{T}$. This will be used repeatedly.
We conclude that if $a \in A$ is unitary then
\bes \left \|
\sum_{t} \hat{V}_t (\delta_t \cdot a x_t \otimes h_t) \right \| =
\left \| \sum_{t} \hat{V}_t (\delta_t \cdot  x_t \otimes h_t)
\right \| .
\ees
For general $a \in A$, we may write $a =
\sum_{i=1}^4 \lambda_i u_i$, where $u_i$ is unitary and
$|\lambda_i| \leq 2 \|a \|$. Thus,
\bes \left \| \sum_{t}
\hat{V}_t (\delta_t \cdot a x_t \otimes h_t) \right \| = \left
\|\sum_{i=1}^4  \lambda_i \sum_{t} \hat{V}_t (\delta_t u_i \cdot
x_t \otimes h_t) \right \| \leq 8 \|a\| \left \| \sum_{t}
\hat{V}_t (\delta_t \cdot  x_t \otimes h_t) \right \| .
\ees
In fact, we will soon see that $V_0$ is a representation, so this is
quite a lousy estimate. But we proved it only to show that $V_0(a)$
can be extended to a well defined operator on $K$.

It is immediate that $V_0$ is linear and multiplicative. To see
that it is $*$-preserving, let $s,t \in \cS$, $x \in X(s), x' \in
X(t)$ and $h,h' \in H$.
\begin{align*}
\langle V_0(a)^* \hat{V}_s (\delta_s \cdot x \otimes h),
\hat{V}_t(\delta_t \cdot x' \otimes h') \rangle
& = \langle  \hat{V}_s (\delta_s \cdot x \otimes h), V_0(a) \hat{V}_t(\delta_t \cdot x' \otimes h') \rangle \\
& = \langle  \hat{V}_s (\delta_s \cdot x \otimes h),  \hat{V}_t(\delta_t \cdot ax' \otimes h') \rangle \\
& = \langle \hat{V}_{(s-t)_-}^* \hat{V}_{(s-t)_+} (\delta_s \cdot x \otimes h),  \delta_t \cdot ax' \otimes h' \rangle \\
& = \langle \hat{T}_{(s-t)_-}^* \hat{T}_{(s-t)_+} (\delta_s \cdot x \otimes h),  \delta_t \cdot ax' \otimes h' \rangle \\
& = \langle \hat{T}_{(s-t)_-}^* \hat{T}_{(s-t)_+} (\delta_s \cdot a^* x \otimes h), \delta_t \cdot x' \otimes h'\rangle\\
& = \langle \hat{V}_s (\delta_s \cdot a^* x \otimes h),  \hat{V}_t(\delta_t \cdot x' \otimes h') \rangle \\
& = \langle V_0(a^*) \hat{V}_s (\delta_s \cdot x \otimes h),
\hat{V}_t(\delta_t \cdot x' \otimes h') \rangle.
\end{align*}
Thus, $V_0(a)^* = V_0(a^*)$.

By (\ref{eq:V_0onH}), $H$ reduces $V_0 (A)$, and $V_0(a) \big|_H =
\sigma(a) \big|_H$ (under the appropriate identifications). The assertion about nondegeneracy of $V_0$ is clear from the definitions.

To define $V_s$ for $s > 0$, we will show that the rule
\be\label{eq:definition V_s1}
V_s(x_s) \hat{V}_t(\delta_t \cdot x_t
\otimes h) = \hat{V}_{s+t} (\delta_{s+t} \cdot x_s \otimes x_t
\otimes h)
\ee
can be extended to a well defined operator on $K$.
Let $\sum \hat{V}_{t_i}(\delta_{t_i} \cdot x_i \otimes h_i) $ be a
finite sum in $K$, and let $s \in \cS, x_s \in X(s)$. To estimate
\begin{align*}
\| \sum \hat{V}_{t_i + s}(\delta_{t_i + s} \cdot & x_s \otimes x_i
\otimes h_i) \|^2 = \\
&= \sum \langle \hat{V}_{t_i + s}(\delta_{t_i + s} \cdot x_s \otimes x_i \otimes h_i), \hat{V}_{t_j + s}(\delta_{t_j + s} \cdot x_s \otimes x_j \otimes h_j) \rangle \\
&= \sum \langle \hat{V}_s  \hat{V}_{t_i}(\delta_{t_i + s} \cdot x_s \otimes x_i \otimes h_i), \hat{V}_s \hat{V}_{t_j}(\delta_{t_j + s} \cdot x_s \otimes x_j \otimes h_j) \rangle \\
&= \sum \langle  \hat{V}_{t_i}(\delta_{t_i + s} \cdot x_s \otimes
x_i \otimes h_i), \hat{V}_{t_j}(\delta_{t_j + s} \cdot x_s \otimes
x_j \otimes h_j) \rangle,
\end{align*}
we look at each summand of the last equation. Denoting $\xi_i =
x_i \otimes h_i$, we have
\begin{align*}
\big\langle \hat{V}_{t_i}(\delta_{t_i + s} \cdot x_s \otimes
\xi_i), & \hat{V}_{t_j}(\delta_{t_j + s} \cdot x_s \otimes \xi_j)
\big\rangle =\\
&= \big\langle \hat{V}_{(t_i - t_j)_-}^* \hat{V}_{(t_i - t_j)_+}(\delta_{t_i + s} \cdot x_s \otimes \xi_i), \delta_{t_j + s} \cdot x_s \otimes \xi_j \big\rangle \\
&= \big\langle \hat{T}_{(t_i - t_j)_-}^* \hat{T}_{(t_i - t_j)_+}(\delta_{t_i + s} \cdot x_s \otimes \xi_i), \delta_{t_j + s} \cdot x_s \otimes \xi_j \big\rangle \\
&= \big\langle \delta_{t_j + s} \cdot x_s \otimes \left(I \otimes \widetilde{T}_{(t_i-t_j)_-}^*\right) \left(I \otimes \widetilde{T}_{(t_i-t_j)_+}\right) \xi_i, \\
& \quad \quad \quad \delta_{t_j + s} \cdot x_s \otimes \xi_j \big\rangle \\
&= \big\langle \delta_{t_j} \cdot \left(I \otimes \widetilde{T}_{(t_i-t_j)_-}^*\right) \left(I \otimes \widetilde{T}_{(t_i-t_j)_+}\right) \xi_i, \delta_{t_j} \cdot |x_s|^2  \xi_j \big\rangle \\
&= \big\langle \hat{T}_{(t_i - t_j)_-}^* \hat{T}_{(t_i - t_j)_+}(\delta_{t_i} \cdot \xi_i), \delta_{t_j} \cdot |x_s|^2 \xi_j \big\rangle \\
&= \big\langle \hat{V}_{t_i}(\delta_{t_i } \cdot |x_s| \xi_i), \hat{V}_{t_j}(\delta_{t_j} \cdot |x_s| \xi_j) \big\rangle \\
&= \big\langle V_0(|x_s|) \hat{V}_{t_i}(\delta_{t_i } \cdot
\xi_i), V_0(|x_s|) \hat{V}_{t_j}(\delta_{t_j} \cdot \xi_j)
\big\rangle ,
\end{align*}
(again, this argument works also if some $\xi$'s are in $H$). This
means that
\begin{align*}
\| \sum \hat{V}_{t_i + s}(\delta_{t_i + s} \cdot x_s \otimes x_i
\otimes h_i) \|^2
&= \| V_0 (|x_s|) \sum \hat{V}_{t_i}(\delta_{t_i} \cdot x_i \otimes h_i) \|^2 \\
&\leq \| V_0 (|x_s|)\|^2 \left \| \sum \hat{V}_{t_i}(\delta_{t_i}
\cdot x_i \otimes h_i) \right \|^2,
\end{align*}
so the mapping $V_s$ defined in (\ref{eq:definition V_s1}) does
extend to a well defined operator on $K$. Now it is clear from the
definitions that for all $s \in \cS$, $(V_0, V_s)$ is a covariant
representation of $X(s)$ on $K$. We now show that it is isometric.
Let $s,t,u \in \cS$, $x, y \in X(s)$, $x_t \in X(t)$, $x_u \in
X(u)$ and $h,g \in H$. Then
\begin{align*}
\langle V_s(x)^* V_s(y) \hat{V}_t \delta_t \cdot x_t \otimes h,&
\hat{V}_u \delta_u \cdot x_u \otimes g \rangle =\\
&= \langle  \hat{V}_{t+s} \delta_{t+s} \cdot y \otimes x_t \otimes h, \hat{V}_{u+s} \delta_{u+s} \cdot x \otimes x_u \otimes g \rangle \\
&= \langle  \hat{V}_{(t-u)_-}^* \hat{V}_{(t-u)_+} \delta_{t+s} \cdot y \otimes x_t \otimes h,  \delta_{u+s} \cdot x \otimes x_u \otimes g \rangle \\
(*)&= \langle  \hat{V}_{(t-u)_-}^* \hat{V}_{(t-u)_+} \delta_t \cdot x_t \otimes h,  \delta_{u} \cdot \langle y, x \rangle x_u \otimes g \rangle \\
&= \langle  \hat{V}_{t} \delta_{t} \cdot x_t \otimes h, \hat{V}_{u} \delta_{u} \cdot \langle y, x \rangle x_u \otimes g \rangle \\
&= \langle V_0(\langle x,y \rangle) \hat{V}_t \delta_t \cdot x_t
\otimes h, \hat{V}_u \delta_u \cdot x_u \otimes g \rangle .
\end{align*}
The justification of (*) was essentially carried out in the proof
that $V_s (x_s)$ is well defined. Let us, for a change, show that
this computation works also for the case $u=0$:
\begin{align*}
\langle V_s(x)^* V_s(y) \hat{V}_t \delta_t \cdot x_t \otimes h,&
 \delta_0 \cdot g \rangle =\\
&= \langle  \hat{V}_{t+s} \delta_{t+s} \cdot y \otimes x_t \otimes h, \hat{V}_{s} \delta_{s} \cdot x \otimes g \rangle \\
&= \langle  \hat{V}_{t} \delta_{t+s} \cdot y \otimes x_t \otimes h,  \delta_{s} \cdot x \otimes g \rangle \\
&= \langle  \hat{T}_{t} \delta_{t+s} \cdot y \otimes x_t \otimes h,  \delta_{s} \cdot x \otimes g \rangle \\
&= \langle  \delta_{s} \cdot y \otimes T_t(x_t) \otimes h,  \delta_{s} \cdot x \otimes g \rangle \\
&= \langle  T_t(x_t) \otimes h, \sigma (\langle y, x\rangle) g \rangle \\
&= \langle  \hat{T}_t \delta_t \cdot x_t \otimes h, V_0 (\langle y, x\rangle) \delta_0 \cdot g \rangle \\
&= \langle  \hat{V}_t \delta_t \cdot x_t \otimes h, V_0 (\langle y, x\rangle) \delta_0 \cdot g \rangle \\
&= \langle V_0(\langle x,y \rangle) \hat{V}_t \delta_t \cdot x_t
\otimes h, \delta_0 \cdot g \rangle .
\end{align*}

We have constructed a family $V = \{V_s \}_{s \in \cS}$ of maps
such that $(V_0, V_s)$ is an isometric covariant representation of
$X(s)$ on $K$. To show that $V$ is a product system representation
of $X$, we need to show that the ``semigroup property" holds.

Let $h \in H$, $s,t,u \in \cS$, and let $x_s, x_t, x_u$ be in
$X(s), X(t), X(u)$, respectively. Then
\begin{align*}
V_{s+t} (x_s \otimes x_t) \hat{V}_u (\delta_u \cdot x_u \otimes h)
& = \hat{V}_{s+t+u}(\delta_{s+t+u} \cdot x_s \otimes x_t \otimes x_u \otimes h) \\
& = V_s (x_s) \hat{V}_{t+u}(\delta_{t+u} \cdot x_t \otimes x_u \otimes h) \\
& = V_s (x_s) V_t (x_t) \hat{V}_{u}(\delta_{u} \cdot x_u \otimes
h),
\end{align*}
so the semigroup property holds.

We have yet to show that $V$ is a minimal, regular dilation of
$T$. To see that it is a regular dilation, let $s \in \cS - \cS$,
$x_+ \in X(s_+), x_- \in X(s_-)$ and $h = \delta_0 \cdot h, g =
\delta_0 \cdot g \in H$. Using the fact that $\hat{V}$ is a
regular dilation of $\hat{T}$, we compute:
\begin{align*}
\langle \widetilde{V}_{s_-}^* \widetilde{V}_{s_+} (x_+ \otimes \delta_0
\cdot h),  (x_- \otimes \delta_0 \cdot g) \rangle &= \langle
\hat{V}_{s_+} (\delta_{s_+} x_+ \otimes h),
\hat{V}_{s_-} (\delta_{s_-} x_- \otimes g) \rangle \\
&= \langle \hat{V}_{s_-}^* \hat{V}_{s_+} (\delta_{s_+} x_+ \otimes
h),
\delta_{s_-} x_- \otimes g \rangle \\
&= \langle \hat{T}_{s_-}^* \hat{T}_{s_+} (\delta_{s_+} x_+ \otimes
h),
 \delta_{s_-} x_- \otimes g \rangle \\
&= \langle \widetilde{T}_{s_+} (x_+ \otimes h),
\widetilde{T}_{s_-} (x_- \otimes g) \rangle \\
&= \langle \widetilde{T}_{s_-}^* \widetilde{T}_{s+} (x_+ \otimes h),
 x_- \otimes g \rangle .
\end{align*}

$V$ is a minimal dilation of $T$, because
\begin{align*}
K &= \bigvee \{\hat{V}_s(\delta_s \cdot(x \otimes h)) : s \in \cS,
x \in X(s),
h \in H \} \\
&= \bigvee \{V_s(x) (\delta_0 \cdot h) : s \in \cS, x \in X(s), h
\in H \} .
\end{align*}

Finally, let us note that item \ref{it:V*1} from the statement of the proposition is true for any
minimal isometric dilation (of any c.c. representation of a
product system over any semigroup). Indeed, let $V$ be a minimal
isometric dilation of $T$ on $K$. Let $x_s \in X(s), x_t \in X(t)$
and $h \in H$. Then
\begin{align*}
P_H V_s(x_s) V_t(x_t) h & = P_H V_{s+t} (x_s \otimes x_t)h \\
& = T_{s+t} (x_s \otimes x_t)h = T_s(x_s) T_t(x_t) h \\
& = P_H V_s(x_s) P_H V_t(x_t) h.
\end{align*}
But $K = \bigvee \{V_s(x)h : s \in \cS, x \in X(s), h \in H \}$,
so $P_H V_s(x_s) P_H = P_H V_s(x_s)$, from which item
\ref{it:V*1} follows.
\end{proof}

It is worth noting that, as commensurable semigroups are
countable, if $\cS = \sum_{i=1}^\infty \cS_i$, then, using the
notation of the above proposition, separability of $H$ implies
that $K$ is separable.

\begin{corollary}\label{cor:normcondregdil}
Let $X = \{X(n)\}_{n \in \mb{N}^k}$ be a product system over $\mb{N}^k$, and let $T$ be a representation of $X$ such that
\be\label{eq:norm_cond}
\sum_{j=1}^k \|\widetilde{T}_{{\bf e_j}(1)}\widetilde{T}_{{\bf e_j}(1)}^*\| \leq 1.
\ee
Then $T$ has a minimal regular isometric dilation.
\end{corollary}
\begin{proof}
From (\ref{eq:def:hat}) together with Proposition \ref{prop:technology}, it follows that
\bes
\|\hat{T}_{{\bf e_j}(1)}\| = \|\widetilde{T}_{{\bf e_j}(1)}\|.
\ees
By (\ref{eq:norm_cond}), this means that
\bes
\sum_{j=1}^k \|\hat{T}_{{\bf e_j}(1)}\|^2 = \sum_{j=1}^k \|\hat{T}_{{\bf e_j}(1)} \hat{T}_{{\bf e_j}(1)}^*\| \leq 1.
\ees
By \cite[Proposition 9.2]{SzNF70}, $\hat{T}$ has a regular isometric dilation.
The proof is completed by invoking Proposition \ref{prop:mainprop}.
\end{proof}

\section{Regular isometric dilations of doubly commuting representations}\label{sec:dbly}
It is well known that in order that a $k$-tuple $(T_1, T_2,
\ldots, T_k)$ of contractions have a commuting isometric dilation,
it is not enough to assume that the contractions commute. One of
the simplest sufficient conditions that one can impose on $(T_1,
T_2, \ldots, T_k)$ is that it \emph{doubly commute}, that is
$$T_j T_k = T_k T_j \quad {\rm and} \quad T_j^* T_k= T_k T_j^* $$
for all $j \neq k$. Under this assumption, the $k$-tuple $(T_1,
T_2, \ldots, T_k)$ actually has \emph{regular} unitary dilation. In fact,
if the $k$-tuple $(T_1, T_2, \ldots, T_k)$ doubly commutes then it
also has a minimal \emph{doubly commuting} regular \emph{isometric}
dilation (see \cite[Proposition 3.5]{ShalitNotes} for the simple
explanation). This fruitful notion of double commutation can be
generalized to representations as follows.
\begin{definition}
A representation $(\sigma,T)$ of a product system $X$ over $\cS$
is said to \emph{doubly commute} if \bes (I_{{\bf e_k}(s_k)}
\otimes \widetilde{T}_{{\bf e_j}(s_j)})
  (t \otimes I_H) (I_{{\bf e_j}(s_j)} \otimes \widetilde{T}_{{\bf e_k}(s_k)}^*) = \widetilde{T}_{{\bf e_k}(s_k)}^* \widetilde{T}_{{\bf e_j}(s_j)}
\ees
for all $j \neq k$ and all nonzero $s_j\in \cS_j, s_k\in \cS_k$, where
$t$ stands for the isomorphism between $X({\bf e_j}(s_j)) \otimes
X({\bf e_k}(s_k))$ and $X({\bf e_k}(s_k)) \otimes X({\bf
e_j}(s_j))$, and $I_{s}$ is shorthand for $I_{X(s)}$.
\end{definition}
The following theorem appeared already as \cite[Theorem 3.10]{S08} (for the case $\cS = \mb{N}^k$). We give here a new proof. $\cS$ is assumed to be as specified in Section \ref{subsec:notation}.
\begin{theorem}\label{thm:dbly}
Let $A$ be a $C^*$-algebra, let $X = \{X(s)\}_{s \in \cS}$ be a
product system of $A$-correspondences over $\cS$, and let
$(\sigma,T)$ be doubly commuting c.c. representation of $X$ on a
Hilbert space $H$. There exists a Hilbert space $K \supseteq H$
and a minimal, doubly commuting, regular isometric representation
$V$ of $X$ on $K$.
\end{theorem}
\begin{proof}
Construct $\cH$ and $\hat{T}$ as in Proposition \ref{prop:technology}.

We now show that $\hat{T}_{{\bf e_j}(s_j)}$ and $\hat{T}_{{\bf
e_k}(s_k)}$ doubly commute for all $j \neq k$, and all $s_j\in
\cS_j, s_k\in \cS_k$. Let $t \in \cS$, $x \in X(t), y \in X({\bf
e_j}(s_j))$ and $h \in H$. Using the assumption that $T$ is a
doubly commuting representation,
\begin{align*}
\hat{T}_{{\bf e_k}(s_k)}^* \hat{T}_{{\bf e_j}(s_j)} (\delta_{t+{\bf e_j}(s_j)} & \cdot x \otimes y \otimes h)
= \\
&= \hat{T}_{{\bf e_k}(s_k)}^* \left(\delta_{t} \cdot x \otimes \widetilde{T}_{{\bf e_j}(s_j)} (y \otimes h) \right) \\
&= \delta_{t+{\bf e_k}(s_k)} \cdot x \otimes \widetilde{T}_{{\bf e_k}(s_k)}^* \widetilde{T}_{{\bf e_j}(s_j)} (y \otimes h) \\
&= \delta_{t+{\bf e_k}(s_k)} \cdot x \otimes \left( (I_{{\bf e_k}(s_k)} \otimes \widetilde{T}_{{\bf e_j}(s_j)})
  (t \otimes I_H) (I_{{\bf e_j}(s_j)} \otimes \widetilde{T}_{{\bf e_k}(s_k)}^*)
  (y \otimes h) \right) \\
&= \hat{T}_{{\bf e_j}(s_j)} \hat{T}_{{\bf e_k}(s_j)}^* (\delta_{t+{\bf e_j}(s_j)} \cdot x \otimes y \otimes h) ,
\end{align*}
where we have written $t$ for the
isomorphism between $X({\bf e_j}(s_j)) \otimes X({\bf e_k}(s_k))$ and
$X({\bf e_k}(s_k)) \otimes X({\bf e_j}(s_j))$, and we haven't written the
isomorphisms between $X(s) \otimes X(t)$ and $X(s+t)$.

By a straightforward extension of \cite[Proposition 9.2]{SzNF70},
there exists a minimal, regular isometric dilation $\hat{V} =
\{\hat{V}_s \}_{s \in \cS}$ of $\hat{T}$ on some Hilbert space
$\cK$, such that $\hat{V}_{{\bf e_j}(s_j)}$ and $\hat{V}_{{\bf
e_k}(s_k)}$ doubly commute for all $j\neq k,s_j\in \cS_j, s_k \in
\cS_k$.
Proposition \ref{prop:mainprop} gives
a minimal, regular isometric dilation $V$ of $T$ on some
Hilbert space $K$.

To see that $V$ is doubly commuting, one computes what one should
using the fact that $\hat{V}$ is a minimal, doubly commuting,
regular isometric dilation of $\hat{T}$ (all the five adjectives
attached to $\hat{V}$ play a part). This takes about 4 pages of
handwritten computations, so is omitted. Let us indicate how it is
done. For any $i\in\Omega$, $s_i \in \cS_i$, write $\widetilde{V}_i$
for $\widetilde{V}_{X({\bf e_i}(s_i))}$, $I_i$ for $I_{X({\bf
e_i}(s_i))}$, and so on. Taking $j\neq k$, $s_j \in \cS_j, s_k \in
\cS_k$, operate with
$$\widetilde{V}_k (I_k\otimes \widetilde{V}_j)(t_{j,k}\otimes I_J)(I_j \otimes \widetilde{V}_k^*)$$
and with
$$\widetilde{V}_k \widetilde{V}_k^* \widetilde{V}_j $$
on a typical element of $X({\bf e_j}(s_j)) \otimes  K$ of the
form: \be\label{eq:element} x \otimes \hat{V}_s (\delta_s \cdot
x_s \otimes h) , \ee to see that what you get is the same. One has
to separate the cases where ${\bf e_k}(s_k) \leq s$ and ${\bf
e_k}(s_k) \nleq s$ (this is the case where the fact that $\hat{V}$
is a doubly commuting semigroup comes in). Because $\widetilde{V}_k$
is an isometry, and the elements (\ref{eq:element}) span $X({\bf
e_j}(s_j)) \otimes  K$, one has
$$\widetilde{V}_k^* \widetilde{V}_j =  (I_k\otimes \widetilde{V}_j)(t_{j,k}\otimes I_J)(I_j \otimes \widetilde{V}_k^*) .$$
That will conclude the proof.
\end{proof}

\section{A sufficient condition for the existence of a regular isometric dilation}\label{sec:further}
Using the above methods, one can, quite easily, arrive at the
following result, which is, for the case $\cS = \mathbb{N}^k$, one half of Theorem 3.5 of \cite{S08}.
We prove it for $\cS$ satisfying the conditions described in Section \ref{subsec:notation}.
\begin{theorem}\label{thm:reg}
Let $X$ be a product system over $\cS$, and let
$T$ be a c.c. representation of $X$. If
\be\label{eq:NS}
\sum_{u\subseteq v}(-1)^{|u|}\left(I_{s[v]-s[u]} \otimes \widetilde{T}^*_{s[u]}\widetilde{T}_{s[u]} \right) \geq 0
\ee
for all finite subsets $v\subseteq \Omega$ and all $s\in\cS$, then $T$ has a regular isometric dilation.
\end{theorem}
\begin{proof}
Here are the main lines of the proof. Construct $\hat{T}$ as in Proposition \ref{prop:technology}.
From (\ref{eq:NS}), it follows that
$\hat{T}$ satisfies
$$\sum_{u\subseteq v}(-1)^{|u|}\hat{T}^*_{s[u]}\hat{T}_{s[u]} \geq 0 ,$$
for all finite subsets $v\subseteq \Omega$ and all $s\in\cS$,
which, by a not-very-difficult extension of \cite[Theorem 9.1]{SzNF70}, is
a necessary and sufficient condition for the existence of a
regular isometric dilation $\hat{V}$ of $\hat{T}$. The result now
follows from Proposition \ref{prop:mainprop}.
\end{proof}

Among other reasons, this example has been put forward to illustrate the limitations of our method. By \cite[Theorem
3.5]{S08}, when $\cS = \mathbb{N}^k$, equation (\ref{eq:NS}) is a \emph{necessary}, as well as a sufficient, condition that $T$ has a regular isometric dilation. But our construction ``works only in one direction", so we are able to prove only sufficient conditions (roughly speaking). We believe that, using the methods of \cite{S08} combined with commensurability considerations, one would be able to show that (\ref{eq:NS}) is indeed a necessary condition for the existence of a regular isometric dilation (over $\cS$). Whether or not the constructions of Section \ref{sec:rep} can be modified to give the other direction remains to be answered.

\section{Isometric dilation of a fully coisometric product system representation}\label{sec:iso_dil}

For any $r = (r_1, \ldots, r_k) \in \mathbb{R}^k$, we denote $r_+
:= (\max\{r_1,0\}, \ldots, \max\{r_k,0\})$ and $r_- := r_+ - r$.
Throughout this section, $\cS$ will be a subsemigroup of $\Rp^k$
such that for all $s \in \cS - \cS$, both $s_+$ and $s_-$ are in
$\cS$. The semigroup that we are most interested in, namely
$\Rp^k$, satisfies this condition. For applications in Part \ref{part:II} of this thesis
we will need the
following theorem for $\cS = \mathbb{N}^k$, which also satisfies this
condition.
\begin{theorem}\label{thm:isoDilFC}
Let $\cS$ be as above, let $X = \{X(s)\}_{s \in \cS}$ be a product
system of unital $\cA$-correspondences over $\cS$, and let
$(\sigma,T)$ be a fully coisometric representation of $X$ on $H$,
with $\sigma$ unital. Then there exists a Hilbert space $K
\supseteq H$ and a minimal, fully coisometric and isometric
representation $(\rho,V)$ of $X$ on $K$, with $\rho$ unital, such
that
\begin{enumerate}
    \item $P_H$ commutes with $\rho (\cA)$, and $\rho(a) P_H = \sigma(a) P_H$, for all $a \in \cA$.
    \item\label{it:dilation1} $P_H V_s(x)\big|_H = T_s(x)$ for all $s \in \cS$, $x \in X(s)$.
    \item\label{it:V*2} $P_H V_s(x)\big|_{K \ominus H} = 0$ for all $s \in \cS$, $x \in X(s)$.
\end{enumerate}
If $\sigma$ is nondegenerate and $X$ is essential (that is, $\cA
X(s)$ is dense in $X(s)$ for all $s \in \cS$) then $\rho$ is also
nondegenerate. If $\cA$ is a $W^*$-algebra, $X$ is a product
system of $W^*$-correspondences and $(\sigma,T)$ is a
representation of $W^*$-correspondences, then $(\rho,V)$ is also a
representation of $W^*$-correspondences.
\end{theorem}
\begin{proof}
The proof is very similar to the proof of Proposition \ref{prop:mainprop}, so we will not go into all the details whenever they were taken care of in that proof.

Let $\cH = \oplus_{s \in \cS} X(s)\otimes_\sigma H$, and let
$\hat{T}$ be the semigroup of coisometries constructed in the
discussion preceding Proposition \ref{prop:technology}.

Since $\hat{T}$ is a semigroup of coisometries, there exists a
minimal, \emph{regular} unitary dilation $W = \{W_s\}_{s\in \cS}$
of the semigroup $\{\hat{T}_s^*\}_{s\in \cS}$ on a Hilbert space
$\cK \supseteq \cH$ (this follows from \cite[Proposition 9.2]{SzNF70}). We denote $\hat{V}_s = W_s^*$. We
have for all $s \in \cS-\cS$
\be P_\cH \hat{V}_{s_+} \hat{V}_{s_-}^*
P_\cH = \hat{T}_{s_+} \hat{T}_{s_-}^* ,
\ee
Since the semigroup
$\hat{V}$ consists of commuting unitaries, and since commuting
unitaries doubly commute, we also have
\be\label{eq:reg_dil} P_\cH
\hat{V}_{s_-}^* \hat{V}_{s_+} P_\cH = \hat{T}_{s_+}
\hat{T}_{s_-}^* .
\ee
This triviality turns out to be crucial: it
will allow us to compute the inner products in $\cK$.

Introduce the Hilbert space $K$,
\bes K = \bigvee
\{\hat{V}_s(\delta_s \cdot(x \otimes h)) : s \in \cS, x \in X(s),
h \in H \} .
\ees
We consider $H$ as embedded in $K$ (or in $\cH$ or in $\cK$)
by the identification
\bes h \leftrightarrow \delta_0 \cdot(1
\otimes h) .
\ees
(This is where we use the fact that $\sigma$ is
unital). We turn to the definition of the representation $V$ of
$X$ on $K$. First, note that $\sigma(a)h$ is identified with
$\delta_0 \cdot 1 \otimes_{\sigma}\sigma(a)h = \delta_0 \cdot a
\otimes_{\sigma}h$. Next, we define a left action of $\cA$ on $\cH$ by
\bes
a \cdot (\delta_s \cdot x \otimes h) = \delta_s \cdot ax
\otimes h ,
\ees
for all $a \in \cA, s \in \cS, x \in X(s)$ and $h
\in H$. As we have explained in the proof of Proposition \ref{prop:mainprop}, this
gives rise to a well defined $*$-representation
that commutes
with $\hat{T}_s$ and $\hat{T}_s^*$ for all $s \in \cS$.

We now define a representation $(\rho,V)$ of $X$ on $K$, exactly as in the proof of Proposition \ref{prop:mainprop}. First, we define
$\rho$ by the rule
\be \label{eq:V_e definition}
\rho(a) \hat{V}_s (\delta_s \cdot x_s \otimes h) = \hat{V}_s (\delta_s \cdot a x_s \otimes h) .
\ee
Using (\ref{eq:reg_dil}),
one shows that $\rho(a)$ extends to a bounded map on $K$. It then
follows by direct computation that $\rho$ is a $*$-representation. When $(\sigma,T)$ is a representation of $W^*$-correspondences, we
also have to show that $\rho$ is a \emph{normal} representation.
Let $\{a_\gamma\} \subseteq {\rm ball}_1(\cA)$ be a net converging
in the weak operator topology to $a \in {\rm ball}_1(\cA)$. It is
known (for an outline of a proof, see \cite{MS03}) that the
mapping taking $b\in \cA$ to $b \otimes I_H \in B(X(s)
\otimes_\sigma H)$ is continuous in the ($\sigma$-)weak
topologies. Thus, for all $s \in \cS, x \in X(s)$ and $h \in H$,
$$a_\gamma  x \otimes h \longrightarrow a  x \otimes h $$
in the weak topology of $X(s) \otimes_\sigma H$. It follows that
$$\delta_s \cdot a_\gamma  x \otimes h \longrightarrow \delta_s \cdot a  x \otimes h $$
in the weak topology of $K$, so
$$\hat{V}_s \delta_s \cdot a_\gamma  x \otimes h \longrightarrow \hat{V}_s \delta_s \cdot a  x \otimes h $$
weakly. This implies that $\rho(a_\gamma) \rightarrow \rho(a)$ in the weak operator topology of $B(K)$, so $\rho$ is normal.

Note that $H$ reduces $\rho (A)$, and that $\rho(a) \big|_H =
\sigma(a) \big|_H$ (under the appropriate identifications).
Indeed, putting $t = 0$ in equation (\ref{eq:V_e definition})
gives \bes \rho(a) (\delta_0 \cdot 1 \otimes h) = \delta_0 \cdot a
\otimes h = \delta_0 \cdot 1  \otimes \sigma(a)h .
\ees
The assertions regarding the unitality and nondegeneracy of $\rho$ are
clear from the definitions.

We have completed the construction of $\rho$, and we proceed to
define the representation $V$ of $X$ on $K$. For $s
> 0$, we define $V_s$
by the rule
\be \label{eq:definition V_s}
V_s(x_s)
\hat{V}_t(\delta_t \cdot x_t \otimes h) = \hat{V}_{s+t}
(\delta_{s+t} \cdot x_s \otimes x_t \otimes h) .
\ee

One has to use (\ref{eq:reg_dil}) to show that $V_s(x_s)$ can be extended to a well defined operator on $K$, but once
that is done, it is
easy to see that for all $s \in \cS$, $(\rho, V_s)$ is a covariant representation of $X(s)$ on
K. We now show that it is isometric. This computation is included so the reader has an opportunity to appreciate the role
played by equation (\ref{eq:reg_dil}). Let $s,t,u \in \cS$, $x, y \in X(s)$, $x_t \in X(t)$,
$x_u \in X(u)$ and $h,g \in H$. Then
\begin{align*}
& \langle V_s(x)^* V_s(y) \hat{V}_t \delta_t \cdot x_t \otimes h, \hat{V}_u \delta_u \cdot x_u \otimes g \rangle \\
&= \langle  \hat{V}_{t+s} \delta_{t+s} \cdot y \otimes x_t \otimes h, \hat{V}_{u+s} \delta_{u+s} \cdot x \otimes x_u \otimes g \rangle \\
&= \langle  \hat{V}_{(t-u)_-}^* \hat{V}_{(t-u)_+} \delta_{t+s} \cdot y \otimes x_t \otimes h,  \delta_{u+s} \cdot x \otimes x_u \otimes g \rangle \\
(*)&= \langle  \hat{T}_{(t-u)_+} \hat{T}_{(t-u)_-}^* \delta_{t+s} \cdot y \otimes x_t \otimes h,  \delta_{u+s} \cdot x \otimes x_u \otimes g \rangle \\
&= \langle  \delta_{u+s} \cdot y \otimes \left(I \otimes \widetilde{T}_{(t-u)_+}\right)\left(I \otimes \widetilde{T}_{(t-u)_-}^* \right)  (x_t \otimes h),
\delta_{u+s} \cdot x \otimes x_u \otimes g \rangle \\
&= \langle  \delta_{u} \cdot \left(I \otimes \widetilde{T}_{(t-u)_+}\right)\left(I \otimes \widetilde{T}_{(t-u)_-}^* \right)
  (x_t \otimes h),  \delta_{u} \cdot \langle y, x\rangle x_u \otimes g \rangle \\
&= \langle  \hat{T}_{(t-u)_+} \hat{T}_{(t-u)_-}^* \delta_{t} \cdot x_t \otimes h ,  \delta_{u} \cdot \langle y, x\rangle x_u \otimes g \rangle \\
&= \langle  \hat{T}_{(t-u)_+} \hat{T}_{(t-u)_-}^* \delta_{t} \cdot \langle x, y\rangle x_t \otimes h ,  \delta_{u} \cdot x_u \otimes g \rangle \\
(*)&= \langle  \hat{V}_{(t-u)_-}^* \hat{V}_{(t-u)_+} \delta_{t} \cdot \langle x, y\rangle x_t \otimes h,  \delta_{u} \cdot  x_u \otimes g \rangle \\
&= \langle  \hat{V}_{t} \delta_{t} \cdot \langle x, y\rangle x_t \otimes h, \hat{V}_{u} \delta_{u} \cdot  x_u \otimes g \rangle \\
&= \langle \rho(\langle x,y \rangle) \hat{V}_t \delta_t \cdot x_t \otimes h, \hat{V}_u \delta_u \cdot x_u \otimes g \rangle .
\end{align*}
(The equations marked by (*) are where we use (\ref{eq:reg_dil}).)
This shows that $V_s(x)^* V_s(y) = \rho(\langle x,y \rangle)$, so
$(\rho,V)$ is indeed an isometric representation. To see that it
is fully coisometric, is enough to show that for all $s\in\cS$,
$\widetilde{V}_s$ is onto. It is clear that
$${\rm Im}(\widetilde{V}_s) = \bigvee\{\hat{V}_{t+s} (\delta_{t+s} \cdot x_s \otimes x_t \otimes h) : t\in \cS, x_s \in X(s), x_t \in X(t), h\in H \} .$$
But if $t\in \cS$, $x_t \in X(t)$ and $h\in H$, then
\begin{align*}
\hat{V}_t (\delta_t \cdot x_t \otimes h)
&= \hat{V}_t \hat{V}_s \hat{V}_s^* (\delta_t \cdot x_t \otimes h) \\
(*)&= \hat{V}_t \hat{V}_s \hat{T}_s^* (\delta_t \cdot x_t \otimes h) \\
&= \hat{V}_{t+s}  (\delta_{t+s} \cdot x_t \otimes \widetilde{T}_s^* h) \in {\rm Im}(\widetilde{V}_s) ,
\end{align*}
where (*) is justified because $\hat{V}^*_s$ is an extension of
$\hat{T}^*_s$ (as is any unitary dilation of an isometry). This
shows that $\widetilde{V}_s$ is onto, so it is a unitary, hence $V$ is
fully coisometric.

Let $s \in \cS, x \in X(s)$ and
$h = \delta_0 \cdot 1 \otimes h \in H$. We compute:
\begin{align*}
P_H V_s(x) \big|_H h &= P_H V_s(x) \delta_0 \cdot 1 \otimes h \\
&= P_H \hat{V}_s (\delta_s \cdot x \otimes h) \\
&= P_H P_{\cH} \hat{V}_s \big|_\cH (\delta_s \cdot x \otimes h) \\
&= P_H \hat{T}_s (\delta_s \cdot x \otimes h) \\
&= P_H (\delta_0 \cdot 1 \otimes T_s(x)h) = T_s(x)h .
\end{align*}
Item \ref{it:V*2} in the statement of the theorem, the ``semigroup property" of $V$, as well as minimality, follow as in the proof of Proposition \ref{prop:mainprop}.
\end{proof}

\chapter{Strong commutativity}\label{chap:strong_commutativity}

In this chapter we define the main technical condition that we need in order to simultaneously dilate
a pair of commuting CP-semigroups to a pair of commuting E-semigroups. This condition is \emph{strong commutativity}. After studying the definition and its consequences in the first three sections, we turn in the fourth section to discuss examples. Most of the material here is from \cite{ShalitCP0Dil}. Section \ref{sec:GNS} is from \cite{ShalitCPDil}. The example of quantized convolution semigroups did not appear elsewhere.

\section{Strongly commuting CP maps}
Let $\Theta$ and $\Phi$ be CP maps on $\cM$. We define the Hilbert
space $\cM \otimes_\Phi \cM \otimes_\Theta H$ to be the Hausdorff
completion of the algebraic tensor product $\cM
\otimes_\textrm{alg} \cM \otimes_\textrm{alg} H$ with respect to
the inner product
$$\langle a \otimes b \otimes h, c \otimes d \otimes k \rangle = \langle h, \Theta(b^* \Phi(a^* c) d) k \rangle .$$
\begin{definition}\label{def:SC}
Let $\Theta$ and $\Phi$ be CP maps on $\cM$. We say that they
\emph{commute strongly} if there is a unitary $u: \cM \otimes_\Phi
\cM \otimes_\Theta H \rightarrow \cM \otimes_\Theta \cM
\otimes_\Phi H$ such that:
\begin{itemize}
\item[(i)] $u(a \otimes_\Phi I \otimes_\Theta h) = a \otimes_\Theta I \otimes_\Phi h$ for all $a \in \cM$ and $h \in H$.
\item[(ii)] $u(ca\otimes_\Phi b \otimes_\Theta h) = (c \otimes I_M \otimes I_H)u(a\otimes_\Phi b \otimes_\Theta h)$ for $a,b,c \in \cM$ and $h \in H$.
\item[(iii)] $u(a\otimes_\Phi b \otimes_\Theta dh) = (I_M \otimes I_M \otimes d)u(a\otimes_\Phi b \otimes_\Theta h)$ for $a,b \in \cM$, $d \in \cM'$ and $h \in H$.
\end{itemize}
\end{definition}
The notion of strong commutation was introduced by Solel in
\cite{S06}. Note that if two CP maps commute strongly, then they
commute. The converse is false (for concrete examples see
Section \ref{sec:examples}). In
Section \ref{sec:examples} we shall give many examples of strongly
commuting pairs of CP maps, and for some von Neumann algebras we shall
give a complete characterization of strong commutativity. For the time being
let us just state the fact that if $H$ is a finite dimensional Hilbert space,
then any two commuting CP maps on $B(H)$ strongly commute (see Section \ref{subsec:H_finite}).
The ``true" significance of strong
commutation comes from a bijection between pairs of strongly
commuting CP maps and product systems over $\mathbb{N}^2$ with
c.c. representations (\cite{S06}, Propositions 5.6 and 5.7, and
the discussion between them). It is this bijection that enables one
to characterize all pairs of strongly commuting CP maps on $B(H)$
\cite[Proposition 5.8]{S06}.

In the next chapter we will work with the spaces $\cM
\otimes_{P_1} \cM  \cdots  \cM \otimes_{P_n} H$, where $P_1,
\ldots, P_n$ are CP maps. These spaces are defined in a way
analogous to the way that the spaces $\cM \otimes_\Theta \cM
\otimes_\Phi H$ were defined in the beginning of this section. The
following results are important for dealing with such spaces.

\begin{lemma}\label{lem:SC1}
Assume that $P_{n-1}$ and $P_n$ commute strongly. Then there exists a unitary
$$v : \cM \otimes_{P_1} \cM \otimes_{P_2} \cdots \otimes_{P_{n-1}} \cM \otimes_{P_n} H
\rightarrow
\cM \otimes_{P_1} \cM \otimes_{P_2} \cdots \otimes_{P_n} \cM \otimes_{P_{n-1}} H
$$
such that
\begin{enumerate}
    \item $v(I \otimes_{P_1} \cdots \otimes_{P_{n-1}} I \otimes_{P_n} h) = I \otimes_{P_1} \cdots \otimes_{P_n} I \otimes_{P_{n-1}} h$, for all $h \in H$,
    \item For all $X \in \cM$,
    $$v \circ (X \otimes I \cdots I \otimes I)
    = (X \otimes I \cdots  I \otimes I)\circ v ,$$
    \item For all $X \in \cM'$,
    $$v \circ (I \otimes I \cdots I \otimes X)
    = (I \otimes I \cdots  I \otimes X)\circ v .$$
\end{enumerate}
\end{lemma}
\begin{proof}
Let $u : \cM \otimes_{P_{n-1}} \cM \otimes_{P_n} H \rightarrow \cM
\otimes_{P_n} \cM \otimes_{P_{n-1}} H$ be the unitary that makes
$P_{n-1}$ and $P_n$ commute strongly. Define
$$v = I_E \otimes u ,$$
where $E$ denotes the $W^*$-correspondence (over $\cM$) $\cM
\otimes_{P_1} \cM \otimes_{P_2} \cdots \otimes_{P_{n-3}} \cM$
equipped with the inner product
$$\langle a_1 \otimes \cdots \otimes a_{n-3}, b_1 \otimes \cdots \otimes b_{n-3} \rangle
= P_{n-3}\left(a_{n-3}^*\cdots P_1(a_1^* b_1) \cdots
b_{n-3}\right) .$$

The fact that $v$ commutes with $\cM \otimes I \otimes \cdots
\otimes I$ and $I \otimes I \cdots I \otimes \cM'$ and satisfies
the three conditions listed above are clear from the definition
and from the properties of $u$. The fact that $u$ is surjective
implies that $v$ is, too. It is left to show that $v$ is an
isometry (and this will also show that it is well defined). Let
$\sum a_i \otimes_{P_{n-2}} b_i \otimes_{P_{n-1}} c_i
\otimes_{P_n} h_i$ be an element of $E \otimes_{P_{n-2}} \cM
\otimes_{P_{n-1}} \cM \otimes_{P_n} H$.
\begin{align*}
& \|v (\sum a_i \otimes_{P_{n-2}} b_i \otimes_{P_{n-1}} c_i \otimes_{P_n} h_i) \|^2
 = \\
& = \lel \sum a_i \otimes_{P_{n-2}} u(b_i \otimes_{P_{n-1}} c_i \otimes_{P_n} h_i), \sum a_j \otimes_{P_{n-2}} u(b_j \otimes_{P_{n-1}} c_j \otimes_{P_n} h_j) \rir
 = \\
& = \sum_{i,j} \lel u(b_i \otimes_{P_{n-1}} c_i \otimes_{P_n} h_i), P_{n-2}\left(\langle a_i, a_j \rangle\right) u(b_j \otimes_{P_{n-1}} c_j \otimes_{P_n} h_j) \rir
 = (*) \\
& = \sum_{i,j} \lel u(b_i \otimes_{P_{n-1}} c_i \otimes_{P_n} h_i), u\left(P_{n-2}\left(\langle a_i, a_j \rangle\right) b_j \otimes_{P_{n-1}} c_j \otimes_{P_n} h_j\right) \rir
 = (**) \\
& = \sum_{i,j} \lel b_i \otimes_{P_{n-1}} c_i \otimes_{P_n} h_i, P_{n-2}(\langle a_i, a_j \rangle) b_j \otimes_{P_{n-1}} c_j \otimes_{P_n} h_j \rir
 = \\
& =\|\sum a_i \otimes_{P_{n-2}} b_i \otimes_{P_{n-1}} c_i \otimes_{P_n} h_i \|^2
\end{align*}
the equality marked by (*) follows from the fact that $u$
intertwines the actions of $\cM$ on $\cM \otimes_{P_{n-1}} \cM
\otimes_{P_{n}} H$ and $\cM \otimes_{P_{n}} \cM \otimes_{P_{n-1}}
H$, and the one marked by (**) is true because $u$ is unitary.
\end{proof}

\begin{lemma}\label{lem:SC2}
Assume that $P$ and $Q$ are strongly commuting CP maps on $\cM$.
Then there exists an isomorphism $v = v_{P,Q}$ of $\cM$-correspondences
$$v : \cM \otimes_P \cM \otimes_Q \cM \rightarrow \cM \otimes_Q \cM \otimes_P \cM $$
such that
$$v(I \otimes_P I \otimes_Q I) = I \otimes_Q I \otimes_P I .$$
\end{lemma}

\begin{proof}
For any two CP maps $\Theta, \Phi$ let $W_{\Theta,\Phi}$ be the
Hilbert space isomorphism
$$W_{\Theta, \Phi} : \cM \otimes_\Theta \cM \otimes_\Phi \cM \otimes_I H \rightarrow  \cM \otimes_\Theta \cM \otimes_\Phi H$$
given by $W_{\Theta, \Phi} (a\otimes_\Theta b \otimes_\Phi c
\otimes_I h) = a \otimes_\Theta b \otimes_\Phi ch$. By a
straightforward computation $W_{\Theta, \Phi}^*$ is given by
$W_{\Theta, \Phi}^* (a \otimes_\Theta b \otimes_\Phi h) = a
\otimes_\Theta b \otimes_\Phi I \otimes_I h$, and by even shorter
computations $W_{\Theta, \Phi}W_{\Theta, \Phi}^*$ and $W_{\Theta,
\Phi}^*W_{\Theta, \Phi}$ are identity maps. For all $a,b,c,x \in
\cM$ and all $y \in \cM'$ we have
\begin{align*}\label{eq:W_properties}
W_{\Theta, \Phi}(xa\otimes_\Theta b \otimes_\Phi c \otimes_I yh) & =
xa\otimes_\Theta b \otimes_\Phi cyh \\ & =
x a\otimes_\Theta b \otimes_\Phi ych \\ & =
(x \otimes I \otimes y)W_{\Theta, \Phi}(a\otimes_\Theta b \otimes_\Phi c \otimes_I h) .
\end{align*}
From this, it also follows that
$$W_{\Theta, \Phi}^*(x \otimes I \otimes y) = (x \otimes I \otimes I \otimes y)W_{\Theta, \Phi}^* \quad \,\, (x \in \cM, y \in \cM') .$$

We now define a map $T : \cM \otimes_P \cM \otimes_Q \cM \otimes_I
H \rightarrow \cM \otimes_Q \cM \otimes_P \cM \otimes_I H$ by
$$T = W_{Q,P}^* \circ u \circ W_{P,Q} ,$$
where $u$ is the map that makes $P$ and $Q$ commute strongly. As a
product of such maps, $T$ is a unitary intertwining the left
actions of $\cM$ and $\cM'$. The $v$ that we are looking for is a
map $v : \cM \otimes_P \cM \otimes_Q \cM \rightarrow \cM \otimes_Q
\cM \otimes_P \cM$ that satisfies $T = v \otimes I_H$. We will
find this $v$ using a standard technique exploiting the self
duality of $\cM \otimes_Q \cM \otimes_P \cM$.

For any $x \in \cM \otimes_Q \cM \otimes_P \cM$ we define a map $
L_x : H \rightarrow \cM \otimes_Q \cM \otimes_P \cM \otimes_I H$
by
$$L_x (h) = x \otimes h \,\, , \,\,(h \in H) .$$
The adjoint is given on simple tensors by $L_x^*(y \otimes h) = \langle x, y\rangle h$.

Now, if there is a $v$ such that $T = v \otimes I_H$, then for all $z\in \cM \otimes_P \cM \otimes_Q \cM$ and $x\in \cM \otimes_Q \cM \otimes_P \cM$ we must have
$$
\langle x, v(z) \rangle h = L_x^* (v(z) \otimes h)  = L_x^* T(z \otimes h) .
$$
This leads us to define, fixing $z \in \cM \otimes_P \cM \otimes_Q
\cM$, a mapping $\varphi$ from $\cM \otimes_Q \cM \otimes_P \cM$
into $\cM$:
$$\varphi (x) h := L_x^*T (z \otimes h) .$$
We now prove that $x \mapsto \varphi(x)^*$ is a bounded, $\cM$-module mapping into $\cM$.

{\bf Into $\cM$:} For all $x \in \cM \otimes_Q \cM \otimes_P \cM$,
$\varphi(x)$ is linear. $\|L_x^* T (z \otimes h) \| \leq \|L_x^*
\|\|T \|\|z \|\|h \| $, so $\varphi(x) \in B(H)$. So
$\varphi(x)^*$ exists and is also a bounded, linear operator on
$H$. Now take $d \in \cM'$. Then
$$\varphi(x)dh = L_x^*T (z \otimes dh) = L_x^*T (I \otimes d)(z \otimes h) = L_x^* (I \otimes d) T (z \otimes h) = d \varphi(x)h$$
($L_x^*$ intertwines $\cM'$ from its very definition) whence $\varphi(x) \in \cM'' = \cM$. Thus, $\varphi(x)^* \in \cM$.

{\bf $\cM$-module mapping:} This is because for all $x,y \in \cM
\otimes_Q \cM \otimes_P \cM$ and all $a \in \cM$ $L_{x+y} = L_x +
L_y$ and $L_{ax} = a L_x$ (and also $L_{xa} = L_x a$).

{\bf Bounded mapping:} From the inequalities $\|L_x^* T (z \otimes
h) \| \leq \|L_x^* \|\|T \|\|z \|\|h \| $ and $\|L_x^*\| \leq
\|x\|$ it follows that $\|\varphi(x)^*\|=\|\varphi(x)\| \leq \|z\|
\|x\|$.

It now follows from the self-duality of $\cM \otimes_Q \cM
\otimes_P \cM$ that for all $ z \in \cM \otimes_P \cM \otimes_Q
\cM$ there exists a $v(z) \in \cM \otimes_Q \cM \otimes_P \cM$
such that \be\label{eq:existv} \langle x,v(z)\rangle h = L_x^* T(z
\otimes h) \ee for all $x \in \cM \otimes_Q \cM \otimes_P \cM, h
\in H$. It is easy to see from (\ref{eq:existv}) that $v(z)$ is a
right $\cM$-module mapping. (\ref{eq:existv}) can be re-written as
$$L_x^* (v(z) \otimes h) = L_x^* T(z \otimes h) ,$$
and, since this holds for all $x$, this means that $(v(z) \otimes
h) = T(z \otimes h)$ (because $\cap_x {\rm Ker}(L_x^*) =
\left(\vee_x {\rm Im}(L_x)\right)^\perp = \{0\}$),
or, in other words, $v
\otimes I = T$. This last equality implies that $v$ is unitary,
and that it has all the properties required. For example, if
$a,b,c,X \in \cM$ and $h \in H$, then
\begin{align*}
v(Xa \otimes b \otimes c) \otimes h & = T(Xa \otimes b \otimes c \otimes h) \\
& = (X \otimes I \otimes I \otimes I)T(a \otimes b \otimes c \otimes h) \\
& = (X \otimes I \otimes I \otimes I)(v(a \otimes b \otimes c)\otimes h) \\
& = \big((X \otimes I \otimes I)v(a \otimes b \otimes c)\big) \otimes h .
\end{align*}
Putting $v_1 = v(Xa \otimes b \otimes c)$ and $v_2=(X \otimes I \otimes I)(v(a \otimes b \otimes c)$ we have that for all $h \in H$
$$0 = \|v_1 \otimes h - v_2 \otimes h \|^2= \|(v_1 - v_2)\otimes h\|^2 = \langle h, \langle v_1 - v_2,v_1 - v_2 \rangle h \rangle ,$$
which implies that $\langle v_1 - v_2,v_1 - v_2 \rangle = 0$, or $v(Xa \otimes b \otimes c) =(X \otimes I \otimes I)(v(a \otimes b \otimes c))$.
\end{proof}

\begin{remark}\label{rem:SC}
\emph{The converse of Lemma \ref{lem:SC2} is also true: if there is an isometry of $\cM$-correspondences $v : \cM \otimes_P \cM \otimes_Q \cM \rightarrow \cM \otimes_Q \cM \otimes_P \cM$ such that $v(I \otimes I \otimes I) = I \otimes I \otimes I$ then $P$ and $Q$ strongly commute. Indeed, to obtain $u : \cM \otimes_P \cM \otimes_Q H \rightarrow \cM \otimes_Q \cM \otimes_P H$ with the desired properties, we simply reverse the construction above. That is, we define $T = v \otimes I$, and}
$$u = W_{Q,P} \circ T \circ W_{P,Q}^* .$$
\end{remark}

\begin{lemma}\label{lem:SC3}
Assume that $P_j$ and $P_{j+1}$ commute strongly, for some $j\leq n-2$. Then there exists a unitary
$$u : \cM \otimes_{P_1} \cdots \otimes_{P_j} \cM \otimes_{P_{j+1}} \cdots \cM \otimes_{P_n} H
\rightarrow
\cM \otimes_{P_1} \cdots \otimes_{P_{j+1}} \cM \otimes_{P_{j}} \cdots \cM \otimes_{P_n} H
$$
such that
\begin{enumerate}
    \item $u(I \otimes_{P_1} \cdots I \otimes_{P_j} I \otimes_{P_{j+1}}I \cdots I \otimes_{P_n} h) = I \otimes_{P_1} \cdots I \otimes_{P_{j+1}} I \otimes_{P_j}I \cdots I \otimes_{P_n} h$,
    \item For all $X \in \cM$,
    $$u \circ (X \otimes I \cdots I \otimes I)
    = (X \otimes I \cdots  I \otimes I)\circ u ,$$
    \item For all $X \in \cM'$,
    $$u \circ (I \otimes I \cdots I \otimes X)
    = (I \otimes I \cdots  I \otimes X)\circ u .$$
\end{enumerate}
\end{lemma}
\begin{proof}
Let $v : \cM \otimes_{P_{j}} \cM \otimes_{P_{j+1}} \cM \rightarrow
\cM \otimes_{P_{j+1}} \cM \otimes_{P_{j}} \cM$ be the unitary that
is described in lemma \ref{lem:SC2}. Introduce the notation
$$E = \cM \otimes_{P_{1}} \cdots \otimes_{P_{j-2}} \cM$$
(understood to be $\mathbb{C}$ if $j=1$ and $\cM$ if $j=2$) and
$$F = \cM \otimes_{P_{j+3}} \cdots \cM \otimes_{P_{n}} H$$
(understood to be $H$ if $j=n-2$). Define
$$u : E \otimes_{P_{j-1}} \cM \otimes_{P_{j}} \cM \otimes_{P_{j+1}} \cM \otimes_{P_{j+2}} F
\rightarrow
E \otimes_{P_{j-1}} \cM \otimes_{P_{j+1}} \cM \otimes_{P_{j}} \cM \otimes_{P_{j+2}} F$$
by
$$u := I_E \otimes v \otimes I_F .$$
$u$ is a well-defined, unitary mapping, possessing the properties asserted.
\end{proof}

Putting together Lemmas \ref{lem:SC1}, \ref{lem:SC2} and
\ref{lem:SC3}, we obtain the following
\begin{proposition}\label{prop:SC}
Let $R_1, R_2, \ldots R_m$, and $S_1, S_2, \ldots, S_n$ be CP maps
such that for all $1 \leq i \leq m$, $1 \leq j \leq n$, $R_i$
commutes strongly with $S_j$. Then there exists a unitary
$$u : \cM \otimes_{R_1} \cdots \otimes_{R_m} \cM \otimes_{S_1} \cdots \otimes_{S_n} H
\rightarrow
\cM \otimes_{S_1} \cdots \otimes_{S_n} \cM \otimes_{R_1} \cdots \otimes_{R_m} H$$
such that
\begin{enumerate}
    \item $u(I \otimes_{R_1} I \cdots  I \otimes_{S_n} h) = I \otimes_{S_1} I \cdots I \otimes_{R_{m}} h$, for all $h \in H$,
    \item For all $X \in \cM$,
    $$u \circ (X \otimes I \cdots I \otimes I)
    = (X \otimes I \cdots  I \otimes I)\circ u ,$$
    \item For all $X \in \cM'$,
    $$u \circ (I \otimes I \cdots I \otimes X)
    = (I \otimes I \cdots  I \otimes X)\circ u .$$
\end{enumerate}
Furthermore, there exists an isomorphism of $\cM$ correspondences
\bes
v : \cM \otimes_{R_1} \cdots \otimes_{R_m} \cM \otimes_{S_1} \cdots \otimes_{S_n} \cM
\rightarrow
\cM \otimes_{S_1} \cdots \otimes_{S_n} \cM \otimes_{R_1} \cdots \otimes_{R_m} \cM
\ees
such that
\bes
v(I \otimes \cdots \otimes I) = I \otimes \cdots \otimes I.
\ees
\end{proposition}
The existence of $u$ and $v$ as above is clear: simply apply the
isomorphisms from the previous lemmas one by one. One might think
that applying these isomorphisms in different orders might lead to
different $u$'s or $v$'s. In the next section we will see, however, that
the order of application does not influence the total outcome (see
Proposition \ref{prop:SCSG}).

\section{Strong commutativity in terms of the GNS representation}\label{sec:GNS}

We now characterize strong commutativity using the GNS representation \cite{Pas}.
This characterization may be interesting for two reasons. First, it shows that
the notion of strong-commutativity is representation free. Second, it provides the connection
between the work of Bhat and Skeide on dilation of one-parameter CP-semigroups originating in \cite{BS00}, and
the constructions made in this thesis.

Let $\Theta$, $\Phi$ and $\cM$ be as in the previous subsection. We will denote by $(E,\xi)$ and $(F,\eta)$ the GNS representations of $\Theta$ and $\Phi$, respectively. That is $E = \cM \otimes_\Theta \cM$, the correspondence formed with the inner product
$$\langle a \otimes b, a' \otimes b' \rangle = b^* \Theta(a^* a') b',$$
$\xi = 1\otimes 1$, and one checks that $\langle \xi, a \xi \rangle = \Theta(a)$ for all $a \in \cM$. $(F , \eta)$ is defined
similarly. See section 2.14 in \cite{BS00}.

\begin{proposition}\label{prop:SCBS}
$\Theta$ and $\Phi$ commute strongly if and only if there exists a unitary (bimodule map)
$$w:E \otimes F \rightarrow F \otimes E$$
sending $\xi \otimes \eta$ to $\eta \otimes \xi$.
\end{proposition}
\begin{proof}
By Lemma \ref{lem:SC2} and the remarks right after, $\Theta$ and $\Phi$ strongly commute, i.e. there exists a unitary $u$
as in Definition \ref{def:SC}, if and only there exists an isomorphism of $\cM-\cM$-correspondences (a unitary bimodule map)
$$v : \cM \otimes_\Theta \cM \otimes_\Phi \cM \rightarrow \cM \otimes_\Phi \cM \otimes_\Theta \cM $$
such that
$$v(1 \otimes_\Theta 1 \otimes_\Phi 1) = 1 \otimes_\Phi 1 \otimes_\Theta 1 .$$
But
$$\cM \otimes_\Theta \cM \otimes_\Phi \cM \cong \left(\cM \otimes_\Theta \cM \right) \otimes \left(\cM \otimes_\Phi \cM \right) = E \otimes F$$
as $W^*$-correspondences via the correspondence isomorphism
$$a \otimes_\Theta b \otimes_\Phi c \mapsto \left(a \otimes_\Theta b \right) \otimes \left(1 \otimes_\Phi c \right).$$
Indeed, this is a well defined isometry because it preserves inner products:
\begin{align*}
\langle \left(a \otimes_\Theta b \right) \otimes \left(1 \otimes_\Phi c \right), \left(a' \otimes_\Theta b' \right) \otimes \left(1 \otimes_\Phi c' \right) \rangle
&= \langle 1 \otimes_\Phi c , b^* \Theta(a^* a') b' \left(1 \otimes_\Phi c' \right)\rangle \\
&= c^* \Phi(b^* \Theta(a^* a') b')c' \\
&= \langle a \otimes b \otimes c, a' \otimes b' \otimes c'\rangle.
\end{align*}
It is onto because the tensor product $E \otimes F$ is balanced. It is clear that this map is a bimodule map. Moreover,
this maps sends $1 \otimes 1 \otimes 1$ to $\xi \otimes \eta$. Thus, the existence of an isomorphism $v$ as above is equivalent to the existence of a an isomorphism
$$w:E \otimes F \rightarrow F \otimes E$$
sending $\xi \otimes \eta$ to $\eta \otimes \xi$.
\end{proof}

\section{Strongly commuting CP-semigroups}
\subsection{The definition}
How should strong commutativity be defined for semigroups of CP maps? A natural guess is the following:

\noindent {\bf NOT the definition: } \emph{Two semigroups of CP maps $\{R_t\}_{t\geq0}$ and
$\{S_t\}_{t\geq0}$ are said to \emph{commute strongly} if for all
$(s,t) \in \Rpt$ the CP maps $R_s$ and $S_t$ commute strongly.}

In \cite{ShalitCP0Dil}, the above was used as the definition of strong commutativity of CP-semigroups. However, we recently discovered that \cite{ShalitCP0Dil} contains a serious gap, and that the above definition is too weak to obtain the full results of \cite{ShalitCP0Dil} without making considerable modifications to the methods used. We therefore give here a more stringent definition, that works.

We need some notation. Let $\{R_t\}_{t\geq0}$ and $\{S_t\}_{t\geq0}$ be two semigroups of CP maps such that $R_s$ strongly commutes with $S_t$ for all $s,t \geq 0$.
Let $\{u_{s,t}\}_{(s,t) \in \Rpt}$ be a family of unitaries $u_{s,t}: \cM \otimes_{R_s} \cM \otimes_{S_t} H \rightarrow \cM \otimes_{S_t} \cM \otimes_{R_s} H$, making the $R_s$ and $S_t$ commute strongly. By Lemma \ref{lem:SC2}, there is a family $\{v_{s,t}\}_{(s,t)\in\Rpt}$ of isomorphisms between $\cM$-correspondences $v_{s,t}: \cM \otimes_{R_s} \cM \otimes_{S_t} \cM \rightarrow \cM \otimes_{S_t} \cM \otimes_{R_s} \cM$ sending $I \otimes I \otimes I$ to $I \otimes I \otimes I$. By Remark \ref{rem:SC}, the existence of such a family $\{v_{s,t}\}_{(s,t)\in\Rpt}$ implies that $R_s$ strongly commutes with $S_t$ for all $s,t \geq 0$, and gives rise to a family $\{u_{s,t}\}_{(s,t) \in \Rpt}$.

Fix $s,s',t \geq 0$. We define an isometry
\bes
\cM \otimes_{R_{s+s'}} \cM \otimes_{S_{t}} \cM \rightarrow  \cM \otimes_{R_{s}} \cM \otimes_{R_{s'}} \cM \otimes_{S_{t}} \cM
\ees
by
\bes
a \otimes_{R_{s+s'}} b \otimes_{S_{t}} c \mapsto a \otimes_{R_{s}} I \otimes_{R_{s'}} b \otimes_{S_{t}} c.
\ees
We also define an isometry
\bes
\cM \otimes_{S_{t}} \cM \otimes_{R_{s+s'}} \cM \rightarrow  \cM \otimes_{S_{t}} \cM \otimes_{R_{s}} \cM \otimes_{R_{s'}} \cM
\ees
by
\bes
a \otimes_{S_{t}} b \otimes_{R_{s+s'}} c \mapsto a \otimes_{S_{t}} b \otimes_{R_{s}} I \otimes_{R_{s'}} c.
\ees
We make similar definitions with the roles of $R$ and $S$ reversed.

\begin{definition}\label{def:SCSG}
Two semigroups of CP maps $\{R_t\}_{t\geq0}$ and
$\{S_t\}_{t\geq0}$ are said to \emph{commute strongly} if for all
$(s,t) \in \Rpt$ the CP maps $R_s$ and $S_t$ commute strongly, and if there is a family $\{v_{s,t}\}_{(s,t) \in \Rpt}$ of isomorphisms of $\cM$-correspondences $v_{s,t}: \cM \otimes_{R_s} \cM \otimes_{S_t} \cM \rightarrow \cM \otimes_{S_t} \cM \otimes_{R_s} \cM$ (making the $R_s$ and $S_t$ commute strongly) such that for all $s,s',t,t'\geq 0$ the following diagrams commute
\bes
\begin{CD}
\cM \otimes_{R_{s+s'}} \cM \otimes_{S_{t}} \cM @>v_{s+s',t}>> \cM \otimes_{S_{t}} \cM \otimes_{R_{s+s'}} \cM\\
@VVV @VVV\\
\cM \otimes_{R_{s}} \cM \otimes_{R_{s'}} \cM \otimes_{S_{t}} \cM @>(v_{s,t} \otimes I)(I \otimes v_{s',t})>> \cM \otimes_{S_{t}} \cM \otimes_{R_{s}} \cM \otimes_{R_{s'}} \cM
\end{CD}
\ees
and
\bes
\begin{CD}
\cM \otimes_{R_{s}} \cM \otimes_{S_{t+t'}} \cM @>v_{s,t+t'}>> \cM \otimes_{S_{t+t'}} \cM \otimes_{R_{s}} \cM\\
@VVV @VVV\\
\cM \otimes_{R_{s}} \cM \otimes_{S_{t}} \cM \otimes_{S_{t'}} \cM @>(I \otimes v_{s,t'})(v_{s,t} \otimes I)>> \cM \otimes_{S_{t}} \cM \otimes_{S_{t'}} \cM \otimes_{R_{s}} \cM
\end{CD}
\ees
where the vertical maps are the isometries from the above discussion.
\end{definition}

\begin{remark}
\emph{Because the maps $v_{s,t}$ above are unitaries, it follows that strong commutation is a symmetric relation.}
\end{remark}

The above definition means that not only do $R_s$ and $S_t$ strongly commute for all $s,t$, the way in which they strongly commute must be compatible with the semigroup structures of $R$ and $S$.

\subsection{More technicalities regarding strongly commuting semigroups}
\begin{proposition}\label{prop:SCSG}
If the CP-semigroups $\{R_t\}_{t\geq0}$ and $\{S_t\}_{t\geq0}$
commute strongly, then, for all $(s,t),(s',t') \in \Rpt$, the
associated maps
$$v_{R_s,S_t}: \cM \otimes_{R_s} \cM \otimes_{S_t} \cM \rightarrow \cM \otimes_{S_t} \cM \otimes_{R_s} \cM ,$$
and
$$v_{R_{s'},S_{t'}}: \cM \otimes_{R_{s'}} \cM \otimes_{S_{t'}} \cM \rightarrow \cM \otimes_{S_{t'}} \cM \otimes_{R_{s'}} \cM ,$$
(see lemma \ref{lem:SC2}) satisfy the following identity :
\be\label{eq:SCSG} (I \otimes I \otimes
v_{R_{s'},S_{t'}})(v_{R_s,S_t} \otimes I \otimes I) = (v_{R_s,S_t}
\otimes I \otimes I)(I \otimes I \otimes v_{R_{s'},S_{t'}}). \ee
\end{proposition}

\begin{proof}
Let $a,b,c,d,e \in \cM$. Assume that $v_{R_s,S_t} (a \otimes_{R_s}
b \otimes_{S_t} c) = \sum_{i=1}^m A_i \otimes_{S_t} B_i
\otimes_{R_s} C_i$, and that $v_{R_{s'},S_{t'}} (I
\otimes_{R_{s'}} d \otimes_{S_{t'}} e) = \sum_{j=1}^n \gamma_i
\otimes_{S_{t'}} \delta_j \otimes_{R_{s'}} \epsilon_j$. Operating
on $a \otimes_{R_s} b \otimes_{S_t} c \otimes_{R_{s'}} d
\otimes_{S_{t'}} e$ with the operator on the left-hand side of equation
(\ref{eq:SCSG}), we obtain
\begin{align*}
& (I \otimes I \otimes v_{R_{s'},S_{t'}})(v_{R_s,S_t} \otimes I \otimes I)(a \otimes b \otimes c \otimes d \otimes e) = \\
& (I \otimes I \otimes v_{R_{s'},S_{t'}}) \sum_{i=1}^m A_i \otimes B_i \otimes C_i \otimes d \otimes e = (*)\\
& \sum_{i=1}^m A_i \otimes B_i \otimes C_i \cdot v_{R_{s'},S_{t'}} (I \otimes d \otimes e) = \\
& \sum_{i=1}^m \sum_{j=1}^n A_i \otimes B_i \otimes C_i \gamma_j \otimes \delta_j \otimes \epsilon_j ,
\end{align*}
where the equality marked by (*) is justified because
$v_{R_{s'},S_{t'}}$ is a left $\cM$-module map. Operating on $a
\otimes_{R_s} b \otimes_{S_t} c \otimes_{R_{s'}} d
\otimes_{S_{t'}} e$ with the operator on the right-hand side of equation
(\ref{eq:SCSG}), we obtain
\begin{align*}
& (v_{R_s,S_t} \otimes I \otimes I)(I \otimes I \otimes v_{R_{s'},S_{t'}}) (a \otimes b \otimes c \otimes d \otimes e) = (*) \\
& (v_{R_s,S_t} \otimes I \otimes I)(a \otimes b \otimes c \cdot v_{R_{s'},S_{t'}} (I \otimes d \otimes e)) = \\
& \sum_{j=1}^n (v_{R_s,S_t} \otimes I \otimes I) (a \otimes b \otimes c \gamma_j \otimes \delta_j \otimes \epsilon_j) = \\
& \sum_{j=1}^n v_{R_s,S_t} (a \otimes b \otimes c \gamma_j) \otimes \delta_j \otimes \epsilon_j = (**)\\
& \sum_{j=1}^n v_{R_s,S_t} (a \otimes b \otimes c ) \cdot \gamma_j \otimes \delta_j \otimes \epsilon_j = \\
& \sum_{i=1}^m \sum_{j=1}^n A_i \otimes B_i \otimes C_i \gamma_j \otimes \delta_j \otimes \epsilon_j ,
\end{align*}
where the equality marked by (*) is justified because
$v_{R_{s'},S_{t'}}$ is a left $\cM$-module map, and the one marked
by (**) is OK because $v_{R_{s},S_{t}}$ is a right $\cM$-module
map. So equation (\ref{eq:SCSG}) holds for all $s,s',t,t' \geq$,
and this proof is complete.
\end{proof}

Let $\{R_t\}_{t\geq0}$ and $\{S_t\}_{t\geq0}$ be CP-semigroups such that for all $(s,t) \in \Rpt$ the CP maps $R_s$ and $S_t$ commute strongly.
Let $\mathfrak{p} = \{0 = s_0 < s_1
< \ldots < s_m = s\}$ be a partition of $[0,s]$. We define
$$H_{\mathfrak{p}}^R = \cM \otimes_{R_{s_1}} \cM \otimes_{R_{s_2-s_1}} \cdots \cM \otimes_{R_{s_m-s_{m-1}}} H$$
and we define (for a partition $\mathfrak{q}$) $H_{\mathfrak{q}}^S$ in a similar manner. If $q = \{0 = t_0 < t_1 < \ldots < t_n = t\}$, we also define
$$H_{\mathfrak{p},\mathfrak{q}}^{R,S} = \cM \otimes_{R_{s_1}} \cdots \otimes_{R_{s_m-s_{m-1}}} \cM \otimes_{S_1}
\cdots \otimes_{S_{t_n-t_{n-1}}} H .$$
$H_{\mathfrak{q},\mathfrak{p}}^{S,R}$ is defined similarly.

Now for any two such partitions $\mathfrak{p},\mathfrak{q}$, Proposition \ref{prop:SC} gives
a unitary
$$
U_{\mathfrak{p},\mathfrak{q}} : H_{\mathfrak{p},\mathfrak{q}}^{R,S} \rightarrow H_{\mathfrak{q},\mathfrak{p}}^{S,R}
$$
with desirable properties. Our definition of strong commutativity is such that forces these unitaries to be consistent one with the other.

Let us describe precisely in what sense these unitaries are consistent. We fix our attentions to two intervals, $[0,s]$ and $[0,t]$. Let us order the set of all partitions of an interval by refinement. If $\mathfrak{p} \leq \mathfrak{p'}$  are partitions of $[0,s]$, and $\mathfrak{q} \leq \mathfrak{q'}$ are partitions of $[0,t]$ we define an isometry
\bes
V^{R,S}_{(\mfp, \mfq),(\mfp' , \mfq')}: H_{\mathfrak{p},\mathfrak{q}}^{R,S} \rightarrow H_{\mathfrak{p'},\mathfrak{q'}}^{R,S}
\ees
as follows.
In the particular case where $\mfp = \mfp'$ and $\mfq = \{0 = t_0 < \cdots < t_k < t_{k+1}< \cdots < t_n = t\}$ and $\mfq' = \{0 = t_0 < \cdots < t_k < \tau < t_{k+1}< \cdots < t_n = t\}$, we define
\begin{align*}
& V^{R,S}_{(\mfp, \mfq),(\mfp' , \mfq')} (\xi \otimes_{R_{s_m-s_{m-1}}} T_1 \otimes_{S_{t_1}} \cdots \otimes_{S_{t_k - t_{k-1}}} T_{k+1} \otimes_{S_{t_{k+1} - t_{k}}} T_{k+2} \otimes  \cdots T_n \otimes_{S_{t_n-t_{n-1}}} \otimes h) = \\
& \xi \otimes_{R_{s_m-s_{m-1}}} T_1 \otimes_{S_{t_1}} \cdots \otimes_{S_{t_k - t_{k-1}}} T_{k+1} \otimes_{S_{\tau - t_k}} I \otimes_{S_{t_{k+1} - \tau}} T_{k+2} \otimes  \cdots  T_n \otimes_{S_{t_n-t_{n-1}}} \otimes h,
\end{align*}
(we ``inserted an $I$ at time $\tau$") where $\xi \in \cM \otimes_{R_{s_1}} \cdots \otimes_{R_{s_{m-1} - s_{m-2}}} \cM$, $T_1, \ldots, T_n$ are in $\cM$ and $h \in H$.
If $\mfp'$ is a one-point-refinement of $\mfp$ and $\mfq = \mfq'$ then we define $V^{R,S}_{(\mfp, \mfq),(\mfp' , \mfq')}$ in a similar manner. Finally, for arbitrary $\mfp \leq \mfp'$ and $\mfq \leq \mfq'$ we define $V^{R,S}_{(\mfp, \mfq),(\mfp' , \mfq')}$ by composing the one-point-refinement maps.

In the same way, we define maps
\bes
V^{S,R}_{(\mfq, \mfp),(\mfq' , \mfp')}: H_{\mfq,\mfp}^{S,R} \rightarrow H_{\mfq',\mfp'}^{S,R}.
\ees

One can see that Definition \ref{def:SCSG} is equivalent to the following definition.
\begin{definition}\label{def:SCSG2}
Two semigroups of CP maps $\{R_t\}_{t\geq0}$ and
$\{S_t\}_{t\geq0}$ are said to \emph{commute strongly} if for all
$(s,t) \in \Rpt$ the CP maps $R_s$ and $S_t$ commute strongly, and if there is a family $\{u_{s,t}\}_{(s,t) \in \Rpt}$ of unitaries $u_{s,t}: \cM \otimes_{R_s} \cM \otimes_{S_t} H \rightarrow \cM \otimes_{S_t} \cM \otimes_{R_s} H$ (making the $R_s$ and $S_t$ commute strongly), such that the maps $U_{\mfp,\mfq}$ obtained from the family $\{u_{s,t}\}_{(s,t) \in \Rpt}$ as in Proposition \ref{prop:SC} satisfies for all $s,t \geq 0$ and all partitions $\mfp \leq \mfp'$ of $[0,s]$ and $\mfq \leq \mfq'$ of $[0,t]$,
\be\label{eq:superSCSG}
V^{S,R}_{(\mfq, \mfp),(\mfq' , \mfp')} \circ U_{\mfp,\mfq} = U_{\mfp',\mfq'} \circ V^{R,S}_{(\mfp, \mfq),(\mfp' , \mfq')}.
\ee
\end{definition}

\begin{remark}\emph{
To prove that the above definition implies Definition \ref{def:SCSG} one may imitate the proof of Lemma \ref{lem:SC2}, and it suffices to assume that equation (\ref{eq:superSCSG}) holds only for partitions of the form $\mfq = \mfq' = \{0,t\}$, $\mfp = \{0,s_1 + s_2\}$, $\mfp' = \{0,s_1,s_1 + s_2\}$, and $\mfp = \mfp' = \{0,s\}$, $\mfq = \{0,t_1 + t_2\}$, $\mfq' = \{0,t_1,t_1 + t_2\}$.}
\end{remark}

\section{Examples of strongly commuting maps and semigroups}\label{sec:examples}
In this section we give some examples of strongly commuting CP maps and of strongly commuting
CP-semigroups. In special cases we are able to state necessary and sufficient conditions for strong commutativity of CP maps.

\subsection{Endomorphisms, automorphisms, and composition with automorphisms}
By \cite[Lemma 5.4]{S06}, there
are plenty of examples of CP maps $\Theta$,$\Phi$ that commute
strongly:
\begin{enumerate}
\item\label{it:end} If $\Theta$ and $\Phi$ are endomorphisms that commute then they commute strongly.
\item\label{it:aut} If $\Theta$ and $\Phi$ commute, and if either one of them is an automorphism, then they commute strongly.
\item If $\alpha$ is a normal automorphism that commutes with $\Theta$, and $\Phi = \Theta \circ \alpha$, then $\Theta$ and $\Phi$ commute strongly.
\end{enumerate}

We note that item \ref{it:aut} does
not remain true if \emph{automorphism} is replaced by
\emph{endomorphism}. Here is an example.
\begin{example}
\emph{Take $\cM = B(\ell^2(\mathbb{N}))$, and identify every operator with its matrix representation with respect to the 
standard basis. Let $\Theta$ be the map that takes a matrix to its diagonal, and let $\Phi$ be given by 
conjugation with the right shift. 
$\Theta$ is a (unital) CP map, $\Phi$ is a (non-unital) $*$-endomorphism, 
these two maps commute, but not strongly. See \cite[Example 2.13]{ShalitCPDil} for details.}
\end{example}

\begin{proposition}\label{prop:SCendo}
Let $\alpha = \{\alpha_t\}_{t\geq 0}$, $\beta = \{\beta_t\}_{t\geq 0}$ be two E-semigroups acting on $\cM$, and assume that for all $s,t \geq 0$, $\alpha_s \circ \beta_t = \beta_t \circ \alpha_s$. Then $\alpha$ and $\beta$ strongly commute.
\end{proposition}
\begin{proof}
We construct a family $\{v_{s,t}\}_{(s,t)\in\Rpt}$ as required by Definition \ref{def:SCSG}.
Note that in $\cM \otimes_{\alpha_s} \cM \otimes_{\beta_t} \cM$ we have the equality $a \otimes_{\alpha_s} b \otimes_{\beta_t} c = a \otimes_{\alpha_s} I \otimes_{\beta_t} \beta_t(b) c$. Thus, there is an isomorphism $v_{s,t} : \cM \otimes_{\alpha_s} \cM \otimes_{\beta_t} \cM \rightarrow \cM \otimes_{\beta_t} \cM \otimes_{\alpha_s} \cM$ completely determined by the mapping
\bes
a \otimes_{\alpha_s} I \otimes_{\beta_t} c \mapsto a \otimes_{\beta_t} I \otimes_{\alpha_s} c.
\ees
Clearly, $v_{s,t}(I \otimes I \otimes I) = I \otimes I \otimes I$. We have yet to show that the family $\{v_{s,t}\}_{(s,t)\in\Rpt}$ participates in commutative diagrams as in Definition \ref{def:SCSG}. Denote by $V$ the isometry $\cM \otimes_{\alpha_{s_1+s_2}} \cM \otimes_{\beta_t} \cM \rightarrow \cM \otimes_{\alpha_{s_1}} \cM \otimes_{\alpha_{s_2}} \cM \otimes_{\beta_t} \cM$ sending $a \otimes b \otimes c$ to $a \otimes I \otimes b \otimes c$, and denote by $W$ the similar isometry $\cM \otimes_{\beta_t} \cM \otimes_{\alpha_{s_1+s_2}} \cM \rightarrow \cM \otimes_{\beta_t} \cM \otimes_{\alpha_{s_1}} \cM \otimes_{\alpha_{s_2}} \cM$. For all $a,c \in \cM$, we have
\begin{align*}
W (v_{s_1+s_2,t} (a \otimes_{\alpha_{s_1+s_2}} I \otimes_{\beta_t} c)) &= W (a \otimes_{\beta_t} I \otimes_{\alpha_{s_1+s_2}} c) \\
&= a \otimes_{\beta_t} I \otimes_{\alpha_{s_1}} I \otimes_{\alpha_{s_2}} c,
\end{align*}
while, on the other hand,
\begin{align*}
(v_{s_1,t} \otimes I)(I \otimes v_{s_2,t}) (V (a \otimes_{\alpha_{s_1+s_2}} I \otimes_{\beta_t} c ))
&= (v_{s_1,t} \otimes I)(I \otimes v_{s_2,t}) (a \otimes_{\alpha_{s_1}} I \otimes_{\alpha_{s_2}} I \otimes_{\beta_t} c) \\
&= (v_{s_1,t} \otimes I) (a \otimes_{\alpha_{s_1}} I \otimes_{\beta_t} I \otimes_{\alpha_{s_2}} c) \\
&= a \otimes_{\beta_t} I \otimes_{\alpha_{s_1}} I \otimes_{\alpha_{s_2}} c .
\end{align*}
That establishes one commutative diagram. The other is similar.
\end{proof}

At a first glance, this proposition might not seem very interesting in the
context of dilating CP-semigroups to endomorphism semigroups. However, we find this item \emph{very}
interesting, because one expects a good dilation theorem not to complicate the situation in any sense.
For example, in Theorem \ref{thm:scudil}, in order to prove the existence of an E-dilation, we will have to assume that the CP-semigroups $\{R_t\}_{t\geq0}$ and $\{S_t\}_{t\geq0}$ are unital, but the E-dilation that we will construct will also be unital.
Another example, again from Theorem \ref{thm:scudil}: if the CP-semigroups act on a type $I$ factor, then so does the minimal E$_0$-dilation that we shall construct. The importance of the above proposition is that it ensures that if $\{\alpha_t\}_{t\geq0}$ and $\{\beta_t\}_{t\geq0}$ are an E-dilation of $\{R_t\}_{t\geq0}$ and $\{S_t\}_{t\geq0}$ (a pair of strongly commuting CP-semigroups), then $\alpha$ and $\beta$ commute strongly.

\begin{proposition}
Let $\alpha = \{\alpha_t\}_{t\geq 0}$ be a semigroup of normal $*$-automorphisms on $\cM$, and let $\theta = \{\theta_t\}_{t\geq 0}$ be a CP-semigroup on $\cM$, and assume that for all $s,t \geq 0$, $\alpha_s \circ \theta_t = \theta_t \circ \alpha_s$. Then $\alpha$ and $\theta$ strongly commute.
\end{proposition}
\begin{proof}
The proof is similar to the proof of Proposition \ref{prop:SCendo}, and is omitted.
\end{proof}

\subsection{Semigroups on $B(H)$}
It is a well known fact that if $\Theta$ and $\Phi$ are
CP-semigroups on $B(H)$, then for each $t$ there are two
($\ell^2$-independent) row contractions $\{T_{t,i}\}_{i=1}^{m(t)}$ and
$\{S_{t,j}\}_{j=1}^{n(t)}$ ($m(t),n(t)$ may be equal to $\infty$) such that for all $a \in B(H)$
\be\label{eq:conj1}
\Theta_t(a) = \sum_{i}T_{t,i}aT_{t,i}^* ,
\ee
and
\be\label{eq:conj2}
\Phi_t(a) = \sum_{j}S_{t,j}aS_{t,j}^* \,.
\ee
We
shall call such semigroups \emph{conjugation semigroups}, as they
are given by conjugating an element with a row contraction. It now
follows from \cite[Proposition 5.8]{S06}, that for all $(s,t)\in\Rpt$, $\Theta_t$ and $\Phi_s$
commute strongly if and only if there is an $m(t)n(s)\times m(t)n(s)$
unitary matrix
$$u(s,t) = \left(u(s,t)_{(i,j)}^{(k,l)}\right)_{(i,j),(k,l)}$$
such that for all $i,j$,
\be\label{eq:SCunitary}T_{t,i}S_{s,j} =
\sum_{(k,l)}u(s,t)_{(i,j)}^{(k,l)}S_{s,l}T_{t,k} .
\ee
As a simple example, if $\Theta_s$ and $\Phi_t$ are given by (\ref{eq:conj1}) and (\ref{eq:conj2}), and $S_{t,j}$ commutes with $T_{s,i}$ for all $i,j$, then $\Phi_t$ and $\Theta_s$ strongly commute.

\subsection{Maps on $B(H)$, $H$ finite dimensional}\label{subsec:H_finite}
\begin{proposition}\label{prop:SC_H_finite}
Let $\Phi$ and $\Psi$ be two commuting CP maps on $B(H)$, with $H$ a finite
dimensional Hilbert space. Then $\Phi$ and $\Psi$ strongly commute.
\end{proposition}
\begin{proof}
Assume that $\Phi$ is given by
$$\Phi(a) = \sum_{i=1}^m S_i a S_i^* $$
and that $\Psi$ is given by
$$\Psi(a) = \sum_{j=1}^n T_j a T_j^* ,$$
where $\{S_1, \ldots , S_m\}$ and $\{T_1, \ldots , T_n\}$ are row contractions
and $m,n \in \mathbb{N}$.
Because $\Phi$ and $\Psi$ commute, we have that
$$\sum_{i,j=1}^{mn} S_i T_j a T_j^* S_i^* = \sum_{i,j=1}^{mn} T_j S_i a S_i^* T_j^* $$
for all $a \in B(H)$. By the lemma on page 153 of \cite{G04} this implies that there exists an $mn \times mn$ unitary matrix $u$ such that
$$ S_i T_j = \sum_{(k,l)}u_{(i,j)}^{(k,l)}T_l S_k ,$$
and this means precisely that $\Phi$ and $\Psi$ strongly commute.
\end{proof}

We note here that the lemma cited above is stated in \cite{G04} for unital CP maps, but the proof works for the non-unital case as well. The reason that the assertion of the proposition fails for $B(H)$ with $H$ infinite dimensional is that in that case we may have $mn=\infty$, and the lemma is only true for a CP maps given by finite sums.

\subsection{Conjugation semigroups on general von Neumann algebras}
Let $\cM$ be a von Neumann algebra acting on a Hilbert space $H$.
If $\Theta$ and $\Phi$ are CP maps
given by
\bes
\Theta(a) = \sum_{i=1}^m T_{i}aT_{i}^* ,
\ees
and
\bes
\Phi(a) = \sum_{j=1}^n S_{j}aS_{j}^* ,
\ees
where $T_{i},S_{j}$ are all in $\cM$, then a sufficient condition for strong commutation of $\Theta$ and $\Phi$ is the existence of a unitary matrix
$$u = \left(u_{(i,j)}^{(k,l)}\right)_{(i,j),(k,l)}$$
such that for all $i,j$,
\bes
T_{i}S_{j} = \sum_{(k,l)}u_{(i,j)}^{(k,l)}S_{l}T_{k} .
\ees
Indeed, by \cite[Proposition 5.6]{S06}, it is enough to show that there are
are two $\cM'$ correspondences $E$ and $F$, together with an
$\cM'$-correspondence isomorphism
$$t:E \otimes_{\cM'} F  \rightarrow F \otimes_{\cM'} E $$
and two c.c. representations $(\sigma,T)$ and $(\sigma,S)$ of $E$
and $F$, respectively, on $H$, such that:
\begin{enumerate}
\item\label{it:rep1} for all $a\in \cM$, $\widetilde{T}(I_E \otimes
a)\widetilde{T}^* = \Theta(a)$,
\item\label{it:rep2} for all $a\in \cM$, $\widetilde{S}(I_F \otimes
a)\widetilde{S}^* = \Phi(a)$,
\item\label{it:commute} $\widetilde{T}(I_E \otimes \widetilde{S}) = \widetilde{S}(I_F \otimes
\widetilde{T})\circ(t \otimes I_H)$.
\end{enumerate}
We construct these correspondences as follows. Let
$$E = \oplus_{i=1}^m \cM' \,\,\,\, {\rm and } \,\,\,\, F = \oplus_{j=1}^n \cM',$$
with the natural inner product and the natural actions of $\cM'$.
If we denote by $\{e_i \}_{i=1}^m$ and $\{f_j \}_{j=1}^n$ the
natural ``bases" of these spaces, then we can define
$$t(e_{i}\otimes f_{j}) =
\sum_{(k,l)}u_{(i,j)}^{(k,l)}f_{l}\otimes e_{k}.$$ We define
$\sigma$ to be the identity representation. Now $E \otimes_\sigma
H \cong \oplus_{i=1}^m H$, and $F \otimes_\sigma H \cong
\oplus_{j=1}^n H$, and on these spaces we define $\widetilde{T}$ and
$\widetilde{S}$ to be the row contractions given by $(T_1, \ldots ,
T_m)$ and $(S_1, \ldots, S_n)$. Some straightforward calculations
shows that items (\ref{it:rep1})-(\ref{it:commute}) are fulfilled.

\subsection{Maps on $\mathbb{C}^n$ or $\ell^\infty$}\label{subsec:SC_comm}
Let $\cM = \mathbb{C}^n$ or
$\ell^\infty(\mathbb{N})$, considered as the algebra of diagonal
matrices acting on the Hilbert space $H = \mathbb{C}^n$ or
$\ell^2(\mathbb{N})$. In this context, a unital CP map is just a
stochastic matrix, that is, a matrix $P$ such that $p_{ij}\geq 0$
for all $i,j$ and such that for all $i$,
$$\sum_{j}p_{ij} = 1.$$
Indeed, it is straightforward to check that such a matrix gives rise to a
normal, unital, completely positive map. On the other hand, for all $i$, the composition
of a normal, unital, completely positive map with the normal state projecting onto the $i$th element
must be a normal state, so it has to be given by a nonnegative element in $\ell^1$ with norm $1$. Thus, every CP map is given by such a matrix.

Given two such matrices $P$ and $Q$, we ask when do they strongly
commute. To answer this question, we first find orthonormal bases
for $\cM \otimes_P \cM \otimes_Q H$ and $\cM \otimes_Q \cM
\otimes_P H$. If $\{e_i\}$ is the column vector with $1$ in the $i$th
place and $0$'s elsewhere, it is easy to see that the set $\{e_i
\otimes_P e_j \otimes_Q e_k\}_{i,j,k}$ spans $\cM \otimes_P \cM
\otimes_Q H$, and $\{e_i \otimes_Q e_j \otimes_P e_k\}_{i,j,k}$
spans $\cM \otimes_Q \cM \otimes_P H$. We compute
\begin{align*}
\langle e_i \otimes_P e_j \otimes_Q e_k, e_m \otimes_P e_p
\otimes_Q e_q \rangle &= \langle e_k, Q(e_j^*P(e_i^*
e_m)e_p)e_q\rangle \\
&=  \delta_{i,m} \delta_{j,p} \delta_{k,q}q_{kj}p_{ji}.
\end{align*}
Thus,
$$\{(q_{kj}p_{ji})^{-1/2} \cdot e_i \otimes_P e_j \otimes_Q e_k
: i,j,k \,\,{\rm such \,\, that}\,\, q_{kj}p_{ji} \neq 0\}$$ is an
orthonormal basis for $\cM \otimes_P \cM \otimes_Q H$, and
similarly for $\cM \otimes_Q \cM \otimes_P H$. If $u: \cM
\otimes_P \cM \otimes_Q H \rightarrow \cM \otimes_Q \cM \otimes_P
H$ is a unitary that makes $P$ and $Q$ commute strongly, then for
all $i,k$ we must have
$$u(e_i\otimes_P a \otimes_Q e_k) = (e_i\otimes 1 \otimes e_k)u(e_i\otimes_P a \otimes_Q e_k)
= e_i \otimes_Q b \otimes_P e_k ,$$ thus for all $i,j$, the spaces
$V_{i,j}:=\{e_i\otimes_P a \otimes_Q e_k : a \in \cM\}$ and $W_{i,j}:=\{e_i\otimes_Q
a \otimes_P e_k : a \in \cM\}$ must be isomorphic. Thus, a
necessary condition for strong commutativity is that for all
$i,k$,
\be\label{eq:card}
|\{j :  q_{kj}p_{ji} \neq 0\} | = |\{j :
p_{kj}q_{ji} \neq 0\} | ,
\ee
where $| \cdot |$ denotes
cardinality. This condition is also sufficient, because we may define a
unitary between each pair $V_{i,j}$ and $W_{i,j}$, sending $e_i\otimes_P 1 \otimes_Q e_k$ to $e_i\otimes_Q 1 \otimes_P e_k$ and doing whatever on the complement. %
By the way, this example shows that when two CP maps commute strongly, there may be a great many unitaries
that implement the strong commutation.

One can impose certain block structures on $P$ and $Q$ that will
guarantee that (\ref{eq:card}) is satisfied. Since we are in
particularly interested in semigroups, we shall be content with
the following observation. Let $P$ and $Q$ be two commuting, \emph{irreducible}, stochastic matrices.
Then $\cP_t := e^{-t}e^{tP}$ and $\cQ_t := e^{-t}e^{tQ}$ are two commuting, stochastic semigroups with strictly positive elements, and thus for all $s,t$, $\cP_s$ and $\cQ_t$ commute strongly.

To illustrate, let
\[
P = \frac{1}{3}\left[
\begin{array}
[c]{ccc}%
1 & 1 & 1 \\
1 & 1 & 1 \\
1 & 1 & 1
\end{array}
\right]  \,\, , \,\,
Q = \left[
\begin{array}
[c]{ccc}%
{1}/{2} & 0 & {1}/{2} \\
{1}/{4} & {1}/{2} & {1}/{4} \\
{1}/{4} & {1}/{2} & {1}/{4}
\end{array}
\right]  .
\]
One may check that $P$ and $Q$ commute, but do not satisfy (\ref{eq:card}), hence they do not commute strongly. So we see that strong commutativity may fail even in the simplest cases. However, $P$ and $Q$ are \emph{irreducible}, thus for all $s,t$, $\cP_s$ and $\cQ_t$ commute strongly.

\begin{remark}
\emph{We wish to call to attention a curious fact: \emph{the above CP map $Q$ does not strongly commute with $Q^2$!}
(this is readily verified using the necessary and sufficient condition (\ref{eq:card})). This phenomenon suggests that strong commutativity may not be a healthy concept.
(However, note that $Q$ is not embeddable in a CP-semigroup. Indeed, if it were it would have to be invertible).}
\end{remark}

\subsection{Quantized convolution semigroups}\label{subsec:QCS}

In \cite{Markiewicz}, Daniel Markiewicz introduced and studied a class of CP$_0$-semigroups, called \emph{quantized convolution semigroups}. Below we show that every quantized convolution semigroup strongly commutes with any other.

Let $\{W_z: z\in \mb{C}\}$ denote the system of Weyl unitaries on $L^2(\mb{R})$. That is,
\bes
W_{a+ib} = e^{i\frac{ab}{2}}e^{iaQ}e^{ibP},
\ees
where
\bes
P = \frac{1}{i}\frac{d}{dx} \,\, , \,\, Q = M_x .
\ees

The Weyl unitaries span a strong operator dense subspace of $B(L^2(\mb{R}))$, and they satisfy the following relations:
\bes
W_z W_{w} = e^{i\Im(z\overline{w})/2}W_{z+w}
\ees
and therefore also
\be\label{eq:W2}
W_z W_{w} = e^{i\Im(z\overline{w})}W_{w}W_z .
\ee
($\Im(z)$ denotes the imaginary part of a complex number $z$).

\begin{lemma}\label{lem:daniel}{\bf (Claim 5.2, \cite{Markiewicz})}
Let $\rho$ be a finite variation Borel complex measure on $\mb{C}$. If
\bes
\int_\mb{C} W_z d\rho(z) = 0,
\ees
then $\rho = 0$.
\end{lemma}

Every infinitely divisible Borel probability measure $\mu$ on $\mb{C}$ gives rise to a unique one-parameter family $\{\mu_t\}_{t\geq 0}$ of Borel probability measures on $\mb{C}$, satisfying $\mu_0 = \delta_0$, $\mu_s \ast \mu_t = \mu_{s+t}$, and $\mu_1 = \mu$.

\begin{definition}
Given an infinitely divisible Borel probability measure on $\mb{C}$, the \emph{quantized convolution semigroup associated with $\mu$} is the CP$_0$-semigroup $\{\phi_t^\mu\}_{t\geq 0}$ on $B(L^2(\mb{R}))$ given by
\bes
\phi_t^\mu(a) = \int_{\mb{C}} W_z a W_z^* d\mu_t(z) .
\ees
\end{definition}

\begin{proposition}\label{prop:main}
Let $\mu$ and $\nu$ be two infinitely divisible Borel probability measures on $\mb{C}$. Then for all $s,t \geq 0$, $\phi^\mu_s$ strongly commutes with $\phi^\nu_t$.
\end{proposition}

\begin{proof}
Write $H = L^2(\mb{R})$. It suffices to show that if $\mu$ and $\nu$ are two Borel probability measures on $\mb{C}$, then the CP maps $\phi$ and $\theta$ acting on $B(H)$, given by
\bes
\phi(a) = \int_{\mb{C}} W_z a W_z^* d\mu(z) ,
\ees
and
\bes
\theta(a) = \int_{\mb{C}} W_z a W_z^* d\nu(z) ,
\ees
strongly commute with each other.

For every $a,b \in B(H)$ and $h\in H$, we define $f_{a,b,h} \in L^2(\mu \times \nu) \otimes H$ by
\bes
f_{a,b,h}(z,w) = aW_z^*bW_w^*h.
\ees
Note that $f_{a,b,h} \in L^2(\nu \times \mu) \otimes H$, too.
\begin{lemma}\label{lem:fisdense}
$\{f_{a,b,h}:a,b \in B(H), h \in H\}$ is total in $L^2(\mu \times \nu) \otimes H$.
\end{lemma}
\begin{proof}
Denote $G = \overline{\textrm{span}}\{f_{a,b,h}: a,b \in B(H), h \in H \}$. In \cite[Lemma 5.1]{Markiewicz} it is shown that the set of functions $\{w \mapsto bW^*_w h : b \in B(H), h \in H\}$ is total in $L^2(\nu) \otimes H$. It follows that $G$ is invariant under $I \otimes B(L^2(\nu) \otimes H)$, thus the orthogonal projection  onto $G$ is in $B(L^2(\mu)) \otimes I \otimes I$. Thus, $G = M \otimes L^2(\nu) \otimes H$, with $M$ a closed subspace of $L^2(\mu)$. Let $F \in M^\perp$.
The proof will be completed by showing that $F = 0$.

For every $g,h \in H$,
\bes
\int_\mb{C}\int_\mb{C} \lel F(z) g, W_z^* W_w^* h \rir d\mu(z) d\nu(w) = \langle F \otimes 1 \otimes g, f_{I,I,h}\rangle = 0,
\ees
so $\int_\mb{C} W_w \left(\int_\mb{C} F(z) W_z g d\mu(z) \right) d \nu(w) =0$.
By two repeated applications of Lemma \ref{lem:daniel}, we find that $F = 0$.
\end{proof}
\begin{lemma}\label{lem:L2}
The map sending $a\otimes_\phi b \otimes_\theta h$ to the function $f_{a,b,h}$
extends to a unitary $V_{\phi,\theta} : B(H) \otimes_\phi
B(H) \otimes_\theta H \rightarrow L^2(\mu \times \nu) \otimes H$. For all $a,b,c \in B(H)$, $h \in H$
\bes
V_{\phi,\theta}(I \otimes_\phi I \otimes_\theta h) = f_{I,I,h},
\ees
and
\bes
V_{\phi,\theta}(ca \otimes_\phi b \otimes_\theta h) = (I \otimes I \otimes c) f_{a,b,h}.
\ees
\end{lemma}
\begin{proof}
\begin{align*}
\lel f_{a,b,h}, f_{a',b',h'} \rir &= \int_\mb{C} \int_\mb{C} \lel a W_z^* b W_w^* h, a' W_z^* b' W_w^* h' \rir d\mu(z) d\nu(w) \\
&= \int_\mb{C} \int_\mb{C} \lel h, W_w b^* W_z a^* a' W_z^* b' W_w^* h' \rir d\mu(z) d\nu(w) \\
&= \lel h, \theta(b^* \phi(a^* a')b')h'\rir \\
&= \lel a \otimes_\phi b \otimes_\theta h, a' \otimes_\phi b' \otimes_\theta h' \rir,
\end{align*}
thus, $V_{\phi,\theta}$ extends to an isometry. By Lemma \ref{lem:fisdense}, $V_{\phi,\theta}$ is surjective. The other stated properties of $V_{\phi,\theta}$ are obvious.
\end{proof}

\begin{lemma}
There exists a unitary
\bes
U_{\mu,\nu}: L^2(\mu \times \nu) \otimes H \rightarrow L^2(\nu \times \mu) \otimes H,
\ees
such that
\bes
U_{\mu,\nu} f_{I,I,h} = f_{I,I,h},
\ees
and
\bes
U_{\mu,\nu} (I \otimes I \otimes c) f_{a,b,h} = (I \otimes I \otimes c) U_{\mu,\nu} f_{a,b,h} .
\ees
\end{lemma}
\begin{proof}
For every continuous and bounded function $F$ on $\mb{C}^2$ and every $\varphi \in L^2(\mu \times \nu) \otimes H$ (or in $L^2(\nu \times \mu) \otimes H$), define $F\cdot \varphi$ to be the function
\bes
F\cdot \varphi (z,w) = F(z,w) \varphi(z,w).
\ees
We define $U_{\mu,\nu}$ by
\be\label{eq:U}
\left(U_{\mu,\nu} (F \cdot f_{a,W_y,h})\right) (z,w) = F(w,z)e^{i\Im(y\overline{w} + z\overline{y})}f_{a,W_y,h} (z,w).
\ee
It is easy to see that $\{a \otimes_\phi W_y \otimes_\theta h: a \in B(H), y \in \mb{C}, h \in H\}$ is total in $B(H) \otimes_\phi B(H) \otimes_\theta H$, thus $\{f_{a,W_y,h}:a \in B(H), y \in \mb{C}, h \in H\}$ is total in $L^2(\mu \times \nu) \otimes H$.
If we show that the mapping given by (\ref{eq:U}) preserves inner products, then it can be extended to the required unitary. Let $a,b \in B(H)$, $x,y \in \mb{C}$, $g,h \in H$ and let $G,F$ be bounded continuous functions on $\mb{C}^2$.
Then, using (\ref{eq:W2}) together with $W_z^*  = W_{-z}$,
\begin{align*}
&\lel U_{\mu,\nu} (G \cdot f_{a,W_x,g}), U_{\mu,\nu} (F \cdot f_{b,W_y,h}) \rir_{L^2(\nu \times \mu) \otimes H} \\
&= \int_\mb{C} \int_\mb{C} G(w,z)e^{i\Im(x\overline{w} + z\overline{x})-i\Im(y\overline{w} + z\overline{y})} \overline{F}(w,z) \lel a W_z^* W_x W_w^* g, b W_z^* W_y W_w^* h \rir d\nu(z) d\mu(w) \\
&= \int_\mb{C} \int_\mb{C} Ge^{i\Im(x\overline{w} + z\overline{x})-i\Im(y\overline{w} + z\overline{y})} \overline{F}e^{-i\Im(x\overline{w} + z\overline{x})+i\Im(y\overline{w} + z\overline{y})} \lel a W_w^* W_x W_z^* g, b W_w^* W_y W_z^* h \rir d\nu(z) d\mu(w) \\
&= \int_\mb{C} \int_\mb{C} G(w,z) \overline{F}(w,z) \lel a W_w^* W_x W_z^* g, b W_w^* W_y W_z^* h \rir d\mu(w) d\nu(z) \\
&= \lel G \cdot f_{a,W_x,g}, F \cdot f_{b,W_y,h} \rir_{L^2(\mu \times \nu) \otimes H}.
\end{align*}
\end{proof}

\emph{Completion of the proof of Proposition \ref{prop:main}.} The unitary $u_{\phi,\theta}: B(H) \otimes_\phi
B(H) \otimes_\theta H \rightarrow B(H) \otimes_\theta B(H)
\otimes_\phi H$ as in Definition \ref{def:SC} is given by
\bes
u_{\phi,\theta} = V^*_{\theta,\phi}\circ U_{\mu,\nu} \circ V_{\phi,\theta}.
\ees
\end{proof}

\begin{theorem}\label{thm:QCSG}
Let $\mu$ and $\nu$ be two infinitely divisible Borel probability measures on $\mb{C}$. Then  $\phi^\mu$ strongly commutes with $\phi^\nu$.
\end{theorem}
\begin{proof}
It is enough to prove that if $\mu$, $\nu$ and $\sigma$ are three Borel probability measures on $\mb{C}$, then the CP maps $\phi$, $\theta$ and $\psi$ acting on $B(H)$, given by
\bes
\phi(a) = \int_{\mb{C}} W_z a W_z^* d\mu(z) ,
\ees
\bes
\theta(a) = \int_{\mb{C}} W_z a W_z^* d\nu(z) ,
\ees
and
\bes
\psi(a) = \int_{\mb{C}} W_z a W_z^* d\sigma(z) ,
\ees
then the following diagram is commutative,
\bes
\begin{CD}
B(H) \otimes_{\theta} B(H) \otimes_{\phi \circ \psi} H @>{u_{\theta,\phi \circ \psi}}>> B(H) \otimes_{\phi \circ \psi} B(H) \otimes_{\theta} H\\
@VVV @VVV\\
B(H) \otimes_{\theta} B(H) \otimes_{\phi} B(H) \otimes_{\psi} H @>{(I \otimes u_{\theta,\psi})(v_{\theta,\phi} \otimes I)}>> B(H) \otimes_{\phi} B(H) \otimes_{\psi} B(H) \otimes_{\theta} H
\end{CD}
\ees
where the maps $u_{\theta,\phi},\ldots$ are the maps constructed in Proposition \ref{prop:main}, $v_{\theta,\phi}$ is the map associated with $u_{\theta,\phi}$ as in Lemma \ref{lem:SC2}, and the vertical maps are the same as in Definition \ref{def:SCSG}. First, some lemmas.

For every $a,b,c \in B(H)$ and $h\in H$, we define $f_{a,b,c,h} \in L^2(\mu \times \nu \times \sigma) \otimes H$ by
\bes
f_{a,b,c,h}(z,w,t) = aW_z^*bW_w^*cW_t^*h.
\ees

\begin{lemma}
The map sending $a\otimes_\phi b \otimes_\theta c \otimes_\psi h$ to the function $f_{a,b,c,h}$
extends to a unitary $V_{\phi,\theta,\psi} : B(H) \otimes_\phi
B(H) \otimes_\theta B(H) \otimes_\psi  H \rightarrow L^2(\mu \times \nu \times \sigma) \otimes H$. For all $a,b,c,d \in B(H)$, $h \in H$
\bes
V_{\phi,\theta,\psi}(I \otimes_\phi I \otimes_\theta I \otimes_\psi h) = f_{I,I,I,h},
\ees
and
\bes
V_{\phi,\theta,\psi}(da \otimes_\phi b \otimes_\theta c \otimes_\psi h) = (I \otimes I \otimes I \otimes d) f_{a,b,c,h}.
\ees
\end{lemma}
\begin{proof}
Just like Lemma \ref{lem:L2}.
\end{proof}

For $F$ a bounded continuous function on $\mb{C}\times \mb{C}$, we define a bounded continuous function $\hat{F}$ on $\mb{C}\times \mb{C} \times \mb{C}$ by
\bes
\hat{F}(z,w,t) = F(z,w+t).
\ees
We define a map $W: L^2(\mu \times \nu \ast \sigma) \otimes H \rightarrow L^2(\mu \times \nu \times \sigma) \otimes H$ by
\bes
F \cdot f_{a,b,h} \mapsto \hat{F} \cdot f_{a,b,I,h}.
\ees
We define a similar map $L^2(\nu \ast \sigma \times \mu) \rightarrow L^2(\nu \times \sigma \times \mu)$.
\begin{lemma}
$W$ is an isometry, intertwining the left actions of $B(H)$, that makes the following diagram commute:
\bes
\begin{CD}
B(H) \otimes_{\theta} B(H) \otimes_{\phi \circ \psi} H @>{V_{\theta,\phi \circ \psi}}>> L^2(\mu \times \nu \ast \sigma) \otimes H\\
@VVV @VV{W}V\\
B(H) \otimes_{\theta} B(H) \otimes_{\phi} B(H) \otimes_\psi H @>V_{\theta,\phi,\psi}>> L^2(\mu \times \nu \times \sigma) \otimes H
\end{CD}
\ees
\end{lemma}
\begin{proof}
The lemma follows from straightforward computations/observations. Let us check, for example, that $W$ is an isometry.
\begin{align*}
&\langle \hat{F} \cdot f_{a,b,I,h}, \hat{F'}\cdot f_{a',b',I,h'} \rangle =\\
&= \int_\mb{C} \int_\mb{C} \int_\mb{C} F(z,w+t)F'(z,w+t) \langle h, W_tW_w b^* W_z a^* a' W_z^* b' W_w^* W_t^* \rangle d \mu(z) d \nu(w) d \sigma(t) \\
&= \int_\mb{C} \int_\mb{C} \int_\mb{C}  F(z,w+t)F'(z,w+t)  \langle h, W_{t+w} b^* W_z a^* a' W_z^* b' W_{w+t}^* \rangle d \mu(z) d \nu(w) d \sigma(t) \\
&= \int_\mb{C} \int_\mb{C}   F(z,u)F'(z,u)  \langle h, W_{u} b^* W_z a^* a' W_z^* b' W_{u}^* \rangle d \mu(z) d \nu \ast \sigma (u)\\
&= \langle F \cdot f_{a,b,h}, F' \cdot f_{a',b',h'} \rangle .
\end{align*}
\end{proof}

\emph{Completion of the proof of Theorem \ref{thm:QCSG}.} The above lemmas reduce the task to showing that the following diagram commutes:
\bes
\begin{CD}
L^2(\mu \times \nu \ast \sigma) \otimes H @>>> L^2(\nu \ast \sigma \times \mu) \otimes H\\
@VVV @VVV\\
L^2(\mu \times \nu \times \sigma) \otimes H  @>>> L^2(\nu \times \sigma \times \mu) \otimes H
\end{CD}
\ees
where the maps are the natural ones arising all along this section.

Let $f_{a,W_y,h} \in L^2(\mu \times \nu \ast \sigma) \otimes H$. The arrow $L^2(\mu \times \nu \ast \sigma) \otimes H \rightarrow L^2(\nu \ast \sigma \times \mu) \otimes H$ takes $f_{a,W_y,h}$ to the function $g \in L^2(\nu \ast \sigma \times \mu) \otimes H$ given by
\bes
g(z,w) = e^{i\Im(y\overline{w} + z\overline{y})}f_{a,W_y,h} (z,w).
\ees
The map $L^2(\nu \ast \sigma \times \mu) \otimes H \rightarrow L^2(\nu \times \sigma \times \mu) \otimes H$ takes $g$ to a function $h \in L^2(\nu \times \sigma \times \mu) \otimes H$ given by
\bes
h(z,w,t) = e^{i\Im(y\overline{t} + (z+w)\overline{y})}f_{a,I,W_y,h} (z,w,t).
\ees
On the other hand, the arrow $L^2(\mu \times \nu \ast \sigma) \otimes H \rightarrow L^2(\mu \times \nu \times \sigma) \otimes H$ takes $f_{a,W_y,h}$ to the function $f_{a,W_y,I,h} \in L^2(\mu \times \nu \times \sigma) \otimes H$. The map $L^2(\mu \times \nu \times \sigma) \otimes H \rightarrow L^2(\nu \times \sigma \times \mu) \otimes H$ takes $f_{a,W_y,I,h}$ to a function $k \in L^2(\nu \times \sigma \times \mu) \otimes H$ given by
\bes
k(z,w,t) = e^{i\Im(y\overline{t} + z\overline{y})}f_{a,W_y,I,h}(z,w,t) = e^{i\Im(y\overline{t} + z\overline{y})} a W_z^* W_y W_w^* W_t^* h ,
\ees
(first, $f_{a,W_y,I,h}(z,w,t)$ is multiplied by the factor $e^{i\Im(y\overline{w} + z\overline{y})}$, and then $w$ and $t$ are swapped).
But,
\begin{align*}
e^{i\Im(y\overline{t} + z\overline{y})} a W_z^* W_y W_w^* W_t^* h &=
e^{i\Im(y\overline{t} + z\overline{y})} e^{i\Im(-y\overline{w})} a W_z^* W_w^* W_y W_t^* h \\
&= e^{i\Im(y\overline{t} + (z+w)\overline{y})}  a W_z^* W_w^* W_y W_t^* h \\
&= h(z,w,t).
\end{align*}
That completes the proof.
\end{proof}

\chapter{E$_0$-dilation of strongly commuting CP$_0$-semigroups}\label{chap:unital}

In this chapter we prove one of the main results of this thesis:
\emph{every pair of strongly commuting CP$_0$-semigroups has an
E$_0$-dilation}. This is accomplished in the first three sections, which are borrowed from
\cite{ShalitCP0Dil}. In Section \ref{sec:type}, most of which is taken from \cite{ShalitWhatType}, we study the dilation that is constructed in terms of cocycle conjugacy. This analysis (together with results of Markiewicz) allows us to close the chapter with a class of concrete examples of strongly commuting CP$_0$-semigroups for which the E$_0$-dilation that we constructed is cocycle conjugate to the familiar CCR flow.

\section{Overview of the Muhly--Solel approach to dilation}\label{sec:MS}
In this section we describe the approach of Muhly and Solel to
dilation of CP-semigroups on von Neumann algebras. This approach
was used by Muhly and Solel to dilate CP-semigroups over
$\mathbb{N}$ and $\mathbb{R}_+$ \cite{MS02}, and later by Solel
for semigroups over $\mathbb{N}^2$ \cite{S06}. Our program is to
adapt this approach for semigroups over $\cS = \Rpt$.

\subsection{The basic strategy}\label{subsec:strategy}
Let $\Theta$ be a CP-semigroup over the semigroup $\cS$, usually
acting on a von Neumann algebra $\cM$ of operators in $B(H)$. The
dilation is carried out in two main steps. In the first step, a
product system of $\cM'$-correspondences $X$ over $\cS$
is constructed, together with a c.c. representation $(\sigma,T)$
of $X$ on $H$, such that for all $a \in \cM, s \in \cS$,
\be\label{eq:rep_rep}
\Theta_s (a) = \widetilde{T_s} \left(I_{X(s)}
\otimes a \right) \widetilde{T_s}^* ,
\ee
where $T_s$ is the
restriction of $T$ to $X(s)$. In \cite[Proposition 2.21]{MS02}, it
is proved that for any c.c. representation $(\sigma, T)$ of a
$W^*$-correspondence $\cE$ over a $W^*$-algebra $\cN$, the mapping
$a \mapsto \widetilde{T_s} \left(I_{X(s)} \otimes a \right)
\widetilde{T_s}^*$ is a normal, completely positive map on
$\sigma(\cN)'$ (for all $s$). It is also shown that if $T$ is
isometric then this map is multiplicative. Having this in mind,
one sees that a natural way to continue the process of dilation
will be to ``dilate" $(\sigma, T)$ to an isometric representation.
\begin{definition}
Let $\cA$ be a $C^*$-algebra, $X$ be a product system of
$\cA$-correspondences over the semigroup $\cS$, and $(\sigma, T)$
a c.c. representation of $X$ on a Hilbert space $H$. An
\emph{isometric dilation} of $(\sigma, T)$ is an isometric
representation $(\rho, V)$ of $X$ on a Hilbert space $K \supseteq
H$, such that
\begin{itemize}
    \item[(i)] $H$ reduces $\rho$ and $\rho(a)\big|_H = P_H \rho(a)\big|_H = \sigma (a)$, for all $a\in\cA$;
    \item[(ii)] for all $s \in \cS, x \in X_s$, one has   $P_H V_s (x)\big|_{K \ominus H} = 0$;
    \item[(iii)] for all $s \in \cS, x \in X_s$, one has  $P_H V_s (x)\big|_{H} = T_s (x)$.
\end{itemize}
\end{definition}
Such a dilation is called \emph{minimal} in case the smallest
subspace of $K$ containing $H$ and invariant under every $V_s
(x)$, $x \in X, s \in \cS$, is all of $K$.

It will be convenient at times to regard an isometric dilation as
a quadruple $(K,u,V,\rho)$, where $(\rho, V)$ are as above and
$u:H \rightarrow K$ is an isometry.

Constructing a minimal isometric dilation $(K,u,V,\rho)$ of the
representation $(\sigma, T)$ appearing in equation
(\ref{eq:rep_rep}) constitutes the second step of the dilation
process. Then one has to show that if $\cR = \rho(\cM')'$, and
$\alpha$ is defined by \bes \alpha_s (a) := \widetilde{V_s}
\left(I_{X(s)} \otimes a \right) \widetilde{V_s}^* \,\,,\,\, a \in \cR
, \ees then the quadruple $(K,u,\cR,\alpha)$ is an E-dilation for
$(\Theta, \cM)$. In \cite{MS98}, \cite{MS02} and \cite{S06}, it is
proved that any c.c. representation of a product system over
$\mathbb{N}$, $\mathbb{R}_+$ or $\mathbb{N}^2$ (in the latter two,
$X$ is assumed to be a product system of $W^*$-correspondences, and
$\sigma$ is assumed to be normal), has a minimal isometric
dilation. Moreover, it is shown that if $X$ is a product system of
$W^*$-correspondences and $\sigma$ is assumed to be normal then
$\rho$ is also normal. When the product system is over
$\mathbb{N}$ or $\mathbb{R}_+$, the minimal isometric dilation is
also unique. From these results, the authors deduce the existence
of an E-dilation of a CP-semigroup $\Theta$ acting on a von
Neumann algebra $\cM$. When $\Theta$ is a CP-semigroup over $\cS = \mathbb{R}_+$ and $H$
is separable, then $\alpha$ is shown to be an
E-semigroup that is a minimal dilation.

\subsection{Description of the construction of the product system and representation for one parameter semigroups}\label{subsec:des_MS}
In this subsection we give a detailed description of Muhly and
Solel's construction of the product system and c.c. representation
associated with a one-parameter CP-semigroup \cite{MS02}. This construction lies on the foundations set by Arveson in \cite{Arv97b}. We
shall use this construction in Section \ref{sec:rep_SC}. We note that
the original construction in \cite{MS02} was carried out for CP$_0$-semigroups,
but it works just as well for CP-semigroups, and that no use is made of the continuity with respect to $t$.

Let $\Theta = \{\Theta_t \}_{t \geq 0}$ be a CP-semigroup
acting on a von Neumann algebra $\cM$ of operators in $B(H)$. Let $\mfB(t)$ denote the collection of partitions of the
closed unit interval $[0,t]$, ordered by refinement. For $\mfp \in
\mfB(t)$, we define a Hilbert space $H_{\mfp,t}$ by
$$H_{\mfp,t} : = \cM \otimes_{\Theta_{t_1}} \cM \otimes_{\Theta_{t_2 - t_1}} \cM \otimes \cdots \otimes_{\Theta_{t-t_{n-1}}} H ,$$
where $\mfp = \{0 = t_0 < t_1 < t_2 < \cdots < t_n = t \}$, and the right-hand side of the above equation is the Hausdorff completion of the algebraic tensor product $\cM \otimes \cM \otimes \cdots \otimes H$ with respect to the inner product
\begin{align*}\langle T_1 \otimes \cdots \otimes T_n \otimes h, & S_1 \otimes \cdots \otimes S_n \otimes k \rangle = \\
& \langle h, \Theta_{t-t_{n-1}} (T_n^* \Theta_{t_{n-1}-t_{n-2}} (T_{n-1}^* \cdots \Theta_{t_1} (T_1^* S_1) \cdots S_{n-1}) S_n) k \rangle .
\end{align*}
$H_{\mfp,t}$ is a left $\cM$-module via the action $S \cdot (T_1 \otimes \cdots \otimes T_n \otimes h) = ST_1 \otimes \cdots \otimes T_n \otimes h$. We now define the intertwining spaces
$$\cL_{\cM} (H,H_{\mfp,t}) = \{X\in B(H,H_{\mfp,t}): \forall S \in \cM . X S = S \cdot X \} .$$
The inner product
$$\langle X_1, X_2 \rangle := X_1^* X_2 ,$$
for $X_i \in \cL_{\cM} (H,H_{\mfp,t})$, together with the right and left actions
$$(XR)h: = X(Rh) ,$$
and
$$(RX)h := (I \otimes \cdots \otimes I \otimes R) Xh ,$$
for $R \in \cM', X \in \cL_{\cM} (H,H_{\mfp,t})$, make $\cL_{\cM} (H,H_{\mfp,t})$ into a $W^*$-correspondence over $\cM'$.

The Hilbert spaces $H_{\mfp,t}$ and $W^*$-correspondences $\cL_{\cM} (H,H_{\mfp,t})$ form inductive systems as follows. Let $\mfp,\mfp' \in \mfB(t)$, $\mfp \leq \mfp'$. In the particular case where $\mfp = \{0 = t_0 < \cdots < t_k < t_{k+1}< \cdots < t_n = t\}$ and $\mfp' = \{0 = t_0 < \cdots < t_k < \tau < t_{k+1}< \cdots < t_n = t\}$, we can define a Hilbert space isometry $v_0 : H_{\mfp,t} \rightarrow H_{\mfp',t}$ by
\begin{align*}
v_0 (T_1 \otimes \cdots \otimes T_{k+1} \otimes T_{k+2} \otimes & \cdots \otimes T_n \otimes h) = \\
& T_1 \otimes \cdots \otimes T_{k+1} \otimes I \otimes T_{k+2} \otimes \cdots \otimes T_n \otimes h.
\end{align*}
This map gives rise to an isometry of $W^*$-correspondences $v : \cL_{\cM} (H,H_{\mfp,t}) \rightarrow \cL_{\cM} (H,H_{\mfp',t})$ by $v(X) = v_0 \circ X$.

Now, if $\mfp \leq \mfp'$ are any partitions in $\mfB(t)$, then we can define $v_{0,\mfp,\mfp'} : H_{\mfp,t} \rightarrow H_{\mfp',t}$ and $v_{\mfp,\mfp'} : \cL_{\cM} (H,H_{\mfp,t}) \rightarrow \cL_{\cM} (H,H_{\mfp',t})$ by composing a finite number of maps such as $v_0$ and $v$ constructed in the previous paragraph, and we get legitimate arrow maps. Now one can form two different direct limits:
$$H_t := \underrightarrow{\lim }(H_{\mfp,t},v_{0,\mfp,\mfp'}) $$
and
$$E(t) := \underrightarrow{\lim }(\cL_{\cM} (H,H_{\mfp,t}),v_{\mfp,\mfp'}) .$$
The inductive limit also supplies us with embeddings of the blocks $v_{0,\mfp,\infty} : H_{\mfp,t} \rightarrow H_{t}$ and $v_{\mfp,\infty} : \cL_{\cM} (H,H_{\mfp,t}) \rightarrow E(t)$. One can also define intertwining spaces $\cL_{\cM} (H,H_{t})$, each of which has the structure of an $\cM'$-correspondence, and these spaces are isomorphic as $W^*$-correspondences to the spaces $E(t)$. $\{E(t)\}_{t \geq 0}$ is the product system of $\cM'$-correspondences that we are
looking for. We have yet to describe the c.c. representation $(\sigma, T)$ that will ``represent" $\Theta$ as in
equation (\ref{eq:rep_rep}) (with $X(s)$ replaced by $E(s)$).

The sought after representation is the so called ``identity representation", which we now describe. First, we set $\sigma = T_0 = {\bf id}_{\cM'}$. Next, let $t > 0$. For $\mfp = \{0 = t_0 < \cdots < t_n = t\}$, the formula
$$\iota_\mfp (h) = I \otimes \cdots \otimes I \otimes h$$
defines an isometry $\iota_\mfp : H \rightarrow H_{\mfp,t}$, with adjoint given by the formula
$$\iota_{\mfp}^*(X_1 \otimes \cdots \otimes X_n \otimes h) = \Theta_{t-t_{n-1}}(\Theta_{t_{n-1}-t_{n-2}}(\cdots (\Theta_{t_1}(X_1)X_2) \cdots X_{n-1})X_n)h .$$
For $\mfp'$ a refinement of $\mfp$, one computes $\iota_{\mfp}^* = \iota_{\mfp'}^* \circ v_{0,\mfp,\mfp'}$. This induces a unique map $\iota_t^* : H_t \rightarrow H$ that satisfies $\iota_t^* \circ v_{0,\mfp,\infty} = \iota_{\mfp}^*$. The c.c. representation $T_t$ on $E(t)$ is given by
$$T_t (X) = \iota_t^* \circ X ,$$
where we have identified $E(t)$ with $\cL_{\cM} (H,H_{t})$.

\section{Representing strongly commuting CP-semigroups}\label{sec:rep_SC}
As we mentioned in the previous section, our
program is to prove that every two strongly commuting CP$_0$-semigroups have an E$_0$-dilation using the Muhly-Solel approach,
which consists of two main steps. In this section we concentrate
on the first step: the representation of a pair of strongly
commuting CP-semigroups using a product system representation
via a formula such as equation (\ref{eq:rep_rep}) above.

Throughout this section and the following one, $\cM$ will be a von Neumann algebra
acting on a Hilbert space $H$. There is a natural correspondence
between two parameter semigroups of maps and pairs of commuting
one parameter semigroups. Indeed, if $\{R_t\}_{t\geq0}$ and
$\{S_t\}_{t\geq0}$ are two semigroups that commute (that is, for
all $t,s\geq0$, $R_s S_t = S_t R_s$) then we can define a two
parameter semigroup $P_{(s,t)} = R_s S_t$. And if we begin with a
semigroup $\{P_{(t,s)}\}_{(t,s)\in \Rpt}$, then we can define a
commuting pair of semigroups by $R_t = P_{(t,0)}$ and $S_t =
P_{(0,t)}$. It is not trivial that $P$ is continuous (in the relevant sense)
 if and only if
$R$ and $S$ are -- it follows from the fact that $(s,X) \mapsto
R_s(X)$ is jointly continuous in the weak topology (we shall make
this argument precise in Lemma \ref{lem:continuity1}). From now on
we fix the notation in the preceding paragraph, and we shall use
either $\{P_{(t,s)}\}_{(t,s)\in \Rpt}$ or the pair
$\{R_t\}_{t\geq0}$ and $\{S_t\}_{t\geq0}$ to denote a fixed
two-parameter CP-semigroup, and we shall assume in addition that $R$ and $S$ strongly commute. Note also that if
$\{\alpha_t\}_{t\geq0}$ and $\{\beta_t\}_{t\geq0}$ are
\emph{commuting} E-dilations of $\{R_t\}_{t\geq0}$ and
$\{S_t\}_{t\geq0}$ acting on the same von Neumann algebra, then
$\{\alpha_t \beta_s\}_{t,s\geq0}$ is an E-dilation of
$\{P_{(t,s)}\}_{(t,s)\in \Rpt}$, and vice versa.

\subsection{Representing a pair of strongly commuting CP-semigroups
 via the identity representation}\label{subsec:repvia}

By the discussion in Section \ref{subsec:SC_comm}, two CP maps that commute strongly may do so in more than one way.
Once and for all we fix a family $\{u_{s,t}\}_{(s,t) \in\Rpt}$ of unitaries that makes $R$ and $S$ commute strongly
(as in Definition \ref{def:SCSG}), and we also fix
the family of corresponding associated maps $\{v_{R_s,S_t}\}$ (as in Lemma \ref{lem:SC2}).

Let $\{E(t)\}_{t\geq0}$,
$\{F(t)\}_{t\geq0}$ denote the product systems (of
$W^*$-correspondences over $\cM '$) associated with
$\{R_t\}_{t\geq0}$ and $\{S_t\}_{t\geq0}$, respectively, and let
$T^E$, $T^F$ be the corresponding identity representations (as
described in Section \ref{subsec:des_MS}). For $s,t \geq 0$, we
denote by $\theta_{s,t}^E$ and $\theta_{s,t}^F$ the unitaries
$$\theta_{s,t}^E : E(s)\otimes_{\cM'}E(t) \rightarrow E(s+t) ,$$
and
$$\theta_{s,t}^F : F(s)\otimes_{\cM'}F(t) \rightarrow F(s+t) .$$

\begin{proposition}\label{prop:tech}
For all $s,t \geq 0$ there is an isomorphism of
$W^*$-correspondences \be \varphi_{s,t}:E(s) \otimes_{\cM'} F(t)
\rightarrow F(t) \otimes_{\cM'} E(s). \ee The isomorphisms
$\{\varphi_{s,t}\}_{s,t\geq0}$, together with the identity
representations $T^E$, $T^F$, satisfy
the ``commutation" relation: \be\label{eq:commutation_relation}
\tT_s^E (I_{E(s)} \otimes \tT_t^F) = \tT_t^F (I_{F(t)} \otimes
\tT_s^E) \circ (\varphi_{s,t} \otimes I_H) \quad , t,s \geq0 . \ee
\end{proposition}
\begin{proof}
We shall adopt the notation used in subsection \ref{subsec:des_MS}
(with a few changes), and follow the proof of
\cite[Proposition 5.6]{S06}. Fix $s,t \geq 0$. Let $\mathfrak{p} = \{0 = s_0 < s_1
< \ldots < s_m = s\}$ be a partition of $[0,s]$. We define
$$H_{\mathfrak{p}}^R = \cM \otimes_{R_{s_1}} \cM \otimes_{R_{s_2-s_1}} \cdots \cM \otimes_{R_{s_m-s_{m-1}}} H$$
and we define (for a partition $\mathfrak{q}$) $H_{\mathfrak{q}}^S$ in a similar manner. If $q = \{0 = t_0 < t_1 < \ldots < t_n = t\}$, we also define
$$H_{\mathfrak{p},\mathfrak{q}}^{R,S} = \cM \otimes_{R_{s_1}} \cdots \otimes_{R_{s_m-s_{m-1}}} \cM \otimes_{S_1}
\cdots \otimes_{S_{t_n-t_{n-1}}} H .$$
$H_{\mathfrak{q},\mathfrak{p}}^{S,R}$ is defined similarly. We can
go on to define
$H_{\mathfrak{q},\mathfrak{p},\mathfrak{p}'}^{S,R,S}$,
$H_{\mathfrak{q},\mathfrak{p},\mathfrak{q}',\mathfrak{p}'}^{S,R,S,R}$,
etc.

Recall that $E(s)$ is the direct limit of the directed system
$(\cL_\cM (H,H_{\mathfrak{p}}^R), v_{\mathfrak{p,p'}})$.
Similarly, we shall write $(\cL_\cM (H,H_{\mathfrak{q}}^S),
u_{\mathfrak{q,q'}})$ for the directed system that has $F(t)$ as
its limit. We write $v_{\mathfrak{p},\infty}$,
$u_{\mathfrak{q},\infty}$ for the limit isometric embeddings.

We proceed to construct an isomorphism
$$\varphi_{s,t} : E(s)\otimes F(t) \rightarrow F(t) \otimes E(s)$$
that has the desired property. Let $\mathfrak{p}  = \{0 = s_0 <
s_1 < \ldots < s_m = s\}$ and $\mathfrak{q}  = \{0 = t_0 < t_1 <
\ldots < t_n = t\}$ be partitions of $[0,s]$ and $[0,t]$,
respectively. Denote by $\Gamma_{\mfp,\mfq}$ the map from $\cL_\cM
(H,H_{\mathfrak{p}}^R) \otimes \cL_\cM (H,H_{\mathfrak{q}}^S)$
into $\cL_\cM (H,H_{\mathfrak{q},\mathfrak{p}}^{S,R})$ given by $X
\otimes Y \mapsto (I \otimes I \cdots I \otimes X)Y$. As explained
in \cite[Lemma 3.2]{MS02}, $\Gamma_{\mfp,\mfq}$ is an
isomorphism. We define $\Gamma_{\mfq,\mfp}$ to be the
corresponding map from $\cL_\cM (H,H_{\mathfrak{q}}^S) \otimes
\cL_\cM (H,H_{\mathfrak{p}}^R)$ into $\cL_\cM
(H,H_{\mathfrak{p},\mathfrak{q}}^{R,S})$. Let $u :
H_{\mathfrak{q},\mathfrak{p}}^{S,R} \rightarrow
H_{\mathfrak{p},\mathfrak{q}}^{R,S}$ be the isomorphism from
Proposition \ref{prop:SC}, and define $\Psi : \cL_\cM
(H,H_{\mathfrak{q},\mathfrak{p}}^{S,R}) \rightarrow \cL_\cM
(H,H_{\mathfrak{p},\mathfrak{q}}^{R,S})$ by $\Psi(Z) = u \circ Z$.
The argument from \cite[Proposition 5.6]{S06} can be repeated
here to show that $\Psi$ is an isomorphism of
$W^*$-correspondences. Define $t_{\mathfrak{p},\mathfrak{q}} :
\cL_\cM (H,H_{\mathfrak{p}}^R) \otimes \cL_\cM
(H,H_{\mathfrak{q}}^S) \rightarrow \cL_\cM (H,H_{\mathfrak{q}}^S)
\otimes \cL_\cM (H,H_{\mathfrak{p}}^R)$ by
$$t_{\mathfrak{p},\mathfrak{q}} = \Gamma_{\mfq,\mfp}^{-1}\circ \Psi \circ \Gamma_{\mfp,\mfq} .$$
Define maps $W_1 : H \rightarrow H_{\mathfrak{p}}^R$ and $W_2 : H
\rightarrow H_{\mathfrak{q}}^S$ by $W_1 h = I \otimes_{R_1} \cdots
I \otimes_{R_{s_m-s_{m-1}}}  h$ and $W_2 h = I \otimes_{S_1}
\cdots I \otimes_{S_{t_n-t_{n-1}}} h$. Also, let $U_1 :
H_{\mathfrak{p}}^R \rightarrow
H_{\mathfrak{q},\mathfrak{p}}^{S,R}$ and $U_2 : H_{\mathfrak{q}}^S
\rightarrow H_{\mathfrak{p},\mathfrak{q}}^{R,S}$ be the maps $U_1
\xi = I \otimes_{S_1} I \cdots I \otimes_{S_{t_n-t_{n-1}}} \xi$
and $U_2 \eta = I \otimes_{R_1} I \cdots I
\otimes_{R_{s_m-s_{m-1}}} \eta$. Just as in \cite{S06}, we have
that
\begin{equation}\label{eq:WU}
W_1^* U_1^* = W_2^* U_2^* u ,
\end{equation}
and that, for $X \in \cL_\cM (H,H_{\mathfrak{p}}^R)$, we have
$U_1^* (I \otimes \cdots I \otimes X) = X W_2^*$. Now, for $X \in
\cL_\cM (H,H_{\mathfrak{p}}^R)$ and $Y \in \cL_\cM
(H,H_{\mathfrak{q}}^S)$,
\begin{equation}\label{eq:UGamma}
U_1^* \Gamma_{\mfp,\mfq} (X \otimes Y) = U_1^* (I \otimes I \cdots
I \otimes X)Y = X W_2^* Y.
\end{equation}
If we define $(T_\mathfrak{p}^R, id)$ \footnote{Watch out - we
have here a little problem with notation - this resembles $T_t^E,
T_t^F$ that we defined above, but both the subscript and the superscript have here a different meaning.} to be the identity representation
of $\cL_\cM (H,H_{\mathfrak{p}}^R)$, and $(T_\mathfrak{q}^S, id)$
to be the identity representation of $\cL_\cM
(H,H_{\mathfrak{q}}^S)$, (see the closing paragraph in subsection
\ref{subsec:des_MS}), then (\ref{eq:UGamma}) implies that, for $h
\in H$, \be\label{eq:REP} W_1^* U_1^* (\Gamma_{\mfp,\mfq}(X
\otimes Y))h = T_\mathfrak{p}^R (X) T_\mathfrak{q}^S (Y)h =
\widetilde{T}_\mathfrak{p}^R (I \otimes \widetilde{T}_\mathfrak{q}^S)
(X\otimes Y\otimes h). \ee On the other hand, using (\ref{eq:WU})
and an analog of (\ref{eq:REP}),
\begin{align*}
W_1^* U_1^* (\Gamma_{\mfp,\mfq}(X \otimes Y))h & =
W_1^* U_1^* (\Psi^{-1} \Gamma_{\mfq,\mfp} \circ t_{\mathfrak{p},\mathfrak{q}} (X \otimes Y))h \\
& = W_1^* U_1^* u^*(\Gamma_{\mfq,\mfp} \circ t_{\mathfrak{p},\mathfrak{q}} (X \otimes Y))h \\
& = W_2^* U_2^* (\Gamma_{\mfq,\mfp} \circ t_{\mathfrak{p},\mathfrak{q}} (X \otimes Y))h \\
& = \widetilde{T}_\mathfrak{q}^S (I \otimes \widetilde{T}_\mathfrak{p}^R)(t_{\mathfrak{p},\mathfrak{q}} (X \otimes Y) \otimes h) .
\end{align*}

Let us summarize what we have accumulated up to this point. For
fixed $s,t \geq 0$, and any two partitions $\mathfrak{p}$,
$\mathfrak{q}$ of $[0,s]$ and $[0,t]$, respectively, we have a
Hilbert space isomorphism
$$t_{\mathfrak{p},\mathfrak{q}} : \cL_\cM (H,H_{\mathfrak{p}}^R) \otimes \cL_\cM (H,H_{\mathfrak{q}}^S) \rightarrow \cL_\cM (H,H_{\mathfrak{q}}^S) \otimes \cL_\cM (H,H_{\mathfrak{p}}^R)$$
satisfying \be\label{eq:com} \widetilde{T}_\mathfrak{p}^R (I \otimes
\widetilde{T}_\mathfrak{q}^S) = \widetilde{T}_\mathfrak{q}^S (I \otimes
\widetilde{T}_\mathfrak{p}^R)(t_{\mathfrak{p},\mathfrak{q}} \otimes
I_H) . \ee These maps induce an isomorphism
$t_{\mathfrak{p},\infty}: \cL_\cM (H,H_{\mathfrak{p}}^R) \otimes
F(t) \rightarrow F(t) \otimes \cL_\cM (H,H_{\mathfrak{p}}^R)$ that
satisfies
\be\label{eq:tpinfty} t_{\mathfrak{p},\infty} (I \otimes
u_{\mathfrak{q},\infty}) = (u_{\mathfrak{q},\infty} \otimes I)
t_{\mathfrak{p},\mathfrak{q}} .
\ee
(The definition we gave for strong commutativity of semigroups is tailor-made to make the previous sentence true).
Plugging (\ref{eq:tpinfty}) in
(\ref{eq:com}) we obtain
$$
\widetilde{T}_\mathfrak{p}^R (I \otimes \widetilde{T}_\mathfrak{q}^S) = \widetilde{T}_\mathfrak{q}^S (u_{\mathfrak{q},\infty}^* \otimes \widetilde{T}_\mathfrak{p}^R)(t_{\mathfrak{p},\infty} \otimes I_H) (I \otimes u_{\mathfrak{q},\infty} \otimes I_H) .
$$
The discussion before Theorem 3.9 in \cite{MS02} imply that
$\widetilde{T}_t^F (u_{\mathfrak{q},\infty} \otimes I) =
\widetilde{T}_{\mathfrak{q}}^S$, or, letting $p_\mathfrak{q}$ denote
the projection in $F(t)$ onto $u_{\mathfrak{q},\infty} (\cL_\cM
(H,H_{\mathfrak{q}}^S))$,
$$\widetilde{T}_t^F (p_{\mathfrak{q}} \otimes I) = \widetilde{T}_{\mathfrak{q}}^S  (u_{\mathfrak{q},\infty} ^* \otimes I_H).$$
The last two equations sum up to
$$
\widetilde{T}_\mathfrak{p}^R (I \otimes \widetilde{T}_t^F)(I \otimes p_{\mathfrak{q}} \otimes I_H) = \widetilde{T}_t^F (p_{\mathfrak{q}} \otimes \widetilde{T}_\mathfrak{p}^R)(t_{\mathfrak{p},\infty} \otimes I_H) (I \otimes p_{\mathfrak{q}} \otimes I_H) ,
$$
which implies, in the limit,
$$
\widetilde{T}_\mathfrak{p}^R (I \otimes \widetilde{T}_t^F) = \widetilde{T}_t^F (I_{F(t)} \otimes \widetilde{T}_\mathfrak{p}^R)(t_{\mathfrak{p},\infty} \otimes I_H)  .
$$
Repeating this ``limiting process" in the argument $\mathfrak{p}$,
we obtain a map $t_{\infty,\infty} : E(s) \otimes F(t) \rightarrow
F(t) \otimes E(s)$, which we re-label as $\varphi_{s,t}$, that
satisfies (\ref{eq:commutation_relation}). The above procedure can
be done for all $s,t \geq 0$, giving isomorphisms
$\{\varphi_{s,t}\}$ satisfying the commutation relation
(\ref{eq:commutation_relation}).
\end{proof}

Our aim now is to construct a product system $X$ over $\Rpt$ and a
c.c. representation $T$ of $X$ that will lead to a representation
of $\{P_{(s,t)}\}_{(s,t)\in \Rpt}$ as in equation
(\ref{eq:rep_rep}). Proposition \ref{prop:tech} is a key
ingredient in the proof that the representation that we define
below gives rise to such a representation. But before going into
that we need to carefully construct the product system $X$.

We define
$$X(s,t) := E(s) \otimes F(t) ,$$
and
$$\theta_{(s,t),(s',t')}: X(s,t) \otimes X(s',t') \rightarrow X(s+s',t+t') ,$$
by
$$\theta_{(s,t),(s',t')} =  (\theta_{s,s'}^E \otimes \theta_{t,t'}^F)\circ (I \otimes \varphi_{s',t}^{-1} \otimes I) .$$
To show that $\{ X(s,t) \}_{t,s\geq0}$ is a product system, we
shall need to show that ``the $\theta$'s make the tensor product
into an associative multiplication", or simply: \be\label{eq:ass}
\theta_{(s,t),(s'+s'',t'+t'')} \circ (I \otimes
\theta_{(s',t'),(s'',t'')}) = \theta_{(s+s',t+t'),(s'',t'')} \circ
(\theta_{(s,t),(s',t')} \otimes I) , \ee for
$s,s',s'',t,t',t''\geq 0$.

\begin{proposition}\label{prop:assoc}
$X = \{ X(s,t) \}_{t,s\geq0}$ is a product system. That is,
equation (\ref{eq:ass}) holds.
\end{proposition}
\begin{proof}
The proof is nothing but a straightforward and tedious computation, using Proposition
\ref{prop:SCSG}.

Let $s,s',s'',t,t',t'' \geq 0$, and let $\mfp, \mfp', \mfp'',
\mfq, \mfq', \mfq''$ be partitions of the corresponding intervals.
It is enough to show that the maps on both sides of equation
(\ref{eq:ass}) give the same result when applied to an element of
the form
$$\zeta = X \otimes Y \otimes X' \otimes Y' \otimes X'' \otimes Y'',$$
where $X \in \cL_\cM (H, H^R_\mfp)$, $Y \in \cL_\cM (H,
H^S_\mfq)$, etc. Let us operate first on $\zeta$ with the right-hand side of
(\ref{eq:ass}).

Now,
$$\theta_{(s,t),(s',t')} (X \otimes Y \otimes X' \otimes Y') =
\left( \theta_{\mfp,\mfp'}^E \otimes \theta_{\mfq,\mfq'}^F \right)
\left(X \otimes t_{\mfp',\mfq}^{-1}(Y \otimes X') \otimes Y'
\right) ,$$
where $\theta_{\mfp,\mfp'}^E$ is the restriction of
$\theta_{s,s'}^E$ to $\cL_\cM (H,H_\mfp^R) \otimes \cL_\cM
(H,H_{\mfp'}^R)$, $\theta_{\mfq,\mfq'}^F$ is defined similarly,
and $t_{\mfp',\mfq}$ is the map defined in Proposition
\ref{prop:tech}. Looking at the definition of $t_{\mfp',\mfq}$, we see
that $t_{\mfp',\mfq}^{-1} (Y \otimes X') =
\Gamma_{\mfp',\mfq}^{-1}\left(U_{\mfp' \leftrightarrow \mfq} \circ
(I\otimes Y) X' \right)$. Here $U_{\mfp' \leftrightarrow \mfq}$
denotes the unitary $H_{\mfp',\mfq}^{R,S} \rightarrow
H_{\mfq,\mfp'}^{S,R}$ given by Proposition \ref{prop:SC}. Assume
that
$$U_{\mfp' \leftrightarrow \mfq} \circ
(I\otimes Y) X' = \sum_{i}(I \otimes x_i)y_i .$$ Then
$$\Gamma_{\mfp',\mfq}^{-1}\left(U_{\mfp' \leftrightarrow \mfq} \circ
(I\otimes Y) X' \right) = \sum_i x_i \otimes y_i ,$$ therefore,
$$\theta_{(s,t),(s',t')} (X \otimes Y \otimes X' \otimes Y') =
\sum_i (I \otimes X)x_i \otimes (I \otimes y_i)Y'.$$
So,
\begin{align*}
& \theta_{(s+s',t+t'),(s'',t'')} \circ (\theta_{(s,t),(s',t')}
\otimes I) \zeta = \\
& \sum_i \left(\theta_{\mfp' \vee \mfp + s', \mfp''}^E \otimes
\theta_{\mfq' \vee \mfq + t', \mfq''}^F \right) \Big[(I \otimes
X)x_i \otimes \Gamma_{\mfp'', \mfq' \vee q + t'}^{-1} \left(
U_{\mfp'' \leftrightarrow \mfq' \vee q + t'} \circ (I \otimes (I \otimes
y_i)Y')X'' \right) \otimes Y'' \Big] .
\end{align*}
Repeated application of Proposition \ref{prop:SCSG} shows that,
and this is a crucial point, $U_{\mfp'' \leftrightarrow \mfq' \vee q + t'} = (I
\otimes U_{\mfp'' \leftrightarrow \mfq}) (U_{\mfp'' \leftrightarrow \mfq'} \otimes I)$. Thus
$$U_{\mfp'' \leftrightarrow \mfq' \vee q + t'} \circ (I \otimes (I \otimes
y_i)Y')X'' = (I \otimes U_{\mfp'' \leftrightarrow \mfq})(I \otimes I \otimes y_i)
\left(U_{\mfp'' \leftrightarrow \mfq'} (I \otimes Y')X'' \right).$$
Write $U_{\mfp'' \leftrightarrow \mfq'} \circ (I \otimes Y')X''$ as $\sum_j (I \otimes a_j)b_j$.
Then we have
\begin{align*}
U_{\mfp'' \leftrightarrow \mfq' \vee q + t'} \circ (I \otimes (I \otimes
y_i)Y')X''
& = \sum_j (I \otimes U_{\mfp'' \leftrightarrow \mfq})(I \otimes I \otimes y_i)
(I \otimes a_j)b_j \\
& = \sum_j \left(I \otimes \big[U_{\mfp'' \leftrightarrow \mfq} \circ (I \otimes y_i)a_j \big]
\right) b_j .
\end{align*}
We now write $U_{\mfp'' \leftrightarrow \mfq} \circ (I \otimes y_i)a_j$ as $\sum_k (I \otimes A_{i,j,k})B_{i,j,k}$.
With this notation, we get
\begin{align*}
& \theta_{(s+s',t+t'),(s'',t'')} \circ (\theta_{(s,t),(s',t')}
\otimes I) \zeta = \\
& \sum_{i,j,k} \left((I \otimes I \otimes X)(I \otimes x_i)A_{i,j,k} \right)
\otimes \left((I \otimes I \otimes B_{i,j,k})(I \otimes b_j)Y'' \right) .
\end{align*}

Now let us operate first on $\zeta$ with the left-hand side of
(\ref{eq:ass}), repeating all the steps that we have made above:
\begin{align*}
\theta_{(s',t'),(s'',t'')} (X' \otimes Y' \otimes X'' \otimes Y'')
& = \left( \theta_{\mfp',\mfp''}^E \otimes \theta_{\mfq',\mfq''}^F \right)
\left(X' \otimes t_{\mfp'',\mfq'}^{-1}(Y' \otimes X'') \otimes Y''
\right) \\
& = \sum_j (I \otimes X')a_j \otimes (I \otimes b_j)Y'',
\end{align*}
thus,
\begin{align*}
& \theta_{(s,t),(s'+s'',t'+t'')} \circ (I \otimes \theta_{(s',t'),(s'',t'')}) \zeta = \\
& \sum_j \left(\theta_{\mfp,\mfp'' \vee \mfp' + s''}^E \otimes
\theta_{\mfq,\mfq'' \vee \mfq' + t''}^F \right)
\Big[ X  \otimes \Gamma_{\mfp'' \vee \mfp' + s'', \mfq}^{-1} \left(
U_{\mfp'' \vee \mfp' + s'' \leftrightarrow \mfq} \circ (I \otimes (I \otimes
Y)X')a_j \right) \otimes (I \otimes b_j) Y'' \Big] .
\end{align*}
As above, we factor $U_{\mfp'' \vee \mfp' + s'' \leftrightarrow \mfq}$ as
$(U_{\mfp'' \leftrightarrow \mfq} \otimes I) (I \otimes U_{\mfp' \leftrightarrow \mfq})$, to obtain

\begin{align*}
 U_{\mfp'' \vee \mfp' + s'' \leftrightarrow \mfq} \circ (I \otimes (I \otimes
Y)X')a_j  & =
\sum_i (U_{\mfp'' \leftrightarrow \mfq} \otimes I) \circ (I \otimes (I \otimes x_i)y_i)a_j \\
& = \sum_i (I \otimes I \otimes x_i) (U_{\mfp'' \leftrightarrow \mfq} \otimes I) \circ (I \otimes y_i)a_j \\
& = \sum_{i,k} (I \otimes I \otimes x_i) (I \otimes A_{i,j,k})B_{i,j,k} .
\end{align*}
So we get
\begin{align*}
& \theta_{(s,t),(s'+s'',t'+t'')} \circ (I \otimes \theta_{(s',t'),(s'',t'')}) \zeta = \\
& \sum_{i,j,k} \left((I \otimes I \otimes X)(I \otimes x_i)A_{i,j,k} \right)
\otimes \left((I \otimes I \otimes B_{i,j,k})(I \otimes b_j)Y'' \right) ,
\end{align*}
and this is exactly the same expression as we obtained for $\theta_{(s+s',t+t'),(s'',t'')} (\theta_{(s,t),(s',t')} \otimes I) \zeta$.
\end{proof}

\begin{theorem}\label{thm:rep_SC}
There exists a two parameter product system of
$\cM'$-correspondences $X$, and a completely contractive,
covariant representation $T$ of $X$ into
$B(H)$, such that for all $(s,t) \in \Rpt$ and all $a \in \cM$,
the following identity holds: \be\label{eq:rep_SC}
\tT_{(s,t)}(I_{X(s,t)} \otimes a)\tT_{(s,t)}^* = P_{(s,t)}(a) .
\ee
Furthermore, if $P$ is unital, then $T$ is fully coisometric.
\end{theorem}
\begin{proof}
As above, define
$$X(s,t) := E(s) \otimes F(t) .$$
By Proposition \ref{prop:assoc}, $X$ is a product system. For $s,t
\geq 0$, $\xi \in E(s)$ and $\eta \in F(t)$, we define a representation $T$ of $X$ by
$$T_{(s,t)}(\xi \otimes \eta) := T_s^E (\xi) T_t^F (\eta) .$$
It is clear that for fixed $s,t \geq 0$, $T_{(s,t)}$, together
with $\sigma = {\bf id}_{\cM'}$, extends to a covariant
representation of $X(s,t)$ on $H$. In addition, \be\label{eq:coi}
\tT_{(s,t)} = \tT_s^E (I_{E(s)}\otimes \tT_t^F)  , \ee so
$\|\tT_{(s,t)}\| \leq 1$. By \cite[Lemma 3.5]{MS98}, $T_{(s,t)}$
is completely contractive. Also, if $P$ is unital, so are $R$
and $S$, thus $T^E$ and $T^F$ are fully coisometric, whence $T$ is
fully coisometric. We turn to show that for $x_1 \in X(s_1,t_1),
x_2 \in X(s_2,t_2)$, $T_{(s_1 +s_2,t_1 + t_2)} (x_1 \otimes x_2) =
T_{(s_1,t_1)}(x_1) T_{(s_2,t_2)}(x_2)$.

Let $\xi_i \in E(s_i), \eta_i \in F(t_i), \,\, i=1,2$. Put $\Phi =
I_{E(s_1)} \otimes \varphi_{s_2,t_1} \otimes I_{F(t_2)}$. Treating
the maps $\theta_{s_1,s_2}^E , \theta_{t_1,t_2}^F$ as identity
maps, we have that $\Phi : X(s_1 +s_2, t_1 + t_2) \rightarrow
X(s_1,t_1) \otimes X(s_2,t_2)$. We need to show that
$$T_{(s_1 + s_2,t_1 + t_2)}\left(\Phi^{-1}(\xi_1 \otimes \eta_1 \otimes \xi_2 \otimes \eta_2) \right)
= T_{(s_1,t_1)}(\xi_1 \otimes \eta_1) T_{(s_2,t_2)}(\xi_2 \otimes
\eta_2) .$$ But for this it suffices to show that
$$T_{(s,t)} \left( \varphi_{s,t}^{-1}(\eta \otimes \xi) \right) = T_{(0,t)}(\eta) T_{(s,0)}(\xi) \quad ,\quad \xi \in E(s), \eta \in F(t) .$$
Let $h \in H$. Now, on the one hand, recalling
(\ref{eq:commutation_relation}), we have
$$\tT_{(s,0)}(I_{E(s)} \otimes \tT_{(0,t)}) (\varphi_{s,t}^{-1}(\eta \otimes \xi) \otimes h) = \tT_{(0,t)}(I_{F(t)} \otimes \tT_{(s,0)})(\eta \otimes \xi \otimes h ) = T_{(0,t)}(\eta) T_{(s,0)}(\xi) h .$$
On the other hand, writing $\sum \xi_i \otimes \eta_i$ for
$\varphi_{s,t}^{-1}(\eta \otimes \xi)$, we have
\begin{align*}
\tT_{(s,0)}(I_{E(s)} \otimes \tT_{(0,t)}) (\varphi_{s,t}^{-1}(\eta
\otimes \xi) \otimes h)
& = \sum \tT_{(s,0)}(\xi_i \otimes T_{(0,t)}(\eta_i) h) \\
& = \sum T_{(s,0)}(\xi_i) T_{(0,t)}(\eta_i) h \\
& = T_{(s,t)}(\sum \xi_i \otimes \eta_i) h\\
& = T_{(s,t)}(\varphi_{s,t}^{-1}(\eta \otimes \xi))h
\end{align*}
so we conclude that $T_{(0,t)}(\xi) T_{(s,0)}(\eta) =
T_{(s,t)}(\varphi_{s,t}^{-1}(\eta \otimes \xi))$, as required.

Finally, using \cite[Theorem 3.9]{MS02} (which is equation \ref{eq:rep_rep} for one parameter CP-semigroups), we easily compute for
$a \in \cM$:
\begin{align*}
\tT_{(s,t)}(I_{X(s,t)} \otimes a)\tT_{(s,t)}^* & =
\tT_{(s,0)}(I_{E(s)} \otimes \tT_{(0,t)})(I_{E(s)} \otimes
I_{F(t)} \otimes a)
(I_{E(s)} \otimes \tT_{(0,t)}^*)\tT_{(s,0)}^* \\
& = \tT_{(s,0)}(I_{E(s)} \otimes S_t (a) ) \tT_{(s,0)}^* \\
& = R_s (S_t (a)) = P_{(s,t)} (a) .
\end{align*}
This concludes the proof.
\end{proof}

\section{E$_0$-dilation of a strongly commuting pair of CP$_0$-semigroups}\label{sec:E0dilation}

In the last two sections we worked out the two
main steps in the Muhly-Solel approach to dilation. In this
section we will put together these two steps and take care of the
remaining technicalities. It is convenient to begin by proving a
few technical lemmas.

\subsection{CP-semigroups and some of their continuity properties}
\begin{lemma}\label{lem:semigroup}
Let $\cN$ be a von Neumann algebra, let $\cS$ be an abelian, cancellative semigroup with unit $0$, and let $X$ be a product system of $\cN$-correspondences over $\cS$. Let $W$ be completely contractive covariant representation of $X$ on a Hilbert space $G$, such that $W_0$ is unital. Then the family of maps
\bes
\Theta_s : a \mapsto \widetilde{W}_s (I_{X(s)} \otimes a) \widetilde{W}_s^* \,\, , \,\, a \in W_0 (\cN)',
\ees
is a semigroup of CP maps (indexed by $\cS$) on $W_0 (\cN)'$. Moreover, if $W$ is an isometric (a fully coisometric) representation, then $\Theta_s$ is a $*$-endomorphism (a unital map) for all $s\in\cS$.
\end{lemma}
\begin{proof}
By Proposition 2.21 in \cite{MS02}, $\{\Theta_s\}_{s\in\cS}$ is a family of contractive, normal, completely positive maps on $W_0(\cN)'$. Moreover, these maps are unital if $W$ is a fully coisometric representation, and they are $*$-endomorphisms if $W$ is an isometric representation. All that remains is to check that $\Theta = \{\Theta_s \}_{s\in\cS}$ satisfies the semigroup condition $\Theta_s \circ \Theta_t = \Theta_{s + t}$. Fix $a \in W_0 (\cN)'$. For all $s,t\in\cS$,
\begin{align*}
\Theta_s (\Theta_t (a))
&= \widetilde{W}_s \left(I_{X(s)} \otimes \left(\widetilde{W}_t (I_{X(t)} \otimes a) \widetilde{W}_t^*\right)\right) \widetilde{W}_s^* \\
&= \widetilde{W}_s (I_{X(s)} \otimes \widetilde{W}_t) (I_{X(s)}\otimes I_{X(t)} \otimes a)(I_{X(s)} \otimes \widetilde{W}_t^*) \widetilde{W}_s^* \\
&= \widetilde{W}_{s + t} (U_{s,t}\otimes I_G)(I_{X(s)}\otimes I_{X(t)}\otimes a)(U_{s,t}^{-1}\otimes I_G)\widetilde{W}_{s + t}^* \\
&= \widetilde{W}_{s + t} (I_{X(s\cdot t)}\otimes a)\widetilde{W}_{s + t}^* \\
&= \Theta_{s + t}(a) .
\end{align*}
Using the fact that $W_0$ is unital, we have
\begin{align*}
\Theta_0(a) h
&= \widetilde{W_0} (I_N \otimes a) \widetilde{W_0}^* h \\
&= \widetilde{W_0} (I_N \otimes a) (I \otimes h) \\
&= W_0(I_N)ah \\
&= ah ,
\end{align*}
thus $\Theta_0(a) = a$ for all $a\in \cN$.
\end{proof}
\begin{lemma}\label{lem:continuity1}
Let $\{R_t\}_{t\geq0}$ and $\{S_t\}_{t\geq0}$ be two commuting CP-semigroups
on $\cM \subseteq B(H)$, where $H$ is a separable Hilbert space.
Then the two parameter CP-semigroup $P$ defined by
$$P_{(s,t)} := R_s S_t$$
is a CP-semigroup, that is, for all $a \in \cM$, the map $\Rpt \ni
(s,t) \mapsto P_{(s,t)}(a)$ is weakly continuous. Moreover, $P$ is
\emph{jointly} continuous on $\Rpt \times \cM$, endowed with the
standard$\times$weak-operator topology.
\end{lemma}
\begin{proof}
Let $(s_n,t_n) \rightarrow (s,t)\in \Rpt$, and let $a_n
\rightarrow a \in \cM$. By \cite[Proposition 4.1]{MS02}, CP-semigroups are jointly continuous in the
standard$\times$weak-operator topology, so
$S_{t_n}(a_n) \rightarrow S_t(a)$ in the weak operator topology.
By the same proposition used once more,
$$P_{(s_n,t_n)}(a_n) = R_{s_n}(S_{t_n}(a_n)) \rightarrow R_s(S_t(a)) = P_{(s,t)}(a)$$
where convergence is in the weak operator topology.
\end{proof}

The above lemma shows that, given two commuting CP$_0$-semigroups $\{R_t\}_{t\geq0}$ and
$\{S_t\}_{t\geq0}$, we can form a two-parameter CP$_0$-semigroup $\{P_{(s,t)}\}= \{R_s S_t\}_{s,t\geq0}$ which  satisfies
the natural continuity conditions. For the theorem below, we will need $P$ to satisfy a stronger type of continuity. This is the subject of the next two lemmas.

\begin{lemma}\label{lem:continuity2}
Let $\cS$ be a topological semigroup with unit $0$, and let $\{W_s\}_{s\in\cS}$ be a semigroup over $\cS$ of CP maps on a von Neumann algebra $\cR \subseteq B(H)$. Let $\cA\subseteq \cR$ be a sub $C^*$-algebra of $\cR$ such that for all $a \in \cA$,
$$W_s(a) \stackrel{WOT}{\longrightarrow} a$$
as $s \rightarrow 0$. Then  for all $a \in \cA$,
$$W_{t+s}(a) \stackrel{SOT}{\longrightarrow} W_t (a)$$
as $s \rightarrow 0$.
\end{lemma}
\begin{proof}
One can repeat, almost word for word, the proof of the first half of Proposition 4.1 in \cite{MS02},
which addresses the case $\cS=\Rp$.
\end{proof}

\begin{lemma}\label{lem:continuity3}
Let $\Theta = \{\Theta_t\}_{t\geq 0}$ be a CP-semigroup on $\cM \subseteq B(H)$, where $H$ is a separable Hilbert space. Then $\Theta$ is jointly strongly continuous, that is, for all $h\in H$,
the map
$$(t,a) \mapsto \Theta_t(a)h$$
is continuous in the standard$\times$strong-operator topology.
\end{lemma}
\begin{proof}
First, assume that $\Theta$ is an E-semigroup. Let $(t_n,a_n)
\rightarrow (t,a)$ in the standard$\times$strong-operator topology
in $\Rp \times \cM$, and $h\in H$. \bes \|\Theta_{t_n} (a_n) h -
\Theta_t(a)h \|^2 = \|\Theta_{t_n} (a_n) h\|^2 - 2 {\rm Re}\langle
\Theta_{t_n} (a_n) h,\Theta_{t} (a) h \rangle + \|\Theta_{t} (a)
h\|^2 , \ees since $\Theta$ is continuous in the
standard$\times$weak-operator topology, it is enough to show that
$\|\Theta_{t_n} (a_n) h\|^2 \rightarrow \|\Theta_{t} (a) h\|^2$.
But \bes \| \Theta_{t_n} (a_n) h\|^2 = \langle \Theta_{t_n} (a_n^*
a_n) h, h \rangle \rightarrow \langle \Theta_{t} (a^* a) h, h
\rangle = \| \Theta_{t} (a) h\|^2 , \ees because $a_n^* a_n
\rightarrow a^* a$ in the weak-operator topology, and $\Theta$ is
jointly continuous with respect to this topology. Thus $\Theta$ is
also jointly continuous with respect to the strong-operator
topology.

Now let $\Theta$ be an arbitrary CP-semigroup, and let $(K,u,\cR,\alpha)$ be an E-dilation of $\Theta$. Then
for all $a \in \cM, t \in \Rp$,
$$\Theta_t(a) = u^* \alpha_t(u a u^* ) u ,$$
whence $\Theta$ inherits the required type of joint continuity from $\alpha$.
\end{proof}

From the above lemma one immediately obtains:

\begin{proposition}\label{prop:continuity4}
Let $\{R_t\}_{t\geq0}$ and
$\{S_t\}_{t\geq0}$ be two CP-semigroups on $\cM \subseteq B(H)$, where $H$ is a separable Hilbert space. Then the
two parameter CP-semigroup $P$ defined by
$$P_{(s,t)} := R_s S_t$$
is strongly continuous, that is, for all $a \in \cM$, the map $\Rpt \ni (s,t) \mapsto P_{(s,t)}(a)$ is strongly continuous. Moreover, $P$ is \emph{jointly} continuous on $\Rpt \times \cM$, endowed with the standard$\times$strong-operator topology.
\end{proposition}

\subsection{The existence of an E$_0$-dilation}
We have now gathered enough tools to prove our main
result.
\begin{theorem}\label{thm:scudil}
Let $\{R_t\}_{t\geq0}$ and $\{S_t\}_{t\geq0}$ be two strongly
commuting CP$_0$-semigroups on a von Neumann algebra $\cM\subseteq
B(H)$, where $H$ is a separable Hilbert space. Then the two
parameter CP$_0$-semigroup $P$ defined by
$$P_{(s,t)} := R_s S_t$$
has a minimal E$_0$-dilation $(K,u,\cR,\alpha)$. Moreover, $K$ is separable.
\end{theorem}
\begin{proof}
We split the proof into the following steps:
\begin{enumerate}
\item Existence of a $*$-endomorphic dilation $(K,u,\cR,\alpha)$ for $(\cM,P)$.
\item Minimality of the dilation.
\item Continuity of $\alpha$ on $\cM$.
\item Separability of $K$.
\item Continuity of $\alpha$.
\end{enumerate}

{\bf Step 1: Existence of a $*$-endomorphic dilation}

Let $X$ and $T$ be the product system (of $\cM'$-correspondences) and the fully coisometric
product system representation given by Theorem \ref{thm:rep_SC}. By Theorem \ref{thm:isoDilFC}, there is
a covariant isometric and fully coisometric representation $(\rho,V)$ of $X$ on some Hilbert space $K \supseteq H$, with $\rho$ unital.
Put $\widetilde{\cR} = \rho(M')'$, and let $u$ be the isometric inclusion $H \rightarrow K$. Note that, since $u H$ reduces $\rho$, $p:= u u^* \in \widetilde{\cR}$. We define a semigroup
$\widetilde{\alpha}=\{\widetilde{\alpha}_s\}_{s\in\Rpt}$ by
\bes
\widetilde{\alpha}_s(b) = \widetilde{V}_s(I \otimes b)\widetilde{V}_s^* \,\, , \,\, s\in\Rpt ,
b\in \widetilde{\cR}.
\ees
By Lemma \ref{lem:semigroup} above, $\widetilde{\alpha}$ is a semigroup of unital, normal
$*$-endomorphisms of $\widetilde{\cR}$. The (first part of the) proof of Theorem
2.24 in \cite{MS02} works in this situation as well, and shows
that
\be\label{eq:M=pRp}
\cM = u^* \widetilde{\cR} u
\ee
and that
\be\label{eq:dilationeq}
P_s (u^* b u) = u^* \widetilde{\alpha}_s(b) u
\ee
for all
$b\in \widetilde{\cR}$, $s \in \Rpt$. Note that we \emph{cannot} use that theorem directly, because
for fixed $s\in\cS$, $X(s)$ is not necessarily the identity representation of $P_s$. For the sake of completeness, we repeat the argument (with some changes).

By Theorem \ref{thm:isoDilFC},
for all $a \in \cM'$, $u^* \rho(a) u = \sigma (a)$, and by definition, $\sigma(a) = a$, thus
\bes\label{eq:corner}
u^* \widetilde{\cR} u = u^*\rho(\cM')' u = (u^*\rho(\cM')u)' = (\cM')' = \cM,
\ees
where the second equality follows from the fact that $uH$ reduces $\rho(\cM')$. This establishes (\ref{eq:M=pRp}), which allows us to make the identification $\cM = p \widetilde{\cR} p \subseteq \widetilde{\cR}$. To obtain (\ref{eq:dilationeq}), we fix $s \in \Rpt$ and $b \in \widetilde{\cR}$, and we compute
\begin{align*}
P_s(u^* b u)
&= \tT_s(I \otimes u^* bu) \tT_s^*  \\
(*)&= u^*\tV_s(I \otimes u)(I \otimes u^* bu)(I \otimes u^*) \tV_s^* u  \\
(**)&= u^*\tV_s(I \otimes  b) \tV_s^* u  \\
&= u^* \widetilde{\alpha}_s (b) u .
\end{align*}
The equalities marked by (*) and (**) are justified by items \ref{it:dilation1} and \ref{it:V*2} of Theorem \ref{thm:isoDilFC}, respectively. Equation (\ref{eq:dilationeq}) implies that $p$ is a coinvariant
projection. Since $\widetilde{\alpha}$ is unital, we have $\widetilde{\alpha}_t(p) \geq p$ for all $t \in \Rpt$, that is, $p$  is an increasing projection.

Even though we started out with a minimal isometric representation
$V$ of $T$, we cannot show that $\widetilde{\alpha}$ is a
minimal dilation of $P$. We define
\be\label{eq:redefine R}
\cR =
W^* \left(\bigcup_{t\in\Rpt}\widetilde{\alpha}_t(\cM) \right) .
\ee
This von Neumann algebra is invariant under $\widetilde{\alpha}$, and we denote $\alpha = \widetilde{\alpha}|_{\cR}$.
Now it is immediate that $(p,\cR,\alpha)$ is a $*$-endomorphic dilation of $(\cM,P)$. Indeed, for all $b\in\cR$ and all $t \in\Rpt$,
$$p \alpha_t (b) p = p \widetilde{\alpha}_t(b) p = P_t (pbp) ,$$
because $(p,\widetilde{\cR},\widetilde{\alpha})$ is a dilation of $(\cM,P)$. It is also clear that $\cM = p\cR p$.

The only issue left to handle is the continuity of $\alpha$. We now define two one-parameter semigroups on $\cR$: $\beta =
\{\beta_t\}_{t\geq 0}$ and $\gamma = \{\gamma_t\}_{t\geq 0}$ by $\beta_t = {\alpha}_{(t,0)}$ and $\gamma_t =
{\alpha}_{(0,t)}$. Clearly, $\beta$ and $\gamma$ are semigroups of normal, unital $*$-endomorphisms of $\cR$. If we show that $K$ is separable, then by Lemma \ref{lem:continuity1}, once we show that $\beta$ and $\gamma$ are E$_0$-semigroups -- that is, possess the required weak continuity -- then we have shown that $\alpha$ is an E$_0$-semigroup. The rest of the proof is dedicated to showing that $\beta$ and $\gamma$ are E$_0$-semigroups and that $K$ is separable. But before we do that, we must show that the dilation is minimal, and, in fact, a bit more.

{\bf Step 2: Minimality of the dilation}

What we really need to prove is that
\be\label{eq:partition1}
K = \bigvee \alpha_{(s_m,t_n)}(\cM) \alpha_{(s_m,t_{n-1})}(\cM) \cdots \alpha_{(s_m,t_1)}(\cM) \alpha_{(s_{m},0)}(\cM) \alpha_{(s_{m-1},0)}(\cM) \cdots \alpha_{(s_1,0)}(\cM) H
\ee
where in the right-hand side of the above expression we run over all strictly positive pairs $(s,t)\in\Rpt$ and all partitions $\{0=s_0 < \ldots < s_m=s\}$ and $\{0=t_0 < \ldots < t_n = t \}$ of $[0,s]$ and $[0,t]$. We shall also need an analog of (\ref{eq:partition1}) with the roles of the first and second ``time variables" of $\alpha$ replaced, but since the proof is very similar we shall not prove it.

Recall that
$$K = \bigvee \left\{V_{(s,t)}(X(s,t))H : (s,t)\in\Rpt \right\} .$$
Thus, it suffices to show that for a fixed $(s,t)\in\Rpt$,
\be\label{eq:partition2}
V_{(s,t)}(X(s,t))H = \bigvee \alpha_{(s_m,t_n)}(\cM) \cdots \alpha_{(s_m,t_1)}(\cM) \alpha_{(s_{m},0)}(\cM) \alpha_{(s_{m-1},0)}(\cM) \cdots \alpha_{(s_1,0)}(\cM) H
\ee
where in the right-hand side of the above expression we run over all  partitions $\{0=s_0 < \ldots < s_m=s\}$ and $\{0=t_0 < \ldots < t_n = t \}$ of $[0,s]$ and $[0,t]$.

To show that we can consider only $s$ and $t$ strictly positive, we note that if $u,v \in\Rpt$, then
\begin{align*}
V_u(X(u))H
&= \widetilde{V}_u(I_{X(u)}\otimes \widetilde{V}_v)(I_{X(u)}\otimes \widetilde{V}_v^*)(X(u) \otimes H) \\
&= \widetilde{V}_u(I_{X(u)}\otimes \widetilde{V}_v)(I_{X(u)}\otimes \widetilde{T}_v^*)(X(u) \otimes H) \\
&= \widetilde{V}_{u+v}(X(u)\otimes \widetilde{T}_v^* H) \\
&\subseteq V_{u+v}(X(u+v))H .
\end{align*}

We now turn to establish (\ref{eq:partition2}). Recall the notation and constructions of Sections \ref{subsec:des_MS} and \ref{subsec:repvia}:
$$X(s,t) := E(s) \otimes F(t) ,$$
and
$$T_{(s,t)}(\xi \otimes \eta) := T_s^E (\xi) T_t^F (\eta) ,$$
where $(E,T^E)$ and $(F,T^F)$ are the product systems and representations  representing $R$ and $S$ via Muhly and Solel's construction as described in \ref{subsec:des_MS}. By Lemma 4.3 (2) of \cite{MS02}, for all $r > 0$,
$$\bigvee \{(I_{E(r)} \otimes a ) (\widetilde{T}_r^E)^* h : a \in \cM, h \in H \} = \cE_r \otimes_{\cM'} H,$$
where $\cE_r = \cL_M (H,H^R_\mfp)$ with the partition $\mfp = \{0=r_0 < r_1 = r\}$. Similarly,
$$\bigvee \{(I_{F(r)} \otimes a ) (\widetilde{T}_r^F)^* h : a \in \cM, h \in H \} = \cF_r \otimes_{\cM'} H .$$
Fix $s,t>0$. Under the obvious identifications,  if we go over all the partitions $\{0=s_0 < \ldots < s_m=s\}$ and $\{0=t_0 < \ldots < t_n = t \}$ of $[0,s]$ and $[0,t]$, the collection of correspondences
$$\cE_{s_1} \otimes \cE_{s_2-s_1} \otimes \cdots \otimes \cE_{s_m-s_{m-1}} \otimes \cF_{t_1} \otimes \cdots \otimes \cF_{t_n - t_{n-1}}$$
is dense in $X(s,t)$. Using Lemma 4.3 (2) of \cite{MS02} repeatedly, we obtain
\begin{align*}
& \alpha_{(s_m,t_n)}(\cM) \cdots \alpha_{(s_m,t_1)}(\cM) \alpha_{(s_{m},0)}(\cM) \alpha_{(s_{m-1},0)}(\cM) \cdots \alpha_{(s_1,0)}(\cM) H \\
&= \alpha_{(s_m,t_n)}(\cM) \cdots \widetilde{V}_{(s_1,0)}(I_{(s_1,0)} \otimes \cM)\widetilde{V}_{(s_1,0)}^* H \\
&= \alpha_{(s_m,t_n)}(\cM) \cdots \widetilde{V}_{(s_1,0)} (I_{(s_1,0)} \otimes \cM)(\widetilde{T}^E_{s_1})^* H \\
&= \alpha_{(s_m,t_n)}(\cM) \cdots \widetilde{V}_{(s_2,0)}(I_{(s_2,0)} \otimes \cM)\widetilde{V}_{(s_2,0)}^* \widetilde{V}_{(s_1,0)} (\cE_{s_1} \otimes H).
\end{align*}
But
$$\widetilde{V}_{(s_2,0)}^* \widetilde{V}_{(s_1,0)} = (I_{(s_1,0)}\otimes \widetilde{V}_{(s_2-s_1,0)}^*)\widetilde{V}_{(s_1,0)}^*\widetilde{V}_{(s_1,0)} = (I_{(s_1,0)}\otimes \widetilde{V}_{(s_2-s_1,0)}^*) ,$$
so we get
\begin{align*}
& \alpha_{(s_m,t_n)}(\cM) \cdots \alpha_{(s_m,t_1)}(\cM) \alpha_{(s_{m},0)}(\cM) \alpha_{(s_{m-1},0)}(\cM) \cdots \alpha_{(s_1,0)}(\cM) H \\
&= \alpha_{(s_m,t_n)}(\cM) \cdots \widetilde{V}_{(s_2,0)}(I_{(s_2,0)} \otimes \cM)(I_{(s_1,0)}\otimes \widetilde{V}_{(s_2-s_1,0)}^*) (\cE_{s_1} \otimes H) \\
&= \alpha_{(s_m,t_n)}(\cM) \cdots \widetilde{V}_{(s_2,0)}(\cE_{s_1} \otimes \cE_{s_2-s_1} \otimes H).
\end{align*}
Continuing this way, we see that
\begin{align*}
& \alpha_{(s_m,t_n)}(\cM) \cdots \alpha_{(s_m,t_1)}(\cM) \alpha_{(s_{m},0)}(\cM) \alpha_{(s_{m-1},0)}(\cM) \cdots \alpha_{(s_1,0)}(\cM) H \\
&= {V}_{(s,t)} (\cE_{s_1} \otimes \cE_{s_2-s_1} \otimes \cdots \otimes \cE_{s_m-s_{m-1}} \otimes \cF_{t_1} \otimes \cdots \otimes \cF_{t_n - t_{n-1}})H .
\end{align*}
Since this computation works for any partition of $[0,s]$ and $[0,t]$, we have (\ref{eq:partition2}). This, in turn, implies (\ref{eq:partition1}), which is what we have been after.

Now it is a simple matter to show that $(p,\cR,\alpha)$ is a minimal dilation of $(\cM,P)$. First, note that by (\ref{eq:partition1})
$$K = \left[\cR p K \right].$$
In light of (\ref{eq:redefine R}), Definitions \ref{def:min_dil} and \ref{def:min_dil_Arv} and Proposition
\ref{prop:equiv_def_min}, we have to show that the central support of $p$ in $\cR$ is $I_K$. But this follows by a standard (and short) argument, which we omit.

{\bf Step 3: Continuity of $\beta$ and $\gamma$ on $\cM$}

We shall now show that function $\Rp \ni t \mapsto \beta_t(a)$ is strongly continuous from the right for each $a \in \cA := C^* \left(\bigcup_{t\in\Rpt} {\alpha}_t(M)\right)$. Of course, the same is true for $\gamma$ as well.

Let $k_1 = \sum_i {\alpha}_{s_i}(m_i)h_i$ and $k_2 = \sum_j
{\alpha}_{t_j}(n_j)g_j$  be in $K$, where $s_i = (s_i^1, s_i^2) , t_j = (t_j^1, t_j^2) \in \Rpt$,
$m_i,n_j \in \cR$ and $h_i,g_j \in H$. By (\ref{eq:partition1}), we may consider only $s_i^1, t_j^1 >0$. Take $a \in \cM$ and $t > 0$. For the following computations, we may assume that $k_1$ and $k_2$ are
given by finite sums, and we take $t <
\min\{t_j^1,s_i^1\}_{i,j}$. We will abuse notation a bit by denoting $(t,0)$ by $t$. Now compute:
\begin{align*}
\langle{\beta}_t(a)k_1,k_2\rangle
&= \sum_{i,j}\langle{\alpha}_t(a){\alpha}_{s_i}(m_i)h_i,{\alpha}_{t_j}(n_j)g_j\rangle \\
&= \sum_{i,j}\langle {\alpha}_{t_j}(n_j^*) {\alpha}_{t}(a){\alpha}_{s_i}(m_i)h_i,g_j\rangle \\
&= \sum_{i,j}\langle {\alpha}_{t}\left({\alpha}_{t_j-t}(n_j^*) a {\alpha}_{s_i-t}(m_i)\right)h_i,g_j\rangle \\
&= \sum_{i,j}\langle P_{t}\left(p{\alpha}_{t_j-t}(n_j^*) pap{\alpha}_{s_i-t}(m_i)p\right)h_i,g_j\rangle \\
&= \sum_{i,j}\langle P_{t}\left(P_{t_j-t}(p n_j^* p) a P_{s_i-t}(p m_i p)\right)h_i,g_j\rangle \\
&\stackrel{t\rightarrow 0}{\longrightarrow} \sum_{i,j}\langle P_{t_j}(p n_j^* p) a P_{s_i}(p m_i p) h_i,g_j\rangle \\
&= \sum_{i,j}\langle a {\alpha}_{s_i}(m_i)h_i,{\alpha}_{t_j}(n_j) g_j\rangle \\
&= \langle a k_1,k_2\rangle ,
\end{align*}
where we have made use of the joint strong continuity of $P$ (Proposition \ref{prop:continuity4}). This implies that for all $a\in \cM$,  ${\alpha}_t(a) \rightarrow a$ weakly as $t
\rightarrow 0$.
It follows from Lemma \ref{lem:continuity2} that $\beta$ is strongly right continuous on $\bigcup_{t\in\Rpt} {\alpha}_t(\cM)$, whence it is also strongly right continuous on $\cA := C^* \left(\bigcup_{t\in\Rpt} {\alpha}_t(M)\right)$.

{\bf Step 4: Separability of $K$}

As we have already noted in Step 2, from (\ref{eq:partition1}) it follows that
$$K = \bigvee\{\alpha_{u_1}(a_1) \cdots \alpha_{u_k}(a_k) h: u_i \in \Rpt, a_i \in \cM, h \in H \}.$$
We define
$$K_0 = \bigvee\{\gamma_{t_1}(\beta_{s_1}(a_1)) \cdots \gamma_{t_k}(\beta_{s_k}(a_k)) h: s_i,t_i \in \mathbb{Q}_+, a_i \in \cM, h \in H \},$$
and
$$K_1 = \bigvee\{\gamma_{t_1}(\beta_{s_1}(a_1)) \cdots \gamma_{t_k}(\beta_{s_k}(a_k)) h: s_i \in \Rp,t_i \in \mathbb{Q}_+, a_i \in \cM, h \in H \}.$$
$K_0$ is clearly separable. Because of the normality of $\gamma$, the strong right continuity of $\beta$ on $\cM$ and the fact that multiplication is strongly continuous on bounded subsets of $\cR$, we can assert that $K_0 = K_1$, thus $K_1$ is separable. Now from the strong right continuity of $\gamma$ on $\cA$ and the continuity of multiplication, we see that $K = K_1$, whence it is separable.

{\bf Step 5: Continuity of $\alpha$}

Recall that all that we have left to show is that $\beta$ and $\gamma$ possess the desired weak continuity. We shall concentrate on $\beta$.

A short summary of the situation: we have a semigroup $\beta$ of normal, unital $*$-endomorphisms defined on a von Neumann algebra $\cR$ (which acts on a \emph{separable} Hilbert space $K$), and there is a weakly dense $C^*$-algebra $\cA \subseteq \cR$ such that for all $a\in \cA, k\in K$, the function $\Rp \ni \tau \mapsto \beta_\tau (a) k  \in K$ is right continuous.
From this, we want to conclude that for all $b \in \cR$, and all $k_1, k_2 \in K$, the map
$$ \tau \mapsto \langle \beta_\tau (b) k_1, k_2 \rangle$$
is continuous. This problem was already handled by Arveson in
\cite{Arv03} and by Muhly and Solel in \cite{MS02}. For
completeness, we give some shortened variant of their arguments.

For every $b \in \cR$, there is a sequence $\{a_n\}$ in $\cA$
weakly converging to $b$. Thus, for every $b\in \cR$ and every
$k_1, k_2, \in K$, the function $ \tau \mapsto \langle \beta_\tau
(b) k_1, k_2 \rangle$ is the pointwise limit of the sequence of
right continuous functions $ \tau \mapsto \langle \beta_\tau (a_n)
k_1, k_2 \rangle$,  so it is measurable. It now follows from
\cite[Proposition 2.3.1]{Arv03} (which, in turn, follows from
well known results in the theory of operator semigroups) that
$\beta$ has the continuity that makes it into an E$_0$-semigroup.
\end{proof}


Loosely speaking, the whole point of dilation theory is to present a certain object as part of a simpler, better understood object. Theorem \ref{thm:scudil} tells us that under the strong-commutativity assumption, a two-parameter CP$_0$-semigroup can be dilated to a two parameter E$_0$-semigroup. Certainly, E$_0$-semigroups are a very special case of CP$_0$-semigroups, so we have indeed made the situation simpler. But did we really? Perhaps $P$ (the CP$_0$-semigroup) was acting on a very simple kind of von Neumann algebra, but now $\alpha$ (the dilation) is acting on a very complicated one? Actually, we did not say much about the structure of $\cR$ (the dilating algebra). In this context, we have the following partial, but quite satisfying, result.
\begin{proposition}
If $\cM = B(H)$, then $\cR = B(K)$.
\end{proposition}
\begin{proof}
As before, denote the orthogonal projection of $K$ onto $H$ by $p$. Let $q\in B(K)$ be a projection in $\cR'$. In particular, $pq = qp = pqp$, so $qp$ is a projection
in $B(H)$ which commutes with $B(H)$, thus $qp$ is either $0$ or $I_H$.

If it is $0$ then for all $t_i\in\Rpt, m_i\in \cM, h \in H$,
$$q \alpha_{t_1}(m_1)\cdots\alpha_{t_k}(m_k)h = \alpha_{t_1}(m_1)\cdots\alpha_{t_k}(m_k) qp h = 0 ,$$
so $qK = 0$ and $q=0$.

If $qp = I_H$ then for all $0<t_i\in\Rpt, m_i\in \cM, h \in H$,
\begin{align*}
q \alpha_{t_1}(m_1)\cdots\alpha_{t_k}(m_k)h
&= \alpha_{t_1}(m_1)\cdots\alpha_{t_k}(m_k) qp h \\
&= \alpha_{t_1}(m_1)\cdots\alpha_{t_k}(m_k)h ,
\end{align*}
so $qK = K$ and $q=I_K$. We see that the only projections in $\cR'$ are $0$ and $I_K$, so
$\cR' = \mathbb{C}\cdot I_K$, hence $\cR = \cR'' = B(K)$.
\end{proof}

\section{The type of the dynamics arising in the E$_0$-dilation of a two-parameter CP$_0$-semigroup}\label{sec:type}

\begin{definition}
Let $\alpha = \{\alpha_t\}_{t\geq 0}$ and $\beta = \{\beta_t\}_{t \geq 0}$ be two E$_0$-semigroups acting on $B(H)$. An \emph{$\alpha$-cocycle} is a strongly continuous family of unitaries $\{U_t \}_{t\geq 0}$ such that for all $s,t \geq 0$,
\bes
U_{s+t} = U_s \alpha_t(U_t).
\ees
$\beta$ is said to be \emph{cocycle conjugate to $\alpha$} if there is an $\alpha$-cocycle $\{U_t \}_{t\geq 0}$ such that $\beta_t(a) = U_t \alpha_t (a) U_t^*$ for all $a \in B(H), t \geq 0$.
If $\beta$ acts on $B(K)$, where $K$ is a different Hilbert space, then $\alpha$ is said to be cocycle conjugate to $\beta$ if there exists a $*$-automorphism $\theta: B(H) \rightarrow B(K)$ such that $\theta^{-1} \circ \beta \circ \theta$ is cocycle conjugate to $\alpha$.
\end{definition}
Cocycle conjugacy is an equivalence relation, and is sometimes referred to as \emph{cocycle equivalence}.

E$_0$-semigroups can be classified - up to cocycle conjugacy - into 3 ``types": type I, type II and type III. Type I E$_0$-semigroups are the best understood, and include automorphism semigroups. There is a complete classification of type I E$_0$-semigroups, and it is known that if $\alpha$ is a type I E$_0$-semigroup that is not a semigroup of automorphisms, then there is a $d \in \{1,2,\ldots, \infty\}$ such that $\alpha$ is cocycle conjugate to the CCR flow of index $d$. See \cite{Arv03} for the whole story.

In the mid 1990's Bhat proved the following result, known today as ``Bhat's Theorem" \cite{Bhat96}:
\begin{theorem}
\emph{{\bf (Bhat).}} Every CP$_0$-semigroup has a unique minimal
E$_0$-dilation.
\end{theorem}

Bhat's Theorem aroused much interest, and one of the reasons was because it opened up a new way of constructing
E$_0$-semigroups. A possible approach could have been this: construct explicitly a tractable CP$_0$-semigroup, (for example a CP$_0$-semigroup on the algebra of $n \times n$ matrices or more generally a CP$_0$-semigroup with a bounded generator), and look at its minimal E$_0$-dilation. It was hoped at the time that the resulting E$_0$-semigroup would turn out to be an E$_0$-semigroup that has not been seen before.

These hopes were soon extinguished by results of Arveson and Powers.
\begin{theorem}\label{thm:Arv}
\emph{{\bf (Arveson, Theorem 4.8, \cite{Arv99})}} Let $\phi$ be a CP$_0$-semigroup with a bounded generator. The minimal E$_0$-dilation of $\phi$ is of type I.
\end{theorem}
Independently, Powers proved that the minimal E$_0$-dilation of a CP$_0$-semigroup acting on the algebra $M_n(\mathbb{C})$ of $n \times n$ matrices is of type I \cite[Theorem 3.10]{Powers}.

Theorems \ref{thm:scudil} and \ref{thm:Arv} face us against two immediate problems:
\begin{enumerate}
\item Figure out the structure of the E$_0$-dilation of a given two-parameter CP$_0$-semigroup, especially
in the simplest case when the CP$_0$-semigroup acts on $M_n(\mathbb{C})$.
\item Try to see whether new E$_0$-semigroups (necessarily not of type I) can arise as ``parts" of the E$_0$-dilation
of a two-parameter CP$_0$-semigroup which is ``simple" in some sense (e.g. - acts on $M_n(\mathbb{C})$).
\end{enumerate}

In this section, we obtain a partial positive result related to the first problem and a partial negative result related to the second one.
We refer to the notation of Theorem \ref{thm:scudil}. Let $\{\beta_t := \alpha_{(t,0)}\}_{t \geq 0}$ and $\{\gamma_t := \alpha_{(0,t)}\}_{t \geq 0}$ be \emph{the particular commuting E$_0$-semigroups
constructed in the course of the proof of Theorem \ref{thm:scudil}} that dilate $\{R_t\}_{t\geq 0}$ and $\{S_t\}_{t\geq 0}$. The main result of this
section is that if $R$ is not an automorphism semigroup then $\beta$ is cocycle conjugate to the minimal E$_0$-dilation of $R$, and that if $R$ is an automorphism semigroup then $\beta$ is also an automorphism semigroup. In particular, we conclude that if $R$ is not an automorphism semigroup and has a bounded generator (in particular, if $H$ is finite dimensional) then $\beta$ is a type I E$_0$-semigroup.

We are still far from solving the two problems mentioned above. The first problem is not solved because it is not clear whether the cocycle conjugacy classes of $\beta$ and $\gamma$ determine in any reasonable way the two-dimensional dynamic behavior of the E$_0$-semigroup $\{\beta_s \circ \gamma_t \}_{s,t \geq 0}$. Let us be a little more concrete in what we mean by this. One may attempt to define the notion of \emph{cocycle equivalence} of two-parameter E$_0$-semigroups exactly as it was defined for one-parameter semigroups, the only difference being that cocycles are now
\emph{two-parameter} families of unitaries. Now assume that $\beta, \gamma$ and $\beta', \gamma'$ are two pairs of commuting E$_0$-semigroups such that $\beta$ and $\gamma$ are cocycle conjugate to $\beta'$ and $\gamma'$, respectively.
In this situation, it is not clear whether the two-parameter semigroups $\{\beta_s \circ \gamma_t\}_{s,t\geq}$
and $\{\beta'_s \circ \gamma'_t\}_{s,t\geq}$ are cocycle conjugate.

The second problem is not solved because we have not ruled out the possibility that for some $a,b > 0$, the one-parameter E$_0$-semigroup $\alpha = \{\alpha_t\}_{t \geq 0}$ given by
$$\alpha_t := \beta_{at} \circ \gamma_{bt}$$
is one that has not been seen before.

\subsection{Restricting an isometric dilation to a minimal isometric dilation}\label{subsec:restricting}

Let $X = \{X(s)\}_{s \in \cS}$ be a product system, and let $T$ be a c.c. representation of $X$ on a Hilbert space $H$. Let $V$ be an isometric dilation of $T$ on a Hilbert space $K \supseteq H$. Define
$$L = \bigvee_{s\in\cS} V_s (X(s)) H .$$
\begin{definition}
$V$ is called a \emph{minimal dilation} of $T$ if $L = K$.
\end{definition}
For all $s\in \cS$ and $x \in X(s)$, $L$ is invariant under $V_s(x)$. As $T_0$ is always assumed to be a nondegenerate representation, $H \subseteq L$.
We define a map $W_s: X(s) \rightarrow B(L)$ by
$$W_s(x) = V_s(x) \big|_L .$$
$W = \{W_s\}_{s\in\cS}$ is a representation of $X$ on $L$, and one easily checks that it is isometric.
Most importantly for us, $W$ is also a dilation of $T$: if $s \in \cS$, $x \in X(s)$ and $h \in H$, then
\begin{align*}
P_H W_s(x) h &= P_H V_s(x)\big|_{L} h \\
&= T_s(x) h .
\end{align*}
It is obvious that $W$ is a minimal dilation of $T$.
\begin{definition}
$W$ is called the \emph{restriction of $V$  to a minimal isometric dilation of $T$}.
\end{definition}

The discussion establishes the following theorem:
\begin{theorem}\label{thm:restrict0}
Let $X = \{X(s)\}_{s\in\cS}$ be a product system over $\cS$ and let $T$ be a c.c. representation of $X$. Every isometric dilation of $T$ can be restricted to a minimal
isometric dilation of $T$.
\end{theorem}

For our purposes below, we need a specialization of the above theorem:
\begin{proposition}\label{prop:restrict}
Let $X = \{X(t)\}_{t\geq0}$ be a product system over $\mathbb{R}_+$ and let $T$ be a fully-coisometric c.c. representation of $X$ on $H$. Every isometric dilation of $T$ can be restricted to a minimal
isometric and fully-coisometric dilation of $T$.
\end{proposition}

\begin{proof}
All we have to do is to show that the restriction of any isometric dilation of $T$ to a minimal one is fully-coisometric. By a standard computation the minimal isometric dilation of $T$ is unique, up to unitary equivalence. Theorem 3.7 in \cite{MS02} exhibits a minimal isometric dilation of $T$ that is fully-coisometric. Thus any minimal isometric dilation of $T$ is fully-coisometric.
\end{proof}

\subsection{The main results}\label{subsec:main}

Let $R=\{R_t\}_{t\geq 0}$ and $S=\{S_t\}_{t\geq 0}$ be two strongly commuting CP$_0$-semigroups on $B(H)$, with $H$ a separable Hilbert space. Let $\{\beta_t := \alpha_{(t,0)}\}_{t \geq 0}$ and $\{\gamma_t := \alpha_{(0,t)}\}_{t \geq 0}$ be the commuting E$_0$-semigroups
constructed in the course of the proof of Theorem \ref{thm:scudil} that dilate $R$ and $S$.
The semigroup $P_{(s,t)} = R_s \circ S_t$ is given by
\bes
P_{(s,t)}(a) = \widetilde{T}_{(s,t)}(I_{X(s,t)} \otimes a)\widetilde{T}_{(s,t)}^* ,
\ees
where
\be\label{eq:X}
X(s,t) = E(s) \otimes F(t),
\ee
\begin{equation}
T_{(s,t)}(x \otimes y) = T_s^E(x) T_t^F(y) \,\, , \,\, x \in E(s), y \in F(t) ,
\end{equation}
and $E$ and $T^E$, $F$ and $T^F$, are the product systems and representations associated with $R$  and $S$ by the construction from Section \ref{subsec:des_MS}.

Let $V$ be the minimal isometric and fully coisometric dilation of $T$. The dilating E$_0$-semigroups $\beta$ and $\gamma$ are given by
\begin{equation}\label{eq:beta}
\beta_t(A) = \widetilde{V^E_t}(I \otimes A)\widetilde{V^E_t}^* \,\, , \,\, A \in B(K)
\end{equation}
and
$$\gamma_t(A) = \widetilde{V^F_t}(I \otimes A)\widetilde{V^F_t}^* \,\, , \,\, A \in B(K) ,$$
where $V^E$ is the representation of $E$ given by
\begin{equation}\label{eq:V}
V^E_t(x) = V_{(t,0)}(x \otimes 1) \,\, , \,\, x \in E(t) ,
\end{equation}
and $V^F$ is the representation of $F$ given by
$$V^F_t(y) = V_{(0,t)}(1 \otimes y) \,\, , \,\, y \in F(t) .$$

\begin{theorem}\label{thm:aut}
If $R$ is a semigroup of automorphisms, then so is $\beta$.
\end{theorem}
\begin{proof}
When $R$ is an endomorphism semigroup, $E$ is simply the Arveson
product system associated with $R$
\begin{displaymath}
E(t) = \{x \in B(H) : \forall a \in B(H)\, . \, R_t(a)x = x a \} ,
\end{displaymath}
and when $R_t$ is an automorphism $E(t) = \mathbb{C} I$. $T^E_t$ is then simply the identity representation of $E(t)$.

Now $E$ is a one dimensional product system, and $V^E$ is a fully-coisometric and isometric representation of $E$
on a Hilbert space $K$. Then $V^E_s(I) V^E_t(I) = V^E_{s+t}(I \otimes I) = V^E_{s+t}(I)$, so $\{V^E_t(I)\}$ is a semigroup of unitaries. Identifying $K$ with $\mathbb{C} I \otimes K$ and $V^E_t(I)$ with $\widetilde{V_t^E}$ we consider $\{\widetilde{V_t^E}\}$ as a semigroup unitaries. As the formula (\ref{eq:beta}) shows, $\beta$ is given by conjugation with a unitary semigroup, thus $\beta$ is an automorphism semigroup.
\end{proof}

Before proceeding, we write down three (probably well known) facts that we shall need.
\begin{proposition}\label{prop:infdil1}
Let $E$ be a product system of Hilbert spaces over $\mathbb{R}_+$, and let $T$ be a c.c. representation of $T$ on $H$.
Let $V$ be the minimal isometric dilation of $T$, representing $E$ on a Hilbert space $G\supseteq H$. If $T$ is not isometric, then $G$ is infinite dimensional.
\end{proposition}
\begin{proof}
Any dilation of the product system representation $T$ contains the minimal dilation of the single c.c. representation $T_t$ of the correspondence $E(t)$, for all $t$. Thus it is enough to show that the minimal isometric dilation of a
single completely contractive representation that is not isometric represents the correspondence on an
infinite dimensional space. This can be dug out of the proof of \cite[Theorem 2.18]{MS02}.
\end{proof}

\begin{proposition}\label{prop:infdil2}
Assume that $R$'s minimal E$_0$-dilation
acts on $B(G)$, where $G\supseteq H$ is a Hilbert space. If $R$ is not an E$_0$-semigroup itself, then
$G$ is infinite dimensional.
\end{proposition}
\begin{proof}
This follows from the previous proposition and from the uniqueness of the minimal E$_0$-dilation, together
with Muhly and Solel's construction of the minimal E$_0$-dilation in terms of product system representations and
isometric dilations.
\end{proof}
\begin{proposition}\label{prop:Arv}
Let $\gamma$ be an E$_0$-semigroup acting on a separable Hilbert space $G$. Let $P$ be an infinite dimensional
projection in $B(G)$ such that $\gamma_t(P) = P$ for all $t\geq 0$. Let $\sigma$ denote the restriction of $\gamma$ to the invariant corner $PB(G)P = B(PG)$. Then $\sigma$ and $\gamma$ are cocycle conjugate.
\end{proposition}
\begin{proof}
Proposition 2.2.3, \cite{Arv03}.
\end{proof}

\begin{theorem}\label{thm:cocycle}
If $R$ is not a semigroup of automorphisms, then there is an infinite dimensional projection $P \in B(K)$ such that
$\beta_t(P) = P$ for all $t \geq 0$, and such that the restriction of $\beta$ to the invariant corner $PB(K)P = B(PK)$ is conjugate to $R$'s minimal dilation. In particular, $\beta$ is cocycle conjugate to $R$'s minimal
dilation.
\end{theorem}
\begin{proof}
As in Proposition \ref{prop:restrict}, let $W$ denote the restriction of $V^E$ to the minimal isometric (and fully-coisometric) dilation of $T^E$, and denote by $L$ the space
on which it represents $E$. By Proposition \ref{prop:infdil1}, ${\rm dim} L = \infty$.
We compute:
\begin{align*}
\beta_t(P_L) &= \widetilde{V^E_t}(I \otimes P_L)\widetilde{V^E_t}^* \\
&= \widetilde{W}_t(I \otimes P_L)\widetilde{W}_t^* P_L = P_L.
\end{align*}
Let $\sigma$ denote the restriction of $\beta$ to $B(P_L K)$. By Proposition \ref{prop:Arv}, $\beta$ and $\sigma$ are cocycle conjugate. It remains to show that $\sigma$ is the minimal dilation of $R$. But for all $A \in B(L), t\geq 0$,
\begin{align*}
\sigma_t(A) &= \sigma_t(P_L A P_L) \\
&= \beta_t(P_L A P_L) \\
&= \widetilde{V^E_t}(I \otimes P_L A P_L)\widetilde{V^E_t}^* \\
&= \widetilde{W}_t(I \otimes A)\widetilde{W}_t^* .
\end{align*}
On the other hand, $W$ is $T^E$'s minimal isometric dilation. The results of \cite{MS02} show that the minimal E$_0$-dilation of $R$ is given by
\begin{displaymath}
A \mapsto \widetilde{W}_t(I \otimes A)\widetilde{W}_t^* \,\, , \,\, A \in B(L).
\end{displaymath}
The uniqueness of the minimal E$_0$-dilation now implies that $\sigma$ must be the minimal E$_0$-dilation of $R$ (note that uniqueness of the minimal E$_0$-dilation of a CP$_0$-semigroup is up to conjugacy, and not merely up to cocycle conjugacy).
\end{proof}

\begin{corollary}\label{cor:main}
$\beta$ is cocycle conjugate to the minimal dilation of $R$ in all cases except the case where
$R$ is an automorphism semigroup, $S$ is not an automorphism semigroup and $H$ is finite dimensional.
\end{corollary}
\begin{proof}
Assume that $R$ is a semigroup of automorphisms. In this case it is, of course, its own minimal dilation. We know by Theorem \ref{thm:aut} that $\beta$ is also a semigroup of automorphisms. If $H$ is infinite dimensional, then $\beta$ and $R$ are cocycle conjugate (this is the content of \cite[Remark 2.2.4]{Arv03}).

Assume further that $H$ is finite dimensional. If $S$ is also an automorphism semigroup, then $\beta = R$ (and $\gamma = S$). Finally, if $S$ is not a semigroup of automorphisms, then, by Proposition \ref{prop:infdil2}, $K$ must be infinite dimensional, so $\beta$ cannot be cocycle conjugate to $R$.
\end{proof}

\begin{corollary}
Assume that $R$ is not an automorphism semigroup on a finite dimensional space and has a bounded generator. Then $\beta$ is a type I E$_0$-semigroup.
\end{corollary}
\begin{proof}
By Corollary \ref{cor:main}, $\beta$ is cocycle conjugate to $R$'s minimal dilation, which, by Theorem \ref{thm:Arv}, is a type I E$_0$-semigroup.
\end{proof}

\begin{remark}
\emph{By the results in \cite{Arv99}, one may also effectively compute the index of $\alpha$ in terms of natural structures associated with the generator of $\phi$.}
\end{remark}

To conclude this section, let us outline a different proof for part of Theorem \ref{thm:cocycle}, which might be enlightening.
This argument uses the fact that product systems are classifying invariants of E$_0$-semigroups, a fact that we did not use above. To make the following proof precise, one would need to take into account the Borel structure of product systems and measurability of representations.

Recall that to each E$_0$-semigroup Arveson associates a product system, and that two E$_0$-semigroups are
cocycle conjugate if and only if their associated product systems are isomorphic (see Section 2.4, \cite{Arv03}).
By equations (\ref{eq:X}), (\ref{eq:beta}) and (\ref{eq:V}), $\beta$ is an E$_0$-semigroup given by an isometric and fully-coisometric representation of the product system $E$. Proposition 3.2.2 of \cite{Arv03} implies that $E$
must therefore be (isomorphic to) the product system associated with $\beta$. On the other hand, the results of \cite{MS02} show that the minimal E$_0$-dilation of $R$ is also an E$_0$-semigroup that is given by an isometric and fully-coisometric representation of the product system $E$, the same $E$ that Muhly and Solel associate with $R$. Thus, being associated with the same product system, $\beta$ and the minimal dilation of $R$ are cocycle conjugate.

Combining \cite[Corollary 8.7]{Markiewicz} (see Section \ref{subsec:QCS}) with Theorem \ref{thm:scudil} and Corollary \ref{cor:main}, we obtain:
\begin{corollary}\label{cor:quantized}
Let $\mu$ and $\nu$ be two infinitely divisible Borel probability measures on $\mb{C}$. Assume that $\{\phi^\mu_t\}_{t\geq 0}$ has index $m$ and that $\{\phi^\nu_t\}_{t\geq 0}$ has index $n$. Then there exists a Hilbert space $K \supset H := L^2(\mb{R})$, a type \emph{I}$_m$ E$_0$-semigroup $\alpha$ and a type \emph{I}$_n$ E$_0$-semigroup $\beta$ acting on $B(K)$, such that for all $s,t \geq 0$ and all $T \in B(K)$, $\alpha_s (\beta_t (T)) = \beta_t (\alpha_s(T))$
and
\bes
\phi^\mu_s (\phi^\nu_t(P_H T P_H)) = P_H \alpha_s (\beta_t (T)) P_H.
\ees
\end{corollary}

\chapter{E-dilation of two-parameter CP-semigroups - the nonunital case}\label{chap:nonunital}


In this chapter we treat the problem of constructing an E-dilation to a pair of strongly commuting CP-semigroups that do not necessarily preserve the unit. We will use the same basic strategy as in the previous chapter, but several new difficulties will arise along the way. Theorem \ref{thm:rep_SC} applies to show that for every pair of strongly commuting CP-semigroups $R$ and $S$, there exists a product system $X = \{X(s,t)\}_{s,t\geq 0}$ and  representation $T$ of $X$ such that
\bes
R_{s}\circ S_t(a) = \tT_{(s,t)}(I_{X(s,t)} \otimes a)\tT_{(s,t)}^* .
\ees
However, when $R$ and $S$ are not unit preserving $T$ will not be fully coisometric, and Theorem \ref{thm:isoDilFC} cannot be used to obtain an isometric dilation for $T$. It seems that the way to go is to prove the existence of an arbitrary isometric dilation of a product system representation over $\mb{R}_+^2$, but this was not achieved. In the next section, we will prove the existence of an isometric dilation of a product system representation over $\diadp$ (the dyadic rationals in $\mb{R}_+^2$), which seems to be just enough in order to construct a two parameter E-dilation to a pair of strongly commuting CP-semigroups acting on $B(H)$.
The construction of the E-dilation is carried out in Section \ref{sec:Edilation}.

The material in this chapter is taken from \cite{ShalitCPDil}.

\section{Isometric dilation of a product system representation over $\diadp$}\label{sec:iso_dil_d}

Unital CP-semigroups correspond to fully coisometric product system representations. Fully coisometric product system representations are analogous to semigroups of \emph{coisometries} on a Hilbert space. In the context of classical dilation theory of contraction semigroups on a Hilbert space \cite{SzNF70}, the problem of finding an isometric dilation to a semigroup of coisometries is relatively easy (see also \cite{Douglas}). In this section, as we will treat general c.c. representations, we will only be able to construct an isometric dilation for a completely contractive product system representation over the set $\diadp$ of positive dyadic pairs, where
$$\diad := \left\{\left(\frac{k}{2^{n}},\frac{m}{2^{n}}\right) : (k,m) \in \mathbb{Z}^2, n \in \mathbb{N} \right\} $$
is the set of dyadic fractions.
This will be sufficient to lead us to our present goal of dilating a CP-semigroup over $\Rpt$, due to the extendability properties of CP-semigroups discussed in Section \ref{sec:extension}.

Let $\cM$ be a von Neumann algebra, let $X$ be a product system of $\cM$-correspondences over $\diadp$, and let $H$ be a Hilbert space. Assume that $\sigma$ is a normal representation of $\cM$ on $H$. We denote by $\cH_0$ the space of all finitely supported functions $f$ on $\diadp$ such that $f(n) \in X(n) \otimes_\sigma H$, for all $n \in \diadp$. For any $n = (n_1, n_2) \in \diad$, we denote by $n_+$ the element in $\diadp$ having $\max\{n_i,0 \}$ in its $i$-th entry, and we denote $n_- = n_+ - n$.
\begin{definition}\label{def:positive}
Let $\Phi$ be a function on $\diad$ such that $\Phi(n) \in B(X(n_+) \otimes_\sigma H, X(n_-) \otimes_\sigma H)$, $n \in \diad$. We say that $\Phi$ is \emph{positive definite} if $\Phi (0) = I_{M \otimes_\sigma H}$ and if
\begin{enumerate}
\item \label{it:positive} For all $h \in \cH_0$ we have
\bes
\sum_{m,n \in \diadp} \left\langle \left[ I_{X(m-(m-n)_+)} \otimes \Phi(m-n)\right ] h(m), h(n) \right\rangle \geq 0 .
\ees
\item \label{it:self-adjoint}$\Phi(n) = \Phi(-n)^*$, for all $n \in \diadp$.
\item \label{it:covariance} For all $n\in \diadp$, $a \in \cM$,
\bes
 \Phi(n)\left(a \otimes I_H \right) = \left(a \otimes I_H \right) \Phi(n) .
\ees
\end{enumerate}
\end{definition}
In item (\ref{it:covariance}) above $(a \otimes I_H) \xi$ should be interpreted as
$\varphi_{X(m)}(a) x \otimes h$ if $\xi = x \otimes h$, $h \in H, x \in X(n)$, for some $n \neq 0$, and as $ab \otimes h$ if $\xi = b \otimes h$ for $b \in \cM, h \in H$. This will remain our convention throughout.

We note that the proof of Theorem 3.5 in \cite{S08} implies the following fact: \emph{if $\Phi$ is a positive definite function on $X$ as above, then there exists a covariant isometric representation $V$ of $X$ on some Hilbert space $K \supseteq H$, such that $H$ is a reducing subspace for $V_0$ with $V_0(\cdot) \big|_H = \sigma(\cdot)$, and such that}
\be\label{eq:dil}
P_{X(n_-) \otimes H} V(n) \big|_{X(n_+) \otimes H} = \Phi(n) \,\, , \,\, n \in \diad,
\ee
\emph{where $V(n) := \widetilde{V}_{n_-}^*\widetilde{V}_{n_+}$}. This fact is the basis of the proof of the following theorem, so we point out that the definition of $V$ in the above mentioned proof has to be modified in an obvious manner and that straightforward calculations (some almost identical and some different from what appeared in the proof) show that $V$ has all the required properties. The main difference is that one has to show that $V$ has the ``semigroup" property.

\begin{theorem}\label{thm:IsoDilDiad}
Let $\cM$ be a $C^*$-algebra, and let $X=\{X(s,t)\}_{(s,t) \in \diadp}$ be a product system of $\cM$-correspondences. Let $T$ be a c.c. representation of $X$ on a Hilbert space $H$, with $\sigma = T(0,0)$.
Assume that for all $(s,t) \in \diadp$, the Hilbert space $X(s,t) \otimes_{\sigma} H$ is separable.
Then there exists an isometric representation $V$ of $X$ on Hilbert space $K$ containing $H$ such that:
\begin{enumerate}
    \item $P_H$ commutes with $V_{(0,0)} (\cM)$, and $V_{(0,0)}(a) P_H = T_{(0,0)}(a) P_H$, for all $a \in \cM$.
    \item\label{it:dilation} $P_H V_{(s,t)}(x)\big|_H = T_{(s,t)}(x)$ for all $(s,t) \in \diadp$, $x \in X(s,t)$.
    \item\label{it:min} $K = \bigvee \{V(x)h : x \in X, h \in H \}$.
    \item\label{it:V*} $P_H V_{(s,t)}(x)\big|_{K \ominus H} = 0$ for all $(s,t) \in \diadp$, $x \in X(s,t)$.
\end{enumerate}
If $\cM$ is a $W^*$-algebra, $X$ a product system of $W^*$-correspondences and $T_0$ is normal, then
$V_0$ is also normal, that is, $V$ is a representation of $W^*$-product systems.
\end{theorem}
Recall that a dilation satisfying item (\ref{it:min}) above is called a \emph{minimal} dilation.

\begin{proof}
For any $n \in \mathbb{N}$, the triple $(\sigma, T(2^{-n},0), T(0,2^{-n}))$ defines a c.c. representation of the product system $X^{(n)} = \{X(m/{2^n},k/{2^n})\}_{m,k}$. We will denote $X^{(n)}(m,k) = X(m/{2^n},k/{2^n})$. By Theorem 4.4 in \cite{S06}, this representation has a covariant isometric dilation $(\rho_n,V_{1,n},V_{2,n})$ on some Hilbert space which we need not refer to. As $n$ increases we get isometric dilations to the restriction of $T$ to fatter and fatter product systems, but the problem is that we do not know exactly how (and if) they sit one inside the other. Our immediate goal is to define a positive definite function $\Phi$ on $\diad$. The heart of the following idea is taken from Ptak's paper \cite{Ptak}, where the existence of a unitary dilation to a two-parameter semigroup of contractions on a Hilbert space is obtained (the latter result was obtained earlier by S\l oci\'nski in \cite{Slocinski}, using a different method).

First we define, for all $n \in \mathbb{N}, (s,t) \in \diad$ an operator $a_n (s,t)$ in $B(X(s_+,t_+) \otimes H, X(s_-,t_-) \otimes H)$. This is done in the following manner. Fixing $(s,t) \in \diad$, there is some $n_{s,t} \in \mathbb{N}$ such that for all $n \geq n_{s,t}$ there are two integers $m_{s,n}$ and $k_{t,n}$ satisfying $(s,t) = (m_{s,n}\cdot 2^{-n},k_{t,n}\cdot 2^{-n})$, and such that $n_{s,t}$ is the minimal natural number with this property. For $n < n_{s,t}$ we define $a_n(s,t) = 0$. For $n \geq n_{s,t}$ we define
\bes
a_n(s,t) = P_{X(s_-,t_-)\otimes H}\overline{V}_n (m_{s,n},k_{t,n}) \big|_{X(s_+,t_+)\otimes H}
\ees
where $\overline{V}_n (m,k):= \left(I_{X^{(n)}(0,k_-)} \otimes (\widetilde{V}_{1,n}^{m_-})^*\right) (\widetilde{V}_{2,n}^{k_-})^* \widetilde{V}_{1,n}^{m_+}(I_{X^{(n)}(m_+,0)} \otimes \widetilde{V}_{2,n}^{k_+})$ (to be precise, one should multiply the right-hand side by $U_{(0,k_-),(m_-,0)}\otimes I_H$ on the left and $U_{(m_+,0),(0,k_+)}^{-1} \otimes I_H$ on the right, where $U_{\cdot,\cdot}$ are the multiplication maps of $X^{(n)}$).

For fixed $(s,t) \in \diadp$, and for large enough $n$, we have
$$a_n(s,t) = T(s,t):= \widetilde{T}^*_{(s_-,t_-)} \widetilde{T}_{(s_+,t_+)} = \widetilde{T}_{(s,t)} .$$
Fixing $(s,t) \in \diad$, we have  for all large enough $n$
\begin{align*}
a_n(-s,-t)^*
&=\left(P_{X(s_+,t_+) \otimes H} \left(I_{0,k_+} \otimes (\widetilde{V}_{1,n}^{m_+})^*\right) (\widetilde{V}_{2,n}^{k_+})^* \widetilde{V}_{1,n}^{m_-}(I_{m_-,0} \otimes \widetilde{V}_{2,n}^{k_-}) \right)^* \big|_{X(s_-,t_-)\otimes H} \\
&= P_{X(s_-,t_-) \otimes H} \left(I_{m_-,0} \otimes (\widetilde{V}_{2,n}^{k_-})^*\right) (\widetilde{V}_{1,n}^{m_-})^* \widetilde{V}_{2,n}^{k_+}(I_{0,k_+} \otimes \widetilde{V}_{1,n}^{m_+})\big|_{X(s_+,t_+)\otimes H} \\
(*)&=a_n(s,t)
\end{align*}
where we used the shorthand notations $I_{p,q} = I_{X^{(n)}(p,q)}$, $m=m_{s,n},k=k_{t,n}$, and the equality in marked by (*) is true up to multiplication by the product system multiplication maps $U_{\cdot,\cdot}$. Also, it follows immediately from the covariance properties of $(\rho_n,V_{1,n},V_{2,n})$ that $a_n(s,t)$ intertwines the various interpretations of $(a \otimes I_H), a \in \cM$.

Now that $a_n(s,t)$ is defined, we construct a positive definite function $\Phi$ on $\diad$. For every $(s,t) \in \diad$, $\{a_n(s,t)\}_n$
is a sequence of operators in $B(X(s_+,t_+) \otimes H, X(s_-,t_-) \otimes H)$ with norm less than or equal $1$.
As the unit ball of $B(X(s_+,t_+) \otimes H, X(s_-,t_-) \otimes H)$ is weak operator compact, there is a subsequence
$\{n_k\}_{k=1}^\infty$ of $\mathbb{N}$ such that $a_{n_k}(s,t)$ converges in the weak operator topology (our separability assumptions imply that the unit ball of $B(X(s_+,t_+) \otimes H, X(s_-,t_-) \otimes H)$ is metrizable, hence sequentially compact). In fact, since $\diad$ is countable,
a standard diagonalization procedure will produce $\{n_k\}_{k=1}^\infty$ of $\mathbb{N}$ such that $a_{n_k}(s,t)$ converges weakly for \emph{all} $(s,t) \in \diad$. We define
$$\Phi(s,t) = \textrm{wot --}\lim_{k \rightarrow \infty} a_{n_k}(s,t).$$

By the properties that $a_n(s,t)$ possesses, it follows that for $(s,t) \in \diadp$,
$$\Phi(s,t) = T(s,t).$$
Also, $\Phi$ satisfies items (\ref{it:self-adjoint}) and (\ref{it:covariance}) of Definition \ref{def:positive}. For example, for (\ref{it:covariance}) it is enough to check
\begin{align*}
\langle \Phi(s,t) (a \otimes I) e_i, f_j \rangle
&= \lim_{k \rightarrow \infty }\langle a_{n_k}(s,t) (a \otimes I) e_i, f_j \rangle \\
&= \lim_{k \rightarrow \infty }\langle a_{n_k}(s,t)  e_i, (a \otimes I)^* f_j \rangle \\
&= \langle (a \otimes I) \Phi(s,t)  e_i, f_j \rangle .
\end{align*}
(\ref{it:self-adjoint}) follows similarly. Let us prove that it also satisfies item (\ref{it:positive}). Let $h \in \cH_0$, and consider the sum
\be\label{eq:sum}
\sum_{m,n \in \diadp} \left\langle \left[I_{X(m-(m-n)_+)} \otimes \Phi(m-n)\right] h(m), h(n) \right\rangle .
\ee
We are going to show that this sum is greater than $-\epsilon$, for any $\epsilon >0$. The sum in (\ref{eq:sum}) contains only a finite number, say $N$, of non-zero summands. We may take $k$ large enough to satisfy
\bes
\left|\left\langle \left[I_{m-(m-n)_+} \otimes \Phi(m-n)\right] h(m), h(n) \right\rangle - \left\langle \left[I_{m-(m-n)_+} \otimes a_{n_k}(m-n)\right] h(m), h(n) \right\rangle\right| < \frac{\epsilon}{N} ,
\ees
for all $m,n \in \diadp$. If needed, we take $k$ even larger, so that
\bes
a_{n_k}(d) = P_{X(d_-)\otimes H}\overline{V}_{n_k} (m_{d_1,n},k_{d_2,n}) \big|_{X(d_+)\otimes H}
\ees
for all $d = (d_1,d_2) \in \diad$ that appears in a non-zero inner product in (\ref{eq:sum}). In other words, we assume that all dyads appearing non-trivially in (\ref{eq:sum}) have the form $(p\cdot2^{-n_k},q\cdot2^{-n_k})$, $p,q \in \mathbb{Z}$. But then
\begin{align*}
& \sum_{m,n \in \diadp} \langle I_{m-(m-n)_+} \otimes a_{n_k}(m-n) h(m), h(n) \rangle \\
&= \sum_{m,n \in \diadp} \langle I_{m-(m-n)_+} \otimes P_{X((m-n)_-)\otimes H} \widetilde{U}_{(m-n)_-}^* \widetilde{U}_{(m-n)_+} \big|_{X(m-n)_+} h(m), h(n) \rangle \\
(*)&= \sum_{m,n \in \diadp} \langle I_{m-(m-n)_+} \otimes \widetilde{U}_{(m-n)_-}^* \widetilde{U}_{(m-n)_+} h(m), h(n) \rangle \\
(**)&= \sum_{m,n \in \diadp} \langle \widetilde{U}_{n}^* \widetilde{U}_{m} h(m), h(n) \rangle \\
&= \sum_{m,n \in \diadp} \langle  \widetilde{U}_{m} h(m),\widetilde{U}_{n} h(n) \rangle \geq 0.
\end{align*}
The equality marked by (*) follows from identifying $X(d) \otimes H$ with a subspace of $X(d) \otimes G$, where $G$ is the dilation Hilbert space associated with $U$, and $U$ is the isometric dilation of the restriction of $T$ to $X^{(n_k)}$. The equality marked by (**) follows from
\begin{align*}
\widetilde{U}_{n}^* \widetilde{U}_{m}
&= (I_{n-(n-m)_+} \otimes \widetilde{U}^*_{(n-m)_+})\widetilde{U}^*_{n-(n-m)_+}\widetilde{U}_{m-(m-n)_+} (I_{m-(m-n)_+} \otimes
\widetilde{U}_{(m-n)_+})\\
&= I_{m-(m-n)_+} \otimes \widetilde{U}^*_{(m-n)_-} \widetilde{U}_{(m-n)_+},
\end{align*}
because $n-(n-m)_+ = m-(m-n)_+$ and $(n-m)_+ = (m-n)_-$.
Thus
\bes
\sum_{m,n \in \diadp} \left\langle \left[I_{X(m-(m-n)_+)} \otimes \Phi(m-n)\right] h(m), h(n) \right\rangle \geq -\epsilon
\ees
for all $\epsilon$, thus $\sum_{m,n \in \diadp} \left\langle \left[I_{X(m-(m-n)_+)} \otimes \Phi(m-n)\right] h(m), h(n) \right\rangle \geq 0$, for all $h \in \cH_0$.

We have shown that $\Phi$ is a positive definite function on $\diad$. It follows from the remarks before the statement of the theorem that there is a covariant isometric representation $V$ of $X$ on a Hilbert space $K \supseteq H$ satisfying items (1) and (2) in the statement of the theorem. Given an isometric dilation $V$ on $K$, using Theorem \ref{thm:restrict0} we may restrict it to a minimal one. If we take $V$ to be a minimal dilation, (\ref{it:V*}) follows as well.

The proof of the final assertion of the theorem is not different from the proof of the analogous
part in Theorem \ref{thm:isoDilFC}, and we do not wish to repeat it here. We just note that one uses the structure of the isometric dilation given by the proof of \cite[Theorem 3.5]{S08}.
\end{proof}

\section{E-dilation of a strongly commuting pair of CP-maps}\label{sec:Edilation}

Now we are ready to prove the main result of this chapter:
\begin{theorem}\label{thm:main}
Let $\{R_t\}_{t\geq0}$ and $\{S_t\}_{t\geq0}$ be two strongly
commuting CP-semigroups on $B(H)$, where $H$ is a separable Hilbert space. Then the two
parameter CP-semigroup $P$ defined by
$$P_{(s,t)} := R_s S_t$$
has a minimal E-dilation $(K,u,B(K),\alpha)$, where $K$ is separable.
\end{theorem}

\begin{proof}
Let $X$ and $T$ be the product system of Hilbert spaces and the
product system representation given by Theorem \ref{thm:rep_SC}.
We consider the product system $\check{X}=\{X(s)\}_{s\in\diadp}$
of $X$ represented on $H$ by the representation $\check{T}=\{T_s\}_{s\in\diadp}$
of $T$. The proof of \cite[Proposition 4.2]{MS02} shows that $X(t_1,t_2) = E(t_1) \otimes F(t_2)$ is a separable Hilbert space for all $t_1,t_2 \geq 0$, and it follows that for all $s \in \diadp$ the Hilbert space
$X(s) \otimes_{T_0} H$ is separable. By Theorem
\ref{thm:IsoDilDiad}, there is a minimal isometric
dilation ${V} = \{{V}_s\}_{s\in\diadp}$ of $\check{T}$, representing $\check{X}$ on a Hilbert space ${K} \supseteq H$.
We define a semigroup ${\alpha}=\{{\alpha}_s\}_{s\in\diadp}$ on $B({K})$ by
\bes
{\alpha}_s(a) =
\widetilde{{V}}_s(I \otimes a)\widetilde{{V}}_s^* \,\, , \,\, s\in\diadp ,
a\in B({K}).
\ees
By Lemma \ref{lem:semigroup}, ${\alpha}$ is a semigroup of normal
$*$-endomorphisms on $B({K})$. Denote by $p$ the orthogonal projection
of ${K}$ onto $H$. It is clear that $B(H)$ is the corner $B(H) = pB({K})p \subseteq B({K})$.
As in the proof of Theorem \ref{thm:scudil}, we see that ${\alpha}$ is a dilation of $\{P_{s}\}_{s \in \diadp}$.

We define two (``one-parameter") semigroups $\beta =
\{\beta_t\}_{t\in \mathbb{D}_+}$ and $\gamma = \{\gamma_t\}_{t\in
\mathbb{D}_+}$ on $B(K)$ by
\be\label{eq:betagamma}
\beta_t = \alpha_{(t,0)} \quad  \text{\rm and} \quad \gamma_t =
\alpha_{(0,t)} .
\ee
By Proposition \ref{prop:continuity4}, if we will be able to extend $\beta$ and $\gamma$
to continuous E-semigroups $\hat{\beta}$ and $\hat{\gamma}$ over $\Rp$, then the semigroup $\hat{\alpha} = \{\hat{\alpha}_{(s,t)}\}_{(s,t)\in \Rpt}$ given by
$$\hat{\alpha}_{(s,t)} = \hat{\beta}_s \circ \hat{\gamma}_t $$
will be the sought after E-dilation of $P$. The rest of the proof is mostly
dedicated to showing that $\beta$ and $\gamma$ can be continuously extended. As we demonstrate the
extendability of $\beta$ and $\gamma$, we show that $(p, B(K), \alpha)$ is a minimal dilation of $(B(H),\{P_s\}_{s\in\diadp})$, and this will complete the proof of Theorem \ref{thm:main}.

Because ${V}$ is a minimal dilation of $\check{T}$, we have
$$K := \bigvee_{s \in \diadp}{V}_s (X(s))H .$$
In particular, $K$ is separable.

An important observation is this:
\be\label{eq:K}
K = \bigvee \left\{\alpha_{t_1}(m_1)\cdots\alpha_{t_k}(m_k)h:
k \in \mathbb{N}, t_i\in\diadp, m_i\in B(H), h \in H \right\}.
\ee
(When $k=0$, we interpret the product $\alpha_{t_1}(m_1)\cdots\alpha_{t_k}(m_k)h$ as $h$). In Step 2 of the proof of Theorem \ref{thm:scudil} this equality is proved (there the situation was a little simpler and one did not have to consider products $\alpha_{t_1}(m_1)\cdots\alpha_{t_k}(m_k)h$ with $k=0$. The proof, however, holds in this case as well, as long as one does take such products).
In fact, In Step 2 of the proof of Theorem \ref{thm:scudil}, we saw that
\be\label{eq:partitionA}
K = \bigvee \alpha_{(s_m,t_n)}(B(H)) \alpha_{(s_m,t_{n-1})}(B(H)) \cdots \alpha_{(s_m,t_1)}(B(H)) \alpha_{(s_{m},0)}(B(H))  \cdots \alpha_{(s_1,0)}(B(H)) H
\ee
where in the right-hand side of the above expression we run over all pairs $(s,t)\in\diadp$ and all partitions $\{0=s_0 < \ldots < s_m=s\}$ and $\{0=t_0 < \ldots < t_n = t \}$ of $[0,s]$ and $[0,t]$.

Using (\ref{eq:K}), we can show that $(p, B(K), \alpha)$ is a minimal dilation.
Define
\be\label{eq:define R}
\cR =
W^*\left(\bigcup_{s\in\diadp}\alpha_s(B(H)) \right) .
\ee
Note that $K = [\cR H]$.
But the central projection
of $p$ in $\cR$ is the projection on $[\cR p K] = [\cR H] = K$, that is $I_K$. We will now show that $\cR = B(K)$, and this will prove that the central projection of $p$ in $B(K)$ is $I_K$, so by both definitions \ref{def:min_dil} and \ref{def:min_dil_Arv} $(p, B(K), \alpha)$ is a minimal dilation.

To see that $\cR = B(K)$, let $q\in B(K)$ be a projection in $\cR'$. In particular, $pq = qp = pqp$, so $qp$ is a projection
$B(H)$ which commutes with $B(H)$, thus $qp$ is either $0$ or $I_H$.

If it is $0$ then for all $t_i\in\diadp, m_i\in B(H), h \in H$,
$$q \alpha_{t_1}(m_1)\cdots\alpha_{t_k}(m_k)h = \alpha_{t_1}(m_1)\cdots\alpha_{t_k}(m_k) qp h = 0 ,$$
so $qK = 0$ and $q=0$.

If $qp = I_H$ then for all $0<t_i\in\diadp, m_i\in B(H), h \in H$,
\begin{align*}
q \alpha_{t_1}(m_1)\cdots\alpha_{t_k}(m_k)h
&= \alpha_{t_1}(m_1)\cdots\alpha_{t_k}(m_k) qp h \\
&= \alpha_{t_1}(m_1)\cdots\alpha_{t_k}(m_k)h ,
\end{align*}
so $qK = K$ and $q=I_K$. We see that the only projections in $\cR'$ are $0$ and $I_K$, so
$\cR' = \mathbb{C}\cdot I_K$, hence $\cR = \cR'' = B(K)$. This completes the proof of minimality.

We return to showing that $\beta$ and $\gamma$ can be continuously extended to $\Rp$. Let
$$\dpp = \left\{\frac{m}{2^n} : 0<m,n\in \mathbb{N} \right\},$$
put $\diadpp = \dpp \times \dpp$, and define
\be\label{eq:KS}
K_0 = \bigvee \left\{\alpha_{t_1}(m_1)\cdots\alpha_{t_k}(m_k)h:
0<k \in \mathbb{N}, t_i\in\diadpp, m_i\in B(H), h \in H \right\}.
\ee
We shall use (\ref{eq:partitionA}) to prove that $K = K_0$. First, let us show that $H \subseteq K_0$.
Let
$$G_0 = \bigvee_{t\in\diadpp}\alpha_{t}(B(H))H $$
and $G = H \vee G_0$.
For $t \leq s \in \diadpp$, $a \in B(H)$ and $h,g \in H$, we find that
\begin{align*}
\langle \alpha_{t}(p)h,\alpha_{s}(a)g \rangle &=
\langle \alpha_{s}(a^*) \alpha_{t}(p)h, g \rangle \\
&= \langle \alpha_{t}( \alpha_{s-t}(a^*) p)h, g \rangle \\
&= \langle P_{t}( P_{s-t}(pa^* p) p)h, g \rangle \\
&= \langle P_{s}(pa^* p) h, g \rangle \\
&= \langle h,\alpha_{s}(a)g \rangle ,
\end{align*}
so (in a rather trivial way)
\bes
\langle \alpha_{t}(p)h,\alpha_{s}(a)g \rangle \stackrel{t\rightarrow 0}{\longrightarrow} \langle h,\alpha_{s}(a)g \rangle.
\ees
Similarly,
$$\langle \alpha_{t}(p)h,g \rangle  \stackrel{t\rightarrow 0}{\longrightarrow} \langle h,g \rangle .$$
We see that in $G$, $\alpha_{t}(p)h \rightarrow h$ weakly, thus $H$ is in the weak closure of $G_0$ in $G$.
The weak closure being equal to the strong closure, we have $H \subseteq G_0 \subseteq K_0$.

The set
$$\{\alpha_{s}(b) h : s\in \diadpp, b\in \bigcup_{t\in\diadp}\alpha_t(B(K)), h\in H \}$$
is total in $K_0$. To see this, note that every element of the form
$$\alpha_{t_1}(m_1)\cdots\alpha_{t_k}(m_k)h ,$$
with $t_i\in\diadpp, m_i\in B(H)$ and $h \in H$,
can be written as
$$\alpha_s (\alpha_{t_1-s}(m_1)\cdots\alpha_{t_k-s}(m_k)) h ,$$
where $s \in \diadpp$ is smaller than $t_i$, $i=1, 2, \ldots, k$.

Let $\alpha_{s_1}(a_1)h_1$ and $\alpha_{s_2}(a_2)h_2$
be in $K_0$, where $s_1, s_2 \in \diadpp$,
$a_1,a_2 \in B(K)$ and $h_1,h_2 \in H$. Take $a \in B(H)$ and $t \in
\diadp$ such that $t < s_1, s_2$. Now compute:
\begin{align*}
\langle\alpha_t(a)\alpha_{s_1}(a_1)h_1,\alpha_{s_2}(a_2)h_2\rangle
&= \langle \alpha_{s_2}(a_2^*) \alpha_{t}(a)\alpha_{s_1}(a_1)h_1,h_2\rangle \\
&= \langle \alpha_{t}\left(\alpha_{s_2-t}(a_2^*) a \alpha_{s_1-t}(a_1)\right)h_1,h_2\rangle \\
&= \langle P_{t}\left(p\alpha_{s_2-t}(a_2^*) pap\alpha_{s_1-t}(a_1)\right)h_1,h_2\rangle \\
&= \langle P_{t}\left(P_{s_2-t}(p a_2^* p) a P_{s_1-t}(p a_1 p)\right)h_1,h_2\rangle \\
(*)&\stackrel{t\rightarrow 0}{\longrightarrow} \langle P_{s_2}(p a_2^* p) a P_{s_1}(p a_1 p) h_1,h_2\rangle \\
&= \langle a \alpha_{s_1}(a_1)h_1,\alpha_{s_2}(a_2) h_2\rangle .
\end{align*}
If we let $p_0$ denote the orthogonal projection of $K$ on $K_0$, we see that
$p_0 \alpha_{t}(a) p_0 \rightarrow p_0 a p_0$ in the weak operator topology as $t
\rightarrow 0$, for all $a \in B(H)$ (the convergence in $(*)$ is due to Proposition \ref{prop:continuity4}).
Since $H \subseteq K_0$, one has $p \leq p_0$, and $p_0 a p_0 = a$ for all $a \in B(H)$, thus
\be\label{eq:p_0}
\forall a \in B(H). p_0 \alpha_{t}(a) p_0 \rightarrow a \,\, \text{as} \,\, t \rightarrow 0,
\ee
where convergence is in the weak operator topology.

By (\ref{eq:partitionA}), $K$ is spanned by elements of the form
\be\label{eq:g}
g = \alpha_{s_1}(\alpha_{s_2}(\cdots (\alpha_{s_m}(a_m) a_{m-1})\cdots )a_1)h ,
\ee
for $a_i \in B(H), s_i \in \diadp, i=1,\ldots m$ and $h\in H$. Let us show how such an element can be
approximated in norm arbitrarily well by elements of the same form with all $s_i$'s in $\diadpp$ (it is clear that if
all $s_i$'s are in $\diadpp$, then this element is in $K_0$).

Assume that we wish to approximate a fixed element
$g$ as in (\ref{eq:g})
by elements
of the from $\alpha_{t_1}(\cdots(\alpha_{t_m}(b_m) b_{m-1}\cdots)b_1)h'$, where $t_m \in \diadpp$.
We consider
$$g_t = \alpha_{s_1}(\alpha_{s_2}(\cdots(\alpha_{s_m+t}(a_m) a_{m-1})\cdots a_2)a_1)h$$
with $t \in \diadpp$, and compute
\begin{align*}
\langle g_t,g_t \rangle
&=\big\langle \alpha_{s_1}(\cdots(\alpha_{s_m+t}(a_m) a_{m-1}\cdots)a_1)h, \alpha_{s_1}(\cdots(\alpha_{s_m+t}(a_m) a_{m-1}\cdots)a_1)h \big\rangle \\
&= \big\langle\alpha_{s_1}(a_1^* \cdots a_{m-1}^* \alpha_{s_m+t}(a_m^*)\cdots) \alpha_{s_1}(\cdots(\alpha_{s_m+t}(a_m) a_{m-1}\cdots )a_1)h, h \big\rangle \\
&= \big\langle\alpha_{s_1}(a_1^* \alpha_{s_2}(a_2^* \cdots a_{m-1}^* \alpha_{s_m+t}(a_m^* a_m)a_{m-1} \cdots a_{2}))a_1)h, h \big\rangle \\
&= \big\langle\alpha_{s_1}(a_1^* \alpha_{s_2}(a_2^* \cdots a_{m-1}^* P_{s_m+t}(p a^*_m a_m p)a_{m-1} \cdots a_{2}))a_1)h, h \big\rangle \\
&\stackrel{t\rightarrow 0}{\longrightarrow}\big\langle\alpha_{s_1}(a_1^* \alpha_{s_2}(a_2^* \cdots a_{m-1}^* P_{s_m}(p a^*_m a_m p)a_{m-1} \cdots a_{2}))a_1)h, h \big\rangle \\
&= \big\langle\alpha_{s_1}(a_1^* \alpha_{s_2}(a_2^* \cdots a_{m-1}^* \alpha_{s_m}(a_m^* a_m) a_{m-1} \cdots a_{2}))a_1)h, h \big\rangle \\
&= \langle g,g \rangle .
\end{align*}
In addition , we have
\begin{align*}
\langle g,g_t \rangle
&=\big\langle \alpha_{s_1}(\cdots \alpha_{s_{m-1}}(\alpha_{s_m}(a_m) a_{m-1}) \cdots )a_1)h, \alpha_{s_1}(\cdots \alpha_{s_{m-1}}(\alpha_{s_m + t}(a_m) a_{m-1}) \cdots )a_1)h \big\rangle \\
&= \big\langle\alpha_{s_1}(a_1^* \cdots \alpha_{s_{m-1}}(a_{m-1}^* \alpha_{s_m+t}(a^*_m))\cdots) \alpha_{s_1}(\cdots \alpha_{s_{m-1}}(\alpha_{s_m}(a_m) a_{m-1}) \cdots )a_1)h, h \big\rangle \\
&= \big\langle\alpha_{s_1}(a_1^* \cdots \alpha_{s_{m-1}}(a_{m-1}^* \alpha_{s_m}(\alpha_{t}(a_m^*) a_m) a_{m-1})\cdots )a_1)h, h \big\rangle.
\end{align*}
But $a_m = p a_m = p_0 p a_m$, and $p_0$ commutes with $\alpha_t(a_m^*)$, thus
$$\alpha_{t}(a^*_m)a_m = p_0 \alpha_{t}(a^*_m)p_0 a_m \rightarrow a^*_m a_m $$
in the weak operator topology as $t \rightarrow 0$, by (\ref{eq:p_0}). Since $\alpha_{s_i}$ is normal for all $i$, we obtain
$$\langle g,g_t \rangle \stackrel{t\rightarrow 0}{\longrightarrow} \langle g, g \rangle .$$
This allows us to conclude that
$$\|g_t - g\|^2 = \langle g_t , g_t \rangle - 2 \Re \langle g_t, g \rangle + \langle g, g \rangle \stackrel{t\rightarrow 0}{\longrightarrow} 0 ,$$
which shows that $K$ is spanned by elements as $g$ with positive $s_m$. An inductive argument, using the same reasoning
and (\ref{eq:p_0}), shows that $K$ is spanned by elements as $g$ with positive $s_i$ for all $i$, thus $K = K_0$.

Since $p_0 = I_K$, equation (\ref{eq:p_0}) now translates to the weak operator convergence
$$\alpha_t(a) \rightarrow a \,\, \text{as} \,\, t \rightarrow 0 ,$$
for all $a\in B(H)$.
For all $k \in K$, we have
\begin{align*}
\|\alpha_t(a)k - a k \|^2 &=
\langle \alpha_t(a)^* \alpha_t(a) k, k\rangle - 2\Re\langle \alpha_t(a) k , ak \rangle + \|ak\|^2 \\
&= \langle \alpha_t(a^* a) k, k\rangle - 2\Re\langle \alpha_t(a) k , ak \rangle + \|ak\|^2 \\
&\stackrel{s\rightarrow 0}{\longrightarrow} 0,
\end{align*}
thus $\alpha_t(a) \rightarrow a$ in the strong operator topology as $t
\rightarrow 0$, for all $a \in B(H)$.
For all $s\in\diadp$, $\alpha_s$ is normal, thus
$\alpha_t(\alpha_s(a)) = \alpha_s(\alpha_t(a)) \rightarrow
\alpha_s(a)$ in the strong operator topology as $\diadp \ni t \rightarrow 0$.
Denote by $\cA$ the $C^*$-algebra generated by $\bigcup_{s\in\diadp} \alpha_s(B(H))$.
Then we conclude that for all $a \in \cA$, both $\beta_t(a)$ and $\gamma_t(a)$ (recall equation (\ref{eq:betagamma})) converge in the strong operator topology
to $a$ as $\dpp \ni t \rightarrow 0$.

As we have seen above, $W^*(\cA) = \cR = B(K)$, that is, $\cA$ is a $C^*$-algebra strong operator dense in $B(K)$. In particular, $\cA$ acts irreducibly on $K$.
Since $\cA$ contains $B(H)$, it contains nonzero compact operators, and now by
\cite[Proposition 10.4.10]{KRII}, we conclude that
$\cA$ contains \emph{all} compact operators in $B(K)$. In order to use
Theorem \ref{thm:SeLegue} we must show that $\beta_t(1)$ and $\gamma_t(1)$ converge weakly to $1$ as $\dpp \ni t \rightarrow 0$. By a polarization argument, it is enough to show that
\be\label{eq:weakconverge}
\langle \alpha_t(1) k, k \rangle \rightarrow \|k\|^2 \,\,\,\, \text{\rm as}  \,\,\,\, \diadp \ni t \rightarrow 0
\ee
for all $k$ in a total subset of $K$. But taking $h \in H$, $b \in B(K)$ and $s \in \diadpp$, we have that
for $t \leq s$,
$$\alpha_t(1) \alpha_s (b) h = \alpha_t(1) \alpha_s (1) \alpha_s (b) h = \alpha_s(b) h ,$$
because $\alpha_t(1) \leq \alpha_s(1)$. Equation (\ref{eq:weakconverge}) follows.
This means that we may use Theorem \ref{thm:SeLegue} to deduce that $\beta$ and $\gamma$ extend to E-semigroups $\hat{\beta}$ and $\hat{\gamma}$ over $\Rp$. By Proposition \ref{prop:continuity4}, $\{\hat{\beta}_s \circ \hat{\gamma}_t\}_{(s,t)\in \Rpt}$ is a two-parameter E-semigroup dilating $P$.
\end{proof}


\part{Subproduct systems and the dilation theory of $cp$-semigroups}\label{part:II}

\chapter{Subproduct systems of Hilbert $W^*$-correspondences}\label{chap:subproduct}

\begin{definition}\label{def:subproduct_system}
Let $\cN$ be a von Neumann algebra. A \emph{subproduct system of Hilbert $W^*$-correspondences} over $\cN$ is a family $X = \{X(s)\}_{s\in\cS}$ of Hilbert $W^*$-correspondences over $\cN$ such that
\begin{enumerate}
\item $X(0) = \cN$,
\item For every $s,t \in \cS$ there is a coisometric mapping of $\cN$-correspondences
$$U_{s,t}: X(s) \otimes X(t) \rightarrow X(s+t),$$
\item The maps $U_{s,0}$ and $U_{0,s}$ are given by the left and right actions of $\cN$ on $X(s)$,
\item The maps $U_{s,t}$ satisfy the following associativity condition:
\begin{equation}\label{eq:assoc_prod}
U_{s+t,r} \left(U_{s,t} \otimes I_{X(r)} \right) = U_{s,t+r} \left(I_{X(s)} \otimes U_{t,r} \right).
\end{equation}
\end{enumerate}
\end{definition}
The difference between a subproduct system and a product system is that in a subproduct system the maps $U_{s,t}$ are only required to be coisometric, while in a product system these maps are required to be unitaries. Thus, given the image $U_{s,t}(x \otimes y)$ of $x \otimes y$ in $X(s + t)$, one cannot recover $x$ and $y$. Thus, subproduct systems may be thought of as \emph{irreversible} product systems. The terminology is, admittedly, a bit awkward. It may be more sensible -- however, impossible at present -- to use the term \emph{product system} for the objects described above and to use the term \emph{full product system} for product system. We will sometimes add the adjective \emph{full} before the noun \emph{product system} to emphasize that it is a product system and not merely a subproduct system.

\begin{example}\label{expl:full}
\emph{
The simplest example of a subproduct system $F = F_E = \{F(n)\}_{n \in \mb{N}}$ is given by
\bes
F(n) = E^{\otimes n},
\ees
where $E$ is some W$^*$-correspondence. $F$ is actually a product system. We shall call this subproduct system \emph{the full product system (over $E$)}.}
\end{example}

\begin{example}\label{expl:symm}\emph{
Let $E$ be a fixed Hilbert space. We define a subproduct system (of Hilbert spaces) $SSP = SSP_E$ over $\mb{N}$ using the familiar symmetric tensor products (one can obtain a subproduct system from the anti-symmetric tensor products as well). Define
$$E^{\otimes n} = E \otimes \cdots \otimes E ,$$
($n$ times). Let $p_n$ be the projection of $E^{\otimes n}$ onto the symmetric subspace of $E^{\otimes n}$, given by
$$p_n k_1 \otimes \cdots \otimes k_n = \frac{1}{n!}\sum_{\sigma \in S_n} k_{\sigma^{-1}(1)} \otimes \cdots \otimes k_{\sigma^{-1}(n)}.$$
We define
$$SSP(n) = E^{\circledS n} : = p_n  E^{\otimes n},$$
the symmetric tensor product of $E$ with itself $n$ times ($SSP(0) = \mb{C}$).
We define a map $U_{m,n}: SSP(m) \otimes SSP(n) \rightarrow SSP(m+n)$ by
$$U_{m,n} (x \otimes y) = p_{m+n} (x \otimes y).$$
The $U$'s are coisometric maps because every projection, when considered as a map from its domain onto its range, is coisometric. A straightforward calculation shows that (\ref{eq:assoc_prod}) holds (see \cite[Corollary 17.2]{Parthasarathy}). In these notes we shall refer to $SSP$ (or $SSP_E$ to be precise) as the \emph{symmetric subproduct system (over $E$)}. }
\end{example}
\begin{definition}
Let $X$ and $Y$ be two subproduct systems over the same semigroup $\cS$ (with families of coisometries $\{U_{s,t}^X\}_{s,t \in \cS}$ and $\{U_{s,t}^Y\}_{s,t \in \cS}$).  A family $V=\{V_s\}_{s\in \cS}$ of maps $V_{s}: X(s) \rightarrow Y(s)$ is called \emph{a morphism} of subproduct systems if $V_0$ is a unital $*$-isomorphism, if for all $s \in \cS \setminus \{0\}$ the map $V_s$ is a coisometric correspondence map, and if for all $s,t \in \cS$ the following identity holds:
\be\label{eq:iso}
V_{s+t}\circ U_{s,t}^X = U_{s,t}^Y \circ (V_s \otimes V_t) .
\ee
$V$ is said to be an \emph{isomorphism} if $V_s$ is a unitary for all $s \in \cS \setminus \{0\}$.
$X$ is said to be isomorphic to $Y$ if there exists an isomorphism $V: X \rightarrow Y$.
\end{definition}
The above notion of \emph{morphism} is not optimized in any way. It is simply precisely what we need in order to develop dilation theory for $cp$-semigroups.

\begin{definition}\label{def:rep}
Let $\cN$ be a von Neumann algebra, let $H$ be a Hilbert space, and let $X$ be a
subproduct system of Hilbert $\cN$-correspondences over the semigroup
$\cS$. Assume that $T : X \rightarrow B(H)$, and write $T_s$ for
the restriction of $T$ to $X(s)$, $s \in \cS$, and $\sigma$ for
$T_0$. $T$ (or $(\sigma, T)$) is said to be a \emph{completely
contractive covariant representation} of $X$ if
\begin{enumerate}
    \item For each $s \in \cS$, $(\sigma, T_s)$ is a c.c. representation of $X(s)$; and
    \item\label{it:sg} $T_{s+t}(U_{s,t}(x\otimes y)) = T_s(x)T_t(y)$ for all $s,t \in \cS$ and all $x \in X(s), y \in X(t)$.
\end{enumerate}
$T$ is said to be an isometric (fully coisometric) representation if it is an isometric (fully coisometric) representation on every fiber $X(s)$.
\end{definition}
Since we shall not be concerned with any other kind of representation, we shall call a completely contractive covariant representation of a subproduct system simply a \emph{representation}.

\begin{remark}\label{rem:reptilde}
\emph{Item 2 in the above definition of product system can be rewritten as follows:
\bes
\tT_{s+t} (U_{s,t} \otimes I_H) = \tT_s (I_{X(s)} \otimes \tT_t).
\ees
Here $\tT_s: X(s) \otimes_\sigma H \rightarrow H$ is the map given by}
$$\tT_s (x \otimes h) = T_s(x)h .$$
\end{remark}

\begin{example}\label{expl:FockRep}\emph{
We now define a representation $T$ of the symmetric subproduct system $SSP$ from Example \ref{expl:symm} on the symmetric Fock space. Denote by $\mathfrak{F}_+$ the symmetric Fock space
$$\mathfrak{F}_+ = \bigoplus_{n \in \mb{N}} E^{\circledS n} .$$
For every $n \in \mb{N}$, the map $T_n : SSP(n) = E^{\circledS n} \rightarrow B(\mathfrak{F}_+)$  is defined on the $m$-particle space $E^{\circledS m}$ by putting
$$T_n(x) y = p_{n+m}(x \otimes y) $$
for all $x \in X(n), y \in E^{\circledS m}$. Then $T$ extends to a representation of the subproduct system $SSP$ on $\mathfrak{F}_+$ (to see that item \ref{it:sg} of Definition \ref{def:rep} is satisfied one may use again \cite[Corollary 17.2]{Parthasarathy}).}
\end{example}
\begin{definition}
Let $X = \{X(s)\}_{s \in \cS}$ be a subproduct
system of $\cN$-correspondences over $\cS$. A family $\xi = \{\xi_s \}_{s \in \cS}$
is called a \emph{unit} for $X$ if
\be\label{eq:unit}
\xi_s \otimes \xi_t = U_{s,t}^* \xi_{s+t}.
\ee
A unit $\xi = \{\xi_s \}_{s \in \cS}$ is called \emph{unital} if $\lel \xi_s , \xi_s \rir = 1_{\cN}$ for all $s \in \cS$,
it is called \emph{contractive} if $\lel \xi_s , \xi_s \rir \leq 1_{\cN}$ for all $s \in \cS$,
and it is called \emph{generating} if $X(s)$ is spanned by elements of the form
\be\label{eq:spanning}
U_{s_1 + \cdots + s_{n-1},s_n}(\cdots U_{s_1+s_2,s_3}( U_{s_1,s_2}(a_1 \xi_{s_1} \otimes a_2 \xi_{s_2}) \otimes a_3 \xi_{s_3}) \otimes \cdots \otimes a_n \xi_{s_n} a_{n+1} ),
\ee
where $s = s_1 + s_2 + \cdots + s_n$.
\end{definition}

From (\ref{eq:unit}) follows the perhaps more natural looking
\bes
U_{s,t}(\xi_s \otimes \xi_t) = \xi_{s+t}.
\ees
\begin{example}\label{expl:symmUnit}\emph{
A unital unit for the symmetric subproduct system $SSP$ from Example \ref{expl:symm} is given by
defining $\xi_0 = 1$ and
\bes
\xi_{n} = \underbrace{v \otimes v \otimes \cdots \otimes v}_{n \text{ times}}
\ees
for $n \geq 1$. This unit is generating only if $E$ is one dimensional.
}
\end{example}

\chapter{Subproduct system representations and $cp$-semigroups}\label{chap:subncp}

In this chapter, following Muhly and Solel's constructions from \cite{MS02}, we show that subproduct systems and their representations provide a framework for dealing with $cp$-semigroups, and allow us to obtain a generalization
of the classical result of Wigner that any strongly continuous one-parameter group of automorphisms of
$B(H) $ is given by $X\mapsto U_t X U_t^*$ for
some one-parameter unitary group $\{U_t\}_{t\in \mathbb{R}}$.

\section{All $cp$-semigroups come from subproduct system representations}

There is a small overlap of material between the material presented in this section and some of the material of Chapter \ref{chap:unital}. The following Proposition, for example, is very similar to Lemma \ref{lem:semigroup}.

\begin{proposition}\label{prop:semigroup}
Let $\cN$ be a von Neumann algebra and let $X$ be a subproduct system of $\cN$-correspondences over $\cS$, and let $R$ be completely contractive covariant representation of $X$ on a Hilbert space $H$, such that $R_0$ is unital. Then the family of maps
\be\label{eq:reprep}
\Theta_s : a \mapsto \widetilde{R}_s (I_{X(s)} \otimes a) \widetilde{R}_s^* \,\, , \,\, a \in R_0 (\cN)',
\ee
is a semigroup of CP maps on $R_0 (\cN)'$. Moreover, if $R$ is an isometric (a fully coisometric) representation, then $\Theta_s$ is a $*$-endomorphism (a unital map) for all $s\in\cS$.
\end{proposition}
\begin{proof}
By Proposition 2.21 in \cite{MS02}, $\{\Theta_s\}_{s\in\cS}$ is a family of contractive, normal, completely positive maps on $R_0(\cN)'$. Moreover, these maps are unital if $R$ is a fully coisometric representation, and they are $*$-endomorphisms if $R$ is an isometric representation. It remains is to check that $\Theta = \{\Theta_s \}_{s\in\cS}$ satisfies the semigroup condition $\Theta_s \circ \Theta_t = \Theta_{s + t}$. Fix $a \in R_0 (\cN)'$. For all $s,t\in\cS$,
\begin{align*}
\Theta_s (\Theta_t (a))
&= \widetilde{R}_s \left(I_{X(s)} \otimes \left(\widetilde{R}_t (I_{X(t)} \otimes a) \widetilde{R}_t^*\right)\right) \widetilde{R}_s^* \\
&= \widetilde{R}_s (I_{X(s)} \otimes \widetilde{R}_t) (I_{X(s)}\otimes I_{X(t)} \otimes a)(I_{X(s)} \otimes \widetilde{R}_t^*) \widetilde{R}_s^* \\
&= \widetilde{R}_{s + t} (U_{s,t}\otimes I_G)(I_{X(s)}\otimes I_{X(t)}\otimes a)(U_{s,t}^{*}\otimes I_G)\widetilde{R}_{s + t}^* \\
&= \widetilde{R}_{s + t} (I_{X(s\cdot t)}\otimes a)\widetilde{R}_{s + t}^* \\
&= \Theta_{s + t}(a) .
\end{align*}
Using the fact that $R_0$ is unital, we have
\begin{align*}
\Theta_0(a) h
&= \widetilde{R_0} (I_{\cN} \otimes a) \widetilde{R_0}^* h \\
&= \widetilde{R_0} (I_{\cN} \otimes a) (1_{\cN} \otimes h) \\
&= R_0(1_{\cN})ah \\
&= ah ,
\end{align*}
thus $\Theta_0(a) = a$ for all $a\in \cN$.
\end{proof}

We will now show that \emph{every} $cp$-semigroup is given by a subproduct representation
as in (\ref{eq:reprep}) above. We recall some constructions from \cite{MS02} (building on the foundations  set in \cite{Arv97b}). The following material will look familiar to the reader who remembers Section \ref{sec:MS} above, as it is the basis of the constructions made there.

Fix a CP map $\Theta$ on von Neumann algebra $\cM \subseteq B(H)$. We define $\cM \otimes_\Theta H$ to be the Hausdorff completion of the algebraic tensor product $\cM \otimes H$ with respect to the sesquilinear positive semidefinite form
\bes
\lel T_1 \otimes h_1 , T_2 \otimes h_2 \rir = \lel h_1, \Theta(T_1^* T_2) h_2 \rir .
\ees
We define a representation $\pi_\Theta$ of $\cM$ on $\cM \otimes_\Theta H$ by
\bes
\pi_\Theta (S) (T \otimes h) = ST \otimes h,
\ees
and we define a (contractive) linear map $W_\Theta : H \rightarrow \cM \otimes H$ by
\bes
W_\Theta (h) = I \otimes h.
\ees
If $\Theta$ is unital then $W_\Theta$ is an isometry, and if $\Theta$ is an endomorphism then $W_\Theta$ is
a coisometry. The adjoint of $W_\Theta$ is given by
\bes
W_\Theta^* (T \otimes h) = \Theta(T)h .
\ees

For a given $cp$-semigroup $\Theta$ on $\cM$, Muhly and Solel defined in \cite{MS02} a $W^*$-correspondence $E_\Theta$ over $\cM'$ and a c.c.
representation $(\sigma,T_\Theta)$ of $E_\Theta$ on $H$ such that for all $a \in \cM$
\be\label{eq:reprep1}
\Theta(a) = \tT_\Theta \left(I_{E_\Theta} \otimes a \right)\tT_\Theta^* .
\ee
The $W^*$-correspondence $E_\Theta$ is defined as the intertwining space
\bes
E_\Theta = \cL_\cM (H, \cM \otimes_\Theta H),
\ees
where
\bes
\cL_\cM (H, \cM \otimes_\Theta H):= \{X \in B(H,\cM \otimes_\Theta H) \big| \forall T \in \cM. XT = \pi_\Theta(T)X \}.
\ees
The left and right actions of $\cM'$ are given by
\bes
S \cdot X = (I \otimes S)X \quad , \quad X \cdot S = XS
\ees
for all $X \in E_\Theta$ and $S \in \cM'$. The $\cM'$-valued inner product on $E_\Theta$ is defined by
$\lel X, Y \rir = X^* Y$. $E_\Theta$ is called \emph{the Arveson-Stinespring correspondence} (associated with $\Theta$).

The representation $(\sigma,T_\Theta)$ is defined by letting $\sigma = {\bf id}_{\cM'}$, the identity representation
of $\cM'$ on $H$, and by defining
\bes
T_\Theta(X) = W_\Theta^* X .
\ees
$({\bf id}_{\cM'}, T_\Theta)$ is called \emph{the identity representation} (associated with $\Theta$). We remark that the paper \cite{MS02}
focused on unital CP maps, but the results we cite are true for nonunital CP maps, with the proofs unchanged.

The case where $\cM = B(H)$ in the following theorem appears, in essence at least, in \cite{Arv97b}.

\begin{theorem}\label{thm:reprep}
Let $\Theta = \{\Theta_s \}_{s\in\cS}$ be a $cp$-semigroup on a von Neumann algebra $\cM \subseteq B(H)$, and for all $s\in \cS$ let $E(s) := E_{\Theta_s}$ be the Arveson-Stinespring correspondence
associated with $\Theta_s$, and let $T_s := T_{\Theta_s}$ denote the identity representation for $\Theta_s$. Then
$E = \{E(s)\}_{s\in\cS}$ is a subproduct system of $\cM'$-correspondences, and $({\bf id}_{\cM'},T)$ is a representation of $E$ on $H$ that satisfies
\be\label{eq:RepRep}
\Theta_s(a) = \tT_s\left(I_{E(s)} \otimes a \right) \tT_s^*
\ee
for all $a \in \cM$ and all $s \in \cS$. $T_s$ is injective for all $s \in \cS$.
If $\Theta$ is an $e$-semigroup (cp$_0$-semigroup), then $({\bf id}_{\cM'},T)$ is isometric (fully coisometric).
\end{theorem}

\begin{proof}
We begin by defining the subproduct system maps $U_{s,t}:E(s) \otimes E(t) \rightarrow E(s+t)$. We use the constructions made in \cite[Proposition 2.12]{MS02} and the surrounding discussion. We define
\bes
U_{s,t} = V_{s,t}^*\Psi_{s,t} \,\,,
\ees
where the map
\bes
\Psi_{s,t}: \cL_\cM(H,\cM \otimes_{\Theta_s} H) \otimes \cL_\cM(H,\cM \otimes_{\Theta_t} H) \rightarrow \cL_\cM(H,\cM \otimes_{\Theta_t} \cM \otimes_{\Theta_s} H)
\ees
is given by $\Psi_{s,t}(X \otimes Y) = (I \otimes X)Y$, and
\bes
V_{s,t}: \cL_\cM(H,\cM \otimes_{\Theta_{s+t}}H) \rightarrow \cL_\cM(H,\cM \otimes_{\Theta_t} \cM \otimes_{\Theta_s} H)
\ees
is given by $V_{s,t}(X) = \Gamma_{s,t} \circ X$, where $\Gamma_{s,t}: \cM \otimes_{\Theta_{s+t}} H \rightarrow \cM \otimes_{\Theta_t} \cM \otimes_{\Theta_s} H$ is the isometry
\bes
\Gamma_{s,t} : S \otimes_{\Theta_{s+t}} h \mapsto S \otimes_{\Theta_t} I \otimes_{\Theta_s} h .
\ees
By  \cite[Proposition 2.12]{MS02}, $U_{s,t}$ is a coisometric bimodule map for all $s,t \in \cS$. To see that the $U$'s compose associatively as in (\ref{eq:assoc_prod}), take $s,t,u \in \cS$, $X \in E(s), Y \in E(t), Z \in E(u)$, and compute:
\begin{align*}
U_{s,t+u}(I_{E(s)} \otimes U_{t,u})(X \otimes Y \otimes Z)
&= U_{s,t+u}(X \otimes V_{t,u}^*(I \otimes Y)Z) \\
&= V_{s,t+u}^*\left((I \otimes X) V_{t,u}^*(I \otimes Y)Z \right) \\
&= \Gamma_{s,t+u}^*(I \otimes X) \Gamma_{t,u}^*(I \otimes Y)Z
\end{align*}
and
\begin{align*}
U_{s+t,u}(U_{s,t} \otimes I_{E(u)})(X \otimes Y \otimes Z)
&= U_{s+t,u}( V_{s,t}^*(I \otimes X)Y \otimes Z) \\
&= V_{s+t,u}^*\left((I \otimes V_{s,t}^*(I \otimes X)Y)Z\right) \\
&= \Gamma_{s+t,u}^*(I \otimes \Gamma_{s,t}^*)(I \otimes I \otimes X)(I \otimes Y)Z \,\, .
\end{align*}
So it suffices to show that
\bes
\Gamma_{s,t+u}^*(I \otimes X) \Gamma_{t,u}^* = \Gamma_{s+t,u}^*(I \otimes \Gamma_{s,t}^*)(I \otimes I \otimes X)
\ees
It is easier to show that their adjoints are equal. Let $a \otimes h$ be a typical element of $\cM \otimes_{\Theta_{s+t+u}} h$.
\begin{align*}
\Gamma_{t,u}(I \otimes X^*)\Gamma_{s,t+u} (a \otimes_{\Theta_{s+t+u}} h)
&= \Gamma_{t,u}(I \otimes X^*)(a \otimes_{\Theta_{t+u}} I \otimes_{\Theta_s} h) \\
&= \Gamma_{t,u}(a \otimes_{\Theta_{t+u}} X^*(I\otimes_{\Theta_s} h)) \\
&= a \otimes_{\Theta_{u}} I \otimes_{\Theta_{t}}X^*(I\otimes_{\Theta_s} h).
\end{align*}
On the other hand
\begin{align*}
(I \otimes I \otimes X^*)(I \otimes \Gamma_{s,t})\Gamma_{s+t,u} (a \otimes_{\Theta_{s+t+u}} h)
&= (I \otimes I \otimes X^*)(I \otimes \Gamma_{s,t}) (a \otimes_{\Theta_{u}} I \otimes_{\Theta_{s+t}} h)  \\
&= (I \otimes I \otimes X^*) (a \otimes_{\Theta_{u}} I \otimes_{\Theta_{t}} I \otimes_{\Theta_{s}} h)\\
&= a \otimes_{\Theta_{u}} I \otimes_{\Theta_{t}}X^*(I\otimes_{\Theta_s} h).
\end{align*}
This shows that the maps $\{U_{s,t}\}$ make $E$ into a subproduct system.

Let us now verify that $T$ is a representation of subproduct systems. That $({\bf id}_{\cM'},T_s)$ is a c.c. representation
of $E(s)$ is explained in \cite[page 878]{MS02}. Let $X \in E(s), Y \in E(t)$.
\bes
T_{s+t}(U_{s,t}(X \otimes Y)) = W_{\Theta_{s+t}}^* \Gamma_{s,t}^*(I \otimes X)Y,
\ees
while
\bes
T_{s}(X) T_t(Y) = W_{\Theta_{s}}^*XW_{\Theta_{t}}^*Y.
\ees
But for all $h \in H$,
\begin{align*}
W_{\Theta_{t}} X^* W_{\Theta_{s}} h &= W_{\Theta_{t}} X^* (I \otimes_{\Theta_s} h) \\
&= I \otimes_{\Theta_t} X^* (I \otimes_{\Theta_s} h) \\
&= (I \otimes X^*) (I \otimes_{\Theta_t} I \otimes_{\Theta_s} h) \\
&= (I \otimes X^*) \Gamma_{s,t}(I \otimes_{\Theta_{s+t}} h) \\
&= (I \otimes X^*) \Gamma_{s,t}W_{\Theta_{s+t}} h ,
\end{align*}
which implies
$W_{\Theta_{s}}^*XW_{\Theta_{t}}^*Y = W_{\Theta_{s+t}}^* \Gamma_{s,t}^*(I \otimes X)Y$, so
we have the desired equality
\bes
T_{s+t}(U_{s,t}(X \otimes Y)) = T_{s}(X) T_t(Y) .
\ees
Equation (\ref{eq:RepRep}) is a consequence of (\ref{eq:reprep1}).
The injectivity of the identity representation has already been noted by Solel in \cite{S06} (for all $h, g \in H$ and $a \in M$, $\langle W^*_\Theta X a^* h, g \rangle =  \langle Xa^* h, I \otimes g \rangle = \langle (a^* \otimes I)Xh, I \otimes g \rangle = \langle Xh, a \otimes g \rangle$ ).
The final assertion of the theorem is trivial (if $\Theta_s$ is a $*$-endomorphism, then $W_{\Theta_s}$ is a coisometry).
\end{proof}

\begin{definition}\label{def:identityrep}
Given a $cp$-semigroup $\Theta$ on a $W^*$ algebra $\cM$, the pair $(E,T)$ constructed in Theorem \ref{thm:reprep} is called \emph{the identity representation} of $\Theta$, and $E$ is called \emph{the Arveson-Stinespring subproduct system} associated with $\Theta$.
\end{definition}

\begin{remark}\label{rem:identityproduct}
\emph{If follows from \cite[Proposition 2.14]{MS02}, if $\Theta$ is an $e$-semigroup, then the Arveson-Stinespring
subproduct system is, in fact, a (full) product system.}
\end{remark}

\begin{remark}\emph{
In \cite{Markiewicz}, Daniel Markiewicz has studied the Arveson-Stinespring subproduct system of a CP$_0$-semigroup over $\mb{R}_+$ acting on $B(H)$, and has also shown that it carries a structure of a \emph{measurable Hilbert bundle}. }
\end{remark}

\section{Essentially all injective subproduct system representations come from $cp$-semigroups}\label{sec:essentially}

The following generalizes and is motivated by \cite[Proposition 5.7]{S06}. We shall also repeat some arguments from  \cite[Theorem 2.1]{MS07}.

By Theorem \ref{thm:reprep}, with every $cp$-semigroup $\Theta = \{\Theta_s \}_{s \in \cS}$ on $\cM \subseteq B(H)$ we
can associate a pair $(E, T)$ - the identity representation of $\Theta$ - consisting of a subproduct system $E$ (of correspondences over $\cM'$) and an injective subproduct system c.c. representation $T$. Let us write $(E,T) = \Xi(\Theta)$.
Conversely, given a pair $(X,R)$ consisting of a subproduct system $X$ of correspondences over $\cM'$ and a c.c. representation $R$ of $X$ such that $R_0$ is unital, one may define by equation (\ref{eq:reprep}) a $cp$-semigroup $\Theta$ acting on $\cM$, which we denote as $\Theta = \Sigma(X,R)$. The meaning of equation (\ref{eq:RepRep}) is that $\Sigma \circ \Xi$ is the identity map on the set of $cp$-semigroups of $\cM$.
We will show below that $\Xi \circ \Sigma$, when restricted to pairs $(X,R)$ such that $R$ is injective, is also, essentially, the identity. When $(X,R)$ is not injective, we will show that $\Xi \circ \Sigma(X,R)$ ``sits inside" $(X,R)$.

\begin{theorem}\label{thm:essentially_inverse}
Let $\cN$ be a W$^*$-algebra, let $X = \{X(s) \}_{s \in \cS}$ be a subproduct system of $\cN$-correspondences, and let $R$ be a c.c. representation of $X$ on $H$, such that $\sigma := R_0$ is faithful and nondegenerate. Let $\cM \subseteq B(H)$ be the commutant $\sigma(\cN)'$ of $\sigma(\cN)$. Let $\Theta = \Sigma(X,R)$, and let $(E,T) = \Xi(\Theta)$. Then there is a morphism of subproduct systems $W: X \rightarrow E$ such that
\be\label{eq:isosubrep}
R_s = T_s \circ W_s \,\, ,\,\, s \in \cS.
\ee
$W_s^* W_s = I_{X(s)} - q_s$, where $q_s$ is the orthogonal projection of $X(s)$ onto $\textrm{Ker}R_s$.
In particular, $W$ is an isomorphism if and only if $R_s$ is injective for all $s \in \cS$.
\end{theorem}
\begin{remark}
\emph{The construction of the morphism $W$ below basically comes from \cite{S06,MS07}, and it remains only to show that it respects the subproduct system structure. The details are carried out for completeness.}
\end{remark}
\begin{proof}
We may construct a subproduct system $X'$ of $\cM'$-correspondences (recall that $\cM' = \sigma(\cN)$), and a representation $T'$ of $X'$ on $H$ such that $T'_0$ is the identity, in such a way that $(X,T)$ may be naturally identified with $(X',T')$. Indeed, put
\bes
X'(0) = \cM' \,\,,\,\, X'(s) = X(s) \,\, \textrm{for } s \neq 0,
\ees
where the inner product is defined by
\bes
\langle x, y \rangle_{X'} = \sigma ( \langle x, y \rangle_{X} ),
\ees
and the left and right actions are defined by
\bes
a \cdot x \cdot b := \sigma^{-1}(a) x \sigma^{-1}(b),
\ees
for $a,b \in \cM'$ and $x,y \in X'(s)$, $s \in \cS \setminus \{0\}$. Defining $T'_0 = id$ and $W_0 = \sigma$; and $T'_s = T_s$ for and $W_s = id$ for $s \in \cS \setminus \{0\}$, we have that $W$ is a morphism $X \rightarrow X'$ that sends $T$ to $T'$.

Assume, therefore, that $\cN = \cM'$ and that $\sigma$ is the identity representation.

We begin by defining for every $s \neq 0$
\bes
v_s: \cM \otimes_{\Theta_s} H \rightarrow X(s) \otimes H
\ees
by
\bes
v_s(a \otimes h) = (I_{X(s)} \otimes a)\widetilde{R}_s^*h.
\ees
We claim that for all $s \in \cS$ the map $v_s$ is a well-defined isometry. To see that, let $a,b \in \cM$ and $h,g \in H$, and compute:
\begin{align*}
\langle a \otimes_{\Theta_s} h,  b \otimes_{\Theta_s} g \rangle
&= \langle h , \Theta_s(a^* b)g \rangle \\
&= \langle h , \widetilde{R}_s(I_{X(s)} \otimes a^* b)\widetilde{R}_s^* g \rangle \\
&= \langle (I_{X(s)} \otimes a)\widetilde{R}_s^* h , (I_{X(s)} \otimes b)\widetilde{R}_s^* g \rangle.
\end{align*}
$[(I_{X(s)} \otimes \cM)\widetilde{R}_s^*H]$ is invariant under $I_{X(s)} \otimes \cM$, thus the projection onto the orthocomplement of this subspace is in $(I_{X(s)} \otimes \cM)' = \cL (X(s)) \otimes I_H$, so it has the form $q_s \otimes I_H$ for some projection $q_s \in \cL(X(s))$. In fact, $q_s$ is the orthogonal projection of $X(s)$ onto $\textrm{Ker} R_s$:
for all $g,h \in H$, $a \in \cM$,
\begin{align*}
\langle \xi \otimes h, (I \otimes a)\widetilde{R}_s^* g \rangle
&= \langle \widetilde{R}_s (\xi \otimes a^*h), g \rangle \\
&= \langle R_s (\xi) a^*h, g \rangle ,
\end{align*}
thus, $R_s(\xi) = 0$ if and only if $\xi \otimes h \in \left(\textrm{Im} v_s\right)^\perp$ for all $h \in H$, that is, $\xi \in q_s X(s)$.

By the definition of $v_s$ and by the covariance properties of $T$, we have for all $a \in \cM$ and $b \in \cM'$,
\bes
v_s(a \otimes I) = (I \otimes a)v_s \,\, , \,\, v_s (I \otimes b) = (b \otimes I)v_s .
\ees

Fix $s \in \cS$ and $x \in E(s)$. For all $\xi \in X(s), h \in H$, write
\bes
\psi(\xi)h = x^* v_s^*(\xi \otimes h) .
\ees
For $a \in \cM$ we have
\begin{align*}
\psi(\xi)a h &= x^*v_s^* (\xi \otimes ah) \\
&= x^*v_s^*(I \otimes a)(\xi \otimes h) \\
&= x^*(a \otimes I)v_s^*(\xi \otimes h) \\
&= a x^* v_s^*(\xi \otimes h) = a \psi(\xi)h.
\end{align*}
Thus the linear map $\xi \mapsto \psi(\xi)$ maps from $X(s)$ into $\cM'$ and it is apparent that this map is bounded. $\psi$
is also a right module map: for all $b \in \cM'$, $\psi(\xi b) h = x^* v^*(\xi b \otimes h) = x^* v^*(\xi  \otimes bh) = \psi(\xi) b h$. From the self duality of $X(s)$ it follows that there is a unique element in $X(s)$, which we denote by $V_s(x)$, such that for all $\xi \in X(s), h \in H$,
\be\label{eq:V_s}
\langle V_s(x), \xi \rangle h = x^* v_s^*(\xi \otimes h).
\ee
For $a,b \in \cM'$, $\xi \in X(s)$ and $h \in H$, we have
\begin{align*}
\langle V_s(a  x b), \xi \rangle h &= \langle V_s((I \otimes a) x b), \xi  \rangle h \\
&= b^* x^* (I \otimes a^*)v_s^* (\xi \otimes h) \\
&= b^* x^* v_s^*  (a^* \xi \otimes h) \\
&= b^* \langle V_s(x), a^* \xi \rangle h \\
&= \langle a V_s(x) b , \xi \rangle h .
\end{align*}
That is, $V_s(a  x b) = a V_s(x) b$. In a similar way one proves that $V_s : E(s) \rightarrow X(s)$ is linear.

Let us now show that $V_s$ preserves inner products. For all $\xi \in X(s)$, write $L_\xi$ for the operator $L_\xi : H \rightarrow X(s) \otimes H$ that maps $h$ to $\xi \otimes h$. One checks that $L_\xi^* (\eta \otimes h) = \langle \xi, \eta \rangle h$, so equation (\ref{eq:V_s}) becomes
\bes
L_{V_s(x)}^* L_\xi = x^* v_s^* L_\xi \,\, , \,\, \xi \in X(s),
\ees
or $L_{V_s(x)} = v_s x$, for all $x \in E(s)$. But since $v_s$ preserves inner products, we have for all $x,y \in E(s)$:
\bes
\langle x, y \rangle = x^* y = x^* v_s^* v_s y = L_{V_s(x)}^* L_{V_s(y)} = \langle V_s(x), V_s(y) \rangle.
\ees

We now prove that $V_s V_s^* = I_{X(s)} - q_s$.
$\xi \in \textrm{Im}V_s ^\perp$ if and only if $L_\xi^* v_s E(s) H = 0$. But by \cite[Lemma 2.10]{MS02},  $E(s)H = \cM \otimes_{\Theta_s} H$, thus $L_\xi^* v_s E(s) H = 0$ if and only if $\langle \xi , \eta \rangle = 0$ for all $ \eta \in (I_{X(s)} - q_s)X(s)$, which is the same as $\xi \in q_s X(s)$.

Fix $h,k \in H$. For $x \in E(s)$, we compute:
\begin{align*}
\langle T_s(x) h, k \rangle &= \langle W_{\Theta_s}^* x h, k \rangle \\
&= \langle x h, I \otimes_{\Theta_s} k \rangle \\
&= \langle v_s x h, v_s ( I \otimes_{\Theta_s} k) \rangle \\
&= \langle V_s(x) \otimes h, \widetilde{R}_s^* k \rangle \\
&= \langle R_s (V_s(x)) h, k \rangle ,
\end{align*}
thus $T_s = R_s \circ V_s$ for all $s\in \cS$. Define $W_s = V_s^*$. Then $T_s = R_s \circ W_s^*$. Multiplying both sides by $W_s$ we obtain $T_s \circ W_s = R_s \circ W_s^* W_s$. But $W_s^* W_s = I - q_s$ is the orthogonal projection onto $\left(\textrm{Ker}R_s\right)^\perp$, thus we obtain (\ref{eq:isosubrep}).

Finally, we need to show that $W = \{W_s \}$ respects the subproduct system structure: for all $s,t \in \cS$, $x \in X(s)$ and $y \in X(t)$, we must show that
\bes
W_{s+t}(U^X_{s,t}(x \otimes y)) = U^E_{s,t}(W_s(x) \otimes W_t(y)).
\ees
Since $T_{s+t}$ is injective, it is enough to show that after applying $T_{s+t}$ to both sides of the above equation we get the same thing. But $T_{s+t}$ applied to the left hand side gives
\bes
T_{s+t}W_{s+t}(U^X_{s,t}(x \otimes y)) = R_{s+t}(U^X_{s,t}(x \otimes y)) = R_s(x) R_t(y),
\ees
and $T_{s+t}$ applied to the right hand side gives
\bes
T_{s+t}(U^E_{s,t}(W_s(x) \otimes W_t(y))) = T_s(W_s(x)) T_t(W_t(y)) = R_s(x) R_t(y).
\ees
\end{proof}

\begin{corollary}\label{cor:onlyproductisorep}
Let $X$ be a subproduct system that has an isometric representation $V$ such that $V_0$ is faithful and nondegenerate. Then $X$ is a (full) product system.
\end{corollary}
\begin{proof}
Let $\Theta = \Sigma(X,V)$. Then $\Theta$ is an $e$-semigroup. Thus, if $(E,T) = \Xi(\Theta)$ is the identity representation of $\Theta$, then, by Remark \ref{rem:identityproduct}, $E$ is a (full) product system. But if $V_0$ is faithful and $V$ is isometric then $V$ is injective. By the above theorem, $X$ is isomorphic to $E$, so it is a product system.
\end{proof}

\section{Subproduct systems arise from $cp$-semigroups. The shift representation}\label{sec:shift}

A question rises: \emph{does every subproduct system arise as the Arveson-Stinespring subproduct system associated with a $cp$-semigroup?} By Theorem \ref{thm:essentially_inverse}, this is equivalent to the question \emph{does every subproduct system have an injective representation?} We shall answer this question in the affirmative by constructing for every such subproduct system a canonical injective representation.

The following constructs will be of most interest when $\cS$ is a countable semigroup, such as $\mb{N}^k$.
\begin{definition}\label{def:shiftrep}
Let $X = \{X(s)\}_{s\in\cS}$ be a subproduct system. The \emph{$X$-Fock space} $\mathfrak{F}_X$ is defined as
\bes
\mathfrak{F}_X = \bigoplus_{s \in \cS} X(s).
\ees
The vector $\Omega := 1 \in \cN = X(0)$ is called \emph{the vacuum vector of $\mathfrak{F}_X$}.
The \emph{$X$-shift} representation of $X$ on $\mathfrak{F}_X$ is the representation
\bes
S^X: X \rightarrow B(\mathfrak{F}_X),
\ees
given by $S^X(x) y = U^X_{s,t}(x\otimes y)$, for all $x \in X(s), y \in X(t)$ and all $s,t \in \cS$.
\end{definition}
Strictly speaking, $S^X$ as defined above is not a representation because it represents $X$ on a C$^*$-correspondence rather
than on a Hilbert space. However, since for any C$^*$-correspondence $E$, $\cL(E)$ is a C$^*$-algebra, one can
compose a faithful representation $\pi : \cL(E) \rightarrow B(H)$ with $S^X$ to obtain a representation on a Hilbert space.

A direct computation shows that $\widetilde{S}^X_s : X(s) \otimes \mathfrak{F}_X \rightarrow \mathfrak{F}_X$ is a contraction, and also that $S^X(x)S^X(y) = S^X(U^X_{s,t}(x\otimes y))$ so $S^X$ is a completely contractive representation of $X$. $S^X$ is also injective because $S^X(x)\Omega = x$ for all $x \in X$. Thus,
\begin{corollary}
Every subproduct system is the Arveson-Stinespring subproduct system of a $cp$-semigroup.
\end{corollary}

\chapter{Subproduct system units and $cp$-semigroups}\label{chap:subunitsncp}

In this section, following Bhat and Skeide's constructions from \cite{BS00}, we show that subproduct systems and their units may also serve as a tool for studying $cp$-semigroups.

\begin{proposition}
Let $\cN$ be a von Neumann algebra and let $X$ be a subproduct system of $\cN$-correspondences over $\cS$, and let
$\xi = \{\xi_s\}_{s\in\cS}$ be a contractive unit of $X$, such that $\xi_0 = 1_\cN$. Then the family of maps
\be\label{eq:unitrep}
\Theta_s : a \mapsto \lel \xi_s, a \xi_s \rir,
\ee
is a semigroup of CP maps on $\cN$. Moreover, if $\xi$ is unital, then $\Theta_s$ is a unital map for all $s\in\cS$.
\end{proposition}
\begin{proof}
It is standard that $\Theta_s$ given by (\ref{eq:unitrep}) is a contractive completely positive map on $\cN$, which is unital if and only if $\xi$ is unital. The fact that $\Theta_s$ is normal goes a little bit deeper, but is also known (one may use \cite[Remark 2.4(i)]{MS02}).

We show that $\{\Theta_s\}_{s \in \cS}$ is a semigroup. It is clear that $\Theta_0(a) = a$ for all $a \in \cN$. For all $s,t \in \cS$,
\begin{align*}
\Theta_s(\Theta_t(a)) &= \lel \xi_s, \lel \xi_t, a \xi_t \rir \xi_s \rir \\
&= \lel \xi_t \otimes \xi_s, , a \xi_t \otimes \xi_s \rir \\
&= \lel U_{t,s}^* \xi_{s+t}, , a U_{t,s}^* \xi_{s+t} \rir \\
&= \lel \xi_{s+t} , a \xi_{s+t} \rir \\
&= \Theta_{s+t}(a).
\end{align*}
\end{proof}

We recall a central construction in Bhat and Skeide's approach to dilation of $cp$-semigroup \cite{BS00}, that goes back to Paschke \cite{Pas}.
Let $\cM$ be a $W^*$-algebra, and let $\Theta$ be a normal completely positive map on $\cM$ \footnote{The construction works also for completely positive maps on C$^*$-algebras, but in Theorem \ref{thm:GNS} below we will need to work with normal maps on W$^*$-algebras.}. \emph{The GNS representation
of $\Theta$} is a pair $(F_\Theta,\xi_\Theta)$ consisting of a Hilbert $W^*$-correspondence $F_\Theta$ and
a vector $\xi_\Theta \in F_\Theta$ such that
\bes
\Theta(a) = \lel \xi_\Theta, a \xi_\Theta \rir \,\, \textrm{ for all } a \in \cM.
\ees
$F_\Theta$ is defined to be the correspondence $\cM \otimes_\Theta \cM$ - which is the self-dual extension of the Hausdorff completion of the algebraic tensor product $\cM \otimes \cM$ with respect to inner product
\bes
\lel a \otimes b, c \otimes d \rir = b^* \Theta(a^* c)d .
\ees
$\xi_\Theta$ is defined to be $\xi_\Theta = 1 \otimes 1$. Note that
$\xi_\Theta$ is a unit vector, that is - $\lel \xi_\Theta,\xi_\Theta \rir = 1$, if and only if $\Theta$ is unital.

\begin{theorem}\label{thm:GNS}
Let $\Theta = \{\Theta_s \}_{s\in\cS}$ be a $cp$-semigroup on a $W^*$-algebra $\cM$. For every $s \in \cS$ let
$(F(s),\xi_{s})$ be the GNS representation of $\Theta_s$. Then
$F = \{F(s)\}_{s\in\cS}$ is a subproduct system of $\cM$-correspondences, and $\xi = \{\xi_s\}_{s \in \cS}$ is a generating contractive unit for $F$ that gives back $\Theta$ by the formula
\be\label{eq:unitRep}
\Theta_s(a) = \lel \xi_s, a \xi_s \rir \,\, \textrm{ for all } a \in \cM.
\ee
$\Theta$ is a cp$_0$-semigroup if and only if $\xi$ is a unital unit.
\end{theorem}
\begin{proof}
For all $s,t \in \cS$ define a map $V_{s,t}:F(s+t) \rightarrow F(s) \otimes F(t)$
by sending $\xi_{s+t}$ to $\xi_s \otimes \xi_t$ and extending to a bimodule map. Because
\begin{align*}
\lel a \xi_s \otimes \xi_t b, c \xi_s \otimes \xi_t d \rir
&= \lel  \xi_t b, \lel a \xi_s , c \xi_s \rir \xi_t d \rir \\
&= \lel  \xi_t b, \Theta_s(a^* c) \xi_t d \rir \\
&= b^* \lel  \xi_t, \Theta_s(a^* c) \xi_t \rir d \\
&= b^* \Theta_{t+s}(a^* c) d \\
&= \lel a \xi_{t+s} b, c \xi_{t+s} d \rir ,
\end{align*}
$V_{s,t}$ extends to a well defined isometric bimodule map from $F(s+t)$ into $F(s) \otimes F(t)$. We define the map $U_{s,t}$ to be the adjoint of $V_{s,t}$ (here it is important that we are working with $W^*$ algebras - in general, an isometry from one
Hilbert $C^*$-module into another need not be adjointable, but bounded module maps between \emph{self-dual} Hilbert modules are always adjointable, \cite[Proposition 3.4]{Pas}). The collection $\{U_{s,t}\}_{s,t \in \cS}$ makes $F$ into a subproduct system. Indeed, these maps are coisometric by definition, and they compose in an associative manner. To see the latter, we check that
$(I_{F(r)} \otimes V_{s,t})V_{r,s+t} = (V_{r,s} \otimes I_{F(t)})V_{r+s,t}$ and take adjoints.
\begin{align*}
(I_{F(r)} \otimes V_{s,t})V_{r,s+t} (a \xi_{r+s+t} b)
&= (I_{F(r)} \otimes V_{s,t}) (a \xi_r \otimes \xi_{s+t} b) \\
&= a \xi_r \otimes \xi_{s} \otimes \xi_t b .
\end{align*}
Similarly, $(V_{r,s} \otimes I_{F(t)})V_{r+s,t} (a \xi_{r+s+t} b) = a \xi_r \otimes \xi_{s} \otimes \xi_t b$. Since
$F(r+s+t)$ is spanned by linear combinations of elements of the form $a \xi_{r+s+t} b$, the $U$'s make $F$
into a subproduct system, and $\xi$ is certainly a unit for $F$.
Equation (\ref{eq:unitRep}) follows by definition of the GNS representation. Now,
\bes
\lel \xi_s, \xi_s \rir = \Theta_s(1) \,\, , \,\, s \in \cS,
\ees
so $\xi$ is a contractive unit because $\Theta_s(1) \leq 1$, and $\xi$ it is unital if and only if $\Theta_s$ is unital
for all $s$.
$\xi$ is in fact more then just a generating unit,
as $F(s)$ is spanned by elements with the form described in equation (\ref{eq:spanning}) with $(s_1, \ldots, s_n)$ a \emph{fixed} $n$-tuple such that
$s_1 + \cdots + s_n = s$.
\end{proof}

\begin{definition}
Given a $cp$-semigroup $\Theta$ on a $W^*$ algebra $\cM$, the subproduct system $F$ and the unit $\xi$
constructed in Theorem \ref{thm:GNS} are called, respectively, \emph{the GNS subproduct system} and \emph{the GNS unit} of $\Theta$. The pair $(F,\xi)$ is called \emph{the GNS representation} of $\Theta$.
\end{definition}

\begin{remark}
\emph{There is a precise relationship between the identity representation (Definition \ref{def:identityrep}) and the GNS representation of a $cp$-semigroup. The GNS
representation of a CP map is the \emph{dual} of the identity representation in a sense that is briefly
described in
\cite{MS05}. This notion of duality has been used
to move from the product-system-and-representation picture to the product-system-with-unit picture, and vice versa. See for example
\cite{Skeide06} and the references therein. It is more-or-less straightforward to use this duality to get Theorem \ref{thm:GNS}
from Theorem \ref{thm:reprep} (or the other way around)}.
\end{remark}

\chapter{$*$-automorphic dilation of an e$_0$-semigroup}\label{chap:aut_dil}

We now apply some of the tools developed above to dilate an e$_0$-semigroup to a semigroup of $*$-automorphisms.
We shall need
the following proposition, which is a modification (suited for \emph{sub}product systems)
of the method introduced in Chapter \ref{chap:representing_representations} for representing a product
system representation as a semigroup of
contractive operators on a Hilbert space.

\begin{proposition}\label{prop:technology1}
Let $\cN$ be a von Neumann algebra
and let $X$ be a subproduct system of unital \footnote{An $\cN$-correspondence
is said to be \emph{unital} if the left action of $\cN$ is
unital. The right action of every unital $C^*$-algebra
on every Hilbert $C^*$-correspondence is unital.}
$W^*$-correspondences over $\cS$. Let $(\sigma,T)$ be a fully coisometric covariant
representation of $X$ on the Hilbert space $H$, and assume that
$\sigma$ is unital. Denote
$$H_s := \big(X(s) \otimes_{\sigma} H \big) \Big/ {\rm Ker}\widetilde{T}_s $$
and
$$\cH = \bigoplus_{s \in \cS} H_s.$$
Then there exists a semigroup of coisometries $\hat{T} = \{\hat{T}_s\}_{s\in\cS}$ on $\cH$ such that for all
$s\in\cS$, $x \in X(s)$ and $h\in H$,
\bes
\hat{T}_s \left(\delta_s \cdot x \otimes h \right) = T_s(x)h .
\ees
$\hat{T}$ also has the property that for all $s \in \cS$ and all $t \geq s$
\be\label{eq:almostIs}
\hat{T}^*_s\hat{T}_s\big|_{H_t} = I_{H_t} \quad , \quad (t \geq s).
\ee
\end{proposition}

\begin{proof}
First, we note that the assumptions on $\sigma$ and on the left action of $\cN$ imply that $H_0 \cong H$ via the identification $a \otimes h
\leftrightarrow \sigma(a)h$. This identification will be made repeatedly
below.

Define $\cH_0$ to be the space of all finitely supported functions
$f$ on $\cS$ such that for all $s \in \cS$,
$$f(s) \in H_s.$$
We equip $\cH_0$
with the inner product
$$\langle \delta_s \cdot \xi, \delta_t \cdot \eta \rangle = \delta_{s,t} \langle \xi, \eta \rangle  ,$$
for all $s,t \in \cS, \xi \in H_s, \eta \in
H_t$. Let $\cH$ be the completion of $\cH_0$ with
respect to this inner product. We have
$$\cH \cong \bigoplus_{s \in \cS} H_s .$$
It will sometimes be convenient to identify the subspace $\delta_s \cdot H_s \subseteq \cH$ with $H_s$, and for $s = 0$
this gives us an inclusion $H \subseteq \cH$.
We
define a family $\hat{T} = \{\hat{T}_s\}_{s \in \cS}$ of operators
on $\cH_0$ as follows. First, we define
$\hat{T}_0$ to be the identity. Now assume that $s>0$. If $t\in \cS$ and $t \ngeq s$, then we define $\hat{T}_s (\delta_t \cdot \xi ) = 0$ for all
$\xi \in H_t$. If $t
\geq s > 0$ we would like to define (as we did in Chapter \ref{chap:representing_representations})
\be\label{eq:That}
\hat{T}_s \left(\delta_t \cdot (x_{t-s}
\otimes x_s \otimes h) \right) = \delta_{t-s} \cdot
\left(x_{t-s}\otimes \widetilde{T}_s (x_s \otimes h) \right),
\ee
but since $X$ is not a true product system, we cannot identify $X(t-s) \otimes X(s)$ with $X(t)$.
For a fixed $t>0$, we define for all $s\leq t$, $\xi \in X(t)$ and $h \in H$
\bes
\check{T}_s \left(\delta_t \cdot (\xi \otimes h) \right) = \delta_{t-s} \cdot \left((I_{X(t-s)} \otimes \tT_s)(U^*_{t-s,s}\xi \otimes h) \right) .
\ees
$\check{T}_s$ can be extended to a well defined contraction from $X(t) \otimes H$ to $X(t-s) \otimes H$, for all $t\geq s$, and has an adjoint given by
\be\label{eq:Tstar}
\check{T}^*_s \delta_{t-s} \cdot \eta \otimes h = \delta_{t}\cdot \left((U_{t-s,s} \otimes I_H) (\eta \otimes \widetilde{T}^*_s h)\right) .
\ee
We are going to obtain $\hat{T}_s$ as the map $H_t \rightarrow H_{t-s}$ induced by $\check{T}_s$. Let $Y  \in H_t$ satisfy $\tT_t (Y) = 0$. We shall show
that $\check{T}_s \delta_t \cdot Y = 0$ in $\delta_{t-s} \cdot H_{t-s}$. But
\bes
\check{T}_s \delta_t \cdot Y = \delta_{t-s} \cdot \left((I_{X(t-s)} \otimes \tT_s)(U^*_{t-s,s}\otimes I_H)Y \right),
\ees
and
\begin{align*}
\tT_{t-s}\left((I_{X(t-s)} \otimes \tT_s)(U^*_{t-s,s}\otimes I_H)Y \right)
(*)&= \tT_{t}(U_{t-s,s} \otimes I_H)(U^*_{t-s,s}\otimes I_H)Y \\
(**)&= \tT_t(Y) = 0,
\end{align*}
where the equation marked by (*) follows from the fact that $T$ is a representation of subproduct systems, and the
one marked by (**) follows from the fact that $U_{t-s,s}$ is a coisometry. Thus, for all $s,t \in \cS$,
$$\check{T}_s \left(\delta_t \cdot {\rm Ker}\tT_t\right) \subseteq \delta_{t-s}\cdot {\rm Ker}\tT_{t-s} ,$$
thus $\check{T}_s$ induces a well defined contraction $\hat{T}_s$ on $\cH$ given by
\be\label{eq:defThat}
\hat{T}_s \left(\delta_t \cdot (\xi \otimes h) \right) = \delta_{t-s} \cdot \left((I_{X(t-s)} \otimes \tT_s)(U^*_{t-s,s}\xi \otimes h) \right) ,
\ee
where $\xi \otimes h$ and $(I_{X(t-s)} \otimes \tT_s)(U^*_{t-s,s}\xi \otimes h)$ stand for these elements' equivalence classes in $\big(X(t) \otimes H \big) \big/{\rm Ker}\widetilde{T}_{t}$ and
$\big(X(t-s) \otimes H \big) \big/{\rm Ker}\widetilde{T}_{t-s}$, respectively.
It follows that we have the following, more precise, variant of (\ref{eq:That}):
\bes
\hat{T}_s \left(\delta_t \cdot \left(U_{t-s,s}(x_{t-s}
\otimes x_s) \otimes h\right) \right) = \delta_{t-s} \cdot
\left(x_{t-s}\otimes \widetilde{T}_s (x_s \otimes h) \right) .
\ees
In particular,
$$\hat{T}_s \left(\delta_s \cdot x_s \otimes h \right) = T_s(x_s)h ,$$
for all $s\in \cS, x_s \in X(s), h \in H $.

It will be very helpful to have a formula for $\hat{T}_s^*$ as well. Assume that $\sum_i \xi_i \otimes h_i \in {\rm Ker}\tT_t$.
\bes
\check{T}^*_s\left(\delta_{t} \cdot \sum_i \xi_i \otimes h_i \right)
= \delta_{s+t} \cdot \left((U_{t,s} \otimes I_H) (\sum_i \xi_i \otimes \tT_s^* h_i) \right) ,
\ees
and applying $\tT_{s+t}$ to the right hand side (without the $\delta$) we get
\begin{align*}
\tT_{s+t}\left((U_{t,s} \otimes I_H) (\sum_i \xi_i \otimes \tT_s^* h_i) \right)
&= \tT_{t}(I_{X(t)} \otimes \tT_{s})(\sum_i \xi_i \otimes \tT_s^* h_i) \\
&= \tT_t (\sum_i \xi_i \otimes \tT_s \tT_s^* h_i) \\
&= \tT_t (\sum_i \xi_i \otimes h_i) = 0,
\end{align*}
because $T$ is a fully coisometric representation. So
$$\check{T}^*_s \left(\delta_t \cdot {\rm Ker}\tT_t\right) \subseteq \delta_{s+t} \cdot {\rm Ker}\tT_{s+t} ,$$
and this means that $\check{T}_s^*$ induces on $\cH$ a well defined contraction which is equal to $\hat{T}^*_s$, and is given by the formula (\ref{eq:Tstar}).

We now show that $\hat{T}$ is a semigroup. Let $s,t,u \in \cS$. If
either $s = 0$ or $t = 0$ then it is clear that the semigroup
property $\hat{T}_s \hat{T}_t = \hat{T}_{s+t}$ holds. Assume that
$s,t >0$. If $u \ngeq s+t$, then both $\hat{T}_{s} \hat{T}_{t}$
and $\hat{T}_{s + t}$ annihilate $\delta_u \cdot \xi$, for all
$\xi \in H_u$. Assuming $u \geq s+t$, we shall show that $\hat{T}_s \hat{T}_t$ and
$\hat{T}_{s+t}$ agree on elements of the form
$$Z = \delta_u \cdot \big(U_{u-t,t}(U_{u-t-s,s} \otimes I)(x_{u-s-t} \otimes x_s \otimes x_t)\big) \otimes h,$$
and since the set of all such elements is total in $H_u$, this will establish the semigroup property.
\begin{align*}
\hat{T}_{s} \hat{T}_{t} Z
&=  \hat{T}_{s} \left(\delta_{u-t}
\left( U_{u-t-s,s}(x_{u-s-t} \otimes x_s) \otimes \widetilde{T}_t(x_t \otimes h)\right) \right)
\\
&= \delta_{u-s-t} \left(x_{u-s-t} \otimes \widetilde{T}_s(x_s \otimes \widetilde{T}_t( x_t \otimes h)) \right) \\
&= \delta_{u-s-t} \left(x_{u-s-t} \otimes \widetilde{T}_s(I
\otimes \widetilde{T}_t )(x_s \otimes x_t \otimes h) \right) \\
&=  \delta_{u-s-t} \left(x_{u-s-t} \otimes \widetilde{T}_{s+t} \left(U_{s,t} (x_s
\otimes
x_t) \otimes h\right) \right) \\
&=  \hat{T}_{t+s} \delta_u \cdot \left(U_{u-t-s,t+s}\left(x_{u-s-t} \otimes U_{s,t} (x_s
\otimes
x_t)\right) \otimes h  \right) \\
&= \hat{T}_{t+s} Z.
\end{align*}
The final equality follows from the associativity condition (\ref{eq:assoc_prod}).

To see that $\hat{T}$ is a semigroup of coisometries, we take $\xi \in X(t), h \in H$, and compute
\begin{align*}
\tT_t \left(\hat{T}_s \hat{T}^*_s \delta_t \cdot (\xi \otimes h) \right)
&= \tT_t \left((I_{X(t)} \otimes \tT_s)(U_{t,s}^* \otimes I_H) (U_{t,s} \otimes I_H)(I_{X(t)} \otimes \tT_s^*) (\xi \otimes h) \right) \\
&= \tT_{s+t} (U_{t,s} \otimes I_H)(I_{X(t)} \otimes \tT_s^*) (\xi \otimes h) \\
&= \tT_t (\xi \otimes \tT_s \tT_s^* h) = \tT_t (\xi \otimes h) ,
\end{align*}
so $\hat{T}_s \hat{T}^*_s$ is the identity on $H_t$ for all $t \in \cS$, thus $\hat{T}_s \hat{T}^*_s = I_{\cH}$. Equation (\ref{eq:almostIs}) follows by a similar computation, which is a omitted.
\end{proof}

We can now obtain a $*$-automorphic dilation for any e$_0$-semigroup over any subsemigroup of $\Rpk$. 
The following result  should be compared with similar-looking results of Arveson-Kishimoto \cite{ArvKish}, Laca \cite{Laca}, Skeide \cite{Skeide07},
and Arveson-Courtney \cite{ArvCourt} (none of these cited results is strictly stronger or weaker than the result we obtain for the case of $e_0$-semigroups).

\begin{theorem}\label{thm:aut_dil}
Let $\Theta$ be a $e_0$-semigroup acting on a von Neumann algebra $\cM$. Then $\Theta$ can be dilated to a semigroup
of $*$-automorphisms in the
following sense: there is a Hilbert space $\cK$, an orthogonal projection $p$ of $\cK$ onto a subspace $\cH$ of $\cK$, a normal, faithful representation $\varphi:\cM \rightarrow B(\cK)$ such that $\varphi(1) = p$, and a semigroup $\alpha = \{\alpha_s\}_{s\in\cS}$
of $*$-automorphisms on $B(\cK)$ such that for all $a \in \cM$ and all $s \in \cS$
\be\label{eq:aut_ext}
\alpha_s(\varphi(a)) \big|_\cH = \varphi(\Theta_s(a)) ,
\ee
so, in particular,
\be\label{eq:aut_dil}
p \alpha_s(\varphi(a)) p = \varphi(\Theta_s(a)) .
\ee
The projection $p$ is increasing for $\alpha$, in the sense that for all $s \in \cS$,
\be\label{eq:increasing}
\alpha_s(p) \geq p .
\ee
\end{theorem}
\begin{remark}
\emph{
Another way of phrasing the above theorem is by using the terminology of ``weak Markov flows", as used in \cite{BS00}. Denoting $\varphi$ by $j_0$, and defining $j_s:= \alpha_s \circ
j_0$, we have that $(B(\cK),j)$ is a weak Markov flow for $\Theta$ on $\cK$, which just means that for all $t \leq s \in \cS$ and all $a \in \cM$,
\be\label{eq:weak}
j_t(1)j_s(a)j_t(1) = j_t(\Theta_{s-t}(a)) .
\ee
Equation (\ref{eq:weak}) for $t=0$ is just (\ref{eq:aut_dil}), and the case
$t \geq 0$ follows from the case $t=0$.}
\end{remark}
\begin{remark}
\emph{
The assumption that $\Theta$ is a \emph{unital} semigroup is essential, since (\ref{eq:aut_dil}) and (\ref{eq:increasing})
imply that $\Theta(1) = 1$.}
\end{remark}
\begin{remark}
\emph{
It is impossible, in the generality we are working in, to hope for a semigroup of automorphisms that extends $\Theta$ in the sense that
\be\label{eq:aut_extension}
\alpha_s(\varphi(a)) = \varphi(\Theta_s(a)) ,
\ee
because that would imply that $\Theta$ is injective.}
\end{remark}
\begin{proof}
Let $(E,T)$ be the identity representation of $\Theta$. Since $\Theta$ preserves the unit, $T$ is a fully coisometric
representation. Let $\hat{T}$ and $\cH$ be the semigroup and Hilbert
space representing $T$ as described in Proposition \ref{prop:technology1}. $\{\hat{T}^*_s\}_{s\in\cS}$
is a commutative semigroup of isometries. By a theorem of Douglas \cite{Douglas}, $\{\hat{T}^*_s\}_{s\in\cS}$
can be extended to a semigroup $\{\hat{V}^*_s\}_{s\in\cS}$ of unitaries acting on a space $\cK \supseteq \cH$. We obtain a semigroup of unitaries
$V = \{\hat{V}_s\}_{s\in\cS}$ that is a dilation of $\hat{T}$, that is
\bes
P_\cH \hat{V}_s \big|_\cH = \hat{T}_s  \,\, , \,\, s \in \cS.
\ees

For any $b\in B(\cK)$, and any $s\in\cS$, we define
\bes
\alpha_s(b) = \hat{V}_s b \hat{V}_s^*.
\ees
Clearly, $\alpha = \{\alpha_s\}_{s\in\cS}$ is a semigroup of $*$-automorphisms.

Put $p = P_{\cH}$, the orthogonal projection of $\cK$ onto $\cH$.
Define $\varphi : \cM \rightarrow B(\cK)$ by $\varphi(a) = p (I \otimes a) p$, where $I \otimes a : \cH \rightarrow \cH$ is given by
\bes
(I \otimes a) \delta_t\cdot x \otimes h = \delta_t \cdot x \otimes a h  \quad , \quad x \otimes h \in E(t) \otimes H.
\ees

$\varphi$ is well defined because $T$ is an isometric representation (so ${\rm Ker}\widetilde{T}_{t}$ is always zero). We have that $\varphi$ is a faithful, normal $*$-representation (the fact that $T_0$ is the identity representation ensures that $\varphi$ is faithful). It is clear that $\varphi(1) = p$.

To see (\ref{eq:increasing}), we note that since $\hat{V}_s^*$ is an extension of $\hat{T}_s^*$, we have
$\hat{T}_s^* = \hat{V}_s^* p = p\hat{V}_s^* p$, thus
\begin{align*}
p \alpha_s(p) p &= p \hat{V}_s p \hat{V}_s^* p \\
&= p \hat{V}_s \hat{V}_s^* p \\
&= p,
\end{align*}
that is, $p \alpha_s(p) p = p$, which implies that $\alpha_s(p) \geq p$.

We now prove (\ref{eq:aut_dil}).
Let $\delta_t \cdot x \otimes h$ be a typical element of $\cH$. We compute
\begin{align*}
p \alpha_s (\varphi(a)) p \delta_t \cdot x \otimes h &= p \hat{V}_s p (I \otimes a) p \hat{V}_s^* p \delta_t \cdot x \otimes h \\
&= \hat{T}_s (I \otimes a) \hat{T}_s^* \delta_t \cdot x \otimes h \\
&= \hat{T}_s (I\otimes a) \delta_{s+t}\cdot (U_{t,s} \otimes I_H) \Big( x \otimes \widetilde{T}_s^*h \Big) \\
&= \hat{T}_s \delta_{s+t}\cdot (U_{t,s} \otimes I_H) \Big( x \otimes (I\otimes a) \widetilde{T}_s^*h \Big) \\
&= \delta_{t} \cdot x  \otimes \left(\widetilde{T}_s(I\otimes a) \widetilde{T}_s^*h\right) \\
&= \delta_{t} \cdot x \otimes (\Theta_s(a) h) \\
&= \varphi(\Theta_s(a)) \delta_t \cdot x \otimes h .
\end{align*}
Since both $p\alpha_s(\varphi(a)) p$ and $\varphi(\Theta_s(a))$ annihilate $\cK \ominus \cH$, we have
(\ref{eq:aut_dil}).

To prove (\ref{eq:aut_ext}), it just remains to show that
\bes
p \alpha_s (\varphi(a)) \big|_{\cH} =  \alpha_s (\varphi(a)) \big|_{\cH},
\ees
that is, that $\alpha_s (\varphi(a)) \cH \subseteq \cH$.
Now, $\hat{V}_s^*$ is an extension of $\hat{T}_s^*$. Moreover (\ref{eq:almostIs}) shows that
if $\xi \in H_u$ with $u \geq s$, then $\| \hat{T}_s(\xi) \| = \|\xi\|$. Thus
\bes
\|\xi\|^2 = \|\hat{V}_s \xi \|^2 = \|P_\cH \hat{V}_s \xi \|^2 + \|(I_\cK - P_\cH) \hat{V}_s \xi \|^2 = \|\hat{T}_s \xi \|^2 + \|(I_\cK - P_\cH) \hat{V}_s \xi \|^2.
\ees
So $\hat{V}_s \xi = \hat{T}_s \xi$ for $\xi \in H_u$ with $u \geq s$. Now, for a typical element $\delta_t \cdot x \otimes h$ in $H_t$, $t\in \cS$, we have
\begin{align*}
\alpha_s (\varphi(a)) \delta_t \cdot x \otimes h &=
\hat{V}_s (I \otimes a) \hat{V}_s^* \delta_t \cdot x \otimes h \\
&= \hat{V}_s (I \otimes a) \hat{T}_s^* \delta_t \cdot x \otimes h \\
&= \hat{V}_s \delta_{s+t} \cdot (U_{s,t} \otimes I_H) \Big( x \otimes (I \otimes a) \tT_s^* h \Big) \\
&= \hat{T}_s \delta_{s+t} \cdot (U_{s,t} \otimes I_H) \Big( x \otimes (I \otimes a) \tT_s^* h \Big) \in \cH,
\end{align*}
because $\delta_{s+t} \cdot x \otimes (I \otimes a) \tT_s^* h \in H_{s+t}$, and $s+t \geq s$.
\end{proof}

\chapter{Dilations and pieces of subproduct system representations}\label{chap:dil}

\section{Dilations and pieces of subproduct system representations}

\begin{definition}
Let $X$ and $Y$ be subproduct systems of $\cM$ correspondences ($\cM$ a W$^*$-algebra) over the same semigroup $\cS$. Denote by $U_{s,t}^X$ and $U_{s,t}^Y$ the coisometric maps that make $X$ and $Y$, respectively, into subproduct systems.
$X$ is said to be a \emph{subproduct subsystem of $Y$} (or simply a \emph{subsystem of $Y$} for short) if for all $s \in \cS$ the space $X(s)$ is a closed subspace of $Y(s)$, and if the orthogonal projections $p_s : Y(s) \rightarrow X(s)$ are bimodule maps that satisfy
\be\label{eq:pU}
p_{s+t} \circ U_{s,t}^Y = U_{s,t}^X \circ (p_s \otimes p_t) \,\, , \,\, s,t \in \cS.
\ee
\end{definition}

One checks that if $X$ is a subproduct subsystem of $Y$ then
\be\label{eq:p}
p_{s+t+u} \circ U^Y_{s,t+u}(I \otimes (p_{t+u}\circ U^Y_{t,u})) = p_{s+t+u} \circ U^Y_{s+t,u}((p_{s+t}\circ U^Y_{s,t})\otimes I) ,
\ee
for all $s,t,u \in \cS$. Conversely, given a subproduct system $Y$ and a family of orthogonal projections $\{p_s\}_{s \in \cS}$ that are bimodule maps satisfying (\ref{eq:p}), then by defining $X(s) = p_s Y(s)$ and
$U_{s,t}^X = p_{s+t} \circ U_{s,t}^Y$ one obtains a subproduct subsystem $X$ of $Y$ (with (\ref{eq:pU}) satisfied).

The following proposition is a consequence of the definitions.
\begin{proposition}
There exists a morphism $X \rightarrow Y$ if and only if $Y$ is isomorphic to a subproduct subsystem of $X$.
\end{proposition}
\begin{remark}\label{rem:subsystem_iso}
\emph{In the notation of Theorem \ref{thm:essentially_inverse}, we may now say that
given a subproduct system $X$ and a representation $R$ of $X$, then the Arveson-Stinespring subproduct system $E$ of $\Theta = \Sigma(X,R)$ is isomorphic to a subproduct subsystem of $X$.}
\end{remark}

The following definitions are inspired by the work of Bhat, Bhattacharyya and Dey \cite{BBD03}.

\begin{definition}\label{def:repdilation}
Let $X$ and $Y$ be subproduct systems of W$^*$-correspondences (over the same W$^*$-algebra $\cM$) over $\cS$, and let $T$ be a representation of $Y$ on a Hilbert space $K$. Let $H$ be some fixed Hilbert space, and let $S=\{S_s\}_{s\in\cS}$ be a family of maps $S_s : X(s) \rightarrow B(H)$. $(Y,T,K)$ is called a \emph{dilation} of $(X,S,H)$ if
\begin{enumerate}
\item $X$ is a subsystem of $Y$,
\item $H$ is a subspace of $K$, and
\item for all $s \in \cS$,  $\widetilde{T}^*_s H \subseteq X(s) \otimes H$ and $\widetilde{T}^*_s \big|_H = \widetilde{S}^*_s$.
\end{enumerate}
In this case we say that $S$ is \emph{an $X$-piece of $T$}, or simply a \emph{piece of $T$}. $T$ is said to be an 
\emph{isometric} dilation of $S$ if $T$ is an isometric representation.
\end{definition}
The third item can be replaced by the three conditions
\begin{itemize}
\item[1'] $T_0(\cdot) P_H = P_H T_0(\cdot) P_H = S_0 (\cdot)$,
\item[2'] $P_H \tT_s \big|_{X(s) \otimes H} = \widetilde{S}_s$ for all $s\in \cS$, and
\item[3'] $P_H \tT_s \big|_{Y(s) \otimes K \ominus X(s) \otimes H} = 0$.
\end{itemize}
So our definition of dilation is identical to Muhly and Solel's definition of dilation of representations when $X = Y$ is a product system \cite[Theorem and Definition 3.7]{MS02}.

\begin{proposition}\label{prop:pieceisrep}
Let $T$ be a representation of $Y$, let $X$ be a subproduct subsystem of $Y$, and let $S$ an $X$-piece of $T$. Then $S$ is a representation of $X$.
\end{proposition}
\begin{proof}
$S$ is a completely contractive linear map as the compression of a completely contractive linear map. Item 1' above together with the coinvariance of $T$ imply that $S$ is coinvariant:
if $a,b \in \cM$ and $x \in X(s)$, then
\begin{align*}
S_s(axb) = P_H T_s(axb) P_H
&= P_H T_0(a) T_s(x) T_0(b) P_H \\
&= P_H T_0(a) P_H T_s(x) P_H T_0(b) P_H \\
&= S_0(a) S_s(x) S_0(b).
\end{align*}
Finally, (using Item 3' above),
\begin{align*}
S_{s+t} (U^X_{s,t}(x \otimes y)) h &= S_{s+t} (p_{s+t} U^Y_{s,t}(x \otimes y)) h \\
&= \widetilde{S}_{s+t} (p_{s+t} U^Y_{s,t}(x \otimes y) \otimes h) \\
&= P_H \widetilde{T}_{s+t} (U^Y_{s,t}(x \otimes y) \otimes h) \\
&= P_H T_s (x) T_t(y) h \\
&= P_H T_s (x) P_H T_t(y) h \\
&= S_s (x) S_t(y) h .
\end{align*}
\end{proof}

\begin{example}\emph{
Let $E$ be a Hilbert space of dimension $d$, and let $X$ be the symmetric subproduct system constructed in Example \ref{expl:symm}. Fix an orthonormal basis $\{e_1, \ldots, e_n\}$ of $E$. There is a one-to-one correspondence between c.c. representations $S$ of $X$ (on some $H$) and commuting row contractions $(S_1, \ldots, S_d)$ (of operators on $H$), given by
\bes
S \leftrightarrow \underline{S} = (S(e_1), \ldots, S(e_d)).
\ees
If $Y$ is the full product system over $E$, then any dilation $(Y,T,K)$ gives rise to a tuple $\underline{T} = (T(e_1), \ldots, T(e_d))$ that is a dilation of $\underline{S}$ in the sense of \cite{BBD03}, and \emph{vice versa}.
Moreover, $\underline{S}$ is then a commuting piece of $\underline{T}$ in the sense of \cite{BBD03}.}
\end{example}

Consider a subproduct system $Y$ and a representation $T$ of $Y$ on $K$. Let $X$ be some subproduct subsystem of $Y$. Define the following set of subspaces of $K$:
\be\label{eq:cPXT}
\cP(X,T) = \{H \subseteq K: \widetilde{T}^*_s H \subseteq X(s) \otimes H\ \textrm{ for all }s \in \cS\}.
\ee
As in \cite{BBD03}, we observe that $\cP(X,T)$ is closed under closed linear spans (and intersections), thus we may define
\bes
K^X(T) = \bigvee_{H \in \cP(X,T)} H.
\ees
$K^X(T)$ is the maximal element of $\cP(X,T)$.
\begin{definition}\label{def:Xpiece}
The representation $T^X$ of $X$ on $K^X(T)$ given by
\bes
T^X(x) h = P_{K^X(T)} T(x) h,
\ees
for $x \in X(s)$ and $h \in K^X(T)$, is called
\emph{the maximal $X$-piece of $T$}.
\end{definition}
By Proposition \ref{prop:pieceisrep}, $T^X$ is indeed a representation of $X$.

\section{Consequences in dilation theory of $cp$-semigroups}

\begin{proposition}\label{prop:dil_rep_dil_CP}
Let $X$ and $Y$ be subproduct systems of W$^*$-correspondences (over the same W$^*$-algebra $\cM$) over $\cS$, and let $S$ and $T$ be representations of $X$ on $H$ and
of $Y$ on $K$, respectively.
Assume that $(Y,T,K)$ is a dilation of $(X,S,H)$. Then the $cp$-semigroup $\Theta$ acting on $T_0(\cM)'$, given by
\bes
\Theta_s(a) = \widetilde{T}_s(I_{Y(s)} \otimes a)\widetilde{T}_s^* \,\, ,\,\, a \in T_0(\cM)',
\ees
is a dilation of the $cp$-semigroup $\Phi$ acting on $S_0(\cM)'$ given by
\bes
\Phi_s(a) = \widetilde{S}_s(I_{X(s)} \otimes a)\widetilde{S}_s^* \,\, ,\,\, a \in S_0(\cM)',
\ees
in the sense that for all $b \in T_0(\cM)'$ and all $s \in \cS$,
\bes
\Phi_s(P_H b P_H) = P_H \Theta_s (b) P_H .
\ees
\end{proposition}
\begin{proof}
This follows from the definitions.
\end{proof}

Although the above proposition follows immediately from the definitions,
we hope that it will prove to be important in the theory of dilations of $cp$-semigroups, because it points to
a conceptually new way of constructing dilations of $cp$-semigroups, as the following proposition and corollary illustrate.

\begin{proposition}\label{prop:exist_isodil}
Let $X=\{X(s)\}_{s\in\cS}$ be a subproduct system, and let $S$ be a fully coisometric representation of $X$ on $H$
such that $S_0$ is unital. If there exists a (full) \emph{product system} $Y = \{Y(s)\}_{s\in\cS}$ such that $X$ is a
subproduct subsystem of $Y$, then $S$ has an isometric and fully coisometric dilation.
\end{proposition}
\begin{proof}
Define a representation $T$ of $Y$ on $H$ by
\be\label{eq:extend}
T_s = S_s \circ p_s,
\ee
where, as above, $p_s$ is the orthogonal projection $Y(s) \rightarrow X(s)$.
A straightforward verification shows that $T$ is indeed a fully coisometric representation of $Y$ on $H$.
By Theorem \ref{thm:isoDilFC}, $(Y,T,H)$ has a minimal isometric and fully coisometric dilation $(Y,V,K)$. $(Y,V,K)$
is also clearly a dilation of $(X,S,H)$.
\end{proof}
\begin{corollary}\label{cor:e0dilwhensub}
Let $\Theta = \{\Theta_s\}_{s\in\cS}$ be a cp$_0$-semigroup and let $(E,T) = \Xi(\Theta)$ be the Arveson-Stinespring
representation of $\Theta$. If there is a (full) \emph{product system} $Y$ such that $E$ is a subproduct subsystem of $Y$, then $\Theta$ has an $e_0$-dilation.
\end{corollary}
\begin{proof}
Combine Propositions \ref{prop:semigroup}, \ref{prop:dil_rep_dil_CP} and \ref{prop:exist_isodil}.
\end{proof}

Thus, the problem of constructing $e_0$-dilations to $cp_0$-semigroups is reduced to the problem of embedding a
subproduct system into a full product system. In the next section we give an example of a subproduct system that cannot be embedded into full product system. When this can be done in general is a challenging open question.

\begin{corollary}\label{cor:edilwhensub}
Let $\Theta = \{\Theta_s\}_{s\in\mb{N}^k}$ be a $cp$-semigroup generated by $k$ commuting CP maps
$\theta_1, \ldots, \theta_k$, and let $(E,T) = \Xi(\Theta)$ be the Arveson-Stinespring
representation of $\Theta$. Assume, in addition, that
\bes
\sum_{i=1}^k \|\theta_i\|\leq 1.
\ees
If there is a (full) \emph{product system} $Y$ such that $E$ is a subproduct subsystem of $Y$, then
$\Theta$ has an $e$-dilation.
\end{corollary}
\begin{proof}
As in (\ref{eq:extend}), we may extend $T$ to a product system representation of $Y$ on $H$, which we also denote by $T$.
Denote by ${\bf e_i}$ the element of $\mb{N}^k$ with $1$ in the $i$th element and zeros elsewhere.
Then
\bes
\sum_{i=1}^k \|\widetilde{T}_{\bf e_i} \widetilde{T}_{\bf e_i}^*\| = \sum_{i=1}^k \|\theta_i\|\leq 1.
\ees
By Corollary \ref{cor:normcondregdil} $S$ has a minimal (regular) isometric dilation. This isometric dilation
provides the required $e$-dilation of $\Theta$.
\end{proof}

\begin{theorem}\label{thm:edil_repdil}
Let $\cM \subseteq B(H)$ be a von Neumann algebra, let $X$ be a subproduct system of $\cM'$-correspondences,
and let $R$ be an injective representation of $X$ on a Hilbert space $H$.
Let $\Theta = \Sigma(X,R)$ be the $cp$-semigroup acting on $R_0(\cM')'$ given by (\ref{eq:reprep}). Assume that
$(\alpha,K,\cR)$ is an $e$-dilation of $\Theta$, and let $(Y,V) = \Xi(\alpha)$ be the Arveson-Stinespring subproduct
system of $\alpha$ together with the identity representation. Assume, in addition, that the map
$\cR' \ni b \mapsto P_H b P_H$ is a
$*$-isomorphism of $\cR'$ onto $R_0(\cM')$. Then $(Y,V,K)$ is a dilation of $(X,R,H)$.
\end{theorem}
\begin{proof}
For every $s \in \cS$, define $W_s : Y(s) \rightarrow B(H)$ by $W_s (y) = P_H V_s(y)P_H$.
We claim that $W = \{W_s\}_{s\in\cS}$ is a representation of $Y$ on $H$. First, note that
$P_H \alpha_s (I-P_H)P_H = \Theta_s(P_H(I-P_H)P_H) = 0$, thus
$P_H \widetilde{V}_s (I \otimes (I-P_H))\widetilde{V}^*_s P_H =0$, and consequently
$P_H \widetilde{V}_s (I \otimes P_H) = P_H \widetilde{V}_s$.
It follows that $W_s (y) = P_H V_s(y) P_H = P_H V_s(y) $. From this it follows that
\begin{align*}
W_s(y_1) W_t(y_2) &= P_HV_s(y_1)P_HV_t(y_2) = P_HV_s(y_1) V_t(y_2) \\
&= P_H V_{s+t}(U^Y_{s,t}(y_1 \otimes y_2)) = W_{s+t}(U^Y_{s,t}(y_1 \otimes y_2)).
\end{align*}

By Theorem \ref{thm:essentially_inverse}, we may assume that $(X,R) = (E,T) = \Xi(\Theta)$ is the
Arveson-Stinespring representation of $\Theta$.
Because $\alpha$ is a dilation of $\Theta$, we have
\bes
\widetilde{W}_s (I \otimes a)\widetilde{W}_s^* = P_H \widetilde{V}_s (I \otimes a)\widetilde{V}^*_s P_H = \Theta_s(a),
\ees
That is, $\Theta = \Sigma(Y,W)$.
Thus, by Theorem \ref{thm:essentially_inverse} and Remark \ref{rem:subsystem_iso}, we may assume that
$E$ is a subproduct subsystem of $Y$, and that $T_s \circ p_s = W_s$, $p_s$ being the projection
of $Y(s)$ onto $E(s)$. In other words, for all $y \in Y$,
\bes
\widetilde{T}_s (p_s \otimes I_H) = P_H \widetilde{W}_s .
\ees
Therefore, $\widetilde{W}_s^* H \subseteq E(s) \otimes H$, and $\widetilde{W}_s^*\big|_H = \widetilde{T}_s^*$.
That is, $(Y,W,H)$ is a dilation of $(E,T,H)$. But $(Y,V,K)$ is a dilation of $(Y,W,H)$,
so it is also a dilation of $(E,T,H)$.
\end{proof}

\begin{remark}
\emph{The assumption that $\cR' \ni b \mapsto P_H b P_H \in \cM'$ is a $*$-isomorphism is satisfied when $\cM = B(H)$ and
$\cR = B(K)$. More generally, it is satisfied whenever the central support of $P_H$ in $\cR$ is $I_K$
(see Propositions 5.5.5 and 5.5.6 in \cite{KRI}). When $\cM$ is a full corner in $\cR$, that is, when $\cR = \cR P_H \cR$,
then the central support of $P_H$ in $\cR$ is $I_K$.}
\end{remark}

Let $(\alpha,K,\cR)$ be an $e$-dilation of a semigroup $\Theta$ on $\cM \subseteq B(H)$. $(\alpha,K,\cR)$ is called a \emph{minimal dilation} if the central support of $P_H$ in $\cR$ is $I_K$ and if
\bes
\cR = W^*\left( \bigcup_{s\in\cS} \alpha_s(\cM) \right).
\ees
\begin{corollary}\label{cor:edil_repdil}
Let $\Theta$ be $cp$-semigroup on $\cM \subseteq B(H)$, and let $(\alpha,K,\cR)$ be a minimal dilation of $\Theta$. Then $\Xi(\alpha)$ is an isometric dilation of $\Xi(\Theta)$.
\end{corollary}

\section{$cp$-semigroups with no $e$-dilations. Obstructions of a new nature}

By Parrot's famous example \cite{Parrot}, there exist $3$ commuting contractions that do not have a commuting isometric dilation. In 1998 Bhat asked whether $3$ commuting CP maps necessarily have a commuting $*$-endomorphic dilation \cite{Bhat98}. Note that it is not obvious that the non-existence of an isometric dilation for three commuting contractions would imply the non-existence of a $*$-endomorphic dilation for $3$ commuting CP maps. However, it turns out that this is the case.

\begin{theorem}\label{thm:parrot}
There exists a $cp$-semigroup $\Theta = \{\Theta_n\}_{n \in \mb{N}^3}$ acting on a $B(H)$ for which there is no $e$-dilation $(\alpha,K,B(K))$. In fact, $\Theta$ has no \emph{minimal} $e$-dilation $(\alpha,K,\cR)$ on any von Neumann algebra $\cR$.
\end{theorem}
\begin{proof}
Let $T_1,T_2,T_3 \in B(H)$ be three commuting contractions that have no isometric dilation and such that $T_1^{n_1} T_2^{n_2} T_3^{n_3} \neq 0$ for all $n  = (n_1, n_2, n_3) \in \mb{N}^3$ (one may take commuting contractions $R_1,R_2,R_3$ with no isometric dilation as in Parrot's example \cite{Parrot}, and define $T_i = R_i \oplus 1$).
For all $n  = (n_1, n_2, n_3) \in \mb{N}^3$, define
\bes
\Theta_n (a) = T_1^{n_1} T_2^{n_2} T_3^{n_3} a  (T_3^{n_3})^*(T_2^{n_2})^* (T_1^{n_1})^* \,\, , \,\, a \in B(H).
\ees
Note that $\Theta = \Sigma(X,R)$, where $X = \{X(n)\}_{n\in\mb{N}^3}$ is the subproduct system given by $X(n) = \mb{C}$ for all $n \in \mb{N}^3$, and $R$ is the (injective) representation that sends $1 \in X(n)$ to $T_1^{n_1} T_2^{n_2} T_3^{n_3}$ (the product in $X$ is simply multiplication of scalars).

Assume, for the sake of obtaining a contradiction, that $\Theta = \{\Theta_n\}_{n\in\mb{N}^3}$ has
an $e$-dilation $(\alpha,K,B(K))$. Let $(Y,V) = \Xi(\alpha)$ be the Arveson-Stinespring subproduct system of $\alpha$ together with the identity representation. By Theorem \ref{thm:edil_repdil}, $(Y,V,K)$ is a dilation of $(X,R,H)$. It follows that $V_1,V_2,V_3$ are a commuting isometric dilation of $T_1,T_2,T_3$ where $V_1 := V(1)$ with $1 \in X(1,0,0)$, $V_2 := V(1)$ with $1 \in X(0,1,0)$, and
$V_3 := V(1)$ with $1 \in X(0,0,1)$. This is a contradiction.

Finally, a standard argument shows that if $(\alpha,K,\cR)$ is a minimal dilation of $\Theta$, then $\cR = B(K)$.
\end{proof}

Until this point, all the results that we have seen in the dilation theory of $cp$-semigroups have been anticipated by the classical theory of isometric dilations. We shall now encounter a phenomena that has no counterpart in the classical theory.

By \cite[Proposition 9.2]{SzNF70}, if $T_1, \ldots, T_k$ is a commuting $k$-tuple of contractions such that
\be\label{eq:normconditionT}
\sum_{i=1}^k \|T_i\|^2 \leq 1 ,
\ee
then $T_1,\ldots, T_k$ has a commuting regular unitary dilation (and, in particular, an isometric dilation). One is tempted to conjecture that if $\theta_1, \ldots, \theta_k$ is a commuting $k$-tuple of CP maps such that
\be\label{eq:normconditiontheta}
\sum_{i=1}^k \|\theta_i\| \leq 1,
\ee
then the tuple $\theta_1, \ldots, \theta_k$ has an $e$-dilation. Indeed, if $\theta_i(a) = T_i a T_i^*$, where $T_1, \ldots, T_k$ is a commuting $k$-tuple satisfying (\ref{eq:normconditionT}), then
it is easy to construct an $e$-dilation of $\theta_1, \ldots, \theta_k$ from the isometric dilation of $T_1, \ldots, T_k$. However, it turns out that (\ref{eq:normconditiontheta}) is far from being sufficient for an $e$-dilation to exist. We need some preparations before exhibiting an example.

\begin{proposition}\label{prop:sprdctcntrexample}
There exists a subproduct system that is not a subsystem of any product system.
\end{proposition}
\begin{proof}
We construct a counter example over $\mb{N}^3$. Let $e_1,e_2,e_3$ be the standard basis of $\mb{N}^3$. We let $X(e_1) = X(e_2) = X(e_3) = \mb{C}^2$. Let $X(e_i + e_j) = \mb{C}^2 \otimes \mb{C}^2$ for all $i,j=1,2,3$. Put $X(n) = \{0\}$ for all $n\in\mb{N}^k$ such that $|n|>2$. To complete the construction of $X$ we need to define the product maps $U^X_{m,n}$.
Let $U^X_{e_i,e_j}$ be the identity on $\mb{C}^2 \otimes \mb{C}^2$ for all $i,j$ except for $i=3,j=2$, and let $U^X_{e_3,e_2}$ be the flip. Define the rest of the products to be zero maps (except the maps $U^X_{0,n}, U^X_{m,0}$ which are identities). This product is evidently coisometric, and it is also associative, because the product of any three nontrivial elements vanishes.

Let $Y$ be a product system ``dilating" $X$. Then for all $k,m,n \in \mb{N}^k$ we have
\bes
U^Y_{k+m,n}(U^Y_{k,m}\otimes I) =  U^Y_{k,m+n}(I \otimes U^Y_{m,n}),
\ees
or
\bes
U^Y_{k+m,n} =  U^Y_{k,m+n}(I \otimes U^Y_{m,n})(U^Y_{k,m}\otimes I)^*,
\ees
and
\bes
U^Y_{k,m+n} = U^Y_{k+m,n}(U^Y_{k,m}\otimes I)(I \otimes U^Y_{m,n})^*.
\ees
Iterating these identities, we have, on the one hand,
\begin{align*}
U_{e_3,e_1+e_2} &= U^Y_{e_3+e_2,e_1}(U^Y_{e_3,e_2}\otimes I)(I \otimes U^Y_{e_2,e_1})^* \\
&= U^Y_{e_2,e_3+e_1}(I \otimes U^Y_{e_3,e_1})(U^Y_{e_2,e_3}\otimes I)^*(U^Y_{e_3,e_2}\otimes I)(I \otimes U^Y_{e_2,e_1})^* \\
&= U^Y_{e_1+e_2,e_3}(U^Y_{e_2,e_1}\otimes I)(I \otimes U^Y_{e_1,e_3})^*(I \otimes U^Y_{e_3,e_1})(U^Y_{e_2,e_3}\otimes I)^*(U^Y_{e_3,e_2}\otimes I)(I \otimes U^Y_{e_2,e_1})^*, \\
\end{align*}
and on the other hand
\begin{align*}
U_{e_3,e_1+e_2} &= U^Y_{e_3+e_1,e_2}(U^Y_{e_3,e_1}\otimes I)(I \otimes U^Y_{e_1,e_2})^* \\
&= U^Y_{e_1,e_3+e_2}(I \otimes U^Y_{e_3,e_2})(U^Y_{e_1,e_3}\otimes I)^*(U^Y_{e_3,e_1}\otimes I)(I \otimes U^Y_{e_1,e_2})^* \\
&= U^Y_{e_1+e_2,e_3}(U^Y_{e_1,e_2}\otimes I)(I \otimes U^Y_{e_2,e_3})^*(I \otimes U^Y_{e_3,e_2})(U^Y_{e_1,e_3}\otimes I)^*(U^Y_{e_3,e_1}\otimes I)(I \otimes U^Y_{e_1,e_2})^*. \\
\end{align*}
Canceling $U^Y_{e_1+e_2,e_3}$, we must have
\begin{align*}
(U^Y_{e_1,e_2}\otimes I)(I \otimes U^Y_{e_2,e_3})^* & (I \otimes U^Y_{e_3,e_2})(U^Y_{e_1,e_3}\otimes I)^*(U^Y_{e_3,e_1}\otimes I)(I \otimes U^Y_{e_1,e_2})^* \\
&= (U^Y_{e_2,e_1}\otimes I)(I \otimes U^Y_{e_1,e_3})^*(I \otimes U^Y_{e_3,e_1})(U^Y_{e_2,e_3}\otimes I)^*(U^Y_{e_3,e_2}\otimes I)(I \otimes U^Y_{e_2,e_1})^*.
\end{align*}
Now, $U^X_{e_i,e_j}$ were unitary to begin with, so the above identity implies
\begin{align*}
(U^X_{e_1,e_2}\otimes I)(I \otimes U^X_{e_2,e_3})^* & (I \otimes U^X_{e_3,e_2})(U^X_{e_1,e_3}\otimes I)^*(U^X_{e_3,e_1}\otimes I)(I \otimes U^X_{e_1,e_2})^* \\
&= (U^X_{e_2,e_1}\otimes I)(I \otimes U^X_{e_1,e_3})^*(I \otimes U^X_{e_3,e_1})(U^X_{e_2,e_3}\otimes I)^*(U^X_{e_3,e_2}\otimes I)(I \otimes U^X_{e_2,e_1})^*.
\end{align*}
Recalling the definition of the product in $X$ (the product is usually the identity), this reduces to
\begin{align*}
I \otimes U^X_{e_3,e_2} = U^X_{e_3,e_2}\otimes I.
\end{align*}
This is absurd. Thus, $X$ cannot be dilated to a product system.
\end{proof}

We can now strengthen Theorem \ref{thm:parrot}:
\begin{theorem}\label{thm:strongparrot}
There exists a $cp$-semigroup $\Theta = \{\Theta_n\}_{n \in \mb{N}^3}$ acting on a $B(H)$, 
such that for all $\lambda > 0$, the semigroup 
$\lambda \Theta := \{\lambda^{|n|} \Theta_n\}_{n \in \mb{N}^3}$ has no $e$-dilation $(\alpha,K,B(K))$, and no \emph{minimal} $e$-dilation $(\alpha,K,\cR)$ on any von Neumann algebra $\cR$.
\end{theorem}
\begin{proof}
Let $X$ be as in Proposition \ref{prop:sprdctcntrexample}. Let $\Theta$ be the $cp$-semigroup generated by the $X$-shift, as in Section \ref{sec:shift}. Of course, $\Theta$, as a semigroup over $\mb{N}^3$, can be generated by three commuting CP maps $\theta_1,\theta_2,\theta_3$. $X$ cannot be embedded into a full product system, so by Theorem \ref{thm:edil_repdil}, $\Theta$ has no minimal $e$-dilation, nor does it have an $e$-dilation acting on a $B(K)$. Note that if $\Theta$ is scaled \emph{its product system is left unchanged} (this follows from Theorem \ref{thm:essentially_inverse}: if you take $X$ and scale the representation $S^X$ you get a scaled version of $\Theta$). So no matter how small you take $\lambda  > 0$, $\lambda \theta_1,\lambda \theta_2,\lambda \theta_3$ cannot be dilated to three commuting $*$-endomorphisms on $B(K)$, nor to a minimal three-tuple on any von Neumann algebra.
\end{proof}

Note that the obstruction here seems to be of a completely different nature from the one in the example given in Theorem \ref{thm:parrot}. The subproduct system arising there is already a product system, and, indeed, the $cp$-semigroup arising there can be dilated once it is multiplied by a small enough scalar.

\part{Subproduct systems over $\mb{N}$ and operator algebras}\label{part:III}

\chapter{Subproduct systems of Hilbert spaces over $\mb{N}$}\label{chap:subproductN}

We now specialize to subproduct systems of Hilbert W$^*$-correspondences over the semigroup $\mb{N}$, so from now on any subproduct system is to be understood as such (soon we will specialize even further to subproduct systems of Hilbert spaces).

\section{Standard and maximal subproduct systems}

If $X$ is a subproduct system over $\mb{N}$, then $X(0) = \cM$ (some von Neumann algebra), $X(1)$ equals some W$^*$-correspondence $E$, and $X(n)$ can be regarded as a subspace of $E^{\otimes n}$. The following lemma allows us to consider $X(m+n)$ as a subspace of $X(m) \otimes X(n)$.

\begin{lemma}\label{lem:projection_subproduct}
Let $X = \{X(n)\}_{n \in \mb{N}}$ be a subproduct system. $X$ is isomorphic to a subproduct system $Y = \{Y(n)\}_{n\in\mb{N}}$ with coisometries $\{U_{m,n}^Y\}_{m,n \in \mb{N}}$ that satisfies
\bes
Y(1) = X(1)
\ees
and
\be\label{eq:p_subset}
Y(m+n) \subseteq Y(m) \otimes Y(n).
\ee
Moreover, if $p_{m+n}$ is the orthogonal projection of $Y(1)^{\otimes (m + n)}$ onto $Y(m+n)$, then
\be\label{eq:p_maps}
U_{m,n}^Y = p_{m+n}\Big|_{Y(m) \otimes Y(n)}
\ee
and the projections $\{p_n\}_{n \in \mb{N}}$ satisfy
\be\label{eq:p_assoc}
p_{k+m+n} = p_{k+m+n}(I_{E^{\otimes k}} \otimes p_{m+n}) = p_{k+m+n}(p_{k+m} \otimes I_{E^{\otimes n}}).
\ee
\end{lemma}
\begin{proof}
Denote by $U^X_{m,n}$ the subproduct system maps $X(s) \otimes X(t) \rightarrow X(s+t)$. Denote $E = X(1)$. We first note that for every $n$ there is a well defined coisometry $U_n : E^{\otimes n} \rightarrow X(n)$ given by composing in any way a sequence of maps $U^X_{k,m}$ (for example, one can take $U_3 = U^X_{2,1}(U^X_{1,1} \otimes I_{E})$ and so on). We define $Y(n) = {\rm Ker}(U_n)^\perp$, and we let $p_n$ be the orthogonal projection from $E^{\otimes n}$ onto $Y(n)$.
$p_n = U_n^* U_n$, so, in particular, $p_n$ is a bimodule map. For all $m,n \in \mb{N}$ we have that
\bes
E^{\otimes (m)} \otimes {\rm Ker}(U_n) \subseteq {\rm Ker}(U_{m+n}) .
\ees
Thus $E^{\otimes (m)} \otimes {\rm Ker}(U_n)^\perp \supseteq {\rm Ker}(U_{m+n})^\perp$, so
$p_{m+n} \leq I_{E^{\otimes m}} \otimes p_n$. This means that (\ref{eq:p_assoc}) holds. In addition, defining $U_{m,n}^Y$ to be $p_{m+n}$ restricted to $Y(m) \otimes Y(n) \subseteq E^{\otimes (m+n)}$ gives $Y$ the associative multiplication of a subproduct system.

It remains to show that $X$ is isomorphic to $Y$. For all $n$, $X(n)$ is spanned by elements of the form $U_n(x_1 \otimes \cdots \otimes x_n)$, with $x_1, \ldots, x_n \in E$. We define a map $V_n : X(n) \rightarrow Y(n)$ by
\bes
V_n \big( U_n(x_1 \otimes \cdots \otimes x_n) \big) = p_n (x_1 \otimes \cdots \otimes x_n).
\ees
It is immediate that $V_n$ preserves inner products (thus it is well defined) and that it maps $X(n)$ onto $Y(n)$. Finally, for all $m,n \in \mb{N}$ and $x \in E^{\otimes m}, y \in E^{\otimes n}$,
\begin{align*}
V_{m+n}\big(U_{m,n}^X (U_m(x) \otimes U_n(y)) \big)
&= V_{m+n}\big(U_{m+n}(x\otimes y) \big) \\
&= p_{m+n} (x \otimes y) \\
&= p_{m+n} (p_m x \otimes p_n y) \\
&= p_{m+n} \big((V_m U_m (x)) \otimes (V_n U_n (y))\big) \\
&= U^Y_{m+n} \big((V_m U_m (x)) \otimes (V_n U_n (y))\big) ,
\end{align*}
and (\ref{eq:iso}) holds.
\end{proof}

\begin{definition}
A subproduct system $Y$ satisfying (\ref{eq:p_subset}), (\ref{eq:p_maps}) and (\ref{eq:p_assoc}) above will be called a \emph{standard} subproduct system.
\end{definition}

Note that a standard subproduct system is a subproduct subsystem of the full product system $\{E^{\otimes n}\}_{n\in\mb{N}}$.
\begin{corollary}\label{cor:discrete_bhat}
Every $cp$-semigroup over $\mb{N}$ has an $e$-dilation.
\end{corollary}
\begin{proof}
The unital case follows from Corollary \ref{cor:e0dilwhensub} together with the above lemma. The nonunital
case follows from a similar construction (where the dilation of a non-fully-coisometric representation is obtained
by using Corollary \ref{cor:normcondregdil} for $k=1$).
\end{proof}

Let $k \in \mb{N}$, and let $E = X(1), X(2), \ldots, X(k)$ be subspaces of $E, E^{\otimes 2}, \ldots, E^{\otimes k}$, respectively, such that the orthogonal projections $p_n : E^{\otimes n} \rightarrow X(n)$ satisfy
$$p_n \leq I_{E^{\otimes i}} \otimes p_j$$
and
$$p_n \leq p_i \otimes I_{E^{\otimes j}}$$
for all $i,j,n \in \mb{N}_+$ satisfying $i+j = n \leq k$. In this case one can define a maximal standard subproduct system $X$ with the prescribed fibers $X(1), \ldots, X(k)$ by defining inductively for $n>k$
$$X(n) = \left(\bigcap_{i+j=n} E^{\otimes i} \otimes X(j)\right) \bigcap \left(\bigcap_{i+j=n} X(i) \otimes E^{\otimes j}\right).$$
It is easy to see that
$$X(n) = \bigcap_{n_1 + \ldots + n_m=n} X(n_1) \otimes \cdots \otimes X(n_m) = \bigcap_{i+j=n} X(i) \otimes X( j).$$
We then have obvious formulas for the projections $\{p_n\}_{n\in \mb{N}}$ as well, for example
$$p_n = \bigwedge_{i+j = n}p_i \otimes p_j \,\, , \,\, (n > k).$$

\section{Examples}

\begin{example}\label{expl:symmmax}
\emph{
In the case $k = 1$, the maximal standard subproduct system with prescribed fiber $X(1) = E$, with $E$ a Hilbert space, is the full product system $F_E$ of Example \ref{expl:full}. If $\dim E = d$, we think of this subproduct system as the product system representing a (row-contractive) $d$-tuple $(T_1, \ldots, T_d)$ of non commuting operators, that is, $d$ operators that are not assumed to satisfy any relations (the idea behind this last remark  must be rather vague at this point, but it shall become clearer as we proceed).
In the case $k = 2$, if $X(2)$ is the symmetric tensor product $E$ with itself then the maximal standard subproduct system with prescribed fibers $X(1), X(2)$ is the symmetric subproduct system $SSP_E$ of Example \ref{expl:symm}. We think of $SSP$ as the subproduct system representing a commuting $d$-tuple.}
\end{example}

\begin{example}\label{expl:dimexp}
\emph{
Let $E$ be a two dimensional Hilbert space with basis $\{e_1,e_2\}$. Let $X(2)$ be the space spanned by $e_1 \otimes e_1, e_1 \otimes e_2$, and $e_2 \otimes e_1$. In other words, $X(2)$ is what remains of $E^{\otimes 2}$ after we declare that $e_2 \otimes e_2=0$. We think of the maximal standard subproduct system $X$ with prescribed fibers $X(1) = E, X(2)$ as the subproduct system representing pairs $(T_1, T_2)$ of operators subject only to the condition $T_2^2 = 0$. $E^{\otimes n}$ has a basis consisting of all vectors of the form $e_\alpha = e_{\alpha_1} \otimes \cdots \otimes e_{\alpha_n}$ where $\alpha = \alpha_1 \cdots \alpha_n$ is a word of length $n$ in ``1" and ``2". $X(n)$ then has a basis consisting of all vectors $e_\alpha$ where $\alpha$ is a word of length $n$ not containing ``22" as a subword.
Let us compute $\dim X(n)$, that is, the number of such words.}

\emph{Let $A_n$ denote the number of words not containing ``22" that have leftmost letter ``1", and let $B_n$ denote the
number of words not containing ``22" that have leftmost letter ``2".
Then we have the recursive relation $A_n = A_{n-1} + B_{n-1}$ and $B_n = A_{n-1}$. The solution of this recursion gives}
$$\dim X(n) = A_n + B_n \approx \left(\frac{1+\sqrt{5}}{2}\right)^n .$$
\emph{As one might expect, the dimension of $X(n)$ grows exponentially fast.}
\end{example}

\begin{example}\label{expl:dimconst}
\emph{Suppose that we want a ``subproduct system that will represent a pair of operators $(T_1, T_2)$ such that
$T_i T_2 = 0$ for $i=1,2$". Although we have not yet made clear what we mean by this, let us proceed heuristically along the lines of the preceding examples. We let $E$ be as above, but now we declare $e_1 \otimes e_2 = e_2 \otimes e_2 = 0$. In other words, we define $X(2) = \{e_1 \otimes e_2, e_2 \otimes e_2\}^\perp$. One checks that the maximal standard subproduct system $X$ with prescribed fibers $X(1) = E, X(2)$ is given by $X(n) = \textrm{span}\{e_1 \otimes e_1 \otimes \cdots \otimes e_1, e_2 \otimes e_1 \otimes \cdots \otimes e_1\}$. This is an example of a subproduct system with two dimensional fibers.}
\end{example}

At this point two natural questions might come to mind. First, \emph{is every standard subproduct system $X$ the maximal subproduct system with prescribed fibers $X(1), \ldots, X(k)$ for some $k \in \mb{N}$?} Second, \emph{does $\dim X(n)$ grow exponentially fast (or remain a constant) for every subproduct system $X$?} The next example answers both questions negatively.

\begin{example}\label{expl:notmax}
\emph{Let $E$ be as in the preceding examples, and let $X(n)$ be a subspace of $E^{\otimes n}$ having basis the set}
$$\{e_\alpha: |\alpha|=n, \alpha \textrm{ does not contain the words } 22, 212, 2112, 21112, \ldots\} .$$
\emph{Then $X = \{X(n)\}_{n \in \mb{N}}$ is a standard subproduct system, but it is smaller than the maximal subproduct system defined by any initial $k$ fibers. Also, $X(n)$ is the span of $e_{\alpha}$ with $\alpha = 11\cdots 11, 21\cdots 11, 121\cdots 11, \ldots, 11 \cdots 12$, thus}
$$\dim X(n) = n+1 ,$$
\emph{so this is an example of a subproduct system with fibers that have a linearly growing dimension.}
\end{example}

Of course, one did not have to go far to find an example of a subproduct system with linearly growing dimension: indeed, the dimension of the fibers of the symmetric subproduct system $SSP_{\mb{C}^d}$ is known to be
\bes
\dim SSP_{\mb{C}^d}(n) = \left(
\begin{array}{c}
 n+d-1 \\
 n
\end{array} \right).
\ees
Taking $d=2$ we get the same dimension as in Example \ref{expl:notmax}. However, $SSP := SSP_{\mb{C}^2}$ and the subproduct system $X$ of Example \ref{expl:notmax} are not isomorphic: for any nonzero $x \in SSP(1)$, the ``square" $U^{SSP}_{1,1}(x \otimes x) \in SSP(2)$ is never zero, while $U^X_{1,1}(e_2 \otimes e_2) = 0$.

Here is an interesting question that we do not know the  answer to: \emph{given a solution $f:\mb{N}\rightarrow \mb{N}$ to the functional inequality}
\bes
f(m+n) \leq f(m) f(n) \,\, , \,\, m,n \in \mb{N},
\ees
\emph{does there exists a subproduct system $X$ such that $\dim X(n) = f(n)$ for all $n \in \mb{N}$?}

\begin{remark}
\emph{One can cook up simple examples of subproduct systems that are not standard. We will not write these examples down, as we already know that such a subproduct system is isomorphic to a standard one.}
\end{remark}

\section{Representations of subproduct systems}\label{sec:representations}

Fix a W$^*$-correspondence $E$. Every completely contractive linear map $T_1:E\rightarrow B(H)$ gives rise to a c.c. representation $T^n$ of the full product system $F_E=\{E^{\otimes n}\}_{n \in \mb{N}}$ by defining
for all $x \in E^{\otimes n}$ and $h \in H$
\be\label{eq:alotofT}
T^n(x)h = \tT_1\big(I_{E} \otimes \tT_1\big) \cdots \big(I_{E^{\otimes (n-1)}}\otimes \tT_1\big) (x \otimes h) ,
\ee
where $\tT_1 : E \otimes H \rightarrow H$ is given by $\tT_1(e \otimes h) = T_1(e) h$. We will denote the operator acting on $x \otimes h$ in the right hand side of (\ref{eq:alotofT}) as $\tT^n$, so as not to confuse with $\tT_n$, which sometimes has a different meaning (namely: if $T$ denotes a c.c. representation of a subproduct system $X$ then
$$\tT_n: X(n) \otimes H \rightarrow H$$
is given by
$$\tT_n (x \otimes h) = T(x) h $$
for all $x \in X(n), h \in H$. Of course, when $X = F_E$, $T$ is a representation of $F_E$ and $T_1$ is the restriction of $T$ to $E$, then $\tT^n = \tT_n$ for all $n$). If $X$ is a standard subproduct system and $X(1) = E$, we obtain a completely contractive representation of $X(n)$ by restricting $T^n$ to $X(n)$. Let us denote this restriction by $T_n$, and denote the family $\{T_n\}_{n \in \mb{N}}$ by $T$.
\begin{proposition}\label{prop:representation}
Let $X$ be a standard subproduct system with projections $\{p_n\}_{n\in \mb{N}}$, and let $T_1:E \rightarrow B(H)$ be a completely contractive map. Construct the family of maps $T = \{T_n\}_{n\in \mb{N}}$, with $T_n:X(n) \rightarrow B(H)$ as in the preceding paragraph. Then the following are equivalent:
\begin{enumerate}
\item\label{it:Trep} $T$ is a representation of $X$.
\item\label{it:tT} For all $m,n \in \mb{N}$,
\be\label{eq:tTrep}
\tT_{m}(I_{X(m)} \otimes \tT_n)(p_m \otimes p_n \otimes I_H)(p^\perp_{m+n} \otimes I_H) = 0 .
\ee
\item\label{it:tTn} For all $n \in \mb{N}$,
\be\label{eq:tTn}
\tT^n (p_n^\perp \otimes I_H) = 0.
\ee
\end{enumerate}
\end{proposition}
\begin{proof}
If $T$ is a representation, then
\bes
\tT_{m}(I_{X(m)} \otimes \tT_n)(p_m \otimes p_n \otimes I_H)(p^\perp_{m+n} \otimes I_H)
= \tT_{m+n}(p_{m+n} \otimes I_H)(p^\perp_{m+n} \otimes I_H)
= 0 ,
\ees
so \ref{it:Trep} $\Rightarrow$ \ref{it:tT}. To prove \ref{it:tT} $\Rightarrow$ \ref{it:tTn} note first that (\ref{eq:tTn}) is clear for $n=1$. Assuming that (\ref{eq:tTn}) holds for $n=1,2, \ldots, k-1$, we will show that it holds
for $n=k$.
\begin{align*}
\tT^k (p_k^\perp \otimes I_H)
&= \tT^1(I \otimes \tT^{k-1}) (p_k^\perp \otimes I_H) \\
&= \tT^1(I \otimes \tT^{k-1})(I_{E} \otimes p_{k-1}^\perp \otimes I_H + I_{E} \otimes p_{k-1} \otimes I_H) (p_k^\perp  \otimes I_H)\\
(*)&= \tT^1(I \otimes \tT^{k-1}(p_{k-1} \otimes I_H)) (p_k^\perp  \otimes I_H)\\
&= \tT_1(I \otimes \tT_{k-1}(p_{k-1} \otimes I_H)) (p_k^\perp  \otimes I_H)\\
(**)&= 0 .
\end{align*}
The equality marked by (*) is true by the inductive hypothesis, and the one marked by (**) follows from (\ref{eq:tTrep}).

Finally, \ref{it:tTn} $\Rightarrow$ \ref{it:Trep}: by (\ref{eq:tTn}) we have $\tT^{n} (p_{n} \otimes I_H) = \tT^{n}$. On the other hand, $\tT^{n} (p_{n} \otimes I_H) = \tT_{n} (p_{n} \otimes I_H)$. Thus
\begin{align*}
\tT_{m+n} (p_{m+n} \otimes I_H)
&= \tT^{m+n} (p_{m+n} \otimes I_H) \\
&= \tT^{m+n} \\
&= \tT^{m}(I_{X(m)} \otimes \tT^n) \\
&= \tT_{m}(I_{X(m)} \otimes \tT_n)(p_m \otimes p_n \otimes I_H) ,
\end{align*}
which shows that $T$ is a representation.
\end{proof}

\begin{proposition}
Let $X$ be the maximal standard subproduct system with prescribed fibers $X(1), \ldots, X(k)$, and let $T_1:E \rightarrow B(H)$ be a completely contractive map. Construct $T$ as in Proposition \ref{prop:representation}. Then $T$ is a representation of $X$ if and only if
\be\label{eq:tTnk}
\tT^n (p_n^\perp \otimes I_H) = 0 \quad \textrm{for all} \quad n=1,2,\ldots, k .
\ee
\end{proposition}
\begin{proof}
The necessity of (\ref{eq:tTnk}) follows from Proposition \ref{prop:representation}. By the same proposition, to show that the condition is sufficient it is enough to show that (\ref{eq:tTnk}) holds for all $n \in \mb{N}$. Given $m \in \mb{N}$, we have $p_m = \bigwedge_q q$, where $q$ runs over all
projections of the form $q = I_{X(i)} \otimes p_j$ or $q = p_i \otimes I_{X(j)}$, with $i,j \in \mb{N}_+$ and $i+j = m$. But then $p_m^\perp = \bigvee_{q} q^\perp$, thus if (\ref{eq:tTnk}) holds for all $n<m$ then it also holds for $n=m$.
\end{proof}

\section{Fock spaces and standard shifts}

\begin{definition}
Let $X$ be a subproduct system of Hilbert spaces. Fix an orthonormal basis $\{e_i\}_{i \in \cI}$ of $E =X(1)$.
$X(n)$, when considered as a subspace of $\mathfrak{F}_X$, is called \emph{the $n$ particle space}.
The \emph{standard $X$-shift (related to $\{e_i\}_{i \in \cI}$) on $\mathfrak{F}_X$} is the tuple of operators $\underline{S}^X = \left(S^X_i\right)_{i \in \cI}$ in $B(\mathfrak{F}_X)$ given by
\bes
S^X_i(x) = U_{1,n}(e_i \otimes x) ,
\ees
for all $i \in \cI$, $n \in \mb{N}$ and $x \in X(n)$.
\end{definition}
It is clear that $S^X_i = S^X(e_i)$, where $S^X$ is the shift representation given by Definition \ref{def:shiftrep}.

If $F$ denotes the usual full product system (Example \ref{expl:full}) then $\mathfrak{F}_F$
is the usual Fock space and the tuple $(S^F_i)_{i \in \cI}$ is the standard shift (the $\cI$ orthogonal shift of \cite{Popescu89}). We shall denote $\mathfrak{F}_F$ as $\mathfrak{F}$ and $(S^F_i)_{i \in \cI}$ as $(S_i)_{i \in \cI}$. It is then obvious that the tuple $\left(S^X_i\right)_{i \in \cI}$ is a row contraction, as it is the compression of the row contraction $(S_i)_{i \in \cI}$. Indeed, assuming (as we may, thanks to Lemma \ref{lem:projection_subproduct}) that $U_{m,n}$ is an orthogonal projection $p_{m+n}:X(m) \otimes X(n) \rightarrow X(m+n)$, and denoting $p = \oplus_{n}p_n$, we have for all $i$ that $S^X_i = p S_i \big|{\mathfrak{F}_X}$.

\begin{example}\label{expl:qcommuting}
\emph{
The $q$-commuting Fock space of \cite{Dey07} also fits into this framework.
Indeed, let (as in \cite{Dey07}) $\Gamma(\mb{C}^d)$ be the full Fock space, let $\Gamma_q(\mb{C}^d)$ denote the $q$-commuting Fock space, and
let $Y(n)$ be the ``$n$ particle $q$-commuting space" with orthogonal projection $p_n : (\mb{C}^d)^n \rightarrow Y(n)$. Then a straightforward calculation shows that the projections $\{p_n\}_{n\in\mb{N}}$ satisfy
equation (\ref{eq:p_assoc}) of Lemma  \ref{lem:projection_subproduct}, thus $Y = \{Y(n)\}_{n\in\mb{N}}$ is a subproduct system (satisfying (\ref{eq:p_subset}) and (\ref{eq:p_maps})).
With our notation from above we have that $\mathfrak{F}_Y = \Gamma_q(\mb{C}^d)$ and that
the tuple $(S^Y_i, \ldots, S^Y_d)$ is the standard $q$-commuting shift.
}
\end{example}

$S^F$, the standard shift of the full product system on the full Fock space, will be denoted by $S$, and will be called simply \emph{the standard shift}.

By the notation introduced in Definition \ref{def:Xpiece}, the symbol $S^X$ is also used to denote the maximal $X$-piece of the standard shift $S$. The following proposition -- which is a generalization of \cite[Proposition 6]{BBD03}, \cite[Proposition 11]{Dey07} and \cite[Proposition 2.9]{Popescu06} -- shows that this is consistent.

\begin{proposition}\label{prop:Xshift}
Let $X$ subproduct subsystem of a subproduct system $Y$. Then the maximal $X$-piece of the standard $Y$-shift is the standard $X$-shift.
\end{proposition}
\begin{proof}
Let $E = Y(1)$, and let $F = F_E$ be the full product system. Viewing
$F(n)\otimes\mathfrak{F}$ as direct sum of $|\cI|^n$ copies of $\mathfrak{F}_F$, $(\widetilde{S})_n$ is just the row isometry
$(S_{i_1}\circ\cdots\circ S_{i_n})_{i_1, \ldots, i_n \in \cI}$ from the space of columns $\mathfrak{F}_F \oplus \mathfrak{F}_F \oplus \cdots$ into $\mathfrak{F}_F$. In other words, for $h \in \mathfrak{F}_F$ and $i_1, \ldots, i_n \in I$,
$$(\widetilde{S})_n \big((e_{i_1} \otimes \cdots \otimes e_{i_n}) \otimes h\big) = S_{i_1}\circ\cdots\circ S_{i_n}h = (e_{i_1} \otimes \cdots \otimes e_{i_n}) \otimes h.$$
This is an isometry, and the adjoint works by sending $(e_{i_1} \otimes \cdots \otimes e_{i_n}) \otimes h \in \mathfrak{F}_F$ back to $(e_{i_1} \otimes \cdots \otimes e_{i_n}) \otimes h \in F(n) \otimes \mathfrak{F}_F$, and by sending the $0,1, \ldots, n-1$ particle spaces to $0$.

Now, if $Z$ is any standard subproduct subsystem of $F$, then
\bes
\left(\widetilde{S^Z}\right)_n = P_{\mathfrak{F}_Z}\left(\widetilde{S}\right)_n \big|_{Z(n) \otimes \mathfrak{F}_Z},
\ees
thus
\be\label{eq:SZ*}
\left(\widetilde{S^Z}\right)_n^* = P_{Z(n) \otimes \mathfrak{F}_Z}\left(\widetilde{S}\right)_n^* \big|_{\mathfrak{F}_Z}.
\ee
Now if $h$ is in the $k$ particle space of $\mathfrak{F}_F$ with $k<n$, then $(\widetilde{S^Z})_n^* h = 0$. If $k\geq n$, then since $Z(k) \subseteq Z(n) \otimes Z(k-n)$ we may write $h = \sum \xi_i \otimes \eta_i$, where $\xi_i \in Z(n)$ and $\eta_i \in Z(k-n)$. Thus by (\ref{eq:SZ*}) we find that
\be\label{eq:tilde*}
(\widetilde{S^Z})_n^* \left(\sum \xi_i \otimes \eta_i \right) = \sum p^Z_n \xi_i \otimes p^Z_{k-n}\eta_i = \sum \xi_i \otimes \eta_i.
\ee
From these considerations it follows that the standard $X$-shift is in fact an $X$-piece of the standard $Y$ shift,
as $(\widetilde{S^Y})_n^* \big|_{\mathfrak{F}_X} = (\widetilde{S^X})_n^*$. It remains to show that the $X$-shift is maximal.

Assume that there is a Hilbert space $H$, $\mathfrak{F}_X \subseteq H \subseteq \mathfrak{F}_Y$, such that the compression
of $S^Y$ to $H$ is an $X$-piece of $Y$, that is, $H \in \cP(X,S^Y)$ (see equation (\ref{eq:cPXT})). Let $h \in H \ominus \mathfrak{F}_X$. We shall prove that $h = 0$. Being orthogonal to all of $\mathfrak{F}_X$, $p^Y_n h$ must be orthogonal to $X(n)$ for all $n$. Thus, we may assume that $h\in Y(n) \ominus X(n)$ for some $n$. But then by (\ref{eq:tilde*})
\bes
(\widetilde{S^Y})_n^* h = h \otimes \Omega.
\ees
But since $H \in \cP(X,S^Y)$, we must have $h \otimes \Omega \in X(n) \otimes H$, and this, together with $h \in Y(n) \ominus X(n)$, forces $h=0$.
\end{proof}

\chapter{Zeros of homogeneous polynomials in noncommuting variables}\label{chap:projective}

In the next chapter we will describe a model theory for representations of subproduct systems. But before that we dedicate this chapter to build a precise connection between subproduct systems together with their representations and tuples of operators that are the zeros of homogeneous polynomials in non commuting variables.

\begin{remark}
\emph{The notions that we are developing give a framework for studying tuples of operators satisfying relations given by homogeneous polynomials. One can go much further by considering subspaces of Fock spaces
and ``representations", i.e., maps of the Fock space into $B(H)$, that give a framework for studying tuples of operators satisfying arbitrary (not-necessarily homogeneous) polynomial and even analytic identities. Gelu Popescu \cite{Popescu06} has already begun developing such a theory.}
\end{remark}

We begin by setting up the usual notation. Let $\cI$ be a fixed set of indices, and let $\mb{C}\langle (x_i)_{i \in \cI}\rangle$ be the algebra of complex polynomials in the non commuting variables $(x_i)_{i \in \cI}$. We denote $\underline{x} = (x_i)_{i \in \cI}$, and we consider $\underline{x}$ as a ``tuple variable". We shall sometimes write $\mb{C}\langle \underline{x}\rangle$ for $\mb{C}\langle (x_i)_{i \in \cI}\rangle$. The set of all words in $\cI$ is denoted by $\mb{F}_\cI^+$. For a word $\alpha \in \mb{F}_\cI^+$, let $|\alpha|$ denote the length of $\alpha$, i.e., the number of letters in $\alpha$.

For every word $\alpha = \alpha_1 \cdots \alpha_k$ in $\cI$ denote $\underline{x}^\alpha = x_{\alpha_1} \cdots x_{\alpha_k}$. If $\alpha = 0$ is the empty word, then this is to be understood as $1$. $k$ is also referred to in this context as the \emph{degree} of the monomial $x^\alpha$. $\mb{C}\langle \underline{x}\rangle$ is by definition the linear span over $\mb{C}$ of all such monomials, and every element in $\mb{C}\langle \underline{x}\rangle$ is called a polynomial. 
A polynomial is called \emph{homogeneous} if it is the sum of monomials of equal degree. A \emph{homogeneous ideal} is a two-sided ideal that is generated by homogeneous polynomials.

If $\underline{T} = (T_i)_{i \in \cI}$ is a tuple of operators on a Hilbert space $H$ and $\alpha = \alpha_1 \cdots \alpha_k$ is a word with letters in $\cI$, we define
$$\underline{T}^\alpha = T_{\alpha_1} T_{\alpha_2} \cdots T_{\alpha_k} .$$
We consider the empty word $0$ as a legitimate word, and define $\underline{T}^0 = I_H$. If $p(x) = \sum_\alpha c_\alpha \underline{x}^\alpha \in \mb{C}\langle \underline{x}\rangle$, we define $p(\underline{T}) = \sum_\alpha c_\alpha \underline{T}^\alpha$.

If $E$ is a Hilbert space with orthonormal basis $\{e_i\}_{i \in \cI}$, An element $e_{\alpha_1} \otimes \cdots \otimes e_{\alpha_k} \in E^{\otimes k}$ will be written in short form as $e_{\alpha}$, where $\alpha = \alpha_1 \cdots \alpha_k$. If $p(x) = \sum_\alpha c_\alpha \underline{x}^\alpha \in \mb{C}\langle \underline{x}\rangle$, we define $p(e) = \sum_\alpha c_\alpha e_\alpha$. Here $e_0$ ($0$ the empty word) is understood as the vacuum vector $\Omega$.

\begin{proposition}\label{prop:XandI}
Let $E$ be a Hilbert space with orthonormal basis $\{e_i\}_{i \in \cI}$. There is an inclusion reversing correspondence between proper homogeneous ideals $I \triangleleft \mb{C}\langle \underline{x}\rangle$ and standard subproduct systems $X = \{X(n)\}_{n \in \mb{N}}$ with $X(1) \subseteq E$. When $|\cI| < \infty$ this correspondence is bijective.
\end{proposition}
\begin{proof}
Let $X$ be such a subproduct system. We define an ideal
\be\label{eq:I^X}
I^X := \textrm{span}\{p\in \mb{C}\langle\underline{x} \rangle : \exists n > 0 .  p(e) \in E^{\otimes n} \ominus X(n)\}.
\ee
Once it is established that $I^X$ is a two-sided ideal the fact that it is homogeneous will follow from the definition. Let $p\in \mb{C}\langle\underline{x} \rangle$ be such that $p(e) \in E^{\otimes n} \ominus X(n)$ for some $n > 0$. It suffices to show that for every monomial $\underline{x}^\alpha$ we have that $\underline{x}^\alpha p(\underline{x}) \in I^X$, that is,
\bes
e_\alpha \otimes p(e) \in E^{\otimes |\alpha| + n} \ominus X(|\alpha| + n).
\ees
But since $X$ is standard, $X(|\alpha| + n) \subseteq X(|\alpha|) \otimes X(n)$, thus
\bes
E^{\otimes |\alpha|} \otimes (E^{\otimes n} \ominus X(n)) \subseteq E^{\otimes |\alpha| + n} \ominus X(|\alpha| + n).
\ees
It follows that $I^X$ is a homogeneous ideal.

Conversely, let $I$ be a homogeneous ideal. We construct a subproduct system $X_I$ as follows. Let $I^{(n)}$ be the set of all homogeneous polynomials of degree $n$ in $I$. Define
\be\label{eq:X_I}
X_I(n) = E^{\otimes n} \ominus \{p(e) : p \in I^{(n)}\}.
\ee
Denote by $p_n$ the orthogonal projection of $E^{\otimes n}$ onto $X_I(n)$. To show that $X_I$ is a subproduct system it is enough (by symmetry) to prove that for all $m,n \in \mb{N}$
\bes
p_{m+n} \leq I_{E^{\otimes m}} \otimes p_n,
\ees
or, in other words, that
\be\label{eq:X_IinEX_I}
X_I(m+n) \subseteq E^{\otimes m} \otimes X_I(n).
\ee
Let $x \in X_I(m+n)$, let $\alpha \in \cI^m$, and let $q \in I^{(n)}$. Since $I$ is an ideal, $\underline{x}^\alpha q(\underline{x})$ is in $I^{(m+n)}$, thus $\langle x, e_\alpha \otimes q(e) \rangle = 0$. This proves (\ref{eq:X_IinEX_I}).

Assume now that $|\cI| < \infty$. We will show that the maps $X \mapsto I^X$ and $I \mapsto X_I$ are inverses of each other. Let $J$ be a homogeneous ideal in $\mb{C}\langle \underline{x}\rangle$. Then
\begin{align*}
I^{X_J} &= \textrm{span}\{p\in \mb{C}\langle\underline{x} \rangle : \exists n > 0 .  p(e) \in E^{\otimes n} \ominus X_J(n)\} \\
(*)&= \textrm{span}\{p\in \mb{C}\langle\underline{x} \rangle :  \exists n > 0 . p(e) \in \{q(e): q \in J^{(n)}\}\} \\
&= \textrm{span}\{p : \exists n > 0 . p\in  J^{(n)}\} \\
(**)&= J,
\end{align*}
where (*) follows from the definition of $X_J$, and (**) from the fact that $J$ is a homogeneous ideal.

For the other direction, let $Y$ be a standard subproduct subsystem of $F_E = \{E^{\otimes n}\}_{n \in \mb{N}}$. Clearly, $(I^Y)^{(n)} = \{p\in \mb{C}\langle\underline{x} \rangle : p(e) \in E^{\otimes n} \ominus Y(n)\}$.
Thus
\begin{align*}
X_{I^Y}(n) &= E^{\otimes n} \ominus \{p(e) : p \in (I^Y)^{(n)}\} \\
&= E^{\otimes n} \ominus  \{p(e) : p \in \{q\in \mb{C}\langle\underline{x} \rangle : q(e) \in E^{\otimes n} \ominus Y(n)\} \\
&= E^{\otimes n} \ominus  (E^{\otimes n} \ominus Y(n)) \\
&= Y(n).
\end{align*}
\end{proof}

We record the definitions of $I^X$ and $X_I$ from the above theorem for later use:

\begin{definition}
Let $E$ be a Hilbert space with orthonormal basis $\{e_i\}_{i \in \cI}$ ($|\cI|$ is not assumed finite). Given a homogeneous ideal $I \triangleleft \mb{C}\langle \underline{x} \rangle$, the subproduct system $X_I$ defined by (\ref{eq:X_I}) will be called the \emph{subproduct system associated with $I$}. If $X$ is a given subproduct subsystem of $F_E$, then the ideal $I^X$ of $\mb{C}\langle \underline{x} \rangle$ defined by (\ref{eq:I^X}) will be called the \emph{ideal associated with $X$}.
\end{definition}

We note that $X_I$ depends on the choice of the space $E$ and basis $\{e_i\}_{i \in \cI}$, but different choices will give rise to isomorphic subproduct systems.

\begin{proposition}\label{prop:change_of_variables}
Let $X$ and $Y$ be standard subproduct systems with $\dim X(1) = \dim Y(1) = d < \infty$. Then $X$ is isomorphic to $Y$ if and only if there is a unitary linear change of variables in $\mb{C}\langle x_1, \ldots, x_d \rangle$ that sends $I^X$ onto $I^Y$.
\end{proposition}

Fix some infinite dimensional separable Hilbert space $H$. As in classical algebraic geometry, given a homogeneous ideal $I \triangleleft \mb{C}\langle \underline{x}\rangle$, it is natural to introduce and to study the \emph{zero set of $I$}
\bes
Z(I) := \{\underline{T} = (T_i)_{i \in \cI} \in B(H)^\cI : \forall p \in I. p(\underline{T}) = 0\}.
\ees
Also, given a set $Z \subseteq B(H)^\cI$, one may form the following two-sided ideal in $\mb{C}\langle \underline{x}\rangle$
\bes
I(Z) := \{p \in \mb{C}\langle \underline{x}\rangle : \forall \underline{T} \in Z . p(\underline{T}) = 0\}.
\ees

In the following theorem we shall use the notation of \ref{sec:representations}. This simple result is the justification for viewing subproduct systems as a framework for studying tuples of operators satisfying certain homogeneous polynomial relations.
\begin{theorem}\label{thm:repZ(I)}
Let $E$ be a Hilbert space with orthonormal basis $\{e_i\}_{i \in \cI}$ (not necessarily with $|\cI| < \infty$), and let $I$ be a proper homogeneous ideal in $\mb{C}\langle (x_i)_{i \in \cI}\rangle$. Let $X_I$ be the associated subproduct system. Let $T_1 : E \rightarrow B(H)$ be a given representation of $E$. Define a tuple $\underline{T} = (T(e_i))_{i \in \cI}$. Construct the family of maps $T = \{T_n\}_{n\in \mb{N}}$, with $T_n:X(n) \rightarrow B(H)$ as in the paragraphs before Proposition \ref{prop:representation}. Then $T$ is a representation of $X$ if and only if $\underline{T} \in Z(I)$.
\end{theorem}
\begin{proof}
On the one hand, $E^{\otimes n} \ominus X_I(n) = \overline{\textrm{span}}\{p(e) : p \in I^{(n)}\}$. On the other hand, for every $p \in I^{(n)}$ and every $h \in H$,
\bes
\widetilde{T}^n (p(e) \otimes h) = p(\underline{T})h.
\ees
Hence, the Theorem follows from Proposition \ref{prop:representation}.
\end{proof}

\begin{lemma}\label{lem:shiftpoly}
Let $J \triangleleft \mb{C}\langle (x_i)_{i \in \cI}\rangle$, $|\cI| < \infty$, be a proper homogeneous ideal. Let $S^{X_J}$ be the $X_J$-shift representation, and define $\underline{T} = (T_i)_{i \in \cI}$ by $T_i = S^{X_J}(e_i)$, $i \in \cI$. If $p \in \mb{C} \langle \underline{x} \rangle$ is a homogeneous polynomial, then $p(\underline{T}) = 0$ if and only if $p \in J$.
\end{lemma}
\begin{proof}
The ``if" part follows from Theorem \ref{thm:repZ(I)}. For the ``only if" part, let $p \notin J$ be a homogeneous polynomial of degree $n$. Applying $p(\underline{T})$ to the vacuum vector $\Omega$, we have
\bes
p(\underline{T}) \Omega = P p(e),
\ees
where $P$ is the orthogonal projection of $E^{\otimes n}$ onto $X_J(n)$. But as $p \notin J$, $p(e)$ is not in $E^{\otimes n} \ominus X_J(n) = \ker P$, thus $P p(e) \neq 0$. In particular, $p(\underline{T}) \neq 0$.
\end{proof}

We have the following noncommutative projective Nullstellansatz.
\begin{theorem}
Let $H$ be a fixed infinite dimensional separable Hilbert space.
Let $J$ be a homogeneous ideal in $\mb{C}\langle (x_i)_{i \in \cI}\rangle$, with $|\cI| < \infty$. Then 
\bes
I(Z(J)) = J.
\ees
In particular,
$Z(J) = \{\underline{0} = (0,0,\ldots)\}$ if and only if $J$ is the ideal generated by all the $x_i, i \in \cI$.
\end{theorem}
\begin{proof}
%
$I(Z(J)) \supseteq J$ is immediate. To see the converse, first note that equality is obvious when $J = \mb{C}\langle \underline{x} \rangle$, so we may assume that $J$ is proper. Also note that since $J$ is homogeneous $Z(J)$ is scale invariant. From this it follows that $I(Z(J))$ is also a homogeneous ideal. Indeed, if $h,g \in H$, and $p(\underline{x}) = \sum_\alpha c_\alpha \underline{x}^\alpha \in I(Z(J))$, then for all $\lambda \in \mb{C}$ one has for every tuple $\underline{T} = (T_i)_{i \in \cI} \in Z(I)$,
\bes
0 = \langle p(\lambda\underline{T})h,g \rangle = \sum_k \left(\sum_{|\alpha| = k} c_\alpha \langle \underline{T}^\alpha h,g\rangle\right) \lambda^{k},
\ees
and since a nonzero univariate polynomial has only finitely many zeros, it follows the homogeneous components of $p$ are all in $I(Z(J))$.

Assume now that $p$ is a homogeneous polynomial not in $J$. Let $S^{X_J}$ be the $X_J$-shift representation, and define $\underline{T} = (T_i)_{i \in \cI}$ by $T_i = S^{X_J}(e_i)$, $i \in \cI$. It is clear that $B(H)^\cI$ contains some unitarily equivalent copy of $\underline{T}$, which we also denote by $\underline{T}$. By Theorem \ref{thm:repZ(I)}, $\underline{T} \in Z(J)$. But by Lemma \ref{lem:shiftpoly}, $p(\underline{T}) \neq 0$, so $p \notin I(Z(J))$. This completes the proof.
\end{proof}

\chapter{Universality of the shift: universal algebras and models}\label{chap:universal}
In \cite{Arv98}, Arveson established a model for commuting, row-contractive tuples.
Using an idea from that paper that appeared also in \cite{BBD03} and \cite{Dey07} -- an idea that rests upon Popescu's ``Poisson Transform" introduced in \cite{Popescu99} (and pushed forward in \cite{MS08} and \cite{Popescu06}) -- we construct below a model for representations of subproduct systems. Roughly speaking, we will show that every representation of a subproduct system $X$ is a piece of a scaled inflation of the shift. Our model should be compared with a similar model obtained by Popescu in \cite{Popescu06}.
We will also see below that the operator algebra generated by the shift $S^X$ is the universal operator algebra generated by a representation of $X$.

We continue to use the notation set in the previous chapter.
Let $X$ be a standard subproduct system of Hilbert spaces over $\mb{N}$, to be fixed throughout this section. Let $p_n:E^{\otimes n} \rightarrow X(n)$ be the projections. Denote $E = X(1)$. Let $\{e_i\}_{i \in \cI}$ be an orthonormal basis for $E$, fixed once and for all.

We denote the standard $X$-shift tuple by $\underline{S^X} = (S^X_i)_{i \in \cI}$ , and we denote the standard $X$-shift representation of $X$ on $\mathfrak{F}_X$ by $S^X$. We consider $\mathfrak{F}_X$ to be a subspace of the full Fock space $\mathfrak{F}$, we denote the full shift by
$\underline{S} = (S_i)_{i\in \cI}$, and we denote the full shift representation of $F$ on $\mathfrak{F} := \mathfrak{F}_F$ by $S$.

Given a representation $T: X \rightarrow B(H)$, we will write $\underline{T} = (T_i)_{i \in \cI}$ for the tuple $(T(e_i))_{i\in \cI}$.

We denote by $\cA_X$ the unital algebra
$$\cA_X := \overline{\textrm{span}} \{\underline{S^X}^\alpha : \alpha \in  \mb{F}_\cI^+\} .$$

We denote by $\cE_X$ the operator system
$$\cE_X := \overline{\textrm{span}}\cA_X \cA_X^*,$$
and by $\cT_X = C^*(\underline{S^X})$ the C$^*$-algebra generated by $S^X_i$, $i\in \cI$ and $I_{\mathfrak{F}_X}$.
We denote by $\cK(\mathfrak{F}_X)$ the algebra of compact operators on $\mathfrak{F}_X$

If $T$ and $U$ are two representations of $X$ on Hilbert spaces $H$ and $K$, respectively, then we define
\bes
T \oplus U
\ees
to be the representation of $X$ on $H \oplus K$ given by $(T \oplus U)(x) = T(x) \oplus U(x)$.
We also define
\bes
T \otimes I_K
\ees
to be the representation of $X$ on $H \otimes K$ given by $(T \otimes I_K)(x) = T(x) \otimes I_K$.

\section{Popescu's ``Poisson Transform"}

After obtaining the results of this section, we discovered that they were obtained earlier by Gelu Popescu \cite{Popescu06}. We are presenting them here since they are important for the sequel.

\begin{proposition}\label{prop:KinE}
$\cK(\mathfrak{F}_X) \subseteq \cE_X$.
\end{proposition}
\begin{proof}
By the definition of representation, we have that $S(e_\alpha) = \underline{S}^\alpha$. By Definition \ref{def:repdilation} and the remarks following it, we have that $\underline{S}^{\alpha *}\big|_{\mathfrak{F}_X} = \left(\underline{S^X}\right)^{\alpha *}$ for all $\alpha$. Let $x \in X(n)$. Let $|\alpha|=k$. If $n<k$ then $\left(\underline{S^X}\right)^{\alpha *} x = 0$. If $n \geq k$, then since $X(n) \subseteq X(k) \otimes X(n-k)$, we may write $x = \sum_i x_k^i \otimes x_m^i$, where $x_k^i \in  X(k)$, $x_m^i \in X(m)$, and $m = n-k$. We have then
\be\label{eq:SXstar}
\underline{S^X}^{\alpha *} x = \underline{S}^{\alpha *} \sum_i x_k^i \otimes x_m^i = \sum_i \langle e_\alpha , x_k^i \rangle x_m^i \in X(m).
\ee
We then have for $x \in X(n)$:
\bes
\left(I - \sum_{|\alpha|=k}\underline{S^X}^{\alpha}\underline{S^X}^{\alpha *} \right)x =
\begin{cases}
x , & n < k \cr
x - \sum_{|\alpha|=k}\underline{S^X}^{\alpha}\sum_i \langle e_\alpha , x_k^i \rangle x_m^i  , & n \geq k
\end{cases} .
\ees
But
\begin{align*}
\sum_{|\alpha|=k}\underline{S^X}^{\alpha}\sum_i \langle e_\alpha , x_k^i \rangle x_m^i
&= \sum_{|\alpha|=k} p_n\left(\sum_i \langle e_\alpha , x_k^i \rangle e_\alpha \otimes x_m^i \right) \\
&= p_n\left(\sum_i \sum_{|\alpha|=k} \langle e_\alpha , x_k^i \rangle e_\alpha \otimes x_m^i \right) \\
&= p_n\left(\sum_i x_k^i \otimes x_m^i \right) \\
&= x.
\end{align*}
We thus conclude that $I - \sum_{|\alpha|=k}\underline{S^X}^{\alpha}\underline{S^X}^{\alpha *} = P_{W}$, where $W = \mb{C} \oplus X(1) \oplus \cdots \oplus X(k-1)$. In particular,
\be\label{eq:PC}
I - \sum_{i\in \cI}S^X_{i} \left(S^X\right)^{*}_i = P_{\mb{C}}.
\ee
Equations (\ref{eq:SXstar}) and (\ref{eq:PC}) give
\bes
(\underline{S^X})^\beta \left(I - \sum_{i\in \cI}S^X_{i} \left(S^X\right)^{*}_i \right)\underline{S^X}^{\alpha *} x
= p_{|\beta|} \langle e_\alpha, x  \rangle e_\beta.
\ees
As the elements $p_{|\beta|}e_\beta$ span $\mathfrak{F}_X$, it follows that $\cK(\mathfrak{F}_X) \subseteq \cE_X$.
\end{proof}

Given a representation $T$ of $X$ on a Hilbert space $H$ and given an integer $m \in \mb{N}$, we denote by $m \cdot T$ the representation
\bes
m \cdot T : X \rightarrow B(\underbrace{H \oplus H \oplus \cdots \oplus H}_{m \textrm{ times}})
\ees
given by $m \cdot T(x) = \underbrace{T(x) \oplus T(x) \oplus \cdots \oplus T(x)}_{m \textrm{ times}}$.
$\underline{T}$ is a row contraction (i.e., $\sum_{i \in \cI} T_i T_i^* \leq I_H$) if and only if $T$ is completely contractive.
When $\underline{T}$ is a row contraction the \emph{defect operator} $\Delta (\underline{T})$ is defined as
\bes
\Delta(\underline{T}) = I - \sum_{i \in \cI} T_i T_i^*,
\ees
and the \emph{Poisson Kernel} \cite{Popescu99} associated with $\underline{T}$ is the family of isometries $\{K_r\left(\underline{T}\right)\}_{0\leq r < 1}$
\bes
K_r\left(\underline{T}\right):H \rightarrow \mathfrak{F} \otimes H,
\ees
given by
\bes
K_r\left(\underline{T}\right) h = \sum_{\alpha \in \mb{F}_\cI^+} e_\alpha \otimes \left(r^{|\alpha|}\Delta(r\underline{T})^{1/2} \underline{T}^{\alpha *} h \right).
\ees
(See the beginning of \cite[Section 8]{Popescu99} for the remark that $\underline{T}$ has ``property (P)", and \cite[Lemma 3.2]{Popescu99} for the fact that these are isometries). When it makes sense, we also define $K_1\left(\underline{T}\right)$ by the same formula with $r=1$. The \emph{Poisson transform} is then defined as a map
\bes
\Phi = \Phi_{\underline{T}}: C^*(\underline{S}) \rightarrow B(H)
\ees
\bes
\Phi(a) = \Phi_{\underline{T}}(a) = \lim_{r \nearrow 1}K_r\left(\underline{T}\right)^* (a \otimes I) K_r\left(\underline{T}\right).
\ees
By \cite[Theorem 3.8]{Popescu99}, $\Phi$ is a unital, completely positive, completely contractive, satisfies
\bes
\Phi(\underline{S}^\alpha \underline{S}^{\beta *}) = \underline{T}^\alpha \underline{T}^{\beta *},
\ees
and is multiplicative on $Alg(\underline{S},I_{\mathfrak{F}})$, the algebra generated by $\underline{S}$ and $I_{\mathfrak{F}}$ ($\Phi$ is in fact an $Alg(\underline{S},I_{\mathfrak{F}})$-morphism).

\begin{theorem}\label{thm:CP1}
Let $T$ be a c.c. representation of $X$ on $H$. There exists a unital, completely positive, completely contractive map
\bes
\Psi: \cE_X \rightarrow B(H)
\ees
that satisfies
\bes
\Psi\left((\underline{S^X})^\alpha (\underline{S^X})^{\beta *}\right) = \underline{T}^\alpha \underline{T}^{\beta *} \,\, , \,\, \alpha, \beta \in \mb{F}_\cI^+
\ees
and
\be\label{eq:Amorphism}
\Psi(ab) = \Psi(a) \Psi(b) \,\, , \,\, a \in \cA_X, b \in \cE_X.
\ee
\end{theorem}
\begin{proof}
By the lemma below, the range of $K_r\left(\underline{T}\right)$ is contained in $\mathfrak{F}_X \otimes H$ for all $0 \leq r < 1$, thus
\bes
(P_{\mathfrak{F}_X} \otimes I_H) K_r\left(\underline{T}\right) = K_r\left(\underline{T}\right).
\ees
We may then define
\begin{align*}
\Psi(\underline{T})(\left((\underline{S^X})^\alpha (\underline{S^X})^{\beta *}\right))
&= \lim_{r \nearrow 1}K_r\left(\underline{T}\right)^* \left(\left((\underline{S^X})^\alpha (\underline{S^X})^{\beta *}\right) \otimes I\right) K_r\left(\underline{T}\right) \\
(*)&= \lim_{r \nearrow 1}K_r\left(\underline{T}\right)^* \left(\left(\underline{S}^\alpha \underline{S}^{\beta *}\right) \otimes I\right) K_r\left(\underline{T}\right) \\
&= \underline{T}^\alpha \underline{T}^{\beta *} ,
\end{align*}
where in (*) we have made use of the coinvariance of $\mathfrak{F}_X$ under $\underline{S}$.
This obviously extends to the desired map on $\cE_X$.
\end{proof}

\begin{lemma}
$K_r\left(\underline{T}\right)H \subseteq \mathfrak{F}_X \otimes H$.
\end{lemma}

\begin{proof}
Let $h \in H$. It suffices to show that for all $n \in \mb{N}$, the element
\bes
\sum_{|\alpha| = n} e_\alpha \otimes \left(r^n\Delta(r\underline{T})^{1/2} \underline{T}^{\alpha *} h \right) = (I \otimes r^n\Delta(r\underline{T})^{1/2}) \sum_{|\alpha| = n} e_\alpha \otimes (\underline{T}^{\alpha *} h)
\ees
is in $X(n) \otimes H$. However, $X(n) \otimes H$ (considered as a subspace of $E^{\otimes n} \otimes H$) is reduced by $(I \otimes r^n\Delta(r\underline{T})^{1/2})$, so it is enough to show that
\bes
\xi := \sum_{|\alpha| = n} e_\alpha \otimes (\underline{T}^{\alpha *} h) \in X(n) \otimes H.
\ees
Let $x \in E^{\otimes n} \ominus X(n)$ and $g \in H$. The proof will be completed by showing that
$\langle \xi, x \otimes g \rangle = 0$.
\begin{align*}
\langle \xi, x \otimes g \rangle
&= \sum_{|\alpha|=n} \langle e_\alpha \otimes T(e_\alpha)^* h, x \otimes g \rangle \\
&= \sum_{|\alpha|=n} \langle e_\alpha, x \rangle \langle h, T(e_\alpha)g \rangle \\
&= \left\langle h, T\left(\sum_{|\alpha|=n} \langle e_\alpha, x \rangle e_\alpha \right) g \right\rangle \\
&= \langle h, \widetilde{T}^n(x \otimes g) \rangle ,
\end{align*}
and by Proposition \ref{prop:representation}, the last expression in this chain of equalities is $0$.
\end{proof}

\section{The universal algebra generated by a tuple subject to homogeneous polynomial identities}

\begin{theorem}\label{thm:universal}
$J \triangleleft \mb{C}\langle (x_i)_{i \in \cI}\rangle$, be a homogeneous ideal.
Then $\cA_{X_J}$ is the universal unital operator algebra generated by a row contraction in $Z(J)$, that is: $\cA_{X_J}$ is a norm closed unital operator algebra generated by a tuple in $Z(J)$, (namely, $(S^{X_J}_i)_{i\in \cI}$), and
if $\mathcal{B} \subseteq B(H)$ is another norm closed unital operator algebra generated by a row contraction $(T_i)_{i\in\cI} \in Z(J)$, then there is a unique unital and completely contractive homomorphism $\varphi$ of $\cA_{X_J}$ onto $\mathcal{B}$, such that $\varphi(S^{X_J}_i) = T_i$ for all $i \in \cI$.
\end{theorem}
\begin{proof}
This follows from Theorems \ref{thm:repZ(I)} and \ref{thm:CP1}.
\end{proof}

\section{A model for representations: every completely bounded representation of $X$ is a piece of an inflation of $S^X$}

We will now construct a model for representations of subproduct systems. 
In \cite[Section 2]{Popescu06}, a similar but different model -- that includes also a fully coisometric part and not only the shift -- has been obtained.

\begin{theorem}\label{thm:model2}
Let $\underline{T}$ be a completely bounded representation of the subproduct system $X$ on a separable Hilbert space $H$, and let $K$ be an infinite dimensional, separable Hilbert space. Then for all $r > \|T\|_{cb}$, $T$ is unitarily equivalent to a piece of
\be\label{eq:shift_r}
S^X \otimes rI_K.
\ee
Moreover, $\|T\|_{cb}$ is equal the infimum of $r$ such that $T$ is a piece of an operator as in  (\ref{eq:shift_r}).
\end{theorem}
\begin{proof}
It is known that $\|T\|_{cb} = \|(T_i)_{i \in \cI}\|_{row}$, where $T_i = T(e_i)$. Thus if $r > r_0 = \|T\|_{cb}$, then
$\sum_{i \in \cI} T_i T_i^* \leq r_0^2 I < r^2 I$. Put $W_i = r^{-1}T_i$, so $\sum_{i\in\cI}W_i W_i^* \leq r_0^2 /r^2 I$. Then $K_1\left(\underline{W}\right)$ is an isometry (it is equal to $K_{r_0/r}(r/r_0 \underline{W} )$, and $r/r_0 \underline{W} $ is a row contraction). Thus we may define a map (as in the proof of Theorem \ref{thm:CP1})
\bes
\Psi: B(\mathfrak{F}_X) \rightarrow B(H)
\ees
by
\bes
\Psi(a) = K_1\left(\underline{W}\right)^* \left(a \otimes I\right) K_1\left(\underline{W}\right).
\ees
$\Psi$ is a normal completely positive unital map that satisfies
\bes
\Psi\left((\underline{S^X})^\alpha (\underline{S^X})^{\beta *}\right) = \underline{W}^\alpha \underline{W}^{\beta *} \,\, , \,\, \alpha, \beta \in \mb{F}_\cI^+ .
\ees
Since $\Psi$ is normal it has a \emph{normal} minimal Stinespring dilation $\Psi(a) = V^* \pi(a) V$, with $\pi : B(\mathfrak{F}_X) \rightarrow B(L)$ a normal $*$-homomorphism and $V : H \rightarrow L$ an isometry. It is well known that $\pi$ is equivalent to a multiple of the identity representation.
Thus we obtain, up to unitary equivalence and after identifying $H$ with $VH$, that
$r^{-1} T_i = P_H \pi(S^X_i) P_H = P_H (S^X_i \otimes I_G) P_H$, for some Hilbert space $G$.
To see that $T$ is a piece of $S^X \otimes I_G$ we need to show that
$(S^X_i \otimes I_G)^* \big|_H = T^*_i$ for all $i \in \cI$.
In other words, we need to show that $P_H \pi(S^X_i) = P_H \pi(S^X_i) P_H$.
But, for all $b \in \cE_X$,
\begin{align*}
P_H \pi(S^X_i) \pi(b) P_H &= P_H \pi(S^X_i b) P_H \\
&= \Psi(S^X_i b) \\
(*)&= \Psi(S^X_i) \Psi (b) \\
&= P_H \pi(S^X_i) P_H \pi(b) P_H,
\end{align*}
where (*) follows from (\ref{eq:Amorphism}). By Proposition \ref{prop:KinE}, the strong operator closure of $\cE_X$ is $B(\mathfrak{F}_X)$. $P_H \pi(S^X_i) = P_H \pi(S^X_i) P_H$ now follows from the minimality and normality of the dilation.

It is clear that $r^{-1}T$ is a also piece of $S^X \otimes I_K$ for every $K$ with $\dim K \geq \dim G$, so we may choose $K$ to be infinite dimensional.

We want to show that necessarily $\dim K \geq \dim H$. Since $S^X \otimes I_K$ is a dilation of $r^{-1}T$, $I_L - \sum_{i \in \cI} S^X_i (S^X_i)^* \otimes I_K$ is a dilation of $I_H - \sum_{i \in \cI} r^{-2}T_i T_i^*$. But the latter operator is invertible so it has rank $\dim H$. Thus the rank of $P_\mb{C} \otimes I_K = I_L - \sum_{i \in \cI} S^X_i(S^X_i)^* \otimes I_K$, which is $\dim K$, must be greater.

Now the final assertion is clear.
\end{proof}

We can now obtain a general von Neumann inequality.
\begin{theorem}\label{thm:vNinequality}
Let $X$ be a subproduct system, and let $T$ be a c.c. representation of $X$ on a Hilbert space $H$. Let $\{e_1, \ldots, e_d\}$ be an orthonormal set in $X(1)$, and define $T_i = T(e_i)$ and $S^X_i = S^X(e_i)$ for $i=1,\ldots,d$. Then for every polynomials $p$ and $q$ in $d$ non commuting variables,
\bes
\|p(T_1, \ldots, T_d) q(T_1, \ldots, T_d)^*\| \leq \|p(S^X_1, \ldots, S^X_d) q(S^X_1, \ldots, S^X_d)^*\|.
\ees
\end{theorem}
\begin{proof}
Since $T$ is a piece of $S^X \otimes r I_K$ for all $r>1$, we have
\bes
p(T_1, \ldots, T_d) q(T_1, \ldots, T_d)^*  = P \Big( p(rS_1, \ldots, rS_d) q(rS_1, \ldots, rS_d)^* \otimes I_K \Big) P
\ees
for some projection $P$, and the result follows by taking $r\searrow 1$.
\end{proof}

\chapter{The operator algebra associated to a subproduct system}\label{chap:operator_algebra}

Let $X$ be a subproduct system. Recall the definitions of $\cA_X$ and $\cE_X$  from the beginning of Chapter \ref{chap:universal}. If $\{e_i\}_{i \in \cI}$ is an orthonormal basis for $X(1)$, then $\cA_X$ is the unital operator algebra generated by $(S^X_i)_{i \in \cI}$ with $S^X_i = S^X(e_i)$. If $\{f_i\}_{i \in \cI}$ is another orthonormal basis then the tuple $(S^X(f_i))_{i \in \cI}$ is not necessarily unitarily equivalent to $(S^X_i)_{i \in \cI}$. For instance (with the above notation), if $X$ and $\{e_1, e_2\}$ are as in Example \ref{expl:notmax}, and
\bes
f_1 = \frac{1}{\sqrt{2}}(e_1 + e_2) \quad, \quad f_2 = \frac{1}{\sqrt{2}}(e_1 - e_2),
\ees
then $S^X_1, S^X_2$ are partial isometries, whereas $T_1 = S^X(f_1)$ and $T_2 = S^X(f_2)$ are not.
Thus, the unitary equivalence of the row $(S^X_i)$ does not determine the isomorphism class of the subproduct system $X$.
\begin{proposition}\label{prop:UE_ISO}
Let $X$ and $Y$ be two subproduct systems with $X(1) = E$ and $Y(1) = F$. Assume that $\{e_i\}_{i \in \cI}$ is an orthonormal basis for $E$ and that $\{f_i\}_{i\in\cI}$ is an orthonormal basis for $F$. Then the shifts $(S^X_i)_{i \in \cI}$ and $(S^Y_i)_{i \in \cI}$ are unitarily equivalent as rows (i.e., there exists a unitary $V: \mathfrak{F}_X \rightarrow \mathfrak{F}_Y$ such that $V S^X_i = S^Y_i V$ for all $i \in \cI$), if and only if there is an isomorphism of subproduct systems $W: X \rightarrow Y$ such that $We_i = f_i$ for all $i \in \cI$.
\end{proposition}
\begin{proof}
If $X$ and $Y$ are isomorphic with the isomorphism $W$ sending $e_i$ to $f_i$, then define a unitary $V: \mathfrak{F}_X \rightarrow \mathfrak{F}_Y$ by
\bes
V = \bigoplus_{n \in \mb{N}} W\big|_{X(n)}.
\ees
$V S^X_i = S^Y_i V$ follows from the properties of $W$. Conversely, a unitary $V$ intertwining $\underline{S}^X$ and $\underline{S}^Y$ must send $\Omega_X$ to $\Omega_Y$. Indeed, such a unitary must send $\{\Omega_X\}^\perp$ (which is equal to $\vee_i \textrm{Im}S^X_i$) onto a subspace of $\{\Omega_Y\}^\perp$ that has codimension $1$ in $\mathfrak{F}_Y$, thus it must send $\{\Omega_X\}^\perp$ onto $\{\Omega_Y\}^\perp$. It follows that $V \Omega_X = \Omega_Y$. Thus, given a unitary $V$ intertwining $\underline{S}^X$ and $\underline{S}^Y$, we may define $W\big|{X(n)} : X(n) \rightarrow Y(n)$ by
\bes
W S^X_\alpha \Omega = V S^X_\alpha \Omega = S^Y_\alpha \Omega ,
\ees
for all $|\alpha| = n$, and it is easy to see that the maps $W\big|_{X(n)}$ assemble to form an isomorphism of subproduct systems.
\end{proof}

In the example preceding the proposition, we saw how the shift ``tuple" $(S^X_1,S^X_2)$ depends essentially on the choice of basis in $E$. However, the closed unital algebra generated by $(S^X_1, S^X_2)$ is isomorphic to the one generated by $(T_1, T_2)$. Similar remarks hold for $\cE_X$ and $\cT_X$.


\begin{example}
\emph{Let $X$ be the subproduct system given by $X(0) = \mb{C}, X(1) = \mb{C}^2$ and $X(n) = 0$ for all $n \geq 2$. Let $Y$ be the subproduct system given by $Y(0) = Y(1) = Y(2) = \mb{C}$ and $Y(n) = 0$ for all $n \geq 3$. Then since $\cE_X$ and $\cE_Y$ contain the compact operators on $\cF_X$ and $\cF_Y$ (the Fock spaces), we have $\cE_X = \cT_X \cong M_3(\mb{C}) \cong \cT_Y = \cE_Y$.}

\emph{On the other hand, let $\{e_1, e_2\}$ be an orthonormal basis for $X(1)$. Then if $\Omega$ is the vacuum vector, then $\cA_X$ is generated by $S^X(\Omega) = I,S^X(e_1),S^X(e_2)$. In the base $\{\Omega,e_1,e_2\}$ for $\cF_X$, these operators have the form}
\bes
\begin{pmatrix}
  1 & 0 & 0  \\
  0 & 1 & 0  \\
  0 & 0 & 1  \\
\end{pmatrix},
\begin{pmatrix}
  0 & 0 & 0  \\
  1 & 0 & 0  \\
  0 & 0 & 0  \\
\end{pmatrix},
\begin{pmatrix}
  0 & 0 & 0  \\
  0 & 0 & 0  \\
  1 & 0 & 0  \\
\end{pmatrix}.
\ees
\emph{Thus,}
\bes \cA_X \cong \left\{
\begin{pmatrix}
  a & 0 & 0  \\
  b & a & 0  \\
  c & 0 & a  \\
\end{pmatrix} \Big| a,b,c \in \mb{C} \right\}.
\ees
\emph{On the other hand, $\cA_Y$ is generated by}
\bes
I =
\begin{pmatrix}
  1 & 0 & 0  \\
  0 & 1 & 0  \\
  0 & 0 & 1  \\
\end{pmatrix},
S^Y(f_1) = \begin{pmatrix}
  0 & 0 & 0  \\
  1 & 0 & 0  \\
  0 & 1 & 0  \\
\end{pmatrix},
\big(S^Y(f_1)\big)^2 = \begin{pmatrix}
  0 & 0 & 0  \\
  0 & 0 & 0  \\
  1 & 0 & 0  \\
\end{pmatrix},
\ees
\emph{where $\{f_1\}$ is an orthonormal basis for $Y(1)$. Thus}

\bes \cA_Y \cong \left\{
\begin{pmatrix}
  a & 0 & 0  \\
  b & a & 0  \\
  c & b & a  \\
\end{pmatrix} \Big| a,b,c \in \mb{C} \right\}.
\ees
\emph{So $\cA_X \ncong \cA_Y$ (in $\cA_X$ the solutions of $T^2 = 0$ form a two dimensional subspace, and in $\cA_Y$ they form
a one dimensional subspace). }
\end{example}

For every subproduct system $X$ there exists a unique completely contractive multiplicative linear functional
$\rho_0 : \cE_X \rightarrow \mb{C}$ that sends $\lambda I$ to $\lambda$ and $S^X_\alpha$ to $0$ when $|\alpha| > 0$.
The existence of $\rho_0$ follows from Theorem \ref{thm:CP1} (using the Poisson Transform),
but it is also clear that $\rho_0$ is just the vector state associated with the vacuum vector $\Omega_X$:
\bes
\rho_0(T) = \langle T \Omega_X, \Omega_X \rangle \,\, , \,\, T \in \cA_X.
\ees
$\rho_0$ can be considered also as a conditional expectation $\rho_0 : \cA_X \rightarrow \mb{C} \cdot \Omega_X$, inducing a direct sum
\be\label{eq:direct_sum}
\cA_X = \rho_0 \cA_X \oplus \ker \rho_0 = \mb{C}\cdot I \oplus \sum_i S_i^X \cA_X .
\ee
$\cA_X$ contains a dense graded subalgebra, with the homogeneous elements of degree $n$ being $S^X(\xi)$, where $\xi \in X(n)$.

\begin{lemma}\label{lem:iso_vacuum}
Let $\varphi : \cA_X \rightarrow \cA_Y$ be an isometric isomorphism. Then $\varphi$ is unital.
\end{lemma}
\begin{proof}
A theorem of Arazy and Solel \cite{ArazySolel} implies that an isometric map between $\cA_X$ and $\cA_Y$ must send $I \in \cA_X$ to an isometry in $\cA_X \cap \cA_X^*$. It follows that $\varphi(I) = cI$, $|c| = 1$. But since $\varphi$ is a homomorphism, then $c = 1$.
\end{proof}

\begin{lemma}\label{lem:vacuum_norm}
For all $n \in \mb{N}$, $\xi \in X(n)$
\bes
\|S^X(\xi)\| = \|S^X(\xi) \Omega_X\| = \|\xi\|.
\ees
\end{lemma}
\begin{proof}
Because $S^X(\xi)$ maps the orthogonal summands $X(k)$ of $\mathfrak{F}_X$ into the orthogonal summands $X(k+n)$, it suffices to show that for all $\eta \in X(k)$, $\|S^X(\xi)\eta\| \leq \|\xi\|\|\eta\|$
(because $S^X(\xi) \Omega_X = \xi$). Now, $S^X(\xi) \eta = p^X_{n+k}(\xi \otimes \eta)$, thus
\bes
\|S^X(\xi) \eta\|^2 \leq \|\xi \otimes \eta\|^2 = \|\xi\|^2 \| \eta\|^2.
\ees
\end{proof}

\begin{lemma}\label{lem:iso_grading}
Let $\varphi : \cA_X \rightarrow \cA_Y$ be an isometric isomorphism that preserves the direct sum decomposition (\ref{eq:direct_sum}). Then $\varphi$ preserves the grading: if $ \xi \in X(n)$ then $\varphi(S^X(\xi))$ is in the norm closure of $\textrm{span}\{S^Y(\eta) : \eta \in Y(n)\}$.
\end{lemma}
\begin{proof}
Since $\varphi$ is a homomorphism, it suffices to show, say, that $\varphi(S^X_1)$ has ``degree one", that is, it is in
the norm closure of $\textrm{span}\{S^Y(\eta) : \eta \in Y(1)\}$. By assumption, we may write $\varphi(S^X_1) = \sum_i a_i S^Y_i + T$, with $T$ in the closure of $\textrm{span}\{S^Y(\eta) : \eta \in Y(n), n \geq 2\}$. But $\varphi^{-1}(\sum_i a_i S^Y_i + T) = S^X_1$, and $\varphi^{-1}(T)$ is in the norm closure of $\textrm{span}\{S^X(\xi) : \eta \in X(n), n \geq 2\}$, so $\varphi^{-1}(\sum_i a_i S^Y_i) = S^X_1 + B$, with $B = -\varphi^{-1}(T)$ (note that $\varphi^{-1}$ also preserves the direct sum decomposition (\ref{eq:direct_sum})).

If $T = 0$ then we are done, so assume $T \neq 0$. Then $B \neq 0$, also. But
\bes
1 = \|S^X_1\| = \|S^X_1 \Omega_X\| < \|(S^X_1 + B)\Omega_X\| \leq \|S^X_1 + B\| = \|\sum_i a_i S^Y_i\|,
\ees
and at the same time
\bes
\|\sum_i a_i S^Y_i\| = \|\sum_i a_i S^Y_i \Omega_Y\| < \|(\sum_i a_i S^Y_i + T) \Omega_Y\| \leq \|\sum_i a_i S^Y_i + T\| = \|S^X_1\| = 1.
\ees
From $T \neq 0$ we arrived at $1<1$, thus $T=0$.
\end{proof}

\begin{theorem}\label{thm:algebra_iso}
$X \cong Y$ if and only if $\cA_X$ and $\cA_Y$ are isometrically isomorphic with an isomorphism that preserves the direct sum decomposition (\ref{eq:direct_sum}), and this happens if and only if $\cA_X$ and $\cA_Y$ are isometrically isomorphic with a grading preserving isomorphism. In fact, if $\varphi : \cA_X \rightarrow \cA_Y$ is a grading preserving isometric isomorphism then there is an isomorphism $V: X \rightarrow Y$ such that for all $T\in \cA_X$, $\varphi(T) = V T V^*$.
\end{theorem}

\begin{proof}
$X \cong Y$ implies $\cA_X \cong \cA_Y$ because these algebras are then generated by unitarily equivalent tuples.

For the converse, we will assume that X and Y are standard subproduct systems. The isomorphism $V : X \rightarrow Y$ is defined on the fiber $X(n)$ by
\bes
V (\xi) = V(S^X(\xi)\Omega_X) = \varphi(S^X(\xi))\Omega_Y \,\, , \,\, \xi \in X(n).
\ees
If it is well defined, then it is onto. Lemma \ref{lem:vacuum_norm} shows that $V$ is an isometry on the fibers:
\bes
\|S^X(\xi)\Omega_X\| = \|S^X(\xi)\| = \|\varphi(S^X(\xi))\| = \|\varphi(S^X(\xi))\Omega_Y\|.
\ees
Lemma \ref{lem:iso_grading} implies that $V(\xi)$ sits in $Y(n)$. $V$ respects the subproduct structure: if $m,n \in \mb{N}$, $\xi \in X(n), \eta \in X(m)$, then
\begin{align*}
V p^X_{m,n}(\xi \otimes \eta) &= V S^X(p^X_{m,n}(\xi \otimes \eta)) \Omega_X \\
&= \varphi(S^X(p^X_{m,n}(\xi \otimes \eta))) \Omega_Y \\
&= \varphi(S^X(\xi) S^X(\eta))\Omega_Y \\
&= \varphi(S^X(\xi)) \varphi(S^X(\eta))\Omega_Y \\
(*)&= p_{m,n}^Y \left( \varphi(S^X(\xi)) \Omega_Y \otimes \varphi(S^X(\eta))\Omega_Y \right) \\
&= p_{m,n}^Y(V(\xi) \otimes V(\eta)).
\end{align*}
(*) follows from the facts $S^Y(y) \Omega_Y = y$ and $S^Y(y_1) S^Y(y_2) \Omega_Y = S^Y(p_{m,n}^Y (y_1 \otimes y_2)) \Omega_Y = p_{m,n}^Y(y_1 \otimes y_2) = p_{m,n}^Y(S^Y(y_1) \Omega_Y \otimes S^Y(y_2) \Omega_Y)$.

Finally, let us show that for all $T\in \cA_X$, $\varphi(T) = V T V^*$.
What we mean by this is that for all $\xi \in X$, $\varphi(S^X(\xi)) = V S^X(\xi) V^*$. Let $\varphi(S^X(\eta))\Omega_Y = V(\eta)$ be a typical element in $\mathfrak{F}_Y$.
\begin{align*}
V S^X(\xi) V^* \varphi(S^X(\eta))\Omega_Y &= V S^X(\xi) \eta \\
&= V p^X(\xi \otimes \eta) \\
&= \varphi(S^X(p^X(\xi \otimes \eta))) \Omega_Y \\
&= \varphi(S^X(\xi) S^X(\eta)) \Omega_Y \\
&= \varphi(S^X(\xi)) \varphi(S^X(\eta)) \Omega_Y ,
\end{align*}
This completes the proof.
\end{proof}

\chapter{Classification of the universal algebras of $q$-commuting tuples}\label{chap:qcommuting}

\begin{definition}
A matrix $q$ is called \emph{admissible} if $q_{ii}=0$ and $0 \neq q_{ij}=q_{ji}^{-1}$ for all $i \neq j$.
\end{definition}

\section{The $q$-commuting algebras $\cA_q$ and their universality}
Let $\{e_1, \ldots,  e_d\}$ be an orthonormal basis for $E := \mb{C}^d$, to be fixed (together with $d$) throughout this section. Let $q \in M_d(\mb{C})$ be an admissible matrix, and let $X_q$ be the maximal standard subproduct system with fibers
\bes
X_q(1) = E \,\, , \,\, X_q(2) = E \otimes E \ominus \textrm{span}\{e_i \otimes e_j - q_{ij} e_j \otimes e_i : 1\leq i, j \leq d, i \neq j \}.
\ees
When $q_{ij} = 1$ for all $i<j$, then $X_{q}$ is the symmetric subproduct system $SSP$. The Fock spaces $\mathfrak{F}_{X_q}$ have been studied in \cite{Dey07}.

For brevity, we shall write $S^q_i$ instead of $S^{X_q}_i$.
We denote by $\cA_q$ the algebra $\cA_{X_q}$. By Theorem \ref{thm:universal}, the algebra $\cA_q$ is the universal
norm closed unital operator algebra generated by a row contraction $(T_1, \ldots, T_d)$ satisfying the relations
\bes
T_i T_j = q_{ij}T_j T_i \,\, , \,\, 1 \leq i < j \leq d.
\ees


\section{The character space of $\cA_q$}\label{sec:charA_q}

Let $\cM_q$ be the space of all (contractive) multiplicative and unital linear functionals on $\cA_q$, endowed with the weak-$*$ topology. We shall call $\cM_q$ \emph{the character space of $\cA_q$}. Every $\rho \in \cM_q$ is uniquely determined by the $d$-tuple of complex numbers $(x_1, \ldots, x_d)$, where $x_i = \rho(S^q_i)$ for $i = 1, \ldots, d$. Since a contractive linear functional is completely contractive, $(x_1, \ldots, x_d)$ must be a row contraction, that is, $|x_1|^2 + \ldots + |x_d|^2 \leq 1$. In other words, $(x_1, \ldots, x_d)$ is in the unit ball $B_d$ of $\mb{C}^d$. The multiplicativity of $\rho$ implies that $(x_1, \ldots, x_d)$ must lie inside the set
\bes
Z_q := \{(z_1, \ldots, z_d)\in B_d: (1-q_{ij})z_iz_j = 0, 1 \leq i < j \leq d\} .
\ees

Conversely, Theorem \ref{thm:universal} implies that every $(x_1, \ldots, x_d) \in Z_q$ gives rise to a character $\rho \in \cM_q$ that sends $S^q_i$ to $x_i$. Thus the map
\bes
\cM_q \ni \rho \mapsto (\rho(S^q_1), \ldots, \rho(S^q_d)) \in Z_q
\ees
is injective and surjective. It is also obviously continuous (with respect to the weak-$*$ and standard topologies). Since $\cM_q$ is compact, we have the homeomorphism
\be
\cM_q \cong Z_q.
\ee

Note that the vacuum state $\rho_0$ corresponds to the point $0 \in Z_q \subset \mb{C}^d$.

When $q_{ij} = 1$, the condition $(1-q_{ij})z_iz_j = 0$ is trivially satisfied, so when $q_{i,j}=1$ for all $i,j$, then $Z_q$ is the unit ball $B_d$. When $q_{ij} \neq 1$, the condition is
that either $z_i = 0$ or $z_j = 0$. Thus, if for all $i,j$, $q_{ij}\neq 1$, then $Z_q$ is the union of $d$ discs glued together at their origins.

\section{Classification of the $\cA_q$, $q_{ij}\neq 1$}

Given a permutation $\sigma$ (on a set with $d$ elements), let $U_\sigma$ be the matrix that induces the same permutation on the standard basis of $\mb{C}^d$.

\begin{proposition}\label{prop:similarity}
Let $q$ and $r$ be two admissible $d \times d$ matrices. Assume that there is a permutation $\sigma \in S_d$ such that $r = U_\sigma q U_\sigma^{-1}$, and let $\lambda_1, \ldots, \lambda_d$ be any complex numbers on the unit circle. Then the map
\be\label{eq:sigma}
e_i \mapsto \lambda_i e_{\sigma(i)}
\ee
extends to an isomorphism of $X_q$ onto $X_r$, and thus the map
\bes
S_i^q \mapsto \lambda_i S_{\sigma(i)}^r
\ees
extends to a completely isometric isomorphism between $\cA_q$ and $\cA_r$.
\end{proposition}
\begin{proof}
For all $n$, the map (\ref{eq:sigma}) extends to a unitary $V_n$ of $E^{\otimes n}$. For $n=2$, this unitary sends $e_i \otimes e_j - q_{ij}e_j \otimes e_i$ to $\lambda_i \lambda_j e_{\sigma(i)} \otimes e_{\sigma(j)} - \lambda_i \lambda_j q_{ij}e_{\sigma(j)} \otimes e_{\sigma(i)}$. But $r = U_\sigma q U_\sigma^{-1}$ implies $r_{\sigma(i)\sigma(j)} = q_{ij}$, thus
\bes
V_2: e_i \otimes e_j - q_{ij}e_j \otimes e_i \mapsto \lambda_i \lambda_j e_{\sigma(i)} \otimes e_{\sigma(j)} - \lambda_i \lambda_j r_{\sigma(i)\sigma(j)}e_{\sigma(j)} \otimes e_{\sigma(i)},
\ees
so $V_2$ is a unitary between $X_q(2)$ and $X_r(2)$ that respects the product. By induction, it follows that
$V = \{V_n\big|_{X_q(n)}\}_n$ is an isomorphism of subproduct systems.
The final assertion follows from Proposition \ref{prop:UE_ISO}.
\end{proof}

\begin{theorem}\label{thm:subproduct_iso_q}
Let $q$ and $r$ be two admissible $d \times d$ matrices such that $q_{ij},r_{ij}\neq 1$ for all $i,j$. Then $X_q$ is isomorphic to $X_r$ if and only if there is a permutation $\sigma \in S_d$ such that $r = U_\sigma q U_\sigma^{-1}$. In this case the isomorphisms are precisely those of the form
\bes
e_i \mapsto \lambda_i e_{\sigma(i)},
\ees
where $\lambda_1, \ldots, \lambda_d$ are any complex numbers on the unit circle, and $\sigma$ is such that $r = U_\sigma q U_\sigma^{-1}$.
\end{theorem}
\begin{proof}
One direction is Proposition \ref{prop:similarity}, so assume that there is an isomorphism of subproduct systems $V: X_q \rightarrow X_r$. Let $f_i := V^{-1}e_i$. There is a $d \times d$ unitary matrix $U = (u_{ij})$ such that $f_i = \sum_j u_{ij}e_j$. As $V$ is an isomorphism of subproduct systems, we have for all $i \neq j$
\bes
Vp_2^{X_q}(f_i \otimes f_j - r_{ij}f_j \otimes f_i) = p_2^{X_r}(e_i \otimes e_j - r_{ij}e_j \otimes e_i) = 0,
\ees
thus
\bes
(\sum_k u_{ik} e_k)\otimes(\sum_l u_{jl} e_l) - r_{ij} (\sum_k u_{jk} e_k)\otimes(\sum_l u_{il} e_l) \in \textrm{span} \{e_m \otimes e_n - q_{mn} e_n \otimes e_m : m \neq n\},
\ees
or
\be\label{eq:sum_in}
\sum_{k,l}(u_{ik}u_{jl} - r_{ij}u_{jk}u_{il})e_k \otimes e_l \in \textrm{span} \{e_m \otimes e_n - q_{mn} e_n \otimes e_m : m \neq n\}.
\ee
The coefficients of the vectors $e_k \otimes e_k$ in the sum above must vanish, thus $u_{ik}u_{jk} - r_{ij}u_{jk}u_{ik} = 0$ for all $i \neq j$. Since $r_{ij}\neq 1$, we must have $u_{jk}u_{ik} = 0$ for all $k$ and all $i \neq j$. Thus the unitary matrix $U$ has precisely one nonzero element in each column, and it therefore must be of the form $U_\sigma^{-1} D$, where $D$ is a diagonal unitary matrix.

Equation (\ref{eq:sum_in}) becomes
\bes
u_{i\sigma(i)}u_{j\sigma(j)}e_{\sigma(i)} \otimes e_{\sigma(j)} - r_{ij}u_{j\sigma(j)}u_{i\sigma(i)}e_{\sigma(j)} \otimes e_{\sigma(i)} \in \textrm{span} \{e_m \otimes e_n - q_{mn} e_n \otimes e_m : m \neq n\},
\ees
but this can only happen if

\bes
u_{i\sigma(i)}u_{j\sigma(j)}e_{\sigma(i)} \otimes e_{\sigma(j)} - r_{ij}u_{j\sigma(j)}u_{i\sigma(i)}e_{\sigma(j)} \otimes e_{\sigma(i)}
\ees
is proportional to
\bes
e_{\sigma(i)} \otimes e_{\sigma(j)} - q_{\sigma(i)\sigma(j)}e_{\sigma(j)} \otimes e_{\sigma(i)} ,
\ees
that is $u_{i\sigma(i)}u_{j\sigma(j)}q_{\sigma(i)\sigma(j)} = u_{j\sigma(j)}u_{i\sigma(i)}r_{ij}$, or $r_{ij} = q_{\sigma(i)\sigma(j)}$. Replacing $\sigma$ with $\sigma^{-1}$, the proof is complete.
\end{proof}

\begin{corollary}\label{cor:auto}
Let $q$ be an admissible $d \times d$ matrix such that there is no permutation $\sigma \in S_d$ such that
$q = U_\sigma q U_\sigma^{-1}$. Assume that $q_{ij} \neq 1$ for all $i,j$. Then the only automorphisms of $X_q$ are unitary scalings of the basis $\{e_1, \ldots, e_d\}$.
\end{corollary}

\begin{theorem}\label{thm:algebra_iso_q}
Let $q$ and $r$ be two admissible $d \times d$ matrices such that $q_{ij},r_{ij}\neq 1$ for all $i,j$. Then $\cA_q$ is isometrically isomorphic to $\cA_r$ if and only if there is a permutation $\sigma \in S_d$ such that $r = U_\sigma q U_\sigma^{-1}$. In this case the isometric isomorphisms between $\cA_q$ and $\cA_r$ are precisely those of the form
\bes
S^q_i \mapsto \lambda_i S^r_{\sigma(i)},
\ees
where $\lambda_1, \ldots, \lambda_d$ are any complex numbers on the unit circle.
\end{theorem}
\begin{proof}
If $r = U_\sigma q U_\sigma^{-1}$, then by Proposition \ref{prop:similarity} and Theorem \ref{thm:algebra_iso} $\cA_q$ and $\cA_r$ are isomorphic (with an isomorphism that preserves the direct sum decomposition (\ref{eq:direct_sum})).

Conversely, assume that $\varphi: \cA_q \rightarrow \cA_r$ is a completely isometric isomorphism. Then $\varphi$ induces a homeomorphism between $\cM_r$ and $\cM_q$ by $\rho \mapsto \rho \circ \varphi$.
Recall that $\cM_q$ and $M_r$ are both homeomorphic to $d$ discs glued together at the origin. Thus the homeomorphism $\rho \mapsto \rho \circ \varphi$ must take $\rho_0$ of ${X_r}$ to $\rho_0$ of ${X_q}$, because these are the unique points in $\cM_r$ and $\cM_q$, respectively, that when removed from $M_r$ and $\cM_q$ leave $d$ disconnected punctured discs. Thus $\varphi$ sends the vacuum state of $\cA_r$ to the vacuum state of $\cA_q$, and must therefore preserve the direct sum decomposition (\ref{eq:direct_sum}). By Theorem \ref{thm:algebra_iso}, there is an isomorphism of subproduct systems $V: X_q \rightarrow X_r$ such that $\varphi(\bullet) = V \bullet V^*$. By Theorem \ref{thm:subproduct_iso_q} we conclude that there is a permutation $\sigma \in S_d$ such that $r = U_\sigma q U_\sigma^{-1}$. It also follows that
$\varphi(S_i^q) = \lambda_i S_{\sigma(i)}^r$.
\end{proof}

\begin{corollary}\label{cor:iso_auto}
Let $q$ be an admissible $d \times d$ matrix such that there is no permutation $\sigma \in S_d$ such that $q = U_\sigma q U_\sigma^{-1}$. Then the only isometric automorphisms of $\cA_q$ are unitary scalings of the shift $\{S^q_1, \ldots, S^q_d\}$.
\end{corollary}

As a corollary of the above discussion we have:
\begin{corollary}
Let $q$ and $r$ be two admissible $d \times d$ matrices such that $q_{ij},r_{ij}\neq 1$ for all $i,j$. Then $\cA_q$ is isometrically isomorphic to $\cA_r$ if and only if $X_q \cong X_r$.
\end{corollary}

\section{$X_q$ and $\cA_q$, $d = 2$}
In the particular case $d=2$, we let a complex number $q$ parameterize the spaces $X_q$ (we may allow also $q = 0$) defined to be the maximal standard subproduct system with fibers
\bes
X_q(1) = \mb{C}^2 \,\, , \,\, X_q(2) = \mb{C}^2 \otimes \mb{C}^2 \ominus \textrm{span}\{e_1 \otimes e_2 - q e_2 \otimes e_1 \}.
\ees
Since $\cM_1 \cong B_2$, $\cA_1$ is not isomorphic to any $\cA_q$ with $q \neq 1$ (recall that when $q \neq 1$, $\cM_q$ is homeomorphic to two discs glued together at the origin).
Thus Theorem \ref{thm:algebra_iso_q}
gives:
\begin{corollary}
Assume that $d=2$. Then $X_q \cong X_r$ if and only if $\cA_q$ is isometrically isomorphic to $\cA_r$, and either one of these happens if and only if either $r = q$ or $r = q^{-1}$.
\end{corollary}

Elias Katsoulis has pointed out to us that the above corollary also follows from the techniques of \cite{DavidsonKatsoulis}.

The above result is reminiscent to the fact that two rotation algebras $A_\theta$ and $A_{\theta'}$ are isomorphic if and only if either $e^{2\pi i \theta} = e^{2\pi i \theta'}$ or $(e^{2\pi i \theta})^{-1} = e^{2\pi i \theta'}$. One cannot help but wonder whether one can draw a deeper connection between these results then the superficial one, in particular, can the classification of  rotation algebras be deduced from the classification of the algebras $\cA_q$?

By Corollaries \ref{cor:auto} and \ref{cor:iso_auto} we have the following.
\begin{corollary}
Let $d=2$ and let $q \neq 1$. Then subproduct system $X_q$ has no automorphisms aside form the unitary scalings of the basis. The algebra $\cA_q$ has no isometric automorphisms other than unitary scalings of the generators.
\end{corollary}

On the other hand, a direct calculation shows that every unitary on $\mb{C}^2$ extends to an automorphism of $X_1$, and thus induces a non-obvious automorphism of $\cA_1$.

\chapter{Standard maximal subproduct systems with $\dim X(1) = 2$ and $\dim X(2) = 3$}\label{chap:A}

Again, let $\{e_1, \ldots,  e_d\}$ be an orthonormal basis for $E := \mb{C}^d$. We will soon turn attention to the case $d=2$. For a matrix $A \in M_d(\mb{C})$, we define the \emph{symmetric part} of $A$ to be $A^s := (A + A^t)/2$ and the \emph{antisymmetric part} of $A$ to be $A^a := (A - A^t)/2$. Denote by $X_A$ the maximal standard subproduct system with fibers
\bes
X_A(1) = E \,\, , \,\, X_A(2) = E \otimes E \ominus \textrm{span}\left\{\sum_{i,j=1}^d a_{ij}e_i \otimes e_j \right\}.
\ees
We will write $S^A$ for the shift $S^{X_A}$. We will also write $\cA_A$ for $\cA_{X_A}$.

\begin{proposition}\label{prop:classifyX_A}
Let $A, B \in M_d(\mb{C})$. Then there is an isomorphism $V: X_A \rightarrow X_B$ if and only if there exists $\lambda \in \mb{C}$ and a unitary $d \times d$ matrix $U$ such that $B = \lambda U^t A U$. In this case, $U$ extends to the isomorphism $V$ between $X_A$ and $X_B$ by $V_1 = U$.
\end{proposition}
\begin{proof}
Let $V: X_A \rightarrow X_B$ be an isomorphism of subproduct systems. There is a $d\times d$ unitary matrix $U = (u_{ij})$ such that
\bes
f_i := V_1 (e_i) = \sum_{j=1}^d u_{ij}e_j   .
\ees
Then
\begin{align*}
0 &= V_1 (p_2^X(\sum_{i,j}a_{ij}e_i \otimes e_j)) \\
&= p_2^Y (\sum_{i,j}a_{ij} f_i \otimes f_j),
\end{align*}
so $\sum_{i,j}a_{ij} f_i \otimes f_j$ must be a spanning vector of $\textrm{span}\left\{\sum_{i,j} b_{ij}e_i \otimes e_j \right\}$. Writing out fully what this means,
\bes
\lambda \sum_{i,j}a_{ij} \sum_{k,l} u_{ik}u_{jl}  e_k \otimes e_l =  \sum_{k,l} b_{kl}e_k \otimes e_l
\ees
for some $\lambda \in \mb{C}$, so
\bes
b_{kl} = \lambda \sum_{i,j}a_{ij}u_{ik}u_{jl} .
\ees
But the right hand side is precisely the $kl$-th element of $\lambda U^t A U$.

Conversely, assuming $B = \lambda U^t A U$, one can read the above argument from finish to start to obtain an isomorphism $V: X_A \rightarrow X_B$.
\end{proof}

We see that for $X_A$ and $X_B$ to be isomorphic the ranks of $A$ and $B$ must be the same, as well as the ranks of their symmetric and anti-symmetric parts. For example, if $A$ is symmetric and $B$ is not then $X_A \ncong X_B$, a result which may not seem obvious at first glance.

\begin{theorem}\label{thm:classifyA_A}
Assume that $d=2$. Let $A,B \in M_2(\mb{C})$ be any two matrices. Then $\cA_{A}$ is isometrically isomorphic to $\cA_B$ if and only if $X_A \cong X_B$, and this happens if and only if there exists $\lambda \in \mb{C}$ and a unitary $2 \times 2$ matrix $U$ such that $B = \lambda U^t A U$.
\end{theorem}

The proof of Theorem \ref{thm:classifyA_A} will occupy the rest of this section. Denote by $\cM_A$ the character space of $\cA_A$, that is, the topological space of contractive multiplicative and unital linear functionals on $\cA_A$, endowed with the weak-$*$ topology.

\begin{lemma}
The topology of $\cM_A$ depends on the rank $r(A^s)$ of the symmetric part $A^s$ of $A$:
\begin{enumerate}
	\item If $r(A^s) = 0$ then $\cM_A \cong B_2$, the unit ball in $\mb{C}^2$.
	\item If $r(A^s) = 1$ then $\cM_A \cong \mb{D}$, the unit disc in $\mb{C}$.
	\item If $r(A^s) = 2$ then $\cM_A$ is homeomorphic to two discs pasted together at the origin.
\end{enumerate}
\end{lemma}

\begin{proof}
We proceed similarly to the lines of \ref{sec:charA_q}.
Every character $\rho \in \cM_A$ is uniquely determined by $\lambda_1 = \rho(S^A_1)$ and $\lambda_2 = \rho(S^A_2)$, which lie in $B_2$. Conversely, every $(\lambda_1, \lambda_2) \in B_2$ that satisfies
\bes
\sum_{i,j}a_{ij}\lambda_i \lambda_j = 0
\ees
gives rise to a character $\rho$ by defining $\lambda_1 = \rho(S^A_1)$ and $\lambda_2 = \rho(S^A_2)$. Thus,
\bes
\cM_A \cong V_A := \left\{(\lambda_i, \lambda_j) \in B_2 : \sum_{i,j}a_{ij} \lambda_i \lambda_j = 0 \right\}.
\ees
Clearly, $V_A = V_{A^s}$. However, every symmetric $2 \times 2$ matrix is complex congruent to one of the following:
\bes
D_0 = \begin{pmatrix}
  0 & 0 \\
  0 & 0 \\
\end{pmatrix},
D_1 = \begin{pmatrix}
  1 & 0 \\
  0 & 0 \\
\end{pmatrix}
\,\, \textrm{or} \,\,
D_2 = \begin{pmatrix}
  1 & 0 \\
  0 & 1 \\
\end{pmatrix},
\ees
i.e., there exists a nonsingular matrix $T$ such that $A^s = T^t D_i T$, for $i = r(A^s)$. But then $V_{A^s} = T^{-1}V_{D_i} \cong V_{D_i}$, so it remains to verify that $V_{D_i}$ is homeomorphic to the spaces listed in the statement of the lemma.
\end{proof}

\begin{corollary}
If $r(A^s) \neq r(B^s)$ then $\cA_A \ncong \cA_B$.
\end{corollary}

We can use this corollary to break down the classification of the algebras $\cA_A$ to the classification of the algebras $\cA_A$ with fixed $r(A^s)$. The easiest case is $r(A^s) = 0$, because then $A$ is either the zero matrix or a multiple of $\begin{pmatrix}
  0 & 1 \\
  -1 & 0 \\
\end{pmatrix}$, and these two matrices give rise to non isomorphic algebras (these are the algebras generated by the full and symmetric shifts, respectively).

The next easiest case is $r(A^s) = 2$.
\begin{lemma}
If $A,B \in M_2(\mb{C})$ and $r(A^s) = r(B^s) = 2$, then $\cA_{A}$ is isometrically isomorphic to $\cA_B$ if and only if $X_A \cong X_B$, and this happens if and only if there exists $\lambda \in \mb{C}$ and a unitary $2 \times 2$ matrix $U$ such that $B = \lambda U^t A U$. Any isometric isomorphism between $\cA_A$ and $\cA_B$ arises as conjugation by the subproduct system isomorphism arising from $U$.
\end{lemma}
\begin{proof}
In light of Theorem \ref{thm:algebra_iso} and Proposition \ref{prop:classifyX_A}, it suffices to show that any isometric isomorphism $\varphi: \cA_A \rightarrow \cA_B$ sends the vacuum state to the vacuum state. But the vacuum state in $\cM_A$ and in $\cM_B$ corresponds to the point where the two discs are glued together. Since $\varphi$ induces a homeomorphism between $\cM_B$ and $\cM_A$, it must send the vacuum state to the vacuum state.
\end{proof}

\begin{remark}\emph{
In the previous section we have seen already that there is a continuum of non-(isometrically)-isomorphic algebras $\cA_{A}$ and subproduct systems $X_{A}$ with $r(A^s) = 2$, namely the algebras $\cA_q$. One can see that these algebras $\cA_A$ are not exhausted by the algebras $\cA_q$ of the previous section. For example, all the algebras $\cA_A$ with $A = \begin{pmatrix}
  1 & 0 \\
  0 & q \\
\end{pmatrix}$, with $q>0$, are non-isomorphic, and only for $q = 1$ is this algebra isomorphic to an $\cA_q$
(in this case $q = -1$).}
\end{remark}

We now come to the trickiest case, $r(A^s) = 1$.


\begin{lemma}\label{lem:symrank1}
If $A,B \in M_2(\mb{C})$ are two symmetric matrices of rank $1$, then there exists $\lambda \in \mb{C}$ and a unitary $2 \times 2$ matrix $U$ such that $B = \lambda U^t A U$, and consequently $X_A \cong X_B$ and $\cA_{A}$ is isometrically isomorphic to $\cA_B$.
\end{lemma}
\begin{proof}
We only have to prove the first assertion, and we may assume that $B = \begin{pmatrix}
  1 & 0 \\
  0 & 0 \\
\end{pmatrix}$. We may also assume that there is a unit vector $v = (v_1, v_2)^t$ such that $A = v v^t$. Now let
\bes
U = \begin{pmatrix}
  \overline{v_1} & \overline{v_2} \\
  \overline{v_2} & -\overline{v_1} \\
\end{pmatrix}.
\ees
Then
\bes
U^t A U = \begin{pmatrix}
  \overline{v_1} & \overline{v_2} \\
  \overline{v_2} & -\overline{v_1} \\
\end{pmatrix}^t v v^t \begin{pmatrix}
  \overline{v_1} & \overline{v_2} \\
  \overline{v_2} & -\overline{v_1} \\
\end{pmatrix}
= \begin{pmatrix}
  \overline{v_1} & \overline{v_2} \\
  \overline{v_2} & -\overline{v_1} \\
\end{pmatrix}^t \begin{pmatrix}
  v_1 & 0 \\
  v_2 & 0 \\
\end{pmatrix} = \begin{pmatrix}
  1 & 0 \\
  0 & 0 \\
\end{pmatrix}.
\ees
\end{proof}


Below we will also need the following lemma.
\begin{lemma}\label{lem:notsymrank1}
Let $A$ be a $2\times2$ matrix for which $r(A^s) = 1$. Then there exists one and only one $q\geq 0$ for which there
is a $\lambda \in \mb{C}$ and a unitary $U$ such that
\bes
\begin{pmatrix}
1 & q \\ -q & 0
\end{pmatrix} = \lambda U^t A U.
\ees
Furthermore, if $A$ is non-symmetric then $A$ is congruent to the matrix
\bes
\begin{pmatrix}
1 & 1 \\ -1 & 0
\end{pmatrix}.
\ees
\end{lemma}
\begin{proof}
Direct verification, using Lemma \ref{lem:symrank1} and the fact that congruations preserve, up to a scalar,
the anti-symmetric part.
\end{proof}

Let us write $A_q$ for the matrix
\bes
A_q = \begin{pmatrix}
1 & q \\ -q & 0
\end{pmatrix}.
\ees
By the above lemma, we may restrict attention only to the algebras $\cA_{A_q}$ with $q \geq 0$.

Recall that the character space $\cM_{A_q}$ of $\cA_{A_q}$ is identified with the closed unit disc $\overline{\mb{D}}$
by
\bes
\cM_{A_q} \ni \rho \longleftrightarrow \rho(S^{A_q}_2) \in \overline{\mb{D}} .
\ees
We write $\rho_z$ for the character that sends $S^{A_q}_2$ to $z \in \overline{\mb{D}}$.
This identifies the vacuum vector $\rho_0$ with the point $0$.
Recall also that if $\varphi : \cA_{A_q} \rightarrow \cA_{A_r}$ is an isometric isomorphism, then it induces a
homeomorphism $\varphi_* : \cM_{A_r} \rightarrow \cM_{A_q}$ given by $\varphi_* \rho = \rho \circ \varphi$.
We write $F_\varphi$ for the homeomorphism $\overline{\mb{D}} \rightarrow \overline{\mb{D}}$ induced by $\varphi$,
that is, $F_\varphi$ is the unique self map of $\overline{\mb{D}}$ that satisfies
\bes
\varphi_* \rho_z = \rho_{F_\varphi(z)} \,\, , \,\, z \in \overline{\mb{D}}.
\ees

Let us introduce the notation
\bes
\cO(0;q,r) = \{F_\varphi(0) \big| \varphi : \cA_{A_q} \rightarrow \cA_{A_r} \textrm{ is an isometric isomorphism}\},
\ees
and
\bes
\cO(0;q) = \cO(0;q,q).
\ees
\begin{lemma}\label{lem:nozero}
Let $q,r \geq 0$. If $q \neq r$ then $0$ does not lie in $\cO(0;q,r)$.
\end{lemma}
\begin{proof}
Assume that $0 \in \cO(0;q,r)$. Then there is some isometric isomorphism $\varphi:\cA_{A_q} \rightarrow \cA_{A_q}$ that preserves the character $\rho_0$. It follows from Theorem \ref{thm:algebra_iso} and Proposition \ref{prop:classifyX_A} that, for some unitary $2 \times 2$ matrix $U$ and some $\lambda \in \mb{C}$, $A_q = \lambda U^t A_r U$. But, as noted in Lemma \ref{lem:notsymrank1}, this is impossible if $r \neq q$.
\end{proof}

\begin{lemma}\label{lem:rotinv}
The sets $\cO(0;q,r)$ are invariant under rotations around $0$.
\end{lemma}
\begin{proof}
For $\lambda$ with $|\lambda| = 1$, write $\varphi_\lambda$ for the isometric isomorphism mapping $S_i^{A_q}$ to
$\lambda S_i^{A_q}$ ($i = 1,2$). For $b = F_\varphi(0) \in \cO(0;q,r)$, consider $\varphi \circ \varphi_\lambda$.
We have $\rho_0((\varphi \circ \varphi_\lambda)(S_2^{A_q})) = \rho_0(\varphi(\lambda S_2^{A_q})) = \lambda \rho_0(\varphi(S_2^{A_q})) = \lambda b$.
Thus $\lambda b \in \cO(0;q,r)$.
\end{proof}

\begin{lemma}
Let $q,r \geq 0$. If $q \neq r$ then $\cA_{A_q}$ is not isometrically isomorphic to $\cA_{A_r}$.
\end{lemma}

\begin{proof}
Assume that $\varphi :\cA_{A_q} \rightarrow \cA_{A_r}$ is an isometric isomorphism. We have
$\rho_0 \circ \varphi = \rho_b$, with $b = F_\varphi(0)$, and $F_\varphi$ is a homeomorphism of $\overline{\mb{D}}$
onto itself.

By definition, $b \in \cO(0;q,r)$. By Lemma \ref{lem:nozero}, $b \neq 0$. Denote $C := \{z:|z|=|b|\}$. By Lemma
\ref{lem:rotinv}, $C \subseteq \cO(0;q,r)$. Consider $C' := F_\varphi^{-1}(C)$. We have that $C' \subseteq \cO(0;r)$.
$C'$ is a simply connected closed path in $\mb{D}$ that goes through the origin. By Lemma \ref{lem:rotinv}, the
interior of $C'$, $\textrm{int}(C')$, is in $\cO(0;r)$. But then $F_\varphi(\textrm{int}(C'))$
is the interior of $C$, and it is in $\cO(0;q,r)$. But then $0 \in \cO(0;q,r)$, contradicting Lemma \ref{lem:nozero}.
\end{proof}

That concludes the proof of Theorem \ref{thm:classifyA_A}.

\chapter{The representation theory of Matsumoto's subshift C$^*$-algebras}\label{chap:subshift}

In \cite{Ma} Kengo Matsumoto introduced a class of C$^*$-algebras that arise from symbolic dynamical systems called ``subshifts" (we note that in the later paper \cite{CM04} Carlsen and Matsumoto suggest another way of associating a C$^*$-algebra with a subshift. Here we are discussing only the algebras originally introduced in \cite{Ma}). These \emph{subshift algebras}, as we shall call them, are strict generalizations of Cuntz-Krieger algebras and have been extensively studied by Matsumoto, T. M. Carlsen and others. For example, the following have been studied: criteria for simplicity and pure-infiniteness; conditions on the underlying dynamical systems for subshift algebras to be isomorphic; the automorphisms of the subshift algebras; K-theory of the subshift algebras; and much more. In this chapter we will use the framework constructed in the previous chapters to give a complete description of all representations of a subshift algebra when the subshift is of finite type.


\section{Subshifts and the corresponding subproduct systems and C$^*$-algebras}

Our references for subshifts are \cite{Ma} and \cite[Chapter 3]{BrinStuck02}.

Let $\cI = \{1, 2, \ldots, d\}$ be a fixed finite set. $\cI^{\mb{Z}}$ is the space of all two-sided infinite sequences, endowed with the product topology. The \emph{left shift} (or simply \emph{the shift}) on $\cI^{\mb{Z}}$ is the homeomorphism $\sigma : \cI^{\mb{Z}} \rightarrow \cI^{\mb{Z}}$ given by $(\sigma(x))_k = x_{k+1}$. Let $\Lambda$ be a shift invariant closed subset of $\cI^{\mb{Z}}$. By this we mean $\sigma(\Lambda) = \Lambda$. The topological dynamical system $(\Lambda, \sigma\big|_\Lambda)$ is called a \emph{subshift}. Sometimes $\Lambda$ is also referred to as the subshift.

If $W$ is a set of words in $1,2,\ldots,d$, one can define a subshift by forbidding the words in $W$ as follows:
\bes
\Lambda_W = \{x \in \cI^{\mb{Z}} : \textrm{no word in $W$ occurs as a block in $x$}\}.
\ees
Conversely, every subshift arises this way: i.e., for every subshift $\Lambda$ there exists a collection of words $W$, called \emph{the set of forbidden words}, such that $\Lambda = \Lambda_W$. In this context, if $W$ can be chosen finite then $\Lambda = \Lambda_W$ is called \emph{a subshift of finite type}, or \emph{SFT} for short. By replacing $\cI$ if needed, we may always assume that $W$ has no words of length one. If $W$ can be chosen such that the longest word in $W$ has length $k+1$ then $\Lambda$ is called a \emph{$k$-step SFT}. A $1$-step SFT is also called a \emph{topological Markov chain}. A basic result is that every SFT is isomorphic to a topological Markov chain (\cite[Proposition 3.2.1]{BrinStuck02}).

For a fixed subshift $(\Lambda, \sigma\big|_\Lambda)$, we set
\bes
\Lambda^k = \{\alpha : \textrm{$\alpha$ is a word with length $k$ occurring in some $x \in \Lambda$}\},
\ees
and $\Lambda_l = \cup_{k=0}^l \Lambda^k$, $\Lambda^* = \cup_{k=0}^\infty \Lambda^k$. With the subshift $(\Lambda, \sigma\big|_\Lambda)$ we associate a subproduct system $X_\Lambda$ as follows. Let $\{e_i\}_{i \in \cI}$ be an orthonormal basis of a Hilbert space $E$. We define
\bes
X_\Lambda (0) = \mb{C},
\ees
and for $n \geq 1$ we define
\bes
X_\Lambda (n) = \textrm{span}\{e_\alpha : \alpha \in \Lambda^n\}.
\ees
We define a product $U_{m,n}: X_\Lambda (m) \otimes X_\Lambda(n) \rightarrow X_\Lambda(m+n)$
by
\bes
U_{m,n}(e_\alpha \otimes e_\beta) =
\begin{cases}
e_{\alpha \beta} , & \textrm{if } \alpha\beta \in \Lambda^{m+n} \cr
0  , & \textrm{else.}
\end{cases}
\ees
Since $\Lambda^{m+n} \subseteq \Lambda^m \cdot \Lambda^n$, $X_\Lambda$ is a standard subproduct system.

\begin{definition}
The C$^*$-algebra associated with a subshift $(\Lambda, \sigma\big|_\Lambda)$ is defined as the quotient algebra
\bes
\cO_\Lambda := \cO_{X_\Lambda} = \cT_{X_\Lambda} / \cK(\mathfrak{F}_{X_\Lambda}).
\ees
\end{definition}
\begin{remark}\emph{
Just to prevent confusion: In \cite{Ma}, $\cO_\Lambda$ was defined as the quotient by the compacts of the C$^*$-algebra generated by the ``creation operators" (that is, the $X$-shift) on $\mathfrak{F}_X$, without using the language of subproduct systems.}
\end{remark}

\section{Subproduct systems that come from subshifts}

\begin{proposition}
Let $X$ be a standard subproduct system such that there is an orthonormal basis $\{e_i\}_{i \in \cI}$ of $X(1)$, with
$\cI$ finite, such that
 \begin{enumerate}
 \item Every $X(n)$, $n\geq1$, is spanned by vectors of the form $e_\alpha$ with $|\alpha| = n$.
 \item For all $m,n \in \mb{N}$, $|\alpha| = n$ and $e_\alpha \in X(n)$, implies that there is some $\beta,\gamma \in \cI^m$
 such that $e_\beta \otimes e_\alpha$ and $e_\alpha \otimes e_{\gamma}$ are in $X(m+n)$.
 \end{enumerate}
 Then there is a shift invariant closed subset $\Lambda$ of $\cI^{\mb{Z}}$ such that
 $X = X_\Lambda$. $X$ is the maximal standard subproduct
 system with prescribed fibers $X(1), X(2), \ldots, X(k+1)$ if and only if $\Lambda$ is $k$-step SFT.
\end{proposition}
\begin{proof}
For all $k \in \mb{N}$, define
\bes
\Lambda^{(k)} = \{\alpha \in \cI^k : e_\alpha \in X(k)\}.
\ees
For all $m \in \mb{Z}, k \in \mb{N}$, define the closed sets
\bes
A_{m,k} = \{x \in \cI^{\mb{Z}} : (x_m, x_{m+1}, \ldots, x_{m+k-1}) \in \Lambda^{(k)}\}.
\ees
Condition (2) implies that $X(k)$ always contains a nonzero vector of the form $e_\alpha$, $|\alpha| = k$. That implies that the family $\{A_{m,k} \}_{m,k}$ has the finite intersection property. Indeed,
\bes
A_{m_1,k_1} \cap A_{m_2,k_2} \supseteq A_{M,K} \neq \emptyset,
\ees
where $M = \min\{m_1, m_2 \}$, $K = \max\{m_2 + k_2, m_1 + k_1\} - M$. By compactness of $\cI^{\mb{Z}}$ we conclude that the closed set
\bes
\Lambda := \bigcap_{m,k}A_{m,k}
\ees
is non-empty. $\Lambda$ is invariant under the left and the right shifts, so $\sigma(\Lambda) = \Lambda$, so $(\Lambda, \sigma\big|_\Lambda)$ is a subshift. By condition (2), $\Lambda^k = \Lambda^{(k)}$. Condition (1) together with the definition of $X_\Lambda$  now imply that $X = X_\Lambda$.

The final assertion follows from the following facts, together with $X = X_\Lambda$. Fact number one:
\bes
E^{\otimes n} \ominus X_\Lambda(n) = \textrm{span}\{e_\alpha : \textrm{$\alpha$ is a forbidden word of length $n$} \}.
\ees
Fact number two: $X$ is the maximal standard subproduct system with prescribed fibers $X(1), \ldots, X(k+1)$ if and only if for every $n > k+1$,
\bes
X(n) = \bigcap_{i+j=n} X(i) \otimes X(j),
\ees
or in other words, if and only if
\begin{align*}
E^{\otimes n} \ominus X(n) &= \bigvee_{i+j=n} \left(E^{\otimes n} \ominus (X(i) \otimes X(j))\right) \\
&= \bigvee_{i+j=n} \left(E^{\otimes i}\otimes (E^{\otimes j} \ominus X(j)) + (E^{\otimes i} \ominus X(i)) \otimes E^{\otimes j}\right).
\end{align*}
Fact number three: $\Lambda$ is a $k$-step SFT if and only if for every $n > k+1$,
\begin{align*}
& \{\textrm{forbidden words of length $n$}\} = \\
& \bigcup_{i+j=n}\left(\cI^i \cdot \{\textrm{forbidden words of length $j$}\} \cup \{\textrm{forbidden words of length $i$}\}\cdot \cI^j \right).
\end{align*}
These facts assemble together to complete the proof.
\end{proof}

Not every subproduct system is isomorphic to one that comes from a subshift. Indeed, in the symmetric subproduct system $SSP$ (see Example \ref{expl:symm}) for any basis $\{e_i\}_{i \in \cI}$ of $X(1)$, the product $e_i \otimes e_j$ for $i \neq j$ is never in $X(2)$, and thus the images $f_i$ and $f_j$ of $e_i$ and $e_j$ in any isomorphic subproduct system $X$ can never be such that $f_i \otimes f_j$ is mapped isometrically to $U^X_{1,1}(f_i \otimes f_j)$. Thus if $SSP$ is isomorphic to $X_\Lambda$ for some subshift $\Lambda$, then $\Lambda$ must be the subshift containing only constant sequences. But such $X_\Lambda$ is clearly not isomorphic to $SSP$.

As another example, the subproduct system $X(0) = \mb{C}$, $X(1) = \mb{C}^2$, and $X(n) = 0$ for $n > 1$, cannot be of the form $X_\Lambda$ for any $\Lambda \subseteq \cI^{\mb{Z}}$.

\section{The representation theory of the C$^*$-algebra associated with a subshift of finite type}

Let $\Lambda$ be a fixed subshift in $\cI^{\mb{Z}}$ (with $\cI = \{1,2, \ldots, d\}$), and let $X = X_\Lambda$ be the associated subproduct system. We will denote the $X$-shift by $S$ (instead of $S^X$) to make some formulas more readable. Let $Z_i$ be the image of $S_i$ in the quotient $\cO_\Lambda$. We define for $i \in \cI$, $k \in \mb{N}$ the sets
\bes
E_i^k = \{\alpha \in \Lambda^k : i \alpha \in \Lambda^* \}.
\ees

\begin{lemma}\label{lem:E_i^k}
If $\Lambda$ is a $k$-step SFT, then for all $i\in\cI$,
\bes
\{\gamma \in \Lambda^* : |\gamma| \geq k,  i \gamma \in \Lambda^* \} = \{\alpha \beta \in \Lambda^*   : \alpha \in E_i^k, \beta \in \Lambda^* \}.
\ees
\end{lemma}
\begin{proof}
Assume that $\gamma \in \Lambda^*$ is such that $|\gamma| \geq k$ and $i \gamma \in \Lambda^*$. Defining $\alpha = \gamma_1 \cdots \gamma_k$ and $\beta = \gamma_{k+1}\cdots \gamma_{k+l}$, we have that $\gamma = \alpha \beta$ where $\alpha \in E_i^k$ and $\beta \in \Lambda^*$.

Conversely, if $\gamma = \alpha \beta \in \Lambda^*$ where $\alpha \in E_i^k$ and $\beta \in \Lambda^*$, then $i \gamma$ must be in $\Lambda^*$. Indeed, if not, then $i \gamma$ must contain a forbidden word. But $\gamma \in \Lambda^*$, thus the forbidden word must be in $i \alpha$ (since $\Lambda$ is a $k$-step SFT). But that is impossible because $\alpha \in E_i^k$.
\end{proof}

\begin{lemma}\label{lem:E=T}
If $\Lambda$ is a $k$-step SFT then for all $i,j \in \cI$, $i \neq j$,
\bes
S_i^*S_j = 0 ,
\ees
and
\be\label{eq:SXimodK}
S_i^*S_i = \sum_{\alpha \in E_i^k} \underline{S}^\alpha \underline{S}^{\alpha*}  \,\, \mod \cK_X.
\ee
Consequently, $\cE_X = \cT_X$.
\end{lemma}
%

\begin{proof}
Since the $S_i$ are partial isometries with orthogonal ranges, we have $S_i^* S_j = 0$ for all $i \neq j$. Since $\cK_X \subseteq \cE_X \subseteq \cT_X$ (Proposition \ref{prop:KinE}), $\cE_X = \cT_X$ will be established once we prove (\ref{eq:SXimodK}).

$S_i^*S_i$ is the projection onto the initial space of $S_i$. Call this space $G$. We have
\bes
G = \textrm{span}\{e_\alpha : \alpha \in \Lambda^* \textrm{ such that } i \alpha \in \Lambda^* \}.
\ees
The space
\bes
G' = \textrm{span}\{e_\alpha : \alpha \in \Lambda^* \textrm{ such that } i \alpha \in \Lambda^* \textrm{ and } |\alpha| \geq k\}
\ees
has finite codimension in $G$. But by Lemma \ref{lem:E_i^k},
\bes
G' = \{e_{\alpha \beta} : \alpha \beta \in \Lambda^*, \alpha \in E_i^k\},
\ees
that is, $G'$ is spanned by $e_\gamma$ where $\gamma$ runs through all legal words beginning with some $\alpha \in E_i^k$. Thus, $G'$ is the range of the projection $\sum_{\alpha \in E_i^k} \underline{S}^\alpha \underline{S}^{\alpha^*}$. Since $G'$ has finite codimension in $G$, we have (\ref{eq:SXimodK}).
\end{proof}

\begin{proposition}
For every subshift $\Lambda$, the $d$-tuple $\underline{Z} = (Z_1, \ldots, Z_d)$ satisfies the following relations:
\be\label{eq:Z1}
p(\underline{Z}) = 0 \,\, , \textrm{ for all } p \in I^{X_\Lambda} ,
\ee
\be\label{eq:Z2}
Z_i^* Z_j = 0 \,\,  , \textrm{ for all } i,j \in \cI \, , i \neq j ,
\ee
and
\be\label{eq:Z3}
\sum_{i=1}^d Z_i Z_i^* = 1.
\ee
In particular, $Z_i$ is a partial isometry for all $i \in \cI$.
If $\Lambda$ is a $k$-step SFT, the $\underline{Z}$ also satisfies
\be\label{eq:Z4}
Z_i^*Z_i = \sum_{\alpha \in E_i^k} \underline{Z}^\alpha \underline{Z}^{\alpha *} \,\, , \textrm{ for all } i \in \cI.
\ee
\end{proposition}

\begin{proof}
The quotient map $\cT_X \rightarrow \cO_\Lambda$ is a $*$-homomorphism, so (\ref{eq:Z1}) follows from Theorem \ref{thm:repZ(I)}. (\ref{eq:Z2}) and (\ref{eq:Z4}) follow from the previous lemma, and (\ref{eq:Z3}) follows from equation (\ref{eq:PC}).
\end{proof}

\begin{theorem}\label{thm:rep_subshift}
Let $\Lambda$ be a $k$-step SFT. Every unital representation $\pi : \cO_\Lambda \rightarrow B(H)$ is determined by a row-contraction $\underline{T} = (T_1, \ldots, T_d)$ satisfying relations (\ref{eq:Z1})-(\ref{eq:Z4}) such that $\pi(Z_i) = T_i$ for all $i \in \cI$. Conversely, every row contraction in $B(H)^d$ satisfying the relations (\ref{eq:Z1})-(\ref{eq:Z4}) gives rise to a unital representation $\pi : \cO_\Lambda \rightarrow B(H)$ such $\pi(Z_i) = T_i$ for all $i \in \cI$.
\end{theorem}
\begin{proof}
It is the second assertion that is non-trivial, and we will try to convince that it is true. By Theorem \ref{thm:CP1}, there is unital completely positive map
\bes
\Psi : \cE_X \rightarrow B(H)
\ees
sending $\underline{S}^\alpha \underline{S}^{\beta*}$ to $\underline{T}^\alpha \underline{T}^{\beta*}$. Since enough of the rank one operators on $\mathfrak{F}_X$ arise as $\underline{S}^\alpha (I - \sum_{i = 1}^d S_i S_i^*) \underline{S}^{\beta*}$ (see equation (\ref{eq:PC})), and because $\underline{T}$ satisfies (\ref{eq:Z3}), we must have that $\Psi(K) = 0$ for every $K \in \cK(\mathfrak{F}_X)$. By Lemma \ref{lem:E=T}, $\cE_X = \cT_X$, and it follows that $\Psi$ induces a positive and unital (hence contractive) mapping
\bes
\pi : \cO_\Lambda \rightarrow B(H)
\ees
that sends $\underline{Z}^\alpha \underline{Z}^{\beta*}$ to $\underline{T}^\alpha \underline{T}^{\beta*}$. Roughly speaking: $\pi$ must be multiplicative because $\underline{Z}$ and $\underline{T}$ satisfy the same relations. In more detail: every product $(\underline{Z}^\alpha \underline{Z}^{\beta*})(\underline{Z}^{\alpha'} \underline{Z}^{\beta'*})$ may be written, using the relations (\ref{eq:Z1})-(\ref{eq:Z4}) as some sum $\sum_{\gamma,\delta} \underline{Z}^\gamma \underline{Z}^{\delta*}$. The mapping $\pi$ then takes this sum to $\sum_{\gamma,\delta} \underline{T}^\gamma \underline{T}^{\delta*}$, and this can be rewritten (using the same relations) as \bes
(\underline{T}^\alpha \underline{T}^{\beta*})(\underline{T}^{\alpha'} \underline{T}^{\beta'*}) =
\pi(\underline{Z}^\alpha \underline{Z}^{\beta*})\pi(\underline{Z}^{\alpha'} \underline{Z}^{\beta'*}).
\ees
This shows that
\bes
\pi\left((\underline{Z}^\alpha \underline{Z}^{\beta*})(\underline{Z}^{\alpha'} \underline{Z}^{\beta'*})\right) = \pi(\underline{Z}^\alpha \underline{Z}^{\beta*})\pi(\underline{Z}^{\alpha'} \underline{Z}^{\beta'*}),
\ees
and since the elements of the form $\underline{Z}^\alpha \underline{Z}^{\beta*}$ span $\cO_\Lambda$, and since $\pi$ is a positive linear map, it follows that $\pi$ is in fact a $*$-representation.
\end{proof}

\begin{remark}
\emph{Note that for $\Lambda = \cI^{\mb{Z}}$ we recover the representation theory of the Cuntz algebra.}
\end{remark}

\end{document}